\documentclass[10]{article}
\usepackage{hyperref}
\usepackage{amsfonts,amssymb,amsmath,amsthm,cite}
\usepackage[toc,page]{appendix}
\usepackage{nicefrac}

\usepackage[applemac]{inputenc}
\usepackage[T1]{fontenc}
\usepackage{amssymb, euscript}
\usepackage[matrix,arrow,curve]{xy}
\usepackage{graphicx}
\usepackage{tabularx}
\usepackage{float}
\usepackage{tikz}
\usepackage{slashed}
\usepackage{mathrsfs}
\usepackage{multirow}
\usepackage{rotating}

\usetikzlibrary{matrix}
\usetikzlibrary{cd}

\usetikzlibrary{decorations.markings}
		\tikzset{double line with arrow/.style args={#1,#2}{decorate,decoration={markings,%
			mark=at position 0 with {\coordinate (ta-base-1) at (0,1pt);
				\coordinate (ta-base-2) at (0,-1pt);},
			mark=at position 1 with {\draw[#1] (ta-base-1) -- (0,1pt);
				\draw[#2] (ta-base-2) -- (0,-1pt);
}}}}

\usepackage{siunitx}

\usepackage{lmodern}
\usepackage[T1]{fontenc}
\usepackage[babel=true]{microtype}

\newcommand{\Sh}{\ensuremath{\mathrm{Sh}}}

\newcommand{\cat}[1]{\ensuremath{\mathbf{#1}}}
\newcommand{\CA}{{\mathcal A}}

\newcommand{\Sets}{\mbox{\textbf{Sets}}}

\newcommand{\alg}[1]{\ensuremath{#1}}

\newcommand{\asstopos}{\ensuremath{\mathcal{T}}}

\newcommand{\functor}[1]{\ensuremath{\underline{#1}}}
\newcommand{\ca}{C*-algebra}

\newcommand{\uA}{\underline{A}}

\newcommand{\p}{\mathfrak{P}}

	\newcommand{\bean}{\begin{eqnarray*}}
	\newcommand{\eean}{\end{eqnarray*}}
\newcommand{\s}{\mathsf{s}}

\DeclareMathSymbol{\onto}{\mathrel}{AMSa}{"10}
\renewcommand{\hbar}{{\mathchar'26\mkern-9muh}}

\usepackage{amsfonts,cite}
\usepackage{graphicx}
\usepackage[applemac]{inputenc}

\usepackage[sc]{mathpazo}
\usepackage{environ}

\DeclareFontFamily{T1}{pzc}{}
\DeclareFontShape{T1}{pzc}{m}{it}{1.8 <-> pzcmi8t}{}
\DeclareMathAlphabet{\mathpzc}{T1}{pzc}{m}{it}

\theoremstyle{plain}
\newtheorem{prop}{Proposition}[section]
\newtheorem{prdf}[prop]{Proposition and Definition}
\newtheorem{lem}[prop]{Lemma}
\newtheorem{cor}[prop]{Corollary}
\newtheorem{thm}[prop]{Theorem}
\newtheorem{theorem}[prop]{Theorem}
\newtheorem{lemma}[prop]{Lemma}
\newtheorem{proposition}[prop]{Proposition}
\newtheorem{corollary}[prop]{Corollary}
\newtheorem{statement}[prop]{Statement}

\theoremstyle{definition}
\newtheorem{defn}[prop]{Definition}
\newtheorem{empt}[prop]{}
\newtheorem{rem}[prop]{Remark}

\theoremstyle{definition}
\newtheorem{notation}[prop]{Notation}
\newtheorem{definition}[prop]{Definition}
\newtheorem{example}[prop]{Example}
\newtheorem{exercise}[prop]{Exercise}

\newtheorem{conjecture}[prop]{Conjecture}

\newtheorem{remark}[prop]{Remark}

\numberwithin{equation}{section}

\newcommand{\vertiii}[1]{{\left\vert\kern-0.25ex\left\vert\kern-0.25ex\left\vert #1
		\right\vert\kern-0.25ex\right\vert\kern-0.25ex\right\vert}}
\newcommand{\Ga}{\Gamma}  
\newcommand{\coker}{\mathrm{coker}}                   
\newcommand{\Coo}{C^\infty}                  

\usepackage[sc]{mathpazo}
\newbox\ncintdbox \newbox\ncinttbox 
\setbox0=\hbox{$-$} \setbox2=\hbox{$\displaystyle\int$}
\setbox\ncintdbox=\hbox{\rlap{\hbox
		to \wd2{\hskip-.125em \box2\relax\hfil}}\box0\kern.1em}
\setbox0=\hbox{$\vcenter{\hrule width 4pt}$}
\setbox2=\hbox{$\textstyle\int$} \setbox\ncinttbox=\hbox{\rlap{\hbox
		to \wd2{\hskip-.175em \box2\relax\hfil}}\box0\kern.1em}

\newcommand{\rep}{\mathfrak{rep}}
\newcommand{\lift}{\mathfrak{lift}}
\newcommand{\desc}{\mathfrak{desc}}

\newcommand{\Id}{\mathrm{Id}}                
\renewcommand{\d}{\mathrm{d}}       

\newcommand{\A}{\mathcal{A}}                 

\renewcommand{\a}{\alpha}                    
\newcommand{\B}{\mathcal{B}}                 
\newcommand{\E}{\mathcal{E}}                 

\newcommand{\C}{\mathbb{C}}                  
\newcommand{\F}{\mathcal{F}}                 
\newcommand{\G}{\mathcal{G}}                 
\newcommand{\D}{\mathcal{D}}                 
\renewcommand{\H}{\mathcal{H}}               
\newcommand{\hookto}{\hookrightarrow}        
\newcommand{\into}{\hookrightarrow}        
\newcommand{\K}{\mathcal{K}}                 
\renewcommand{\L}{\mathcal{L}}               
\newcommand{\La}{\Lambda}                    
\newcommand{\la}{\lambda}                    

	\newcommand{\N}{\mathbb{N}}                  
	\newcommand{\om}{\omega}                     
	\newcommand{\ox}{\otimes}                    
	\newcommand{\eps}{\varepsilon}                    
	\newcommand{\Q}{\mathbb{Q}}                  
	\newcommand{\R}{\mathbb{R}}                  

	\renewcommand{\SS}{\mathcal{S}}              
	\DeclareMathOperator{\supp}{\mathfrak{supp}}            
	\newcommand{\T}{\mathbb{T}}                  
	\renewcommand{\th}{\theta}                   
	\newcommand{\thalf}{\tfrac{1}{2}}            
	\DeclareMathOperator{\tr}{tr}                
	\newcommand{\Z}{\mathbb{Z}}                  
	\newcommand{\8}{\bullet}                     
	
	\newcommand{\sC}{\mathcal{C}}       
	\newcommand{\sF}{\mathcal{F}}       
	\newcommand{\sH}{\mathcal{H}}       
	\newcommand{\sR}{\mathcal{R}}       
	\newcommand{\sS}{\mathcal{S}}       
	\newcommand{\sU}{\mathcal{U}}       
	\newcommand{\sV}{\mathcal{V}}       
	\newcommand{\sX}{\mathcal{X}}       
	\newcommand{\sY}{\mathcal{Y}}       
	\newcommand{\sZ}{\mathcal{Z}}       
	
	\newcommand{\Om}{\Omega}       

	\newcommand{\bydef}{\stackrel{\mathrm{def}}{=}}          

	\newcommand{\al}{\alpha}          
	\newcommand{\bt}{\beta}           
	\newcommand{\Dl}{\Delta}          
	\newcommand{\dl}{\delta}          
	\newcommand{\ga}{\gamma}          
	\newcommand{\sg}{\sigma}          
	\newcommand{\Sg}{\Sigma}          
	\newcommand{\Th}{\Theta}          
	\renewcommand{\th}{\theta}        
	
	\DeclareMathOperator{\rank}{rank}   

	\DeclareMathOperator{\Ind}{Ind}
	\newcommand{\blank}{-}

	\def\<#1|#2>{\langle#1\stroke#2\rangle} 
	\def\?#1|#2?{\{#1\stroke#2\}}        
	
	\renewcommand{\Bar}[1]{\overline{#1}} 

	\def\<#1,#2>{\langle#1,#2\rangle}            
	\def\ee_#1{e_{{\scriptscriptstyle#1}}}       
	\def\wick:#1:{\mathopen:#1\mathclose:}       

	\newbox\ncintdbox \newbox\ncinttbox 
	\setbox0=\hbox{$-$}
	\setbox2=\hbox{$\displaystyle\int$}
	\setbox\ncintdbox=\hbox{\rlap{\hbox
			to \wd2{\box2\relax\hfil}}\box0\kern.1em}
	\setbox0=\hbox{$\vcenter{\hrule width 4pt}$}
	\setbox2=\hbox{$\textstyle\int$}
	\setbox\ncinttbox=\hbox{\rlap{\hbox
			to \wd2{\hskip-.05em\box2\relax\hfil}}\box0\kern.1em}
	
	\newcommand{\stroke}{\mathbin|}   
	
	\newcommand{\End}{\mathrm{End}}       
	\newcommand{\Hom}{\mathrm{Hom}}       
	\newcommand{\Aut}{\mathrm{Aut}}       
	\newcommand{\im}{\mathrm{im}}       

	\newcommand{\ulS}{\functor{\Sigma}}

	\makeatletter
	\providecommand*{\xmapstofill@}{%
		\arrowfill@{\mapstochar\relbar}\relbar\rightarrow
	}
	\providecommand*{\xmapsto}[2][]{%
		\ext@arrow 0395\xmapstofill@{#1}{#2}%
	}
	
	\providecommand*{\xhookrightfill@}{%
		\arrowfill@{\lhook\joinrel\relbar}\relbar\rightarrow
	}
	\providecommand*{\xhookrightarrow}[2][]{%
		\ext@arrow 0395\xhookrightfill@{#1}{#2}%
	}
	\makeatother

	\newcommand{\bs}{\begin{split}}
		\newcommand{\es}{\end{split}}
	\newcommand{\be}{\begin{equation}}
		\renewcommand{\ee}{\end{equation}}
	\newcommand{\bea}{\begin{eqnarray}}
		\newcommand{\eea}{\end{eqnarray}}
	\newcommand{\brray}{\begin{array}}
		\newcommand{\erray}{\end{array}}

	\usetikzlibrary{calc,trees,positioning,arrows,chains,shapes.geometric,%
		decorations.pathreplacing,decorations.pathmorphing,shapes,%
		matrix,shapes.symbols}
	
	\usetikzlibrary{trees,positioning,shapes,shadows,arrows}

	\tikzset{
		basic/.style  = {draw, text width=2cm, drop shadow, font=\sffamily,     rectangle},
		root/.style   = {basic, rounded corners=2pt, thin, align=center,
			fill=green!30},
		level 2/.style = {basic, rounded corners=6pt, thin,align=center,     fill=green!60,
			text width=8em},
		level 3/.style = {basic, thin, align=left, fill=pink!60, text width=6.5em}
	}

	\title{Algebraic topology of $C^*$-algebras}
	
	\author
	{\textbf{Petr R. Ivankov*}\\
		e-mail: * monster.ivankov@gmail.com \\
	}

\begin{document}
\maketitle  
\pagestyle{plain}

\begin{abstract}
	\noindent
\paragraph{}	Any $C^*$-algebra can be regarded as a generalization of locally compact, Hausdorff  topological space $\mathcal X$. From the commutative  Gelfand-Na\u{\i}mark theorem it follows that the spectrum of any commutative $C^*$-algebra is a locally  compact, Hausdorff  space which have the exact information about the $C^*$-algebra.  Here we consider a Gelfand spaces of $C^*$-algebras which can be regarded as a generalization of the spectrum. In case of commutative $C^*$-algebras the Gelfand space coincides with the spectrum. Generally Gelfand spaces are not Hausdorff and provide more detailed information of noncommutative  $C^*$-algebras. Usage of Gelfand spaces of $C^*$-algebra enables us to define some $C^*$-algebraic analogs of several notions of the classical algebraic topology, e.g. fundamental group, cohomology,  multiplicative structure of $K$-theory and Adams operations. Application of the  Gelfand spaces yields a generalization of Dixmier-Douady theory. Also Gelfand space provides interoperability between topological invariants and quantum logic.
\end{abstract}

\section{Spaces of ultrafilters}
\paragraph{} Here we define topological spaces related to meet-semilattice (cf. Definition \ref{lattice_defn}).
\begin{definition}\label{lattice_s_defn}
	If $\mathscr L$ is  the meet-semilattice (cf. Definition \ref{lattice_defn}) then denote by   $\mathfrak{Filters}\left( \mathscr L\right)$  (resp.  $\mathfrak{Ultrafilters}\left( \mathscr L\right)$) the  set of filters  (resp. ultrafilters) (cf. Definitions \ref{filter_defn} and \ref{ultra_filter_defn}). \end{definition}

\begin{lemma}\label{topology_l_lem} 
	Under the hypothesis of the Definition \ref{lattice_s_defn}	 the set 
		\be\label{topology_l_eqn}
		 \left\{\mathfrak{Ultrafilters}\left( \mathscr L\right)_L \bydef\left( \left\{ \left.\xi \in \mathfrak{Ultrafilters}\left( \mathscr L\right)\right| L \in \xi\right\}\subset \mathfrak{Ultrafilters}\left( \mathscr L\right)\right)  \right\}_{L \in \mathscr L}
		\ee
		is a basis of topology on the set  $\mathfrak{Ultrafilters}\left( \mathscr L\right)$).
\end{lemma}

\begin{proof}
	One needs check conditions (a) and (b) of the Definition \ref{top_base_defn},\\
	(a) If $\xi \bydef \left\{ L_\la \right\}_{\la \in \La}\in  \mathfrak{Ultrafilters}\left(\mathscr L \right)$ then 
	$$
	\forall \la \in \La\quad  \xi \in 	\mathfrak{Ultrafilters}\left(\mathscr L\right)_{L_\la}.
	$$
	(b) 
	\bean
	\forall \la', \la'' \in \La\quad \exists \la \in \La\quad L_\la = L_{\la'}\wedge L_{\la''},\\
	\mathfrak{Ultrafilters}\left(\mathscr L \right)_{L_{\la'}}\cap\mathfrak{Ultrafilters}\left(\mathscr L \right)_{L_{\la''}}=\mathfrak{Ultrafilters}\left(\mathscr L \right)_{L_{\la}}.
	\eean
\end{proof}
\begin{definition}\label{topology_l_defn} 
If 	$\mathscr L$ is  the meet-semilattice (cf. Definition \ref{lattice_defn}) then the set $\mathfrak{Ultrafilters}\left( \mathscr L\right)$ (cf. Definition \ref{lattice_s_defn}) supplied with the given by the Lemma \ref{topology_l_lem} topology is the {\it space of ultrafilters} of $\mathscr L$.
\end{definition}
\begin{remark}\label{ultra_com_rem}
	From the Proposition \ref{top_ultra_prop} it turns out that the {space of ultrafilters} of $\mathscr L$ is compact.
\end{remark}

\begin{definition}\label{hausdorff_defn}
	Let $\sX$ be a topological space. The $\sX$-{\it Hausdorff category} (cf. Definition \ref{category_defn}) $\mathfrak{H}\left( \sX\right)$ is such that 
	\begin{itemize}
		\item objects of $\mathfrak{H}\left( \sX\right)$ are continuous surjective maps $\sX \onto \sY$ onto Hausdorff locally compact (cf. Definition \ref{top_locally_compact_defn}) space $\sY$, 
		\item a morphism form $\sX \onto \sY'$ to $\sX \onto \sY''$ is a continuous  map  $\sY' \onto \sY''$ such that the diagram 
		\bean
		\begin{tikzcd}
			&\sX \arrow[ld] \arrow[rd] & \\
			\sY' \arrow[rr]	&& \sY''
		\end{tikzcd}
		\eean
		is commutative.
	\end{itemize}
\end{definition}
\begin{lemma}\label{hausdorff_lem}
	One has:
	\begin{enumerate}
		\item [(i)] for any locally compact topological space $\sX$ there is the natural,   initial object (cf. Definition \ref{initial_ob_defn}) of the  $\sX$-{Hausdorff category} denoted by by $\mathfrak{Hausdorff}\left( \sX\right)$; 
		\item [(ii)] if both $\sX$ and $\sY$ are locally compact spaces then every continuous map $\sX \to \sY$ yields the commutative diagram 
		\be\label{haudorff_eqn}
		\begin{tikzcd}
			\sX  \arrow[r, two heads, "~"] \arrow[d, "f"]   & \mathfrak{Hausdorff}\left( \sX\right) \arrow[d, "\mathfrak{Hausdorff}(f)"]  \\
			\sY  \arrow[r, two heads]    & \mathfrak{Hausdorff}\left( \sY\right)	
		\end{tikzcd}
		\ee
	\end{enumerate}
\end{lemma}
\begin{proof}(i)
	Let $A$ be a commutative $C^*$-algebra of all bounded continuous maps from $\sX$ to $\C$ and let $\sY'$ be the spectrum of $A$ (cf. Definition \ref{spectrum_prime_primtive_defn}).  Any $x \in \sX$ yields the maximal ideal 
	$$
	I_x \bydef \left\{ f \in A | f\left( x\right) = 0\right\}
	$$
	From the Theorem \ref{gelfand-naimark_thm} the ideal $I_x$ yield a point $y \in \sY'$. So there is the natural  map $\phi' :\sX \to \sY'$. If $\sY \bydef \phi'\left( \sX\right)$ then $\phi'$ yields  the surjective map  $\phi :\sX \onto \sY$. One can prove that the map is continuous. From the Theorem \ref{gelfand-naimark_thm} it turns out that the spaces $\sY'$ and $\sY$ are locally compact (cf. Definition \ref{top_locally_compact_defn}) and Hausdorff.  If there is and an object $\sX \onto \sY''$ with nonidentical morphism $\phi'' : \sY''\onto \sY$ then there are $y_1, y_2 \in \sY''$ with $y_1 \neq y_2$ and $\phi''\left( y_1\right) = \phi''\left( y_2\right)$. There is $f \in C_0\left( \sY''\right)$ such that $f\left( y_1\right)  \neq f\left( y_2\right)$. Clearly $f \in A$ but $f \notin C_0\left(
	\sY \right)$ i.e. $C_0\left(\sY \right)  \subsetneqq C_0\left(\sY'' \right)$. It is a contradiction which proves that the existence of $\sY''$ is impossible. \\
	(ii) If both $A$ and $B$ commutative $C^*$-algebras of all bounded continuous maps from $\sX$ and $\sY$ to $\C$ then $A \cong C_b\left( \mathfrak{Hausdorff}\left( \sX\right)\right)$ and $B \cong C_b\left( \mathfrak{Hausdorff}\left( B\right)\right)$ the map $f$ yields the natural $*$-homomorphism $$
	B  \cong C_b\left( \mathfrak{Hausdorff}\left( \sY\right)\right)\to A  \cong C_b\left( \mathfrak{Hausdorff}\left( \sX\right)\right)
	$$
	which induces the natural continuous map $\bt~\mathfrak{Hausdorff}(f): \bt~\mathfrak{Hausdorff}\left(\sX \right) \to \bt~\mathfrak{Hausdorff}\left(\sY \right)$. This lemma follows from $\bt~\mathfrak{Hausdorff}(f)\left( \mathfrak{Hausdorff}\left(\sX \right)\right) = \mathfrak{Hausdorff}\left(\sY \right)$.
\end{proof}
\begin{definition}\label{blowing_uni_h_defn}
	The initial object of the category $\sX$-{Hausdorff category} is the	$\sX$ - \textit{universal Hausdorff quotient space} denoted by $\mathfrak{Hausdorff}\left(\sX \right)$, 
\end{definition}
\begin{definition}\label{hausdorff_funct_defn}
	The Lemma \ref{hausdorff_lem} yields the {\it Hausdorff functor}
	\be\label{hausdorff_funct_eqn}
	\mathfrak{LocComp} \xrightarrow{\mathfrak{Hausdorff}} \mathfrak{Hausdorff}	\ee
	from the category $\mathfrak{LocComp}$ of locally compact topological spaces and continuous maps to the category of Hausdorff locally compact spaces and continuous maps. 
\end{definition}
\begin{lemma}
	Any surjective homomorphism $\varphi: \mathscr L' \onto  \mathscr L''$ of meet semi-lattices yields the surjective continuous maps
	\bea 
	\psi: \mathfrak{Ultrafilters}\left(\mathscr L' \right)\onto 	\mathfrak{Ultrafilters}\left(\mathscr L'' \right),\\
\phi:	\mathfrak{Hausdorff}\left( \mathfrak{Ultrafilters}\left(\mathscr L' \right)\right) \onto \mathfrak{Hausdorff}\left( 	\mathfrak{Ultrafilters}\left(\mathscr L'' \right)\right) 
	\eea 
\end{lemma}
\begin{proof}
	Any $\xi'' = \left\{L''_\la\right\} \in \mathfrak{Ultrafilters}\left(\mathscr L'' \right)$ yield a filter $\xi' \xi \mathfrak{Filters}\left(\mathscr L' \right)$ generated by the set
	$$
	\left\{M \in \mathscr L' |\varphi\left(M \right)  \in \left\{L''_\la\right\} \right\}
	$$
	if we denote by $\mathfrak{Ultrafilters}\left(\varphi \right)^{-1}\left(\xi'' \right)$ the set of ultrafilters  which exceed $\xi'$ then there is the natural map
\bean
	\psi: \mathfrak{Ultrafilters}\left(\mathscr L' \right)\onto 	\mathfrak{Ultrafilters}\left(\mathscr L'' \right),\\
	\psi \left(\xi' \right) = \xi'' \quad \Leftrightarrow	\quad \xi' \in \mathfrak{Ultrafilters}\left(\varphi \right)^{-1}\left(\xi'' \right).
\eean 
	If $L' \in  \mathscr L'$  and $L'' = \varphi\left(L' \right)\in \mathscr L'$ then 
\bean
\mathfrak{Ultrafilters}\left(\varphi \right)^{-1}\left( \mathfrak{Ultrafilters}\left(\mathscr L''\right)_{L''}\right)  = \mathfrak{Ultrafilters}\left(\mathscr L'\right)_{L'}
\eean 
(cf. Notation \ref{topology_l_eqn})
So any  set of basis (cf. Definition \ref{top_base_defn}) of the topology of  $\mathfrak{Ultrafilters}\left(\mathscr L''\right)$  corresponds to a set of subbasis of the topology of $\mathfrak{Ultrafilters}\left(\mathscr L'\right)$. From (c) of the Theorem \ref{top_continuous_thm} it turns out that the map $\psi$ is continuous. Using this fact ans the Lemma \ref{hausdorff_lem} one can prove that the map $\phi$ is continuous.
\end{proof}

\section{Gelfand category and Gelfand functor}
\paragraph{} The commutative  Gelfand-Na\u{\i}mark \ref{gelfand-naimark_thm} theorem yields the equivalence between  following categories:
\begin{itemize}
	\item  unital commutative $C^*$-algebras and $*$-homomorphisms,
	\item compact Hausdorff  spaces and continuous maps.
\end{itemize} 
Here we would like construct a category such that objects of it are all $C^*$-algebras and a functor to the category of topological  spaces  and continuous maps.
\subsection{Preliminaries}
\paragraph{}
	Let $A$ be a $C^*$-algebra. Denote by $A^{\sim}$ the unital $*$-algebra given by
	\be\label{unital_notation_eqn}
	A^\sim\bydef\begin{cases}
		A & \text{if }A \text{ is unital}\\
		A^+ & \text{if }A \text{ is not unital}\\
	\end{cases}
	\ee
	where $A^+$ is given by \eqref{a_plus_eqn}.
	If $A$ is nonunital $C^*$-algebra then 	$A^{\sim}$
	is the minimal unitization of $A$ (cf. Definition \ref{multiplier_min_defn}). If $\sX$ is a locally compact Hausdorff space then we denote by 
	\be\label{compact_notation_eqn}
\sX^\sim\bydef\begin{cases}
	\sX & \text{if }\sX \text{ is compact}\\
	\sX^+ & \text{if }\sX \text{ is not compact}\\
\end{cases}
\ee
where $\sX^+$ is the natural  one point compactification of $\sX$.	
\begin{definition}\label{nonunital_principal_defn}
Let both $A$ and $B$ be  $C^*$-algebras. If  $A$ and $B$ are not unital then a $*$-homomorphism $f: A \to B$ is \textit{principal} if it can be extended up to the principal in the sense  of Definition \ref{principal_defn}. homomorphism $f^+:A^+ \to B^+$. If both $A$ and $B$ are  unital them $A \to B$  is \textit{principal} if it is principal in the sense  of Definition \ref{principal_defn}. 
\end{definition}
\begin{remark}\label{nonunital_principalem}
If both $A$ and $B$ be  $C^*$-algebras then a principal a $*$-homomorphism $f:A \to B$ can be  extended up to the unique principal  in the sense  of Definition \ref{principal_defn} homomorphism $f^\sim:A^\sim \to B^\sim$ (cf. equation \eqref{unital_notation_eqn}).
\end{remark}

\begin{empt}
	If $A$ is a $C^*$-algebra then one has
	\be\label{four_decompositon_eqn}
	\forall a \in A\quad \exists a_1, a_2, a_3, a_4 \in A_+\quad a=a_1 - a_2 + ia_3 - ia_4.
	\ee
\end{empt}

\subsection{Basic construction}
\paragraph{} Any topological space $\sX$ yields the natural  meet-semilattice of open subsets (cf. Definition \ref{lattice_defn} and Appendix \ref{fbt_sec}). 
It is known that there are ultrafilters on locally compact space which are not convergent (cf. Proposition \ref{top_ultra_prop}). However from the Proposition \ref{top_ultra_prop} it follows that if an ultrafilter contains an open set with compact closure then it is convergent.
\begin{definition}\label{good_ultrafilter_defn}
	An ultrafilter of open subsets sets  on locally compact Hausdorff  space  is a {\it finite point} if contains an open set with compact closure.
\end{definition}
\begin{remark}\label{good_ultrafilterem}
	From the Proposition \ref{top_ultra_prop} it follows that any locally compact Hausdorff  space $\sX$ is a set of its finite points.
\end{remark}
\begin{definition}\label{lattice_a_defn}
If $A$ is a $C^*$-algebra then $A$-\textit{semi}-\textit{lattice} $\mathfrak{Lattice}\left( A\right)$ is the meet-semilattice (cf. Definition \ref{lattice_defn}) of open projections of  $A$ (cf. Definition \ref{open_close_ideal_defn}). We denote by  $\mathfrak{Lattice}\left( A\right)$ the $A$-{semi}-{lattice}.
A filter (resp, an ultrafilter) of $A$-semi-lattice is an $A$-\textit{filter} (resp. an $A$-\textit{ultrafilter}) (cf. Definitions \ref{filter_defn} and \ref{ultra_filter_defn}).	We denote this semilattice by $\mathfrak{Lattice}\left( A\right)$. Denote by $\mathfrak{Filters}\left( A\right)$   the  set of $A$-filters. Denote by $\mathfrak{Filters}\left( A\right)$ the topological space of ultrafilters (cf. Definition \ref{topology_l_defn}) of the $\mathfrak{Lattice}\left( A\right)$.
\end{definition}

\begin{lemma}\label{top_lattices_iso_lem}
For any  compact Hausdorff space $\sX$ (cf. Definition \ref{top_locally_compact_defn}) there is an isomorphism of  meet-semilattices
\be
\begin{split}
	\mathfrak{Lattice}_\sX : \mathfrak{Lattice}\left( C\left(  \sX\right)\right)\cong \mathfrak{Lattice}\left(\sX \right).
\end{split}
\ee	
where  $\mathfrak{Lattice}\left(\sX \right)$ is the lattice of open subsets of $\sX$.
\end{lemma}
\begin{proof}
	Any left ideal of  $C\left(  \sX\right)$ is a two-sided ideal. This lemma can be deduced from the proof of the Theorem \ref{gelfand-naimark_thm}.
\end{proof}

Any $*$-homomorphism $C\left( \sX\right) \to C\left( \sY\right)$ of  yields a continuous  map $\sY\to \sX$. Any $*$-homomorphisms $C_0\left( \sX\right) \to C_b\left( \sY\right)$ can be uniquely extended up to a $*$-homomorphism $C_b\left( \sX\right) \to C_b\left( \sY\right)$ and yields a continuous map $f:\bt\sY \to \bt \sX$ where both $\bt\sX$ and $\bt \sY$ are {Stone-\v{C}ech compactifications} (cf. Definition \ref{top_stone_cech_defn}.
For any $C^*$-algebra $A$ there is the natural inclusion $M\left(A \right)  \hookto A^{**}$ of the multiplier into enveloping $W^*$-algebra (cf. Definition \ref{env_alg_defn}). For any $C^*$-algebra $B$ a $*$-homomorphism $\varphi :M\left(A \right) \to M\left(B \right)$ induces a homomorphism
\be\label{w*_eqn} 
\varphi^{**} : A^{**} \to B^{**}.
\ee 
\begin{lemma}\label{proj_lem}
If both $A$ and $B$ are $C^*$-algebras  and  $\varphi : M\left(  A\right)  \to M\left( B\right) $ is a $*$-homomorphism then for any open projection $p \in A^{**}$ (cf. Definition \ref{open_close_ideal_defn}) the projection $\varphi^{**}\left( p\right) \in  B^{**}$ is open.
\end{lemma}
\begin{proof}
There is a closed left ideal  of $L \subset B$ such that
\be\label{lb_eqn}
L \bydef \left\{b \in B\left| b =b \varphi^{**} \left(p \right) \right.\right\} 
\ee 
The Lemma \ref{left_ideal_thm} yields an increasing  net  $\left\{b_\la \right\}_{\la \in \La}\subset B$ such that
\be\label{inc_eqn}
\forall b \in L \quad \lim_{\la\in \la}\left\| b - b b_\la \right\| = 0.
\ee
From the Lemma \ref{increasing_convergent_w_lem} it follows that there is a  projection $q \in B^{**}$  such that $0\ge b_\la \uparrow q$
$$
\lim_{\la\in \La} b_\la = q
$$
in the weak* topology of $B^{**}$. 	According to the Definition \ref{open_close_ideal_defn}
the projection $q$ is open. From \eqref{inc_eqn} it turns out that
\be\label{lbq_eqn}
L = \left\{b \in B\left| b =b q \right.\right\}.
\ee
Comparison of \eqref{lb_eqn} and \eqref{lbq_eqn} yields $\varphi^{**} \left(p \right) = q$.
\end{proof}
\begin{definition}\label{good_hom_defn}
	If both $A$ and $B$ are $C^*$-algebras  then a $*$-homomorphism $\varphi : M\left(  A\right)  \to M\left( B\right) $ is {\it good} if for any compact projection $q \in B^{**}$ (cf. Definition \ref{compact_p_defn}) there is a compact projection $p \in A^{**}$ such that
	\be \label{good_hom_eqn}
	q \le \varphi^{**}\left(p \right). 
	\ee
\end{definition}

\begin{lemma}
	A homomorphism $\varphi: C_b\left( \sX\right) \to C_b\left( \sY\right)$ is good  if and only if  it induces the continuous map  $f: \bt\sY \to \bt\sX$ such that $f\left(\sY \right) \subset \sX$.
\end{lemma}
\begin{proof}
	$\Rightarrow$ If $y\in \sY$ then there is a compact neighborhood $\sV\subset \sY$ of $y$ which   corresponds to compact $q \in C_0\left(\sY \right)^{**}$ (cf. the Definition \ref{compact_p_defn} and the Remark \ref{compact_pem}). The set $\sU \bydef f\left(\sV \right)\subset \sX$ is also compact, so it naturally yields the compact projection $p \in C_0\left(\sX \right)^{**}$. From $f\left( \sV\right) = \sU$ it follows that 
	\bean
	q = \varphi^{**}\left(p \right).
		\eean 
	$\Leftarrow$   Any compact projection $q \in C_0\left(\sY \right)^{**}$ corresponds to a compact subset $\sV\subset \sY$. The image 
	$
	f\left(\sV \right) 
$ is compact so it corresponds to a compact projection $p \in C_0\left(\sX \right)^{**}$. It turns out that 	$q \le \varphi^{**}\left(p \right)$.
\end{proof}
\begin{definition}\label{gelfand_category_defn}
	The {\it Gelfand category} $C^*$-$\mathfrak{Gelfand}$ is such that:
	\begin{itemize}
		\item objects of $C^*$-$\mathfrak{Gelfand}$ are $C^*$-algebras,
		\item for any two objects $A$ and $B$ of $C^*$-$\mathfrak{Gelfand}$ the set of arrows from $A$ to $B$ is a set of good $*$-homomorphisms from $M\left(A \right)\to M\left( B\right)$ (cf. Definition \ref{good_hom_defn}). 
	\end{itemize}
	
\end{definition}

\begin{empt}\label{topology_empt} 
	If $\mathfrak{Ultrafilters}\left(A \right)$ 
	is a set of $A$-ultrafilters (cf. Definition \ref{lattice_a_defn}) then foe any for any open projection $p$ (cf. Definition \ref{open_close_ideal_defn}) we define a set
	\be\label{topology_eqn}
	\mathfrak{Ultrafilters}\left(A \right)_p	\bydef \left\{ \xi \in \mathfrak{Ultrafilters}\left(A \right) \left| p \in \xi\right.\right\}\subset 	\mathfrak{Ultrafilters}\left(A \right)
	\ee
\end{empt}
\begin{lem}\label{topology_lem} 
	The set of given by \eqref{topology_eqn} sets is a basis of topology on $\mathfrak{Ultrafilters}\left(A \right)$  (cf. Definition \ref{top_base_defn}).
\end{lem}
\begin{proof}
One needs check conditions (a) and (b) of the Definition \ref{top_base_defn},\\
(a) If $\xi \in \left\{ p_\la \right\}_{\la \in \La} \mathfrak{Ultrafilters}\left(A \right)$ then 
$$
\forall \la \in \La\quad  \xi \in 	\mathfrak{Ultrafilters}\left(A \right)_{p_\la}.
$$
(b) 
\bean
\forall \la', \la'' \in \La\quad \exists \la \in \La\quad p_\la = p_{\la'}\wedge p_{\la''},\\
\mathfrak{Ultrafilters}\left(A \right)_{p_{\la'}}\cap\mathfrak{Ultrafilters}\left(A \right)_{p_{\la''}}=\mathfrak{Ultrafilters}\left(A \right)_{p_{\la}}.
\eean
\end{proof}
\begin{definition}\label{topology_defn} 
The given by the Lemma \ref{topology_lem} topology is \textit{Gelfand topology}, so the set  $\mathfrak{Ultrafilters}\left(A \right)$ becomes a topological space.
\end{definition}
	\begin{lemma}\label{gelfand_u_compact_lem}
	If $A$ is an unital $C^*$-algebra then the space  $\mathfrak{Ultrafilters}\left(A \right)$ is compact. 
\end{lemma}
\begin{proof}
Follows the Proposition \ref{top_ultra_prop},
\end{proof}
	\begin{definition}\label{gelfand_space_defn}
	If $A$ is a $C^*$-algebra then	an ultrafilter  $\xi = \left\{p_\la\right\}\in \mathfrak{Ultrafilters}\left(A \right)$ is an $A$-\textit{finite point} if there is a nontrivial compact projection $p$ (cf. Definition \ref{compact_p_defn}) such that $p_\la \wedge  p\neq \{0\}$ for all $\la\in \La$.
	The \textit{pre}-\textit{Gelfand space} $\mathfrak{PreGelfand}\left(A \right)$ of $C^*$-algebra $A$ is a space  of finite points with the final topology (cf. Definition \ref{top_final_defn}) given by the natural inclusion $\mathfrak{PreGelfand}\left(A \right) \hookto \mathfrak{Ultrafilters}\left(A \right)$, where the topology on $\mathfrak{Ultrafilters}\left(A \right)$ is the   Gelfand topology (cf. Definition \ref{topology_defn}. The \textit{Gelfand space} $\mathfrak{Gelfand}\left(A \right)$ of $C^*$-algebra $A$ is the universal Hausdorff  quotient space of $\mathfrak{PreGelfand}\left(A \right)$ (cf. Definition \ref{blowing_uni_h_defn}), i.e. $\mathfrak{Gelfand}\left(A \right)\bydef \mathfrak{Hausdorff}\left(\mathfrak{PreGelfand}\left(A \right) \right)$ (cf. equation \eqref{hausdorff_funct_eqn}).
\end{definition}
\begin{remark}\label{topology_base_herdem}
	Using the Definition \ref{open_close_ideal_defn} the semilattice $\mathfrak{Lattice}\left(A \right)$ can be replaced with the semilattice of closed left ideals of $A$.
	Using the Lemma \ref{hered_ideal_lem} one can replace the closed left ideals in the Definition \ref{gelfand_space_defn} with right ones or hereditary $C^*$-subalgebras. For any ultrafilter $\xi$ we use a following notation
	\be\label{hered_ideal_eqn}
	L \in \xi \quad \Leftrightarrow \quad B \bydef L \cap L^* \in \xi\quad \Leftrightarrow \quad L^* \in \xi.
	\ee
\end{remark}
\begin{notation}
Similarly to \eqref{topology_eqn} we will use the following notation
	\be\label{topology_q_eqn}
\mathfrak{Gelfanf}\left(A \right)_p	\bydef \mathfrak{Hausdorff}\left(\mathfrak{PreGelfand}\left(A \right)_p \right)
\ee
for the basis of the topology of the Gelfand space. If $p'$ is a closed projection then we define
	\bea\label{u_topology_q_closed_eqn}
	\mathfrak{Ultrafilers}\left(A \right)^{p'}\bydef \left\{\left.\xi \in  \mathfrak{Ultrafilers}\left(A \right)\right| \forall \in \xi\quad p \wedge  p'\neq \{0\} \right\},\\\label{topology_q_closed_eqn}	\mathfrak{PreGelfand}\left(A \right)^{p'}\bydef \left\{\left.\xi \in  \mathfrak{PreGelfand}\left(A \right)\right| \forall \in \xi\quad p \wedge  p'\neq \{0\} \right\},\\
	\label{topology_gq_closed_eqn}
\mathfrak{Gelfanf}\left(A \right)^{p'}	\bydef \mathfrak{Hausdorff}\left(\mathfrak{PreGelfand}\left(A \right)^{p'} \right),\\
\eea
Above subsets correspond to the following $C^*$-algebras
\bea\label{algebra_open_eqn}
A_p \bydef \left\{\left. a \in A\right| pap = a\right\}\subset A,\\
\label{algebra_closed_eqn}
A^{p'} \bydef A^{**} \cap p' A p'\subset A^{**}.
\eea
\end{notation}
\begin{lemma}\label{ultra_cloded_lem}
For any closed projection $p$ one has natural homeomorphisms
\bea\label{ultra_cloded_eqn}
	\mathfrak{Ultrafilers}\left(A \right)^{p}\cong 	\mathfrak{Ultrafilers}\left(A^{p} \right)
	\eea
\end{lemma}

\begin{proof}
According to the Remark \ref{topology_base_herdem} we consider both  $\mathfrak{Lattice}\left( A\right)$ and  $\mathfrak{Lattice}\left( A^p\right)$ (cf. Definition \ref{lattice_a_defn}) as sets of hereditary $C^*$-subalgebra of $A$ and $A^p$ respectively. There are following homomorphisms
\be\label{ip_l_eqn}
\begin{split}
\pi: \mathfrak{Lattice}\left( A\right)\onto \mathfrak{Lattice}\left( A^p\right),\\
B \mapsto pBp;\\
\iota : \mathfrak{Lattice}\left( A^p\right)\onto \mathfrak{Lattice}\left( A\right),\\
C \mapsto \text{the } C^*\text{-norm closure of } \bigcup_{\substack{B \subset A \\ pBp \subset C }} B.
\end{split}
\ee
From $\iota \circ \pi = \Id_{\mathfrak{Lattice}\left( A^p\right)}$ it follows that the homomorphism $\iota$ naturally yields the injective map  
\be\label{u_iota_eqn}
\mathfrak{Ultrafilers}\left(\iota \right) : 	\mathfrak{Ultrafilers}\left(A^p \right)\hookto 	\mathfrak{Ultrafilers}\left(A \right).
\ee
According to the Definition \ref{topology_defn} the  bases of topologies of $\mathfrak{Ultrafilers}\left(A \right)$ and $\mathfrak{Ultrafilers}\left(A^p \right)$ are $\mathfrak{Lattice}\left(A \right)$ and $\mathfrak{Lattice}\left(A^p \right)$ respectively. From this fact one can deduce that the map \eqref{u_iota_eqn} is continuous. If $\xi \in \mathfrak{Ultrafilers}\left(A \right)$ is represented by the family $\left\{B_\la\right\}_{\la\in \La}$ of hereditary $C^*$-subalgebras (cf. Remark \ref{topology_base_herdem}) then from \eqref{ip_l_eqn} it follows that
\be
\xi \notin \mathfrak{Ultrafilers}\left(\iota \right)\left( 	\mathfrak{Ultrafilers}\left(A^p \right)\right)\quad \Leftrightarrow \quad \exists \la \in \La \quad p B_\la p   = \{0\}.
\ee 
Equivalently if $\xi$ is is represented by the family $\left\{p_\la\right\}_{\la\in \La}$ of open  projections then 
\be\label{nn_ee_eqn}
\begin{split}
\xi \notin \mathfrak{Ultrafilers}\left(\iota \right)\left( 	\mathfrak{Ultrafilers}\left(A^p \right)\right)\quad \Leftrightarrow \quad \exists \la \in \La \quad p \wedge p_\la  = 0\quad \Leftrightarrow\\
\Leftrightarrow \quad \xi \notin \mathfrak{Ultrafilers}\left(A\right)^p
\end{split}
\ee 
The equation \eqref{nn_ee_eqn} means that 
\be\label{nnn_ee_eqn}
\begin{split}
\mathfrak{Ultrafilers}\left(A\right) \setminus \mathfrak{Ultrafilers}\left(\iota \right)\left( 	\mathfrak{Ultrafilers}\left(A^p \right)\right)=\mathfrak{Ultrafilers}\left(A\right) \setminus 	\mathfrak{Ultrafilers}\left(A \right)^p
\end{split}
\ee 
The statement of this lemma is a direct consequence of the equation \eqref{nnn_ee_eqn}
\end{proof}

	\begin{lemma}\label{gelfand_compact_lem}
	If $A$ is an unital  $C^*$-algebra then both  spaces $\mathfrak{PreGelfand}\left(A \right)$ and $\mathfrak{Gelfand}\left(A \right)$ are  compact.
\end{lemma}
\begin{proof}
	From $p \le 1_A$ for any projection  it follows that  $\mathfrak{PreGelfand}\left(A \right)= \mathfrak{Ultrafilters}\left(A \right)$.  However the space $\mathfrak{Ultrafilters}\left(A \right)$ is compact (cf. Lemma \ref{gelfand_u_compact_lem}). The space  $\mathfrak{Gelfand}\left(A \right)$ is compact since it is an epimorphic image of the compact space $\mathfrak{PreGelfand}\left(A \right)$.
 \end{proof}
 \begin{corollary}\label{gelfand_compact_cor}
 	If $p\in A^{**}$ is a compact projection  (cf. Definition \ref{compact_p_defn}) then
  one has:
 	\begin{enumerate}
 		\item [(i)] the given by \eqref{algebra_closed_eqn} $C^*$-algebra $A^p$
 		is unital,
 		\item[(ii)] both spaces $\mathfrak{PreGelfand}\left(A \right)^{p}\cong \mathfrak{Ultrafilers}\left(A \right)^{p}$ and $ \mathfrak{Gelfand}\left(A \right)^{p}$ are compact.
 	\end{enumerate}
 	 \end{corollary}
 	 \begin{proof}(i)
	 If $b \in A_+$ is such that $bp = p$ then
\be
\forall a \in A^p \quad \left( pbp \right) a = a, \quad a \left( pbp\right)  = a,
\ee
i.e.
\be
pbp = 1_{A^p}
\ee
(ii) Is a  consequence of (i) and the Lemma \ref{gelfand_compact_lem}.
 	 \end{proof}
 
 \begin{rem}\label{compact_compactem}
 The (ii) of the Corollary \ref{gelfand_compact_cor} is also direct consequence of  the Theorem \ref{gelfand_compact_thm}.
 \end{rem}

\begin{lemma}\label{ineq_lem}
If $A$ is an unital $C^*$-algebra and $p', p''$ are two open projections  then one has
\bean 
p' < p'' \quad \Rightarrow \quad \mathfrak{Ultrafilters}\left(A \right)_{p'}\subsetneqq \mathfrak{Ultrafilters}\left(A \right)_{p''}\quad \Leftrightarrow\\ \Leftrightarrow \mathfrak{PreGelfand}\left(A \right)_{p'}\subsetneqq \mathfrak{PreGelfand}\left(A \right)_{p''} \quad   \Leftrightarrow\quad \mathfrak{Gelfand}\left(A \right)_{p'}\subsetneqq \mathfrak{Gelfand}\left(A \right)_{p''}.
\eean
\end{lemma}
\begin{proof}
	If both $L', L'' \subset A$ a corresponding to $p'$ and $p''$ closed left ideals then from $p'< p''$ it follows that $L' \subsetneqq L''$.
If $a \in L'' \setminus L'$ and $p$ is corresponding to the left ideal $Aa$ projection then
\bean
\mathfrak{PreGelfand}\left(A \right)_{p},~\mathfrak{Ultrafilters}\left(A \right)_{p}\neq \emptyset,\\
\mathfrak{PreGelfand}\left(A \right)_{p} \subset  \mathfrak{PreGelfand}\left(A \right)_{p''}\setminus \mathfrak{PreGelfand}\left(A \right)_{p'}
\eean 
so there is $\xi \in \mathfrak{PreGelfand}\left(A \right)_{p}= \subset  \mathfrak{PreGelfand}\left(A \right)_{p''}\setminus \mathfrak{PreGelfand}\left(A \right)_{p'}$ and $\mathfrak{PreGelfand}\left(A \right)_{p}$ is and open neighborhood of $\xi$. If $\eta \in \mathfrak{Gelfand}\left(A \right)$ is the image of $\xi$ in $\mathfrak{Gelfand}\left(A \right)$ then $\mathfrak{Gelfand}\left(A \right)_p$ is a neighborhood of $\eta$ (cf. \eqref{topology_q_eqn}) such that $\mathfrak{Gelfand}\left(A \right)_{p} \subset  \mathfrak{Gelfand}\left(A \right)_{p''}\setminus \mathfrak{Gelfand}\left(A \right)_{p'}$, so $\eta \in \mathfrak{Gelfand}\left(A \right)_{p''}\setminus \mathfrak{Gelfand}\left(A \right)_{p'}$.
\end{proof}

The following Lemma is a generalization of the Theorem \ref{gelfand-naimark_thm}.
\begin{lemma}\label{lem_g_lem} (Generalized commutative Gelfand theorem).
	If $\sX$ is compact (cf. Definition \ref{top_compact_defn}), Hausdorff  space then  is a natural  homeomorphism $ \mathfrak{Gelfand}_\sX: \mathfrak{Gelfand}\left(C\left( \sX\right) \right)\cong \sX$, where  $\mathfrak{Gelfand}\left(C\left( \sX\right)\right)$ is the Gelfand space.
\end{lemma}
\begin{proof}
	
	From the Lemma \ref{top_lattices_iso_lem}  it turns out that for any locally compact space $\sX$ (cf. Definition \ref{top_locally_compact_defn}) there is an isomorphism of  meet-semilattices
	
	\bean
	\begin{split}
		\mathfrak{Lattice}_\sX : \mathfrak{Lattice}\left( C\left(  \sX\right)\right)\cong \mathfrak{Lattice}\left(\sX \right).
	\end{split}
	\eean  
	It follows that there is a homeomorphism 
	$$
	\mathfrak{Ultrafilters}_\sX : \mathfrak{Ultrafilters}\left( C\left(  \sX\right)\right)\cong \mathfrak{Ultrafilters}\left(\sX \right)
	$$ 	
	of spaces of ultrafilters. Now this lemma follows from the Proposition \ref{top_ultra_prop}.
\end{proof}
\begin{empt}\label{one_point_empt}
If  $A$ is an unital $C^*$-algebra and  $I\subset A$ is  a maximal proper left ideal then $I$ corresponds to an open projection $p_I \in A^{**}$ (cf. Definition \ref{open_close_ideal_defn}). From the Proposition \ref{open_close_ideal_prop} it turns out that $p'_I \bydef 1-p_I\in A^{**}$ is a minimal projection. So for any projection $p \in A^{**}$ one has 
\bean
p \wedge p'_I \neq 0 \quad \Rightarrow \quad  p \wedge p'_I= p'_I \quad \Rightarrow \quad p'_I \le p
\eean 
From the above equation it follows that the set
\be\label{fa_filter_eqn}
\mathfrak{F}_I \bydef \left\{\left. p \in \mathfrak{Lattice}\left( A\right)\right| p'_I < p \right\}
\ee
is a filter.
\end{empt}
\begin{lem}\label{one_point_lem}
For any unital $C^*$-algebra $A$ the set
\be 
 S \bydef \mathfrak{Hausdorff}\left(\left\{\left.\eta \in \mathfrak{PreGelfand}\left(A\right)\right| \eta \ge \mathfrak{F}_I \right\} \right)
\ee
is a singleton, i.e. there is the unique $\xi_I \in \mathfrak{Gelfand}\left(A \right)$ such that $S = \{\xi_I\}$.
\end{lem}
\begin{proof}
If there are $\xi', \xi'' \in S$ with  $\xi' \neq \xi''$  then there are $\eta', \eta'' \in  \mathfrak{PreGelfand}\left(A \right)$ with  $\mathfrak{Hausdorff}\left(\eta' \right) = \xi'$ and $\mathfrak{Hausdorff}\left(\eta'' \right) = \xi''$. Since  $\mathfrak{Gelfand}\left(A \right)$ is compact and Hausdorff   there are open  neighborhoods $\sU'$ and $\sU''$ of $\xi'$ and $\xi''$ with $\sU'\cap \sU''= \emptyset$. One can suppose that $\mathfrak{Hausdorff}^{-1}\left( \sU'\right)$ and $\mathfrak{Hausdorff}^{-1}\left( \sU''\right)$ correspond to $p' \in \mathfrak{Lattice}\left( A\right)$ and $p'' \in \mathfrak{Lattice}\left( A\right)$ with $p'\wedge p'' = 0$. There are neighborhoods $\sV'$ and $\sV''$ of $\xi'$ and $\xi''$ such that closures of $\sV'$ and $\sV''$ are subsets of $\sU'$ and $\sU''$.  One can suppose that $\mathfrak{Hausdorff}^{-1}\left( \sV'\right)$ and $\mathfrak{Hausdorff}^{-1}\left( \sV''\right)$ correspond to $q' \in \mathfrak{Lattice}\left( A\right)$ and $q'' \in \mathfrak{Lattice}\left( A\right)$. If both  $\overline{q}'$ and  $\overline{q}''$ are closures (cf. Definition \ref{closure_defn}) of  $q'$ and $q''$ then $\overline{q}' < p'$ and $\overline{q}'' < p''$ with $\overline{q}'\wedge \overline{q}''=0$. If $p \in \mathfrak{F}_I$ then from $\eta', \eta'' \ge \mathfrak{F}_I$ it turns out that
\bea\label{ee1_eqn}
\forall p' \in  \mathfrak{Lattice}\left( A\right) \quad p' \in \eta' \quad\Rightarrow \quad  p' \wedge  p \neq 0,\\
\label{ee2_eqn}
\forall p'' \in  \mathfrak{Lattice}\left( A\right) \quad p'' \in \eta'' \quad\Rightarrow \quad  p'' \wedge  p \neq 0. 
\eea
If $p'_I \not \le  \overline{q}'$ then $p_I \le 1 - \overline{q}'$ and $ 1 - \overline{q}'\in \mathfrak{F}_I$. However the equality $q' \wedge\left( 1 - \overline{q}'\right) = 0$ contradicts with \eqref{ee1_eqn}. This contradiction rejects the hypothesis $p_I \not \le  \overline{q}'$, so $p_I \le  \overline{q}'$. Similarly one has $p_I \le  \overline{q}''$. But both  inequalities 
\bean
p_I \le  \overline{q}',\\
p_I \le  \overline{q}''
\eean 
are incompatible with  $\overline{q}'\wedge \overline{q}''=0$. This contradiction proves this lemma.
\end{proof}
\begin{theorem}\label{one_point_thm}
	If $A$ is an unital $C^*$-algebra then there is the natural injective map
	\be\label{from_max_to_gelfand_eqn}
	\mathfrak{LeftMax}\left(A\right) \hookto \mathfrak{Gelfand}\left(A\right)
	\ee
	From the set 	$\mathfrak{LeftMax}\left(A\right)$ of maximal left ideals of $A$ into its Gelfand space.
\end{theorem}
\begin{proof}
	The map \eqref{from_max_to_gelfand_eqn} is given by
	\bean
	I \mapsto \xi_I
	\eean 
where  $\xi_I \in \mathfrak{Gelfand}\left(A\right)$ the given by the Lemma \ref{one_point_lem} point.
\end{proof}
\begin{corollary}\label{one_point_cor}
If  $A$ is an nonunital $C^*$ algebra and  $A^+\bydef A \oplus \C$ is the  minimal unitization (cf. Definition \ref{multiplier_min_defn}) then $\mathfrak{Gelfand}\left(A^+\right)= \mathfrak{Gelfand}\left(A\right)^+=\mathfrak{Gelfand}\left(A\right)\cup \left\{\xi_A\right\}$ where $\xi_A$ is  the given by the Lemma \ref{one_point_lem} point.
\end{corollary}
\begin{proof}
Since the $C^*$-algebra $ A^+$ is unital and $A \subset A^+$ is a maximal left ideal this corollary is a consequence of the Lemma \ref{one_point_lem}
\end{proof}
\subsection{The map to the spectrum}
\paragraph{}
From the Theorem \ref{good_thm} it follows that for any good inclusion $
\phi: C_0\left( 
\sX\right) \hookto M\left(A \right)$ there is the natural  open continuous map 
\be
\sX \xrightarrow{{\mathfrak{Gelfand}}\left( \phi\right)} \mathfrak{Gelfand}\left( A\right) 
\ee

\begin{definition}\label{belongs_defn}
	Let $A$ be a $C^*$ -algebra and let $\hat A$ be the spectrum of $A$ (cf. Definition \ref{spectrum_prime_primtive_defn}). An element $\xi \in \mathfrak{Gelfand}\left(A \right)$  \textit{belongs to} $x \in \hat A$ if one has
	\be\label{belongs_eqn}
	\forall L \in \xi \quad \rep_x\left(L\right) \neq \{0\} 
	\ee
	where $\rep_x : A \to B\left(\H_x \right)$ is an irreducible representation (cf. Definition  \ref{irred_defn}). 
	
\end{definition}
\begin{empt}
	Under the hypothesis of the Definition \ref{belongs_defn}  if $x', x'' \in \hat A$ with $x' \neq x''$ and $\xi\in \mathfrak{Gelfand}\left(A \right)$ belongs to both $x'$, and $x''$ then the generated by the set $\left\{ L \cap \ker_{x'}\right\}$ filter exceeds $\xi$. It is impossible since $\xi$ is an ultrafilter. Moreover
	\be\label{not_bel}
	\begin{split}
		\forall x \in \hat A \quad \xi \text{ does not belong to } x \quad \Rightarrow \quad \exists L  \in \xi \quad \rep_x\left( L\right) = \{0\}.
	\end{split}
	\ee
		So from (c) of the Theorem \ref{top_continuous_thm} one has the natural continuous map:
	\be\label{to_spec_eqn} 
	\begin{split}
		\mathfrak{Gelfand}\left(A \right) \xrightarrow{\hat \phi} \hat A
	\end{split}
	\ee
\end{empt}

\subsection{Good $*$-homomorphisms and Gelfand functor}
\paragraph*{}
If both  $A$ and $B$ are $C^*$-algebras and $\varphi :M\left(  A\right) \to M\left( B\right)$ is a good $*$-homomorphism (cf. Definition \ref{good_hom_defn})
	then there is the natural homomorphism of lattices
	\be\label{lat_mor_eqn}
	\begin{split}
		\mathfrak{Lattice}\left(  A\right) \xrightarrow{\mathfrak{Lattice}\left(\varphi \right)}\mathfrak{Lattice} \left( B\right),\\
		p \mapsto \varphi^{**} \left( p \right) \quad \text{(cf. equation \eqref{w*_eqn})}.
	\end{split}
	\ee
	On the other hand the homomorphism $\mathfrak{Lattice}\left(\varphi \right)$ yields a mapping of filters
	\be\label{fil_mor_eqn}
	\begin{split}
		\mathfrak{Filters}\left(  A\right) \xrightarrow{\mathfrak{Filters}\left(\varphi \right)}\mathfrak{Filters} \left( B\right),\\
		\left\{p_\la\right\}_{\la \in \La} \mapsto \text{ a filter generated by }  \left\{\mathfrak{Lattice}\left(\varphi \right)\left( p_\la\right) \right\}_{\la \in \La}
	\end{split}
	\ee
	If $\xi \in \mathfrak{Ultrafilters}\left(  A\right) $ is an ultrafilter then we define a set
	\be\label{uphi-inv_eqn}
	\mathfrak{Ultrafilters}\left(\varphi \right)^{-1}\left(\xi \right) \bydef\left\{\left.\eta \in \mathfrak{Ultrafilters}\left(B\right)\right|\forall q\in \mathfrak{Filters}\left(\varphi\right)\left( \xi\right)\quad q \in  \eta \right\}
	\ee
	For all $\xi', \xi''\in \mathfrak{Ultrafilters}\left(A \right)$ with $\xi' \neq \xi''$ there are $p' \in \xi'$ and $p'' \in \xi''$ such that $p'\wedge p'' = 0$. It turns out that
	\bean
	\xi', \xi''\in \mathfrak{Ultrafilters}\left(A \right)\quad \xi' \neq \xi'' \quad \Rightarrow\\
	\Rightarrow \quad \mathfrak{Ultrafilters}\left(\varphi \right)^{-1}\left(\xi' \right)\cap \mathfrak{Ultrafilters}\left(\varphi \right)^{-1}\left(\xi ''\right)= \emptyset
	\eean 
	Using the above map one can construct the natural map
	\be\label{ultra_mor_eqn}
	\begin{split}
		\mathfrak{Ultrafilters}\left(\varphi \right)\left( \eta\right) = \xi \quad \Leftrightarrow \quad \eta \in \mathfrak{Ultrafilters}\left(\varphi_M \right)^{-1}\left(\xi \right),\\
		\mathfrak{Ultrafilters}\left(B\right) \xrightarrow{\mathfrak{Ultrafilters}\left(\varphi \right)}\mathfrak{Ultrafilters} \left( A\right).
	\end{split}
	\ee

	\begin{theorem}\label{good_thm}
	One has
			\begin{enumerate}
				\item [(i)] 	If the  homomorphism $\varphi : M\left( A\right) \to M\left( B\right)$ good 	  then the map \eqref{ultra_mor_eqn} naturally yields a continuous map
				\be\label{gelfand_map_eqn}
				\mathfrak{Gelfand}\left(\varphi \right) :	\mathfrak{Gelfand}\left( B\right) \to \mathfrak{Gelfand}\left( A\right) 
				\ee 
				
				\item[(ii)] If $\varphi$ is injective then the map  \eqref{ultra_mor_eqn} is surjective.
			\end{enumerate}
			
	\end{theorem}
	\begin{proof}(i) 		If $p \in  A^{**}$ is an open projection (cf. Definition \ref{open_close_ideal_defn}) and $q = \varphi^{**}\left(p \right)\in B^{**}$ then $q$ is an open projection (cf. Lemma \ref{proj_lem})	
		\bean
		\mathfrak{Ultrafilters}\left(\varphi \right)^{-1}\left( \mathfrak{Ultrafilters}\left(B \right)_p\right) = \mathfrak{Ultrafilters}\left(A \right)_{q}
		\eean 
		So any  set of basis (cf. Definition \ref{top_base_defn}) of the topology of  $\mathfrak{Gelfand}\left(A\right)$  corresponds to a set of subbasis of the topology of $\mathfrak{Gelfand}\left(B\right)$. From (c) of the Theorem \ref{top_continuous_thm} it turns out that the map \ref{ultra_mor_eqn} is continuous. Using this fact one can prove that the map \ref{ultra_mor_eqn} yields the natural continuous map between pre-Gelfand spaces (cf. Definition \ref{gelfand_space_defn})
		\bean
		\begin{split}
				\mathfrak{PreGalfanf}\left(B\right) \xrightarrow{\mathfrak{PreGalfanf}\left(\varphi \right)}\mathfrak{PreGalfanf} \left( A\right).
		\end{split}
		\eean
		Now the continuity of the map \eqref{gelfand_map_eqn} follows from the Lemma \ref{hausdorff_lem}.\\
		(ii) If $\varphi$ is injective   then for any $\xi \in \mathfrak{Gelfand}\left(A\right)$   the given by \eqref{uphi-inv_eqn} set $\mathfrak{Ultrafilters}\left(\varphi \right)^{-1}\left(\xi \right)$ is not empty. From \eqref{ultra_mor_eqn} it follows that $\xi \in  \mathfrak{Gelfand}\left(\varphi\right)\left( \mathfrak{Gelfand}\left(A\right) \right)$.
		
	\end{proof}
	
	\begin{definition}\label{gelfand_functor_defn}
		The Theorem \ref{good_thm} yields  a contravariant {\it Gelfand functors} from 
		\bea\label{gelfand_functor_eqn}
		C^*\text{-}\mathfrak{Gelfand}\xrightarrow{\mathfrak{PreGelfand}}\mathfrak{LocComp},\\\label{gelfand_h_functor_eqn}
			C^*\text{-}\mathfrak{Gelfand}\xrightarrow{\mathfrak{Gelfand} \bydef \mathfrak{Hausdorff}~\circ~\mathfrak{PreGelfand}}\mathfrak{Hausdorff}
		\eea
	the Gelfand category (cf. Definition \ref{gelfand_category_defn}) to the categories   $\mathfrak{LocComp}$ (resp. $\mathfrak{Hausdorff}$ of locally compact (resp. Hausdorff locally compact) spaces and continuous maps.  
		Any functor  $\mathfrak{LocComp}\to \mathscr C$ and/or $\mathfrak{Hausdorff}\to \mathscr C$ naturally yields the functor $ C^*\text{-}\mathfrak{Gelfand}
		\to \mathscr C$. This circumstance is a motivation of the term "algebraic topology of $C^*$-algebras".
	\end{definition}
\begin{empt}
	If $A$ is a $C^*$-algebra and $\sX$ is a locally compact Hausdorff space then algebraic tensor product $A \odot C_0\left(\sX \right)$ has the unique $C^*$-cross norm $\left\| \cdot \right\|\to \R_+$ (cf. Appendix \ref{tensor_sec}). One can prove that this norm is given by
	\be\label{a_x_norm_eqn}
	\left\| a_1 \otimes f_1 + ... + a_1 \otimes f_1\right\|\bydef \sup_{\substack{x \in \sX \\s \in \hat A}} \left\| \rho_s\left(  a_1\right)   f_1(x) + ... +\rho_s\left(  a_n\right)   f_n(x)\right\|
	\ee
	where $\hat A$ is the primitive spectrum of $A$ and $\rho_s$ is the corresponding to $s \in \hat A$ irreducible representation (cf. Definitions \ref{spectrum_prime_primtive_defn} and \ref{irred_defn}).
	There is the natural injective $*$-homomorphism 
	\be\label{a_x_hom_eqn} 
	\begin{split}
	A \odot C_0\left(\sX \right) \hookto C_0\left( \sX, A\right) ,\\
a_1 \otimes f_1 + ... + a_1 \otimes f_1\mapsto \left( x \mapsto a_1f_1(x) + ... + a_n f_n\left(x \right) \right) 
	\end{split} 
	\ee
	where  $C_0\left( \sX, A\right)$ is the $C^*$-algebra of vanishing at infinity continuous maps from $\sX$ to $A$. If $A \otimes C_0\left(\sX \right)$ is the completion of $A \odot C_0\left(\sX \right)$ with respect to $C^*$-norm \eqref{a_x_norm_eqn} then the homomorphism \eqref{a_x_hom_eqn} yield the natural isomorphism 
		\be\label{a_x_iso_eqn} 
	\begin{split}
		A \otimes C_0\left(\sX \right) \cong  C_0\left( \sX, A\right) \\
	\end{split} 
	\ee
	of $C^*$-algebras.
\end{empt}

	\begin{lemma}\label{product_lem}
		For any $C^*$-algebra $A$ and any locally compact Hausdorff space $\sX$ there is the natural homeomorphism
		\be\label{suspension_g_eqn}
		\begin{split}
			\mathfrak{Gelfand}\left(C_0\left(\sX \right)\otimes  A \cong C_0\left(\sX, A \right) \right) \cong \sX \times \mathfrak{Gelfand}\left(A \right)
		\end{split}
		\ee
		where the isomorphism $C_0\left(\sX \right)\otimes  A \cong C_0\left(\sX, A \right)$ is given by the equation \eqref{a_x_iso_eqn}.
	\end{lemma}
	
	\begin{proof}
		If both  $A^\sim$ and $\sX^\sim$ are given by the equations \eqref{unital_notation_eqn} and \eqref{compact_notation_eqn} then both  $A^\sim$ and $C\left( \sX^\sim\right)$ are unital $C^*$-algebras.
	Let $A$ be a $C^*$-algebra.  It follows that both natural injective $*$-homomorphisms 
		\bean
		\varphi_A : A^\sim \hookto C\left(\sX^\sim\right)  \otimes A^\sim,\\ 
		\varphi_\sX : C\left( \sX^\sim\right) \hookto  C\left(\sX^\sim \right) \otimes A^\sim
		\eean
		are good (cf. Definition \ref{good_hom_defn}).
		So the Theorem \ref{good_thm}  yields surjective continuous maps
		\bean
		\varphi_A:		\mathfrak{Gelfand}\left(C\left(\sX^\sim, A^\sim \right) \right)\onto \mathfrak{Gelfand}\left(A^\sim  \right),\\
		\varphi_\sX:		\mathfrak{Gelfand}\left(C\left(\sX^\sim, A^\sim \right) \right)\onto  \sX^\sim.
		\eean
		These maps naturally yield a surjective continuous map  
		\be\label{dir_map_eqn}
		\begin{split}
		\varphi:	\mathfrak{Gelfand}\left(C\left(\sX^\sim, A^\sim\right)\right)   \onto  \sX^\sim \times 	\mathfrak{Gelfand}\left(A^\sim  \right)
		\end{split}
		\ee
		If the map $	\varphi$ is not injective then 
		there are different $\xi', \xi'' \in 	\mathfrak{Gelfand}\left(C\left(\sX^\sim, A^\sim\right)\right)$ such that
		\bea\label{pp_ax_eqn}
		\varphi_A  \left( \xi' \right)= \varphi_A  \left( \xi'' \right), \\ 
		\varphi_\sX  \left( \xi' \right)= \varphi_\sX  \left( \xi'' \right) 
		\eea 
		The space $\mathfrak{Gelfand}\left(C\left(\sX^\sim, A^\sim\right)\right)$ is Hausdorff, so there are open subsets\\ $\sU', \sU''\subset \mathfrak{Gelfand}\left(C\left(\sX^\sim, A^\sim\right)\right)$ with
		\be\label{su_ssu_eqn}
		\begin{split}
	\xi' \in \sU',\\
\xi'' \in \sU'',\\
\sU' \cap \sU'' = \emptyset.
		\end{split}
			\ee
		From the above equations it follows that 
		\bean
		\left(\sU' \cap \varphi^{-1}_\sX\left( x\right)\right)  \bigcap \left(\sU'' \cap \varphi^{-1}_\sX\left( x \right)\right) = \emptyset
		\eean 
		where $x \bydef \varphi_\sX\left( \xi'\right)= \varphi_\sX\left( \xi''\right)$. On the other hand  $\varphi^{-1}_\sX\left( x\right)\cong \mathfrak{Gelfand}\left(A^\sim\right)$ so both $\xi'$ and $\xi''$ can be regarded as elements of  $\mathfrak{Gelfand}\left(A^\sim\right)$. From \eqref{su_ssu_eqn} it turns out that  $\xi'$ and $\xi''$  are different  elements of  $\mathfrak{Gelfand}\left(A^\sim\right)$. This circumstance contradicts with the hypothesis \eqref{dir_map_eqn}. From these contradiction one deduce that the map $\varphi$ is injective, therefore $\varphi$ is bijective. Otherwise both spaces $\mathfrak{Gelfand}\left(C\left(\sX^\sim, A^\sim\right)\right)$ and $\sX^\sim \times 	\mathfrak{Gelfand}\left(A^\sim  \right)$ are compact and Hausdorff, so from the Theorem \ref{comp_to_thm} it follows that the map $\varphi$ is a homeomorphism. This lemma follows from the equation
		\bean
		\varphi\left(C_0\left(\sX \right)\otimes  A \right) = \sX \times \mathfrak{Gelfand}\left(A \right).
		\eean 
		
		\end{proof}
		\begin{empt}
		Let $A$ be an unital $C^*$-algebra. If $a \in A$ is a normal element then the Proposition \ref{sp_comm_prop} yields the natural injective $*$-homomorphism $C_0\left( \mathfrak{Sp}_A\left(x \right)\right) \hookto A$. The Theorem \ref{good_thm} yields the natural surjective continuous map $\varphi : \mathfrak{Gelfand}\left(A  \right) \onto \mathfrak{Sp}_A\left(x \right)$. There is the natural injective $*$-homomorphism $C_0\left( \varphi\right): C_0\left( \mathfrak{Sp}_A\left(x \right)\right) \hookto C\left( \mathfrak{Gelfand}\left(A  \right)\right)$. One can consider $a$ as an element of  $C_0\left( \mathfrak{Sp}_A\left(x \right)\right)$ so one has a map
		\be\label{global_map_eqn}
	\mathfrak{Image}:	a \mapsto b \bydef C_0\left( \varphi\right)\left(a \right) 
		\ee 
		from the set of normal elements of $A$ to $C\left( \mathfrak{Gelfand}\left(A  \right)\right)$. 
		\end{empt}
		
	\begin{lemma}\label{normal_map_lem} if $A$ is an unital $C^*$-algebra then one has:
		\begin{enumerate}
			\item [(i)]  the given by  \eqref{global_map_eqn} map $\mathfrak{Image}$ is surjective,
			\item[(ii)] there is a one-to-one correspondence between  self-adjoint elements $a \in A$ and bounded real-valued functions $b \in  C\left( \mathfrak{Gelfand}\left(A  \right), \R\right)$.
		\end{enumerate}
 	\end{lemma}
	\begin{proof}(i)	If $b \in C\left( \mathfrak{Gelfand}\left(A  \right)\right)$ then $b = b' + i b''$ where both $b'$ and $b''$ are self-adjoint. There are real $c', c'' \in \R_+$ such that both $b' + c'$ and $b'' + c''$ are positive operators. From  $b' + c'+ b'' + c''\ge b' + c'$ and $b' + c'+ b'' + c''\ge b'' + c''$. It follows that there are continuous functions $f', f'' : \R_+ \to \left[0, 1\right]$ such that 
	\bean
b' + c'=	f' \left(b' + c'+ b'' + c'' \right),\\
b'' + c''=	f'' \left(b' + c'+ b'' + c'' \right)
	\eean 
	If $f \bydef f' + i f'' : \R_+ \to \C $ then $b + c' + i c'' = f\left(b' + c'+ b'' + c'' \right)$. The Theorem \ref{spectral_thm}  yields an element $a^{**}\in A^{**}$ which corresponds to $b' + c'+ b'' + c''$. From the Definition \ref{coninuous_p_defn} and the Theorem \ref{coninuous_p_thm} it follows that $a^{**}\in A$. Clearly $a \bydef f\left(a^{**} \right) - c_1 - ic_2\in A$ is normal and $\mathfrak{Image}\left(a \right)= b$.\\
	(ii) If $a \in A$ is a self-adjoint element then form the the Definition \ref{coninuous_p_defn} and the Theorem \ref{coninuous_p_thm} it follows that $a$ corresponds to a continuous function $\mathfrak{PreGelfand}\left(A  \right)\to \R$. However there is a one-to-one correspondence between sets of continuous functions  $\mathfrak{PreGelfand}\left(A  \right)\to \R$ and $\mathfrak{Gelfand}\left(A  \right)\to \R$ (cf. proof of the Lemma \ref{hausdorff_lem}). So $a$ yields a  continuous function $b : \mathfrak{Gelfand}\left(A  \right)\to \R$. 	If both self-adjoint elements   $a', a'' \in  A$ are such that $a' \neq a''$ and $\mathfrak{Image}\left(a' \right)= \mathfrak{Image}\left(a'' \right) = b$ then $a' - a'' \neq 0$ and $\mathfrak{Image}\left(a'' \right) - \mathfrak{Image}\left(a' \right)= \mathfrak{Image}\left(a' - a'' \right)\neq 0$. It contradicts with the hypothesis $\mathfrak{Image}\left(a' \right)= \mathfrak{Image}\left(a'' \right) = b$. 
	\end{proof}
	\begin{remark}
	One has
	\begin{equation}\label{unity_image_eqn}
	\mathfrak{Image}\left(1_A \right)= 1_{C\left( \mathfrak{Gelfand}\left(A  \right)\right)  } 
	\end{equation}
	\end{remark}
\begin{definition}\label{covering_defn}
		Let   $A$ and  $\widetilde{A}$ be  $C^*$-algebras.  An injective good     $*$-homomorphism $\lift: M(A) \hookto M\left( \widetilde{A}\right)$ (cf. Definition 
		\ref{good_hom_defn}) is a \textit{noncommutative covering} if
		the given by \eqref{gelfand_map_eqn}	
		\be\label{covering_a_eqn}
		\mathfrak{Gelfand}\left(\widetilde A \right)\xrightarrow{\mathfrak{Gelfand}\left(\lift \right)}\mathfrak{Gelfand}\left(A \right)
		\ee
		continuous map is a (topological) covering (cf. Definition \ref{top_covering_defn}).
		The  \textit{covering group} is given by  
		\be\label{cov_group_eqn}
		G\left(\left.\widetilde{A} \right| A\right) \bydef \left\{\left.g \in \Aut\left(\widetilde A \right)\right| \forall a \in A \quad M\left( g\right)\lift\left(   a \right) = \lift\left(   a \right) \right\}
		\ee
		where $\Aut\left(\widetilde A \right)$ is the group of $*$-automorphisms, and $M\left( g\right) \in \Aut\left(M\left( \widetilde A \right) \right)$ is the natural extension of $g$. 
		
	\end{definition}
	
	  \begin{definition}\label{connected_c_a_defn}
		We say that a $C^*$-algebra $A$ is \textit{connected} if it cannot be represented as a direct sum  $A \cong A' \oplus A''$ of nontrivial $C^*$-algebras $A'$ and $A''$.
		
	\end{definition}
		\begin{definition}\label{fund_defn}
		If both  $A$ and $\widetilde A$ are a connected $C^*$ -algebras (cf. Definition \ref{connected_c_a_defn})then a noncommutative covering $\lift: M\left( A\right)   \hookto M\left( \widetilde{A}\right) $ is \textit{universal}, if for any noncommutative covering   $\lift': M\left( A\right)   \hookto M\left( \widetilde{A}'\right) $ with connected $\widetilde A'$ there is natural noncommutative covering $\lift'': M\left( \widetilde{A}' \right) \hookto M\left( \widetilde{A}\right)$ with commutative diagram
		\bean
		\begin{tikzcd}
			M\left( \widetilde{A}'\right)  \arrow[rr, "M\left( \lift''\right) "] && M\left( \widetilde{A}\right) \\
			& \arrow[lu, "\lift'"] M\left(A \right)  \arrow[ru, "\lift"] &
		\end{tikzcd}
		\eean
		where $M\left( \lift''\right)$ is the unique extension of $\lift''$.
		If the universal covering exists then  the \textit{fundamental group} of $A$ is given by
		\be\label{fund_eqn}
		\pi_1\left(A \right) \bydef G\left(\left.\widetilde{A} \right| A\right)
		\ee
		where $G\left(\left.\widetilde{A} \right| A\right)$ is the covering group (cf. Definition \ref{covering_defn}).
	\end{definition}
	
	\begin{remark}
		From the Lemma \ref{lem_g_lem}  it follows that in case of commutative $C^*$-algebras the theory of coverings of $C^*$-algebras coincides with the theory of coverings of topological spaces.
	\end{remark}
	\begin{definition}\label{dixmier_douady_defn}
		If both  $A$ and $B$ are $C^*$-algebras and $\varphi :M\left(  A\right) \to M\left( B\right)$ is a good $*$-homomorphism (cf. Definition \ref{good_hom_defn}) then the Theorem \ref{good_thm} it follows that there is a continuous map
		\be\label{gelfand_dd_map_eqn}
		\mathfrak{Gelfand}\left(\varphi \right) :	\mathfrak{Gelfand}\left( B\right) \to \mathfrak{Gelfand}\left( A\right). 
		\ee 
		If $	\mathfrak{Gelfand}\left(\varphi \right) $ is not a homeomorphism then there  is a {\it generalized Dixmier-Douady obstruction}.
	\end{definition}
	\begin{remark}
	The definition \ref{dixmier_douady_defn} is motivated by the discussed in the Section \ref{ctr_dd_sec} construction.
	\end{remark}
	
	\section{Bohr space and Gelfand space}
	\paragraph{}
	 Let $A$ be an unital $C^*$-algebra. The Gelfand spectrum $\ulS(\uA)$ is a topos  (cf. Definition \ref{gelfand_spectrum_defn}) but it is not a topological space. If is naturally define a single  space  which contains the full information about  $\ulS(\uA)$. The natural way to define this space by the following equation 
\be\label{bohr_lim_eqn}
	G\left( A\right) \bydef \lim_{C \in \mathcal{C}(A)} \Sigma \left(C \right) = \varprojlim \Sigma \circ  \functor{\alg{A}} 
\ee
		where the limit  $ \varprojlim$ of functor is explained in the Appendix \ref{limits_of_functors_sec}. The lattice $\mathcal{C}(A)$ is not a directed partially ordered set, so the limit  \eqref{bohr_lim_eqn} should not always exist.		
		However the Kochen-Specker Theorem \ref{prop:kochenspecker} implies a structural collapse. If \(A = \mathbb{M}_n(\mathbb{C})\) for $n \ge 3$ (e.g., standard matrix quantum mechanics), the colimit of \(\mathcal{C}(A)\) in commutative \(C^{*}\)-algebras should collapses to $0$, meaning this limit yields the trivial (empty) locale.  It means that in general  the Kochen-Specker Theorem yields an obstruction to the existence of the given by \eqref{bohr_lim_eqn}. Instead of of the limit  \eqref{bohr_lim_eqn} we will construct the Bohr space which accumulates all spaces $\left\{
	 \Sigma \left(C \right)\right\}_{C \in \mathcal{C}(A)}$.
	 \begin{definition}
	 Let  $A$ be an unital $C^*$-algebra, and let $\mathfrak{Comm}\left( A\right)$ be  the meet-semilattice (cf. Definition \ref{lattice_defn}) of commutative $C^*$-subalgebras of $A$. The {\it pre}-{\it Bohr space} of $A$ (write $\mathfrak{PreBohr}\left(A \right)$ ) is the { space of ultrafilters} of $\mathfrak{Comm}\left( A\right)$ (cf. Definition \ref{lattice_s_defn}).   The \textit{Bohr space} $\mathfrak{Bohr}\left(A \right)$ of $A$ is the universal Hausdorff  quotient space of $\mathfrak{PreBohr}\left(A \right)$ (cf. Definition \ref{blowing_uni_h_defn}), i.e. $\mathfrak{Bohr}\left(A \right)\bydef \mathfrak{Hausdorff}\left(\mathfrak{PreBohr}\left(A \right) \right)$ (cf. equation \eqref{hausdorff_funct_eqn}).
	 \end{definition}
	 For any $C \in  \mathcal{C}(A)$ there is the natural inclusion $ \phi_C : \Sigma \left(C \right)\hookto \mathfrak{Bohr}\left(A \right)$. If $\left\{\sX_\la\subset \mathfrak{Bohr}\left(A \right)\right\}_{\la \in \La}$ is a set of finite unions given by
	 \bean
	 \forall \la \in \La\quad \exists ~C_1, ..., C_n \in \mathcal{C}(A)\quad \sX_\la = \Sigma \left(C_1 \right) \cup ... \cup \Sigma \left(C_n \right)
	 \eean
	 then the set $\left\{\sX_\la\right\}$ is directed (cf. Definition \ref{directed_set_defn}). Moreover one has
	 \be \label{bohr_invlim_eqn}
	 \mathfrak{Bohr}\left(A \right)= \varprojlim_{\la\in \La} \sX_\la
	 \ee
	 (cf. Definitions \ref{colim_defn} and \ref{top_init_defn})
	 The equation \eqref{bohr_lim_eqn} can be regarded as a correct proxy of the \eqref{bohr_lim_eqn} one. From the Theorem \ref{good_thm} it turns out that for any $C \in \mathcal{C}(A)$ there is a surjective continuous map
	 \be\label{bohr_c_eqn} 
	 \mathfrak{Gelfand}\left( A \right)\onto	 \Sigma \left(C \right).
	 \ee
	Using  \eqref{bohr_c_eqn} one can construct following surjective continuous  maps
	\bean
	\forall \la \in \La \quad  \mathfrak{Gelfand}\left( A \right)\onto \sX_\la
	\eean
	and taking into account \eqref{bohr_invlim_eqn} one has the natural surjective continuous map
			\be\label{gelfand_to_bohr_eqn}
	\mathfrak{Gelfand}\left( A \right)\onto \mathfrak{Bohr}\left( A \right)
	\ee

	 		\begin{remark}
		If $A$ is an unital commutative $C^*$-algebra then the natural surjective map \eqref{gelfand_to_bohr_eqn}
	is a homeomorphism.  It is not known whether the map \eqref{gelfand_to_bohr_eqn}is a homeomorphism for any unital  $C^*$-algebra $A$.
		\end{remark}

	\section{Gelfand space and $K$-theory}
		\subsection{Basic constructions}
	\paragraph{}
	Any element of $A\otimes \K$ corresponds to a matrix
	\be\label{inf_m_eqn}
	\begin{split}
		\begin{pmatrix}
			a_{1,1}& \ldots & a_{j,1} &  \dots \\
			\vdots& \ddots & \vdots& \ldots\\
			a_{j,1}& \ldots &a_{j,j} &  \ldots \\
			\vdots& \vdots & \vdots & \ddots \\
		\end{pmatrix}\in \K\left( \ell^2\left( A\right) \right)\in \K\left( L^2\left( \N \right)\right),\\
		\text{or}\\
		\begin{pmatrix}
			a_{1,1}& \ldots & a_{n,1}  \\
			\vdots& \ddots & \vdots\\
			a_{n,1}& \ldots &a_{n,n} \\
		\end{pmatrix}\in \K\left( L^2\left( \{1,...,n\} \right)\right) = \mathbb{M}_n\left(A\right). 
	\end{split}
	\ee
	The diagram \eqref{ba_eqn} yields the exact sequence
		\bean
	\{0\}\hookto K_0\left( A\right) \xrightarrow{\phi} K_0\left(A^+\right)\xrightarrow{\psi}\Z \onto \{0\}
	\eean
	such that
	\be\label{unity_k_eqn}
	\forall x \in K_0\left(A \right)\quad \forall z \in \Z \quad  \psi\left(z \left[1_A\right] + \phi\left(x \right)  \right) = z.
	\ee
		
	\begin{definition}\label{elementary_defn}
		If $A$ is an unital  $C^*$-algebra then a given by \eqref{inf_m_eqn} element $a\in 	 \mathbb{M}_n\left(A \right)$ is  $j,k$-\textit{elementary} if 
		\be\label{elementary_eqn}
		a_{p,q} = \begin{cases}
			1_A & p = j \quad q = k \\
			0 & \text{otherwise}.
		\end{cases}
		\ee
		The  $j,k$-{elementary} element is denoted by
		\be\label{elementary_jk_eqn}
		\mathfrak{e}^{j,k}_A\in \mathbb{M}_n\left(A^+ \right)
		\ee
		
	\end{definition}
	
	\begin{empt}
	If $A$ is an unital $C^*$ algebra then the diagonal map 
	\be\label{diag_A}
	\begin{split}
		d_A : A \hookto \mathbb{M}_n\left(A \right) \cong \K\left(A^n \right) ,\\
		a \mapsto 		\begin{pmatrix}
			a& \ldots & 0  \\
			\vdots& \ddots & \vdots\\
			0& \ldots &a \\
		\end{pmatrix}
	\end{split}
	\ee
	is a good $*$-homomorphism (cf. Definition \ref{good_hom_defn}). From the Theorem \ref{good_thm} it follows that there is a surjective continuous map 
	\be\label{ge_fc_eqn}
	\begin{split}
		\mathfrak{Gelfand}\left(\mathbb{M}_n\left(A \right)\right) \onto \mathfrak{Gelfand}\left(A \right)
	\end{split}
	\ee
	of compact Hausdorff  spaces, which yields the natural injective $*$-homomorphism
	\be\label{fin_g_eqn}
	\begin{split}
		C\left( \mathfrak{Gelfand}\left(A \right)\right)  \hookto C\left( \mathfrak{Gelfand}\left(\mathbb{M}_n\left(A \right)\right) \right).
	\end{split}
	\ee

\end{empt}
\begin{lemma}
	If we consider the equation \eqref{fin_g_eqn} then $C\left( \mathfrak{Gelfand}\left(\mathbb{M}_n\left(A \right)\right) \right)$ is a finitely generated projective 	$C\left( \mathfrak{Gelfand}\left(A \right)\right)$-module. Indeed there is the natural isomorphism of left $C\left( \mathfrak{Gelfand}\left(A\right)\right)$-modules $C\left( \mathfrak{Gelfand}\left(\mathbb{M}_n\left(A \right)\right)\right) \cong C\left( \mathfrak{Gelfand}\left(A\right)\right) ^n$.
\end{lemma}
\begin{proof}
	There is the direct sum of left $\mathbb{M}_n\left(A \right)$-ideals
	$$
\mathbb{M}_n\left(A \right) \cong \bigoplus_{j = 1}^n \mathbb{M}_n\left(A \right) \mathfrak{e}^{j,j}_A.
	$$
	This direct sum  corresponds to decomposition of $\mathfrak{Gelfand}\left(\mathbb{M}_n\left(A \right)\right)$ into disconnected union of open closed subsets. For every $j = 1,..., n$ the left  ideal $\mathbb{M}_n\left(A \right) \mathfrak{e}^{j,j}_A$  yields a hereditary $C^*$-subalgebra $ \mathfrak{e}^{j,j}_A\mathbb{M}_n\left(A \right) \mathfrak{e}^{j,j}_A$ (cf. Lemma \ref{hered_ideal_lem}). There is the natural isomorphism $\mathfrak{e}^{j,j}_A\mathbb{M}_n\left(A \right) \mathfrak{e}^{j,j}_A \cong A$ of $A$-$A$-bimodules. From this isomorphism and the Remark \ref{topology_base_herdem} one can deduce that the corresponding to the ideal $\mathbb{M}_n\left(A \right) \mathfrak{e}^{j,j}_A$ subspace of $\mathfrak{Gelfand}\left(\mathbb{M}_n\left(A \right)\right)$ is homeomorphic to $\mathfrak{Gelfand}\left(A\right)$. From this fact one has the following homeomorphism of spaces and isomorphism of $\mathfrak{Gelfand}\left(A\right)$-modules.
	\bea\label{top_n_eqn}
	\mathfrak{Gelfand}\left(\mathbb{M}_n\left(A \right)\right)\cong \mathfrak{Gelfand}\left(A\right)^n,\\
	\label{alg_n_eqn}
C\left( \mathfrak{Gelfand}\left(\mathbb{M}_n\left(A \right)\right)\right) \cong C\left( \mathfrak{Gelfand}\left(A\right)\right) ^n	
	\eea
\end{proof}
If  
$p \in \mathbb{M}_n\left(A \right)$ is a projector then one has the natural disjoint union of open sets
\be\label{matrppbu_eqn}
\begin{split}
	\mathfrak{Gelfand}\left(\mathbb{M}_n\left(A \right)\right)= \mathfrak{Gelfand}\left(\mathbb{M}_n\left(A \right)\right)_{\mathbb{M}_n\left(A \right)p}\bigsqcup \mathfrak{Gelfand}\left(\mathbb{M}_n\left(A \right)\right)_{\mathbb{M}_n\left(A \right)\left(1_{\mathbb{M}_n\left(A \right)} -p\right) }
\end{split}
\ee
(cf. notation \ref{topology_eqn}). Taking into account the equation  \eqref{matrppbu_eqn} one has a direct sum 
\bean
\begin{split}
	C\left( \mathfrak{Gelfand}\left(\mathbb{M}_n\left(A \right)\right)_{\mathbb{M}_n\left(A \right)p} \right)\bigoplus  C\left(  \mathfrak{Gelfand}\left(\mathbb{M}_n\left(A \right)\right)_{\mathbb{M}_n\left(A \right)\left(1_{\mathbb{M}_n\left(A \right)} -p\right) }\right)= \\=	C\left( \mathfrak{Gelfand}\left(\mathbb{M}_n\left(A \right)\right)\right),  
\end{split}
\eean
Both $C\left( \mathfrak{Gelfand}\left(\mathbb{M}_n\left(A \right)\right)_{\mathbb{M}_n\left(A \right)p} \right)$ and $C\left(  \mathfrak{Gelfand}\left(\mathbb{M}_n\left(A \right)\right)_{\mathbb{M}_n\left(A \right)\left(1_{\mathbb{M}_n\left(A \right)} -p\right) }\right)$ are left $C\left(\mathfrak{Gelfand}\left(A \right) \right)$-modules.
i.e. 	$C\left( \mathfrak{Gelfand}\left(\mathbb{M}_n\left(A \right)\right)_{\mathbb{M}_n\left(A \right)p}\right) $ is a finitely generated $C\left(\mathfrak{Gelfand}\left(A \right) \right)$-module. In result one has the following Theorem 
\begin{theorem}\label{k_0_alg_top_thm}
	If $A$ is an unital $C^*$-algebra then one has.
	\begin{enumerate}
		\item [(i)] There is the natural one-to-one correspondence between isomorphism classes of  finitely generated projective $A$-modules and complex vector bundles on the Gelfand space  $\mathfrak{Gelfand}\left(A \right)$.
		\item[(ii)] There is the natural isomorphism
		\be\label{k_0_alg_top_eqn} 
		K_0\left( A\right) \cong K^0\left( \mathfrak{Gelfand}\left(A \right)\right). 
		\ee 
	\end{enumerate}
\end{theorem}
\begin{proof}
	(i) We already know that is the natural one-to-one correspondence between finitely generated projective $A$-modules and finitely generated projective $C\left(\mathfrak{Gelfand}\left(A \right) \right)$-modules. Our statement follows from the Serre Swan Theorem 	\ref{serre_swan_thm}.
	
	(ii) The Theorem \ref{rosen_ak1_thm} states that the comparison map $c_* : K_0(A) \to  K_0^{\mathrm{top}}(A)$ between the functor $K_0$ of algebraic $K$-theory and functor  $K_0^{\mathrm{top}}(A)$ of $K$-theory of $C^*$-algebras is identity. Now this statement follows from (i).
\end{proof}

\begin{remark}
	The Theorem \ref{k_0_alg_top_thm} yields a generalization of characteristic classes, multiplicative structures and Adams operations.
\end{remark}
\begin{empt}\label{k_0_empt}
Here we prove that $K$-theory of $C^*$-algebras can be formulated in terms of the sheaf cohomology. If  $A$ is an unital $C^*$-algebra then denote by $\mathfrak{FinProj}\left( A\right)$ the set of finitely generated left $A$-modules. For any $P, Q \in \mathfrak{FinProj}\left( A\right)$ we write $P \sim Q$ if $P$ and $Q$ are isomorphic $A$-modules. According to the Appendix \ref{algebraic_k_theory} the group  $K^{\mathrm{alg}}_0\left( A\right)$ of the algebraic $K$-theory is the universal Abelian group such that
	\begin{enumerate}
		\item[(a)]  $\mathfrak{FinProj}\left( A\right)$ is the set of generators of $K^{\mathrm{alg}}_0\left( A\right)$.
		\item[(b)] Foe any pair $P, Q\in \mathfrak{FinProj}\left( A\right)$ satisfy to the following relations:
		\bean
		\left[P \oplus Q\right] = \left[P\right]+\left[Q\right],\\
		P \sim Q \quad \Leftrightarrow \quad \left[P\right]=\left[Q\right]
		\eean
		where $\left[P\right],\left[Q\right], \left[P \oplus Q\right]\in K^{\mathrm{alg}}_0\left( A\right)$ are represented by $P, Q, P\oplus Q$ respectively.
	\end{enumerate}
	For any   $P \in \mathfrak{FinProj}\left( A\right)$ there is a projection $p \in A \otimes \K$ such that
	\bean
	P \sim p\left( A \otimes \K\right) 
	\eean 
	From the   Lemma \ref{normal_map_lem} and Proposition \ref{open_closed} the map $\mathfrak{Image}$ yields the one-to-one correspondence between the sets of projections of both $A \otimes \K$ and $C_0\left( \mathfrak{Gelfand} \left( A \otimes \K\right) \right)$. So one can suppose that  $K^{\mathrm{alg}}_0\left( A\right)$ is generated by either the set of projections of   $A\otimes \K$ or the set of projections of  $C_0\left( \mathfrak{Gelfand} \left( A \otimes \K\right) \right)$. If both $p', p'' \in A\otimes \K$ are such that $p'\left( A \otimes \K\right) \cong p''\left( A \otimes \K\right)$ then there is the natural isomorphism
		\be
\psi: 	p'  \left( A\otimes \K\right)\cong 	p''  \left( A\otimes \K\right)
	\ee
	of $A$-modules. One the other hand there is an  isomorphism
	\bean
	\mathfrak{Image}\left( \psi \right): \mathfrak{Image}\left(	p' \right)C_0\left( \mathfrak{Gelfand} \left( A \otimes \K\right) \right)  \cong \mathfrak{Image}\left(	p'' \right)C_0\left( \mathfrak{Gelfand} \left( A \otimes \K\right) \right) ,\\
	\mathfrak{Image}\left(x \right)\mapsto \mathfrak{Image}\left(\psi \left( x\right)  \right) 
	\eean
	of $C\left( \mathfrak{Gelfand} \left( A \ \right) \right)$-modules, i.e.
	\be\label{p'p''_eqn} 
	\begin{split}
p'  \left( A\otimes \K\right)\cong 	p''  \left( A\otimes \K\right)\quad \Rightarrow\\ \Rightarrow\quad  \mathfrak{Image}\left(	p' \right)C_0\left( \mathfrak{Gelfand} \left( A \otimes \K\right) \right)  \cong \mathfrak{Image}\left(	p'' \right)C_0\left( \mathfrak{Gelfand} \left( A \otimes \K\right) \right).	
	\end{split}
	\ee
	If $q', q'' \in C_0\left( \mathfrak{Gelfand} \left( A \otimes \K\right) \right)$ then from the   Lemma \ref{normal_map_lem} and Proposition \ref{open_closed} it follows that there are unique projections $p', p'' \in A \otimes \K$ with 
	\bean
 \mathfrak{Image}\left(p' \right) = q',\\
 \mathfrak{Image}\left(p'' \right)	= q''.
	\eean 
 Both sets $p'  \left( A\otimes \K\right)$ and $p''  \left( A\otimes \K\right)$ are $\C$-linear envelopes of the following two sets
 \bean
 \left\{a \in A\otimes \K \left| a\quad  \text{is normal } \quad \mathfrak{Image}\left(a  \right) \in q' C_0\left( \mathfrak{Gelfand} \left( A \otimes \K\right) \right)\right.\right\},\\
 \left\{a \in A\otimes \K \left| a\quad  \text{is normal } \quad \mathfrak{Image}\left(a  \right) \in q'' C_0\left( \mathfrak{Gelfand} \left( A \otimes \K\right). \right)\right.\right\}.
 \eean 	
 Taking into account \eqref{p'p''_eqn} one has
 	\bean
 \begin{split}
 	p'  \left( A\otimes \K\right)\cong 	p''  \left( A\otimes \K\right)\quad \Leftrightarrow\\ \Leftrightarrow\quad  \mathfrak{Image}\left(	p' \right)C_0\left( \mathfrak{Gelfand} \left( A \otimes \K\right) \right)  \cong \mathfrak{Image}\left(	p'' \right)C_0\left( \mathfrak{Gelfand} \left( A \otimes \K\right) \right).	
 \end{split}
 \eean 
 From these circumstances it follows that the group of $K^{\mathrm{alg}}_0\left( A\right)$ af the algebraic $K$-theory is the universal Abelian group such that
 \begin{enumerate}
 	\item[(a)]  the set of generators of $K^{\mathrm{alg}}_0\left( A\right)$ is the set of projections of $C_0\left( \mathfrak{Gelfand} \left( A \otimes \K\right) \right)$
 	\item[(b)] Any pair of projections of $q', q'' \in C_0\left( \mathfrak{Gelfand} \left( A \otimes \K\right) \right)$ satisfy to the following relations:
 	\bean
 	\left[q'  C_0\left( \mathfrak{Gelfand} \left( A \otimes \K\right) \right) \oplus q''  C_0\left( \mathfrak{Gelfand} \left( A \otimes \K\right) \right)\right] =\\= \left[q'  C_0\left( \mathfrak{Gelfand} \left( A \otimes \K\right) \right)\right]+\left[q''  C_0\left( \mathfrak{Gelfand} \left( A \otimes \K\right) \right)\right],\\
 	q'  C_0\left( \mathfrak{Gelfand} \left( A \otimes \K\right) \right) \sim q''  C_0\left( \mathfrak{Gelfand} \left( A \otimes \K\right) \right) \quad \Leftrightarrow\\ \Leftrightarrow\quad \left[q'  C_0\left( \mathfrak{Gelfand} \left( A \otimes \K\right) \right)\right]=\left[q''  C_0\left( \mathfrak{Gelfand} \left( A \otimes \K\right) \right)\right]
 	\eean
 	where elements
 	\bean
 	\left[q'  C_0\left( \mathfrak{Gelfand} \left( A \otimes \K\right) \right)\right],~\left[q''  C_0\left( \mathfrak{Gelfand} \left( A \otimes \K\right) \right)\right], \\\left[q'  C_0\left( \mathfrak{Gelfand} \left( A \otimes \K\right) \right) \oplus q''  C_0\left( \mathfrak{Gelfand} \left( A \otimes \K\right) \right)\right]\in K^{\mathrm{alg}}_0\left( A\right)
 	\eean  are represented by following projective   $C\left( \mathfrak{Gelfand} \left( A \right) \right)$-modules: 
 		\bean
 q'  C_0\left( \mathfrak{Gelfand} \left( A \otimes \K\right) \right),~q''  C_0\left( \mathfrak{Gelfand} \left( A \otimes \K\right) \right), \\q'  C_0\left( \mathfrak{Gelfand} \left( A \otimes \K\right) \right) \oplus q''  C_0\left( \mathfrak{Gelfand} \left( A \otimes \K\right) \right)\in K^{\mathrm{alg}}_0\left( A\right)\subset C_0\left( \mathfrak{Gelfand} \left( A \otimes \K\right) \right).
 	\eean
 \end{enumerate}
 According to the  Shilov idempotent Theorem (cf. Appendix \ref{commutative_k_empt}) any projection in $C_0\left(\mathfrak{Gelfand}\left(A\otimes \K \right) \right)$ corresponds to an element of the group of the sheaf cohomology 
 \bean 
 H^0_c\left(\mathfrak{Gelfand}\left( A\otimes \K \right); \Z \right)\cong  C_c\left(\mathfrak{Gelfand}\left( A\otimes \K \right); \Z \right)
 \eean 
 with compact support (cf. \ref{commutative_k_empt}).
\end{empt}
\begin{theorem}\label{k_0_thm}
	If $A$ is an unital $C^*$-algebra then there are following  natural isomorphisms
	\be\label{k_00_eqn}
	K_0\left( A\right) \cong K_0\left( C\left( \mathfrak{Gelfand}\left(A \right)\right) \right)\cong  H^0_c\left(\mathfrak{Gelfand}\left( A\otimes \K \right); \Z \right)/\sim
	\ee
	where the equivalence relation $\sim$ comes from isomorphism of $C\left( \mathfrak{Gelfand}\left(A \right)\right)$-modules.	
\end{theorem}
\begin{proof}
The construction \ref{k_0_empt} yields the natural isomorphisms
\bean
	K^{\text{alg}}_0\left( A\right) \cong K^{\text{alg}}_0\left( C_0\left( \mathfrak{Gelfand}\left(A \right)\right)\right) \cong K^0\left( \mathfrak{Gelfand}\left(A \right)\right)\cong H^0_c\left(\mathfrak{Gelfand}\left( A\otimes \K \right); \Z \right)/\sim.
\eean 
where $K^{\text{alg}}$ means the algebraic $K$-theory (cf. Appendix \ref{algebraic_k_theory}). This lemma follows from the Theorem \ref{rosen_ak2_thm}.
\end{proof}
\begin{remark}
	From the Theorem \ref{ga_proj_thm} it follows that that the group  $U\left(A\otimes \K \right)$  acts on  the set $\mathrm{Pr}\left(\left(A\otimes \K \right) ^+ \right)\cong H^0_c\left(\mathfrak{Gelfand}\left( A\otimes \K \right); \Z \right)$, i.e.
	\be\label{g_act_eqn}
	\begin{split}
		U\left(\left(A\otimes \K \right) ^+ \right) \times \mathrm{Pr}\left(A\otimes \K \right) \to \mathrm{Pr}\left(A\otimes \K \right),\\
				U\left(\left(A\otimes \K \right) ^+ \right) \times H^0_c\left(\mathfrak{Gelfand}\left( A\otimes \K \right); \Z \right) \to H^0_c\left(\mathfrak{Gelfand}\left( A\otimes \K \right); \Z \right).
	\end{split}
	\ee
	 Taking into account the Proposition \ref{projection_u_equivalence_prop} one can replace the Theorem \ref{k_0_thm} with the following equation
	\be\label{k_0_eqn}
	K_0\left( A\right) \cong K_0\left( C\left( \mathfrak{Gelfand}\left(A \right)\right) \right)\cong  H^0_c\left(\mathfrak{Gelfand}\left( A\otimes \K \right); \Z \right)/	U\left(\left(A\otimes \K \right) ^+ \right).
	\ee
\end{remark}
\begin{rem}\label{inv_h1em}
	The group  $U\left(\left(A\otimes \K \right) ^+ \right)$ acts on $A\otimes \K$ by internal isomorphisms. So one has the natural action.
	$$
	U\left(\left(A\otimes \K \right) ^+ \right)\times \mathfrak{Gelfand}\left(A\otimes \K \right)\to\mathfrak{Gelfand}\left(A\otimes \K \right).$$ This action yields the trivial action
	$$
	U\left(\left(A\otimes \K \right) ^+ \right)\times H^1_c\left(\mathfrak{Gelfand}\left(A \otimes \K\right); \Z \right)\to H^1_c\left(\mathfrak{Gelfand}\left(A \otimes \K\right); \Z \right).$$ 	
\end{rem}

		\begin{lemma}\label{ring_auto_lem}  
	Any $g \in U\left(\left(A\otimes \K \right) ^+ \right)$ yields a ring automorphism of $H^0_c\left(\mathfrak{Gelfand}\left(A\otimes \K \right)\right) $ with respect to the cup-product (cf. Definition \ref{cup_sheaf_thm}).
\end{lemma}	

\begin{proof}
	The group $H^0_c\left(\mathfrak{Gelfand}\left(A\otimes \K \right); \Z \right)$ is an Abelian group generated by open-closed subsets of $\mathfrak{Gelfand}\left(A\otimes \K \right)$ with the relations
	\be\label{plus_u_eqn}
	\begin{split}
		\sU' \cap \sU''= \emptyset \quad \Rightarrow \quad \left[\sU'\right]+ \left[\sU''\right]= \left[\sU'\cup \sU''\right].
	\end{split}
	\ee
	The cup product satisfies to the following equation
	\be\label{cup_u_eqn}
	\left[\sU'\right]\smile \left[\sU''\right]= \left[\sU'\cap \sU''\right]
	\ee
	Element $g$ yields a homeomorphism of $\mathfrak{Gelfand}\left(A\otimes \K \right)$, so from 
	\bean
	g\left( \sU'\cap \sU''\right)  = g\sU'\cap g\sU''
	\eean 
	it turns that $g$ yields a ring automorphisms.
\end{proof}
	\begin{theorem}\label{k_h_thm}
	If $A$ is an unital $C^*$-algebra then one has
	\bea\label{k_00p_eqn}
	K_0\left( A\right) =H^0_c\left(\mathfrak{Gelfand}_r\left(A\otimes \K \right); \Z \right)/U\left(\left(A\otimes \K \right) ^+ \right).
	\eea
	There is the natural cup product on $H^0_c\left(\mathfrak{Gelfand}_r\left(A\otimes \K \right); \Z \right)/U\left(\left(A\otimes \K \right) ^+ \right)$ which comes from the cup-product on  $H^0_c\left(\mathfrak{Gelfand}_r\left(A\otimes \K \right); \Z \right)$.
\end{theorem}
\begin{proof}
	The $C^*$-algebra $A$ is strongly Morita equivalent to $ A\otimes \K$, so the equation \eqref{k_00p_eqn} follows from the \eqref{k_00_eqn} one. This lemma is the direct consequence of the  Lemma \ref{ring_auto_lem}.
\end{proof}
\
		\begin{empt}
	Let $A$ be  an unital  $C^*$-algebra. From \eqref{stab_k1_eqn} it follows that 
	\bean
	K_1\left(A \right) \cong U_1\left( A \otimes \K \right) / U_1\left( A \otimes \K \right)_0.
	\eean
	If $u \in U_1\left( A \otimes \K\right) $  there is the given by \eqref{diag_eqn} continuous map
	\bean
	\begin{split}
	\Phi: [0,1] \to \mathbb{M}_2 \left(A \otimes \K\right) \subset \mathbb{M}_\infty \left(A \otimes \K\right),\\
t \mapsto	\begin{pmatrix}
		u & 0\\
		0 & 1_{\left( A \otimes \K\right)^+ }
	\end{pmatrix} 
	\begin{pmatrix}
		\cos \frac{\pi}{2} t &  -\sin \frac{\pi}{2} t\\
		\sin \frac{\pi}{2} t &  \cos \frac{\pi}{2} t.
	\end{pmatrix}
	\begin{pmatrix}
		u^* & 0\\
		0 & 1_{\left( A \otimes \K\right)^+ }
	\end{pmatrix}	\begin{pmatrix}
		\cos \frac{\pi}{2} t &  -\sin \frac{\pi}{2} t\\
		\sin \frac{\pi}{2} t &  \cos \frac{\pi}{2} t.
	\end{pmatrix}
	\end{split}
	\eean 
	such that 
	\bean
	\Phi\left(0 \right) = \begin{pmatrix}
		1_{\left( A \otimes \K\right)^+ } & 0\\
		0 & 1_{\left( A \otimes \K\right)^+ }
	\end{pmatrix}\\
	\Phi\left(1 \right) = \begin{pmatrix}
		u & 0\\
		0 & u^*
	\end{pmatrix}
	\eean
	Similarly to the proof of the Theorem \ref{suspension_thm}
three is a set
	\bean
	e_t \bydef \Phi\left( t\right) \begin{pmatrix}
		1_{\left( A \otimes \K\right)^+} & 0\\
		0 &0
	\end{pmatrix}\Phi^{-1}\left(t \right) \quad t \in \left[0,1\right].
	\eean 
According the proof of the Theorem \ref{suspension_thm} the family  $e \bydef (e_t)$ yields is an idempotent in $\mathbb{M}_{2}\left(\left( S\left( A \otimes \K\right)^+\right)^+  \right)$, so one has a map
	\be\label{u_to_k_0eqn}
	U_1\left( \left( A \otimes \K\right)^+ \right) / U_1\left( A \otimes \K \right)_0 \to K_0\left(\left( SA \otimes \K\right)^+ \right) 
	\ee 
	The diagram  \eqref{ba_eqn} yields the following exact sequence
	\be\label{u_to_ex_0eqn}
	\{0\}\hookto K_0\left( SA \otimes \K\right) \hookto K_0\left( \left( SA \otimes \K\right)^+\right)\xrightarrow{\phi}\Z \onto \{0\}
	\ee
		According to the Theorem \ref{suspension_thm}  the equation \eqref{u_to_k_0eqn} yields the isomorphism 
	\be\label{u_to+p_k_0eqn}
	U_1\left( A \otimes \K \right) / U_1\left( A \otimes \K \right)_0 \cong \ker \phi\cong K_0\left( SA \otimes \K\right)
	\ee 
	of Abelian groups.
	If  $\mathfrak{Image}\left( u\right)  \in U_1\left(C \left(\mathfrak{Gelfand}\left(  A \otimes \K\right)^+\right)\right)  $ is a given by the Lemma \ref{normal_map_lem}/  then the equation  \eqref{diag_eqn} yields a continuous map
	\bean
	\Psi: [0,1] \to \mathbb{M}_2\left( C \left(\mathfrak{Gelfand}\left(  A \otimes \K\right)^+\right)\right),\\
	t \mapsto	\begin{pmatrix}
		\mathfrak{Image}\left( u\right) & 0\\
		0 & 1
	\end{pmatrix} 
	\begin{pmatrix}
		\cos \frac{\pi}{2} t &  -\sin \frac{\pi}{2} t\\
		\sin \frac{\pi}{2} t &  \cos \frac{\pi}{2} t.
	\end{pmatrix}
	\begin{pmatrix}
		\mathfrak{Image}\left( u^*\right) & 0\\
		0 & 1
	\end{pmatrix}	\begin{pmatrix}
		\cos \frac{\pi}{2} t &  -\sin \frac{\pi}{2} t\\
		\sin \frac{\pi}{2} t &  \cos \frac{\pi}{2} t.
	\end{pmatrix}
	\eean 
	such that 
	\bean
	\Psi\left(0 \right) = \begin{pmatrix}
		1 & 0\\
		0 & 1
	\end{pmatrix}\\
	\Psi\left(1 \right) = \begin{pmatrix}
		\mathfrak{Image}\left( u\right) & 0\\
		0 & \mathfrak{Image}\left( u^*\right)
	\end{pmatrix}
	\eean
	Similarly to the proof of the Theorem \ref{suspension_thm}
if 
	\bean
	f_t \bydef \Psi\left( t\right) \begin{pmatrix}
		1_{\Psi} & 0\\
		0 &0
	\end{pmatrix}\Psi \left( t\right)^{-1} 
	\eean 
	then the family $f \bydef (f_t)$ is an idempotent in $\mathbb{M}_{2}\left( S\left(C \left(\mathfrak{Gelfand}\left(  A \otimes \K\right)^+\right)\right)^+  \right)$.  
	Similarly to the equations \eqref{u_to_k_0eqn}, \eqref{u_to_ex_0eqn} and \eqref{u_to+p_k_0eqn} one has following homomorphisms
		\be\label{gu_to_k_0eqn}
		\begin{split}
	U_1\left( C\left( \mathfrak{Gelfand}\left(  A \otimes \K\right)^+\right)  \right) / U_1\left( C\left( \mathfrak{Gelfand}\left(  A \otimes \K\right)^+\right)  \right)_0 \to\\\to K_0\left( S\left(  C\left( \mathfrak{Gelfand}\left(  A \otimes \K\right)^+\right)\right)^+ \right) ,\\
	\{0\}\hookto K_0\left( S C\left( \mathfrak{Gelfand}\left(  A \otimes \K\right)^+\right) \right) \hookto \\ \hookto K_0\left( S\left(  C\left( \mathfrak{Gelfand}\left(  A \otimes \K\right)^+\right)\right)^+ \right)\xrightarrow{\psi}\Z \onto \{0\},\\
	U_1\left( C\left( \mathfrak{Gelfand}\left(  A \otimes \K\right)^+\right)\right)/ U_1\left( C\left( \mathfrak{Gelfand}\left(  A \otimes \K\right)^+\right)\right)_0  \cong \\\cong  \ker \psi\cong  K_0\left( S C\left( \mathfrak{Gelfand}\left(  A \otimes \K\right)^+\right) \right) 
\end{split}
	\ee 
	One has $S C\left( \mathfrak{Gelfand}\left(  A \otimes \K\right)^+\right)= C_0\left( S\mathfrak{Gelfand}\left(  A \otimes \K\right)^+\right)$ 
From the Lemma \ref{product_lem} it follows that $S\mathfrak{Gelfand}\left(  A \otimes \K\right)^+\cong \mathfrak{Gelfand}\left( S\left(  A \otimes \K\right)^+\right) $. From the Corollary \ref{one_point_cor} it turns out that 
$$
 C_0\left(\mathfrak{Gelfand}\left( S\left(  A \otimes \K\right)^+\right)\right)^+ \cong C\left( \mathfrak{Gelfand}\left( S\left(  A \otimes \K\right)^+\right)^+\right) 
 $$ Taking into account the Theorem \ref{k_0_thm} one has
 $$
 K_0\left(C\left( \mathfrak{Gelfand}\left( S\left(  A \otimes \K\right)^+\right)^+\right) \right) \cong  K_0\left( S\left( \left(  A \otimes \K\right)^+\right)^+ \right)
 $$
From the above calculation one can obtain the following isomorphism 
$$
\varphi: K_0\left( S\left(  C\left( \mathfrak{Gelfand}\left(  A \otimes \K\right)^+\right)\right)^+ \right)\cong K_0\left( S\left( \left(  A \otimes \K\right)^+\right)^+ \right).
$$
The equations \eqref{u_to_ex_0eqn} and \eqref{gu_to_k_0eqn} contain following exact sequences 
\bean
	\{0\}\hookto K_0\left( SA \otimes \K\right) \hookto K_0\left( \left( SA \otimes \K\right)^+\right)\xrightarrow{\phi}\Z \onto \{0\},\\
		\{0\}\hookto K_0\left( S C\left( \mathfrak{Gelfand}\left(  A \otimes \K\right)^+\right) \right) \hookto \\ \hookto K_0\left( S\left(  C\left( \mathfrak{Gelfand}\left(  A \otimes \K\right)^+\right)\right)^+ \right)\xrightarrow{\psi}\Z \onto \{0\}
\eean
From the equations \ref{unity_image_eqn} and \eqref{unity_k_eqn} one can deduce that the diagram
\bean
\begin{tikzcd}
 K_0\left( S\left(  C\left( \mathfrak{Gelfand}\left(  A \otimes \K\right)^+\right)\right)^+ \right) \arrow[r, "\psi"] \arrow[d, "\varphi"]& \Z \arrow[d, "\cong"]\\
K_0\left( \left( SA \otimes \K\right)^+\right) \arrow[r, "\phi"]& \Z
\end{tikzcd}
\eean 
 is commutative. 
The above exact sequences, the diagram and the five homomorphism lemma yield the natural isomorphism 
\be\label{k_0_iso_eqn}
 K_0\left( S\left(\left(  A \otimes \K\right)^+\right) \right)\cong K_0\left( S C\left( \mathfrak{Gelfand}\left(  A \otimes \K\right)^+\right) \right). 
\ee
From the equations \eqref{u_to+p_k_0eqn}, \eqref{gu_to_k_0eqn} it follows that there is the natural isomorphism
		\be\label{uu_iso_0eqn}
\begin{split}
	U_1\left( A \otimes \K \right) / U_1\left( A \otimes \K \right)_0	\cong\\\cong 	U_1\left( C\left( \mathfrak{Gelfand}\left(  A \otimes \K\right)^+\right)\right)/ U_1\left( C\left( \mathfrak{Gelfand}\left(  A \otimes \K\right)^+\right)\right)_0
\end{split}
\ee 

\end{empt}
\begin{lemma}\label{k_1_coh_lem}
For any unital $C^*$-algebra $A$ there is the natural isomorphism
\be\label{k_1_coh_eqn}
K_1\left(A \right) \cong H^1_c\left( \mathfrak{Gelfand}\left(  A \otimes \K\right); \Z \right). 
\ee
\end{lemma}
\begin{proof}
The equations \eqref{uu_iso_0eqn} and \eqref{stab_k1_eqn} yield the natural isomorphism 
$$
K_1\left(A \right) \cong 	U_1\left( C\left( \mathfrak{Gelfand}\left(  A \otimes \K\right)^+\right)\right)/ U_1\left( C\left( \mathfrak{Gelfand}\left(  A \otimes \K\right)^+\right)\right)_0.
$$
Now isomorphism \eqref{k_1_coh_eqn} is a consequence of the equations \eqref{local_eqn} and \eqref{royden_eqn}.
\end{proof}
\begin{theorem}\label{k_gelfand_thm}
	For any  unital $C^*$-algebra $A$ one has
	\be\label{k_gelfand_eqn}
	\forall j = 0, 1 \quad K_j\left(A \right) \cong K_j\left(C_0\left(\mathfrak{Gelfand}\left(  A \otimes \K\right) \right)  \right)\cong K^j\left(\mathfrak{Gelfand}\left(  A \otimes \K\right) \right).
	\ee
\end{theorem}
\begin{proof} Suppose that $A$ is unital. From the equation \eqref{k_0_iso_eqn} it follows that 
	\bean
	K_0\left( S\left(\left(  A \otimes \K\right)^+\right) \right)\cong K_0\left( S C\left( \mathfrak{Gelfand}\left(  A \otimes \K\right)^+\right) \right) . 
	\eean 
	and taking into account the Theorem \ref{suspension_thm} one has
	\bean
		K_1\left( \left(  A \otimes \K\right)^+ \right)\cong K_1\left(  C\left( \mathfrak{Gelfand}\left(  A \otimes \K\right)^+\right) \right)\cong  K^1\left(   \mathfrak{Gelfand}\left(  A \otimes \K\right)^+ \right). 
	\eean
	From the equation \eqref{k_1_coh_eqn} and the Lemma \ref{gelfand_u_compact_lem} it follows that 
	\bean
	K_1\left( \left(  A \otimes \K\right)^+ \right)\cong H_c^1\left( \mathfrak{Gelfand}\left(  A \otimes \K\right)^+ ;\Z\right)\cong H^1\left( \mathfrak{Gelfand}\left(  A \otimes \K\right)^+ ; \Z\right) 
	\eean 
	From $\mathfrak{Gelfand}\left(  A \otimes \K\right) = \mathfrak{Gelfand}\left(  A \otimes \K\right)\cup \{x_0\}$ and the exact sequence \eqref{long_pair_eqn} one can deduce
	\bean
H^1\left( \mathfrak{Gelfand}\left(  A \otimes \K\right)^+ ;\Z\right) \cong H_c^1\left( \mathfrak{Gelfand}\left(   A \otimes \K\right);\Z \right)\quad \Rightarrow\\	K^1\left(   \mathfrak{Gelfand}\left(  A \otimes \K\right)^+ \right)\cong 	K^1\left(   \mathfrak{Gelfand}\left(  A \otimes \K\right)\right),\\
	\eean 
	In result one has 
	$$
K_1\left(A \otimes \K\right) \cong K_1\left(C_0\left(\mathfrak{Gelfand}\left(  A \otimes \K\right) \right)  \right)\cong K^1\left(\mathfrak{Gelfand}\left(  A \otimes \K\right) \right)
	$$
	From $K_1\left(A \otimes \K\right) \cong K_1\left(A \right)$ it follows that
	\be\label{k_1t_eqn}
	K_1\left(A \right) \cong K_1\left(C_0\left(\mathfrak{Gelfand}\left(  A \otimes \K\right) \right)  \right)\cong K^1\left(\mathfrak{Gelfand}\left(  A \otimes \K\right) \right)
	\ee
	From $K_1\left( SA \right)  \cong K_0\left( A\right)$ and taking into account \eqref{k_1t_eqn} one has
	\bean
	K_0\left(A \right)\cong K^1\left(S \mathfrak{Gelfand}\left(  A \otimes \K\right) \right)  
	\eean 
	However
	\bean
	 K^1\left(S \mathfrak{Gelfand}\left(  A \otimes \K\right) \right)\cong  K^0\left( \mathfrak{Gelfand}\left(  A \otimes \K\right) \right).
	\eean 
	so one has
	\be\label{k_0t_eqn}
K_0\left(A \right) \cong K_0\left(C_0\left(\mathfrak{Gelfand}\left(  A \otimes \K\right) \right)  \right)\cong K^0\left(\mathfrak{Gelfand}\left(  A \otimes \K\right) \right)
\ee
The equation \eqref{k_gelfand_eqn}	is the combination of the \eqref{k_1t_eqn}, \eqref{k_0t_eqn} ones.	
\end{proof}
\begin{corollary}\label{k_gelfand_cor}
	For any   $C^*$-algebra $A$ one has
	\be\label{k_gelfand_c_eqn}
	\forall j = 0, 1 \quad K_j\left(A \right)  \cong K^j\left(\mathfrak{Gelfand}\left(  A \right) \right).
	\ee
\end{corollary}
\begin{proof}
	If $A$ is unital then the statement is the direct consequence of the Theorem \ref{k_gelfand_thm} and the following strong Morita equivalences
	\bean
A \sim 	A \otimes \K,\\
C\left(\mathfrak{Gelfand}\left(  A \right) \right) \sim C\left( \mathfrak{Gelfand}\left(  A \right)\right) \otimes \K.
	\eean 
	
	 If $A$ is not unital then $A^+$ is unital, and there is the diagram
	\bean
	\begin{tikzcd}
		K_0\left( A\right) \arrow[r, hook, "\phi'"] \arrow[d, dashed, "\varphi"]& K_0\left( A^+\right)  \arrow[d, "\kappa"]\\
		K^0\left( \mathfrak{Gelfand}\left( A\right) \right) \arrow[r, hook,"\phi''"]& K^0\left( \mathfrak{Gelfand}\left( A\right)^+ \right) 
	\end{tikzcd}
	\eean 
	where the existence of the isomorphism $\kappa$ is already proven. One needs prove thee existence of $\varphi$. Using \eqref{ba_eqn} on can extend the above diagram up to following one,
	\be\label{k_dia_g_eqn}
	\begin{tikzcd}
		K_0\left( A\right) \arrow[r, hook, "\phi'"] \arrow[d, dashed, "\varphi"]& K_0\left( A^+\right)  \arrow[d, "\kappa"] \arrow[r, two heads, "\psi'"]& K_0\left(\C \right)\arrow[d,  "\cong"] \cong \Z\\
		K^0\left( \mathfrak{Gelfand}\left( A\right) \right) \arrow[r, hook,"\phi''"]& K^0\left( \mathfrak{Gelfand}\left( A\right)^+ \right)\arrow[r,two heads, "\psi''"] & K^0\left(\{+\} \right)\cong \Z
	\end{tikzcd}
	\ee
	where $\{+\} \bydef \mathfrak{Gelfand}\left( A\right)^+\setminus \mathfrak{Gelfand}\left( A\right)$. The commutativity of the right square of \eqref{k_dia_g_eqn} follows from the equations
	\bean
	\kappa\left(\left[A^+\right]\right) = \left[\mathfrak{Gelfand}\left( A\right)^+\right],\\
	\forall z \in \Z \quad \forall x \in K_0\left( A\right) \quad \psi'\left(\phi'\left( x\right)  + z \left[A^+\right]\right) = z\left[1_\C\right],\\
	\forall z \in \Z \quad \forall x \in K^0\left( \mathfrak{Gelfand}\left( A\right) \right) \quad \psi'\left(\phi'\left( x\right)  + z \left[C\left(\mathfrak{Gelfand}\left( A\right)^+  \right) \right]\right) = z\left[1_{C\left(\{+\} \right) }\right].
	\eean 
	The existence of $\varphi$ follows from the  commutativity of the right square of \eqref{k_dia_g_eqn}. From the 5-lemma it turns out that $\varphi$ is an isomorphism, i.e.
	\bean
	\varphi :	K_0\left( A\right)\cong K^0\left( \mathfrak{Gelfand}\left( A\right) \right)
	\eean 
	For any $C^*$-algebra $B$ the isomorphism $K_1\left(B\right)\cong K^1\left( \mathfrak{Gelfand}\left(B\right) \right)$ follows from the Theorem \ref{suspension_thm}
	and the isomorphisms $K_1\left( SB\right)\cong K_0\left( SC_0\left( \mathfrak{Gelfand}\left( B\right)\right)  \right)\cong K^0\left( S\mathfrak{Gelfand}\left( B\right)\right)$.
\end{proof}

\begin{theorem}\label{full_k_thm}
For any unital $C^*$-algebra $A$ and every	 $k \ge 0$ one has
	\bea\label{k0_s_eqn}
	K_{0}\left(A \right) \cong H^1_c\left(S^{2k +1}\mathfrak{Gelfand}\left(A \otimes \K\right); \Z \right)\cong H^1_c\left(\mathfrak{Gelfand}\left(A \otimes \K\right)\times \R^{2k + 1}; \Z \right),\\\label{k1_s_eqn}
	K_{1}\left(A \right)  \cong H^1_c\left(S^{2k}\mathfrak{Gelfand}\left(A \otimes \K\right); \Z \right)\cong H^1_c\left(\mathfrak{Gelfand}\left(A \otimes \K\right)\times \R^{2k}; \Z \right)
	\eea
	or equivalently  
	\be\label{k01_s_eqn} 
	K_{r+1~ \mathrm{mod}~2 }\left(A \right) \cong H^1_c\left(S^{r}\mathfrak{Gelfand}\left(A \otimes \K\right); \Z \right)\cong H^1_c\left(\mathfrak{Gelfand}\left(A \otimes \K\right)\times \R^{r}; \Z \right)
	\ee
	for each $r \ge 0$.
\end{theorem}
\begin{proof}
If $A$ is an unital $C^*$-algebra then from the equation \eqref{k_1_coh_eqn} it follows that 
	\bean
	K_{1}\left(A \right) \cong H^1_c\left(\mathfrak{Gelfand}\left(A \otimes \K\right); \Z \right). 
	\eean
	From the Theorems \ref{suspension_thm} and \ref{bott_map_thm} and the Lemma \ref{product_lem} it follows that 
	\bean
	\forall k \in \N \quad K_{1}\left(A \right) \cong K_{1}\left(S^{2k} A \right) \cong K_{1}\left(A \otimes C_0\left( \R^{2k} \right) \right) \cong H^1_c\left(\mathfrak{Gelfand}\left(A \otimes \K\right)\times \R^{2k}; \Z \right). 
	\eean
	The equation \eqref{k0_s_eqn} can be proven by similar way, applying the  isomorphisms \eqref{ba_eqn} and the Theorem \ref{suspension_thm}
\end{proof}
\subsection{Multiplicative structures}
\paragraph{} The Appendix \ref{mult_k_sec} describes the cup  product on the topological $K$-theory given by 
\be\label{k_top_p_eqn}
\begin{split}
\smile : K^p\left( \sX \right)\times  K^q\left( \sX \right) \to K^{p+q~ \mathrm{mod}~2}\left( \sX \right)
\end{split}
\ee
(cf \eqref{kfn_prod_eqn}). Here for any unital $C^*$-algebra we find a generalization 

\be\label{cup_prodcc_eqn}
\smile : K_p\left(A \right) \otimes K_q\left(A \right) \to K_{p+q~ \mathrm{mod}~2}\left(A \right) \quad \forall p, q \in \{0, 1\}.
\ee
of the product \eqref{k_top_p_eqn}
According to \cite{bredon:sheaf} one has
\be\label{kunnet_sheafrr_eqn}
\begin{split}
	H^p_c\left(\R^r; \Z \right)= \begin{cases}
		0 & p \neq r,\\
		1 & p = r
	\end{cases}
\end{split}
\ee
If $\sX$ is a locally compact Hausdorff space then Theorem \ref{kunnet_sheaf_thm} it follows that
\be\label{kunnet_sheafr_eqn}
\begin{split}
	\bigoplus_{p+q = n} H^p_c\left(\sX ; \Z \right) \otimes H^q_c\left(\R^r; \Z \right)\hookto H^{n}_c\left(\sX \times \R^{r}; \Z \otimes \Z \right)\onto\\ \onto  \bigoplus_{p+q = n+1} \mathrm{Tor}\left( H^p_c\left(\sX ; \Z \right),  H^p_c\left(\R^r; \Z \right) \right). 
\end{split}
\ee
and taking into account \eqref{kunnet_sheafr_eqn} we have
\be\label{shift_hom_eqn}
H^p_c\left( \sX \times \R^k\right) = H^{p + k}_c\left( \sX\right).
\ee

If $A$ is a {\bf stable} $C^*$-algebra then from the Lemma \ref{product_lem} and the Theorem \ref{full_k_thm} it turns out that
\bean
K_1\left(A \otimes C_0\left(\R^r \right) \right)\cong  H^1_c\left(\mathfrak{Gelfand}\left(A\right) \times \R^{r'}; \Z \right),
\eean 
For any $r', r'' > 0$ there are surjective continuous projections 
\bean
p_1 : \mathfrak{Gelfand}\left(A\right) \times \R^{r'+r''}\onto \mathfrak{Gelfand}\left(A\right) \times \R^{r'},\\
p_2 : \mathfrak{Gelfand}\left(A\right) \times \R^{r'+r''}\onto \mathfrak{Gelfand}\left(A\right) \times \R^{r''}
\eean 
If both $\mathscr A'$ and  $\mathscr A''$ are $\Z$-locally constant sheaves (cf. Definition \ref{constant_presheaf_defn}) on $\mathfrak{Gelfand}\left(A\right) \times \R^{r'}$ and $\mathfrak{Gelfand}\left(A\right) \times \R^{r''}$ then inverse images $p_1^*\mathscr A'$ and $p_2^*\mathscr A'$ respectively (cf. Definition \ref{sheaf_inv_im_defn})  are $\Z$-locally constant sheaves on $\mathfrak{Gelfand}\left(A\right) \times \R^{r'+r''}$. There are given by the equation \eqref{cohom_mor_eqn} group homomorphisms
\bean
H^1_c\left(p_1 \right) : H^1_c\left(\mathfrak{Gelfand}\left(A\right) \times \R^{r'}; \mathscr A'\right)\to  H^1_c\left(\mathfrak{Gelfand}\left(A\right) \times \R^{r' + r''}; p_1^*\mathscr A'\right),\\
H^1_c\left(p_2 \right) : H^1_c\left(\mathfrak{Gelfand}\left(A\right) \times \R^{r''}; \mathscr A''\right)\to  H^1_c\left(\mathfrak{Gelfand}\left(A\right) \times \R^{r' + r''}; p_2^*\mathscr A''\right)
\eean 
There is a homomorphism 
\be\label{pssh_eqn}
\begin{split}
  H^1_c\left(\mathfrak{Gelfand}\left(A\right) \times \R^{r'}; \mathscr A'\right)\otimes  H^1_c\left(\mathfrak{Gelfand}\left(A\right) \times \R^{r''}; \mathscr A''\right)\to \\ \to 
H^2_c\left(\mathfrak{Gelfand}\left(A\right) \times \R^{r' + r''}; p_1^*\mathscr A'\otimes p^*_2\mathscr A''\right),\\
u_1 \otimes u_2\mapsto H^1_c\left(p_1 \right)\left(u_1 \right) \smile H^1_c\left(p_2 \right)\left(u_2 \right).
\end{split}
\ee
the the cup product $\smile$ is given by the Theorem \ref{cup_sheaf_thm}. Taking into account that   $\mathscr A'$, $\mathscr A''$, $p^*_1\mathscr A'$, $p^*_2\mathscr A''$ are  $\Z$-locally constant sheaves the map \eqref{pssh_eqn} naturally yields a pairing
\be\label{psshz_eqn}
\begin{split}
	\smile : H^1_c\left(\mathfrak{Gelfand}\left(A\right) \times \R^{r'}; \Z\right)\otimes  H^1_c\left(\mathfrak{Gelfand}\left(A\right) \times \R^{r''}; \Z\right)\to \\ \to 
	H^2_c\left(\mathfrak{Gelfand}\left(A\right) \times \R^{r' + r''}; \Z\otimes\Z\right)\cong H^2_c\left(\mathfrak{Gelfand}\left(A\right) \times \R^{r' + r''}; \Z\right).\\
\end{split}
\ee
If $r' + r'' \ge 1$ then the equation  \eqref{shift_hom_eqn} yield the following group isomorphism $$\phi^2_1: H^2_c\left(\mathfrak{Gelfand}\left(A\right) \times \R^{r'+ r''}; \Z \right)\cong H^1_c\left(\mathfrak{Gelfand}\left(A\right) \times \R^{r'+ r''-1}; \Z \right).$$
From \eqref{k01_s_eqn} it follows that  $\left[v'\right]$, $\left[v''\right]$ and $\left[\phi^2_1\left(	\left[u'\right]\smile \left[u''\right] \right) \right]$ can be regarded as elements of $K$-groups, i.e.
\be\label{h1isk_m_eqn}
\begin{split}
	u'_K \in K_{r'+ 1~ \mathrm{mod}~2 }\left(A \right),\\ 
	u''_K \in K_{r''+1~ \mathrm{mod}~2 }\left(A \right),\\ 
		\left[\phi^2_1\left(u'\smile u'' \right) \right]_K  \in K_{\left( r' + 1\right) +\left(  r'' +1\right)  - 1~ \mathrm{mod}~2 }\left(A \right)= K_{r'+  r'' +1 ~ \mathrm{mod}~2 }\left(A \right)
\end{split}
\ee
where the subscript "$_K$ means that the elements of cohomology are regarded as elements of $K$-groups. We define 
\be\label{smile_d_eqn}
u'_K \smile	u''_K\bydef 	\left[\phi^2_1\left(	u'\smile u'' \right) \right]_K
\ee
	
\begin{theorem}\label{cup_prod_thm}
For any unital $C^*$-algebra $A$ the equation \eqref{smile_d_eqn} yields a multiplicative structure on the group $K_*\left( A\right)$, i.e. there is the natural cup product
\be\label{cup_product_eqn}
\smile : K_p\left(A \right) \otimes K_q\left(A \right) \to K_{p+q~ \mathrm{mod}~2}\left(A \right) \quad \forall p, q \in \{0, 1\}.
\ee 
\end{theorem}
\begin{proof}
	To prove this theorem one should state the associativity of the cup product. It follows from the Proposition \ref{cup_ass_prop}.
\end{proof}
\begin{remark}\label{cup_prodem}
If $A$ be an unital $C^*$-algebra then any element of $x' \in K_0\left( A\right)$ can be represented by $y' \in H^0_c\left( \mathfrak{Gelfand}\left(A \otimes \K\right)\right)$. If $z \in H^1_c\left( \mathfrak{Gelfand}\left(A \otimes \K\right)\right)$ then from the Remark \ref{inv_h1em}  it follows that
\be
\forall g \in U_1\left(\left(A \otimes \K \right)^+  \right)  \quad \left(gx' \right) \smile y' = x' \smile y'
\ee
where the action \eqref{g_act_eqn} is implied. So for any $x \in K_0\left(A \right)$ and $y \in K_1\left(A \right)$  the given by the Theorem  \ref{cup_prod_thm} product $x \smile y$ corresponds to $x' \smile y'\in H^1_c\left( \mathfrak{Gelfand}\left(A \otimes \K\right)\right)$ where $x' \in H^0_c\left( \mathfrak{Gelfand}\left(A \otimes \K\right)\right)$ (resp. $y' \in H^1_c\left( \mathfrak{Gelfand}\left(A \otimes \K\right)\right)$) is a representable of $x$ (resp. $y$), i.e. the product does not depend on choice of representable of $x$.

\end{remark}

\begin{lemma}
If $\sX$ is compact Hausdorff  space and $A \bydef C\left(\sX \right)$  is a compact space then the given by the Theorem \ref{cup_prod_thm} cup product coincides with given by the Appendix \ref{mult_k_sec} one.
\end{lemma}
\begin{proof}
Consider the equation \eqref{kfn_prod_eqn}
  \be\label{top_cup_eqn}
\begin{split}
	K^n\left( \sX \right)\times  K^p\left( \sX \right) \to K^{n+p~\mathrm{mod} ~2}\left( \sX \right)
\end{split}
\ee
In $n = p = 0$ then any element of $K^0\left( \sX\right)$ is given by $\left[E\right]- \left[F\right]$ where $F$ and $F$ are vector bundles. The cup product is given by
\bean
\left( \left[E'\right]- \left[F'\right]\right) \smile \left( \left[E''\right]- \left[F'\right]\right) = \left[E'\boxtimes E''\right]-\left[E'\boxtimes F''\right]-\left[F'\boxtimes E''\right]+\left[F'\boxtimes F''\right].
\eean
Using the Theorem \eqref{serre_swan_thm} one can replace vector bundles with projectors in $\mathbb{M}_\infty\left( C\left(\sX \right) \right)$. For any $m, n\in \N$ there is the natural isomorphism $$
\iota: \mathbb{M}_m\left( C\left(\sX \right) \right) \otimes \mathbb{M}_n\left( C\left(\sX \right) \right) \cong \mathbb{M}_{mn}\left( C\left(\sX \right) \right).
$$
If $p\in \mathbb{M}_m\left( C\left(\sX \right) \right)$ and $q \in\mathbb{M}_q\left( C\left(\sX \right) \right)$ are representatives of bundles $E$ and $F$ then $p \otimes q$ is the representative of $E\boxtimes F$. There are $*$-homomorphisms $\iota': \mathbb{M}_m\left( C\left(\sX \right) \right)\hookto \mathbb{M}_{mn}\left( C\left(\sX \right) \right)$ and $\iota'': \mathbb{M}_n\left( C\left(\sX \right) \right) \hookto \mathbb{M}_{mn}\left( C\left(\sX \right) \right)$ such that 
  \bean
\begin{split}
\iota\left(p \otimes q \right) = \iota'\left(p \right) \iota''\left(q \right).
\end{split}
\eean
There is an  inclusion $\mathbb{M}_{mn}\left( C\left(\sX \right) \right)\subset C\left( \sX\right) \otimes \K$. If $P \in C\left( \sX\right) \otimes \K$ is the projector onto
\bean
\left( A \otimes\K\right)\begin{pmatrix}
		1_{\mathbb{M}_{mn}\left(\sX \right) }& \ldots & 0  \\
		\vdots& \ddots & \vdots\\
		0& \ldots &0 \\
	\end{pmatrix}
\eean
  then  the image  $\mathfrak{Image}\left( P\right)\in C_c\left( \mathfrak{Gelfand}\left(C_0\left(\sX \right)\otimes \K \right)\right)$ (cf. equation \ref{global_map_eqn}) of $P$ is a projector. From our construction it turns out that 
\bean
\mathfrak{Image}\left( \iota\left( p\otimes q\right) \right)\le \mathfrak{Image}\left( P\right),\\
\mathfrak{Image}\left( \iota'\left(p\right) \right)\le \mathfrak{Image}\left( P\right),\\
\mathfrak{Image}\left( \iota''\left(q\right) \right)\le \mathfrak{Image}\left( P\right),\\
\mathfrak{Image}\left( \iota\left( p\otimes q\right) \right)= \mathfrak{Image}\left( \iota'\left( p\right) \right)\mathfrak{Image}\left( \iota''\left( q\right) \right).
\eean
The elements $\mathfrak{Image}\left( \iota\left( p\otimes q\right) \right)$,  $\mathfrak{Image}\left( \iota'\left( p\right) \right)$, $\mathfrak{Image}\left( \iota''\left(  q\right) \right)$ in $H^0_c\left( \mathfrak{Gelfand}\left(C_0\left(\sX \right)\otimes \K \right); \Z\right)$  are representatives of $\left[p\otimes q \right]$, $\left[p\right]$ and $\left[q \right]$ in $K_0\left( C\left(\sX \right) \right)$  (cf. \eqref{k0_s_eqn}).
Otherwise $\left[p\otimes q \right]$ is the representative of $\left[p \right]\smile\left[ q \right]$ and $\left[\mathfrak{Image}\left( \iota'\left( p\right) \right)\mathfrak{Image}\left( \iota''\left( q\right) \right)\right] \in H^0_c\left( \mathfrak{Gelfand}\left(C_0\left(\sX \right)\otimes \K \right); \Z\right)$ is a representative of $\left[\mathfrak{Image}\left( \iota'\left( p\right) \right)\right]\smile  \left[\mathfrak{Image}\left( \iota''\left( q\right) \right)\right] $. So this lemma is proven for the restriction 
  \bean
\begin{split}
	K^0\left( \sX \right)\times  K^0\left( \sX \right) \to K^{0}\left( \sX \right)
\end{split}
\eean
of the product \eqref{top_cup_eqn}. The completion of the proof is left to the reader.
\end{proof}

\subsection{Graded $K$-theoretic ring}
\paragraph{} It is known (cf. Appendix \ref{mult_k_sec}) that the topological  $K$-theory yields a contravariant  functor from the category  locally compact spaces  to the category of graded rings. Here we extend the domain of this functor.
If $A$ is c $C^*$-algebra then from the Exercise \ref{cs1_exer} one has following homomorphisms 
\bea\label{cs1_11_eqn} 
	0 \hookto SA \xrightarrow{\iota} A \otimes C\left(S^1 \right) \xrightarrow{p} A \onto 0,\\
\label{cs1_21_eqn} 
K_0\left( A \otimes C\left(S^1 \right) \right) \cong K_1\left( A \otimes C\left(S^1 \right) \right)\cong K_0\left( A\right) \oplus K_1\left(A \right).
\eea
On the other hand if $\sX \bydef \mathfrak{Gelfand}\left(A \otimes \K\right)$ then the equation \eqref{long_pair_eqn} yields the sequence

	\be\label{coh_suspernsion_eqn}
\begin{split}
{0} \hookto  H^0_c\left(\sX\times S^1, \sX\times {x_0};\Z  \right) \xrightarrow{\iota_*} H^0_c\left(\sX\times S^1;\Z  \right) \xrightarrow{p_*} H^0_{c\cap \sX\times {x_0}}\left(\sX\times {x_0};\Z \right)\xrightarrow{\dl}\\ \xrightarrow{\dl}  H^1_c\left(\sX\times S^1, \sX\times {x_0};\Z  \right) \xrightarrow{\iota_*} H^1_c\left(\sX\times S^1;\Z  \right) \xrightarrow{p_*} H^1_{c\cap \sX\times {x_0}}\left(\sX\times {x_0};\Z \right)\xrightarrow{\dl}...
\end{split}
\ee
Taking into account \eqref{setmunus_h_eqn} one has 
	\bean
\begin{split}
	{0} \hookto  H^0_c\left(\sX\times S^1\setminus \sX\times {x_0};\Z  \right) \xrightarrow{\iota_*} H^0_c\left(\sX\times S^1;\Z  \right) \xrightarrow{p_*} H^0_{c\cap \sX\times {x_0}}\left(\sX\times {x_0};\Z \right)\xrightarrow{\dl}\\ \xrightarrow{\dl}  H^1_c\left(\sX\times S^1\setminus \sX\times {x_0};\Z  \right) \xrightarrow{\iota_*} H^1_c\left(\sX\times S^1;\Z  \right) \xrightarrow{p_*} H^1_{c\cap \sX\times {x_0}}\left(\sX\times {x_0};\Z \right)\xrightarrow{\dl}...
\end{split}
\eean
or, equivalently

\be\label{coh_suspernsione_eqn}
\begin{split}
	{0} \hookto  H^0_c\left(\sX\times \R;\Z  \right) \xrightarrow{\iota_*} H^0_c\left(\sX\times S^1;\Z  \right) \xrightarrow{p_*} H^0_{c\cap \sX\times {x_0}}\left(\sX;\Z \right)\xrightarrow{\dl}\\ \xrightarrow{\dl}  H^1_c\left(\sX\times \R;\Z  \right) \xrightarrow{\iota_*} H^1_c\left(\sX\times S^1;\Z  \right) \xrightarrow{p_*} H^1_{c\cap \sX\times {x_0}}\left(\sX;\Z \right)\xrightarrow{\dl}...
\end{split}
\ee
There is the natural surjective map  $\sX\times S^1\onto \sX$ and injective map  $\sX\cong  \sX\times {x_0}\hookto \sX\times S^1$ such that the composition $\sX \hookto \sX \times S^1 \onto \sX$ is the identical map. It turns out that the sequence \ref{coh_suspernsione_eqn} naturally splits and there is the natural direct sums
\bea 
 H^0_c\left(\sX\times S^1;\Z  \right)\cong H^0_c\left(\sX;\Z  \right)\oplus H^0_c\left(\sX\times \R;\Z  \right),\\\label{h1_d_eqn}
  H^1_c\left(\sX\times S^1;\Z  \right)\cong H^1_c\left(\sX;\Z  \right)\oplus H^1_c\left(\sX\times \R;\Z  \right).
\eea 
From the equation \eqref{k1_s_eqn} it follows that
\bean
K_1\left(C\left(\sX \right)  \right) \cong   H^1_c\left( \mathfrak{Gelfand}\left(C\left(\sX \right) \otimes \K \right);\Z\right) 
\eean 
Taking into account the Theorems  \ref{suspension_thm}, \ref{bott_map_thm} and one has
\bean 
K_0\left(A \right) \cong   H^1_c\left( \mathfrak{Gelfand}\left(A\otimes \K \right)\times \R;\Z\right), 
\eean 
so from \eqref{cs1_21_eqn} and \eqref{h1_d_eqn} one has the group isomorphism 
\be
 H^1_c\left( \mathfrak{Gelfand}\left(A\otimes \K \right)\times S^1;\Z\right)\cong K_1\left( A \otimes C\left(S^1 \right) \right)\cong K_*\left( A\right) 
\ee
Let us define the multiplication on $K_*\left( A\right)$. Let $$x, y \in K_0\left( A\right) .$$  If both $x' \in  H^0_c\left(\sX;\Z  \right)$ and $y' \in  H^1_c\left(\sX\times \R;\Z  \right)$ are given by the equations \eqref{k_00_eqn} and \eqref{k0_s_eqn} representatives of $x$ and $y$ then one can assume that $x', y' \in H^*_c\left(\mathfrak{Gelfand}\left(A\otimes \K \right)\times S^1 \right)$. It turns out that
\be
x' \smile y' \in H^*_c\left(\mathfrak{Gelfand}\left(A\otimes \K \right)\times S^1 \right)
\ee
where the cup product is given by \ref{cup_sheaf_thm}. 

\begin{definition}\label{deg_1_defn}
If both $A$ and $B$ are $C^*$-algebras then a homomorphism $\phi: K_*\left(A \right) \to K_*\left(B \right)$ of Abelian groups \textit{has degree 1} if $\phi\left(K_j\left( A\right)  \right) \subset K_{1-j}\left( B\right)$, for every $j = 0,1$. We say that $\phi$ is {\it of degree 1}.
\end{definition}
\begin{example}\label{degr_1_exanple}
The Theorems \ref{suspension_thm} and \ref{bott_map_thm} yield isomorphisms  $K_*\left(A \right) \cong K_*\left(SA \right)$  and  $K_*\left(SA \right) \cong K_*\left(A \right)$ of degree 1.
\end{example}
\begin{definition}\label{degr_1_defn}
Suppose that a homomorphism  $\phi: K_*\left(A \right) \to K_*\left(B \right)$ of Abelian groups {has degree 1} and both   $K_*\left(B \right)\cong  K_*\left(SB \right)$ and $K_*\left(SA \right)\cong  K_*\left(A \right)$ are given by the Example \ref{degr_1_exanple}.. We say that $\phi$ is a \textit{degree 1 ring homomorphism} if both compositions
\be
\begin{split}
K_*\left(A \right) \xrightarrow{\phi} K_*\left(B \right)\cong  K_*\left(SB \right),\\
K_*\left(SA \right)\cong K_*\left(A \right) \xrightarrow{\phi} K_*\left(B \right)
\end{split}
\ee
are the ring homomorphisms.  Both $K_*\left(B \right)\cong  K_*\left(SB \right)$ and $K_*\left(SA \right)\cong  K_*\left(A \right)$ are degree 1 ring homomorphisms.
\end{definition}

\subsection{Standard Exact Sequence}\label{standard_exact_sequence_sec}

\paragraph{}
The Theorem \ref{standard_exact_sequence_thm} yields the following
	six-term cyclic exact sequence 
	\bean
	\begin{tikzcd}
		K_0\left( J\right)  \arrow[r, "\iota_*"] & K_0\left( A\right) \arrow[r, "\pi_*"]& K_0\left( A/J\right) \arrow[d, "\partial"]  \\
		K_1\left(A/J \right) \arrow[u, "\partial"] & \arrow[l, "\pi_*"]K_1\left( A\right) &\arrow[l, "\iota_*"] K_1\left(J \right) 
	\end{tikzcd}
	\eean
	One can construct the following sequence
		\be\label{sis_ext_eqn}
	\begin{tikzcd}
		K_0\left( J\right) \oplus 	K_1\left( J\right)\arrow[r, "\iota_*"] & K_0\left( A\right) \oplus 	K_1\left( A\right)\arrow[r, "\pi_*"]& K_0\left( A/J\right)\oplus 	K_1\left(A/ J\right) \arrow[d, "\partial"]  \\
		K_1\left(A/J \right)\oplus	K_0\left( A/J\right) \arrow[u, "\partial"] & \arrow[l, "\pi_*"] K_1\left( A\right)\oplus 	K_0\left( A\right) &\arrow[l, "\iota_*"] K_1\left(J \right) \oplus	K_0\left( J\right)
	\end{tikzcd}
	\ee
	or equivalently 
	\be\label{sisf_ext_eqn}
\begin{tikzcd}
	K_*\left( J\right)  \arrow[r, "\iota_*"] & K_*\left( A\right) \arrow[r, "\pi_*"]& K_*\left( A/J\right) \arrow[d, "\partial"]  \\
	K_*\left(A/J \right) \arrow[u, "\partial"] & \arrow[l, "\pi_*"]K_*\left( A\right) &\arrow[l, "\iota_*"] K_*\left(J \right) 
\end{tikzcd}
\ee
	Using  the Theorems  \ref{good_thm}  and \ref{full_k_thm}, the equation \eqref{cohom_mor_eqn} and the Remark \ref{pres_cupem} one can deduce that both $\iota_*$ and $\pi_*$ in \eqref{sisf_ext_eqn} are grading preserving  ring homomorphisms. Let us prove that $\partial$ is a degree 1 ring homomorphism. One can suppose that $A$, $J$ and $A/J$ are stable $C^*$-algebras. Since $J$ is a closed left ideal of $A$ the set $\mathfrak{Gelfand}\left(J\right)$ is an open subset of $\mathfrak{Gelfand}\left(A\right)$. There are continuous maps
\bea\label{map_j_eqn}
\mathfrak{Gelfand}\left(A\right)\onto \mathfrak{Gelfand}\left(J\right),\\\label{map_aj_eqn} \mathfrak{Gelfand}\left(A/J\right)\into \mathfrak{Gelfand}\left(A\right)
\eea
where the map \eqref{map_aj_eqn} is given by the equation \eqref{gelfand_map_eqn}. From our construction it turns out that
\bean
\mathfrak{Gelfand}\left(A\right) = \mathfrak{Gelfand}\left(J\right)\cup \mathfrak{Gelfand}\left(A/J\right),\\
\mathfrak{Gelfand}\left(J\right)\cap \mathfrak{Gelfand}\left(A/J\right)= \emptyset,\\
\mathfrak{Gelfand}\left(J\right)=\mathfrak{Gelfand}\left(A/J\right)\setminus \mathfrak{Gelfand}\left(A/J\right).
\eean 
The set $\mathfrak{Gelfand}\left(A/J\right)$ is closed in $\mathfrak{Gelfand}\left(A\right)$ since $\mathfrak{Gelfand}\left(J\right)$ is open. Both equations \eqref{cohom_seq_eqn} and \eqref{setmunus_h_eqn} yield the following exact sequence 
\be\label{coh_x_eqn}
\begin{split}
	{0} \hookto  H^0_c\left(\mathfrak{Gelfand}\left(J\right);\Z  \right) \xrightarrow{\iota_*} H^0_c\left(\mathfrak{Gelfand}\left(A\right);\Z  \right) \xrightarrow{\pi_*} \\ \xrightarrow{\pi_*} H^0_{c}\left(\mathfrak{Gelfand}\left(A/J\right);\Z \right)\xrightarrow{\dl}H^1_{c}\left(\mathfrak{Gelfand}\left(J\right);\Z \right)\to ...
\end{split}
\ee
Let us prove that $\dl$ in \eqref{coh_x_eqn} yields the homomorphism $\partial K_0\left( A/J\right) \to K_1\left( J\right)$. Any $\left[c\right] \in H^0_{c}\left(\mathfrak{Gelfand}\left(A/J\right);\Z \right)$ can be represented by  $c \in C_{c}\left(\mathfrak{Gelfand}\left(A/J\right);\Z \right)$. The support of $c$ is compact  there is  valued $b \in C_c\left( \mathfrak{Gelfand}\left(A\right); \Z\right)$ such that 
\be
\forall \xi \in \mathfrak{Gelfand}\left(A/J\right)\quad c\left( \xi\right) = y\left( \xi\right) 
\ee 
If  
\be\label{u_x_eqn}
u_c\bydef e^{2\pi i b}
\ee
then $u_c$ can be regarded  an element of both  $U\left(C_0\left( \mathfrak{Gelfand}\left(A/J\right)^+\right)  \right)$ and \\ $H^0_{c}\left(\mathfrak{Gelfand}\left(J\right);U\left( 1\right)\right)$, so $\left[u_c\right]=\partial c \in K_1\left( J\right)$ (cf. equations \eqref{u_x_eqn} and \eqref{exp_map_eqn}).
If 
\be\label{sxxh_shh_eqn}
\begin{split}
		\{0\}  \to  H^0_c\left(\mathfrak{Gelfand}\left(J\right); \Z\right)  \to  H^0_c\left(\mathfrak{Gelfand}\left(J\right); \sR \right)\to \\ H^0_c\left(\mathfrak{Gelfand}\left(J\right); U\left( 1\right)  \right)\xrightarrow{\dl_\SS }  H^1_c\left(\mathfrak{Gelfand}\left(J\right); \Z \right)\to ...
\end{split}
\ee
is  a specializations of \eqref{sxx_shh_eqn}  then there is
\be\label{a_p_eqn}
a' \bydef \dl_\SS u_c \in H^1_c\left(\mathfrak{Gelfand}\left(J\right); \Z \right)
\ee
such that $a'$ is a representative of $\partial c \in K_1\left( J\right)$ (cf. equation \eqref{k1_s_eqn})
\begin{lemma}\label{six_term_lem}
	The exponential map $\partial : K_0\left(A/ J \right)\to K_1\left( J\right)$ in \eqref{sisf_ext_eqn} corresponds to the homomorphism $\dl$ of \eqref{coh_x_eqn} where the equations of the Theorem \ref{k_h_thm} are implied.
\end{lemma}
\begin{proof}
	There is a specialization of exact sequence of complexes \eqref{longx_pair_eqn} given by
\be\label{longcx_pair_eqn} 
0 \hookto C^*_c\left(\mathfrak{Gelfand}\left(J\right);\Z \right) \xrightarrow{\iota} C^*_c\left(\mathfrak{Gelfand}\left(A\right);\Z \right)\xrightarrow{\pi} C^*_c\left(\mathfrak{Gelfand}\left(A/J\right);\Z \right)\onto 0
\ee
Let $\left[c\right] \in H^0_{c}\left(\mathfrak{Gelfand}\left(A/J\right);\Z \right)$ and $c \in C^0_c\left(\mathfrak{Gelfand}\left(A/J\right);\Z  \right)\cong C_c\left(\mathfrak{Gelfand}\left(A/J\right);\Z  \right)$ is a representative of $\left[c\right]$. There is $b \in C^0_c\left(\mathfrak{Gelfand}\left(A\right);\Z  \right)$ such that $c = \pi\left( b\right)$. From $\partial c = 0$ it follows that  $ \pi\circ \partial \left( b\right) = 0$. It turns out that there is $a\in  C^1_c\left(\mathfrak{Gelfand}\left(J\right);\Z \right)$ such that $\partial \left( b\right)= \iota\left(  a\right)$. Since $\iota$ is injective $ \partial^2 \left( b\right)=0$. It follows that $\partial a = 0$, so $a$ is a cycle  According to homological algebra one has 
\be\label{dl_eqn}
\dl \left[c\right] = \left[a\right].
\ee
From $\dl b = \dl_\SS e^{2\pi i b}$ it follows that  $\left[a\right]= a'$ where $a'$ is given by \eqref{a_p_eqn}.
\end{proof}

\begin{theorem}\label{six_term_thm}
The Abelian group homomorphisms $\iota_*$ and $\pi_*$  \eqref{sisf_ext_eqn} are ring homomorphisms, The Abelian group homomorphisms $\partial$ in \eqref{sisf_ext_eqn} yield a  degree 1 ring homomorphisms $K_*\left( A/J\right) \to K_*\left( J\right)$.
\end{theorem}
\begin{proof}
If $\left[p_1\right], \left[p_2\right]\in K_0\left(S\left( A/ J \right) \right)$ then there are representatives \\ $p_1, p_2 \in H^0_c\left( \mathfrak{Gelfand}\left(S\left( A/ J \right)\right); \Z \right) \cong C_c\left( \mathfrak{Gelfand}\left(S\left( A/ J \right)\right); \Z \right)$ of these element. Moreover $\left[p_1\right]\cup \left[p_2\right]= \left[p_1p_2\right]$. If both $p'_1, p'_2 \in \in  C_c\left( \mathfrak{Gelfand}\left(SA\right); \Z\right)$ are such that  
\be
\forall \xi \in \mathfrak{Gelfand}\left(S\left( A/ J \right)\right)\quad p'_1\left( \xi\right) = p_1\left( \xi\right), \quad  p'_2\left( \xi\right) = p_2\left( \xi\right)
\ee
then 
\be
\forall \xi \in \mathfrak{Gelfand}\left(S\left( A/ J \right)\right)\quad \left( p'_1\cdot p_2'\right)\left(\xi \right)  \left( \xi\right) = \left( p_1\cdot p_2\right)\left(\xi \right) 
\ee
where $\cdot$ means the product of integer valued functions. It follows that $\partial \left( \left[p_1\right]\cup \left[p_2\right]\right)$ is given by $e^{2\pi i \left( p'_1\cdot p'_2\right) }$. Let  $p''_1, p''_2 \in H^0_c\left( \mathfrak{Gelfand}\left( J \right); \Z \right)$ representatives of $\left[p''_1\right], \left[p''_2\right]\in K_0\left(J \right)$ such that
\bean
\bt_J\left(  \left[p''_1\right]\right) = \partial \left[p_1\right],\\
\bt_J\left(  \left[p''_2\right]\right) = \partial \left[p_2\right]
\eean 
where $\bt_J$ is the Bott map (cf. Definition \ref{bott_map_defn}). One has 
	\be\label{cup_pu_eqn}
\begin{split}
	\left[p''_1\right]\cup \left[p''_2\right]= \left[\left( p''_1\cdot p''_2\right) \right].\\
	\bt_J\left(  \left[\left( p''_1\cdot p''_2\right)\right]\right)  = \partial \left( \left[p_1\right]\cup \left[p_2\right]\right).
\end{split}
\ee
This theorem is proven for the homomorphism $\partial : K_0\left(A/ J \right)\to K_1\left( J\right)$. The completion of the proof is left to the reader.
\end{proof}

\subsection{Connes' Thom isomorphism and Pimsner-Voiculescu Exact Sequence}
	\paragraph{} If $A$ is a $C^*$-algebra and $\a : \R \to \Aut\left( A\right)$ then   similarly to \eqref{sisf_ext_eqn} the  Theorem  \ref{c_t_theorem} yields  the natural Abelian group isomorphism
	\be\label{c_t_eqn}
	K_*\left(A \right) \cong K_*\left(A \rtimes_\a \R \right) 
	\ee 
We would like to prove that \eqref{c_t_eqn} is a ring isomorphism.
\begin{theorem}\label{connes_thom_thm}
The isomorphism \eqref{c_t_eqn} is a degree one ring isomorphism.
\end{theorem}	
\begin{proof}
	The  six-term sequence \ref{tho_xex_eqn} yields the diagram 
\be\label{thoh_xex_eqn}
\begin{tikzcd}
	K_*\left( A\right) \arrow[r] &K_*\left( \R\hookto CA \rtimes_\ga \R\right)\arrow[r] & K_*\left( A \rtimes_\a \R\right) \arrow[d, "\partial"] \\
	K_*\left(  A \rtimes_\a \R\right) \arrow[u, "\partial"] &K_*\left( \R\hookto CA \rtimes_\ga \R\right)\arrow[l] & K_*\left( A \right)\arrow[l]
\end{tikzcd}
\ee 
From the theorem \ref{six_term_thm} it follows that  The Abelian group homomorphisms $\partial$ in \eqref{thoh_xex_eqn} yield the degree one ring  isomorphism 
\bean
	K_*\left(A \right) \cong K_*\left(A \rtimes_\a \R \right). 
\eean
This theorem is a consequence of  \eqref{six_term_thm} one.
\end{proof}	
\begin{theorem}\label{pimsner_voiculesky_thm}
If the exact sequence
\be\label{pimsner_voiculesky_eqn}
\begin{tikzcd}
	K_*\left( A\right)\arrow[r, "1-\a_*"] & K_*\left(A \right)\arrow[r, "\iota_*"] & K_*\left(A \rtimes_\a \Z \right) \arrow[d, "\partial"]\\
	K_*\left( A\rtimes \Z\right)\arrow[u , "\partial"] & \arrow[l, "\iota_*"] K_*\left(A \right) & \arrow[l, "1-\a_*"]K_*\left(A \right)
\end{tikzcd}
\ee
comes from  \eqref{p_v_eqn} then  Abelian group homomorphism $\iota_*$ and $1-\a_*$ are ring homomorphism. and $\partial$  is a degree 1 ring homomorphism.
\end{theorem}	
\begin{proof}
	 $\a_*$ is a graded ring homomorphism since it comes from continuous map (cf. \eqref{cohom_mor_eqn}), so $1-\a_*$  is also a graded ring homomorphism.
	Below the notation of \ref{p_v_empt} will be used.
	The exact sequence \eqref{p_v_sec_eqn} yields the diagram 
	\be\label{thoh_xexs_eqn}
	\begin{tikzcd}
		K_*\left( SA\right) \arrow[r] &K_*\left( B \rtimes_\bt \R\right)\arrow[r] & K_*\left( A\right) \arrow[d] \\
		K_*\left(  A\right) \arrow[u] &K_*\left(B \rtimes_\bt\R\right)\arrow[l] & K_*\left( SA \right)\arrow[l]
	\end{tikzcd}
	\ee
	such that horizontal  arrows are grading preserved ring homomorphisms. Using the construction \ref{standard_exact_sequence_sec}. the diagram \eqref{thoh_xexs_eqn} can be replaced with the following one
	\be\label{thohh_xexs_eqn}
	\begin{tikzcd}
		K_*\left( A\right) \arrow[r] &K_*\left( B \rtimes_\bt \R\right)\arrow[r] & K_*\left( A\right) \arrow[d] \\
		K_*\left(  A\right) \arrow[u] &K_*\left(B \rtimes_\bt\R\right)\arrow[l] & K_*\left( A \right)\arrow[l]
	\end{tikzcd}
	\ee
where the ring homomorphism 	$K_*\left( A\right)	\to K_*\left( B \rtimes_\bt \R\right)$ (resp. $K_*\left( B \rtimes_\bt \R\right)\to K_*\left( A\right)$ has degree 1 (resp. 0). The Theorem \ref{connes_thom_thm} yields the ring ring homomorphism 	$K_*\left( A\right)	\to K_*\left( B \right)$ (resp. $K_*\left( B  \right)\to K_*\left( A\right)$ having  degree 0 (resp. 1). Taking into account $B \bydef A\rtimes_\a \Z$ and the Theorem \ref{six_term_thm} one completes the proof. 
\end{proof}

\subsection{Connes quantization}\label{connes_q_sec}

\paragraph*{}	
In general the quantization yields a passage from the classical physics to the quantum one. In particular there is a quantization which yields an the transposition from commutative $C^*$-algebra $B$ to noncommutative ones. In particular a commutative  $C^*$-algebra $B$ with the product $\times$ is a depending on real parameter $\hslash$  limit of continuous fields  of noncommutative $C^*$-algebras $\overline B_\hslash$ with product $\times_\hslash$ such that
\begin{itemize}
	\item Any $\overline B_\hslash$ is isomorphic to $B$ with the $C^*$-norm $f \mapsto \left\|f \right\|_{\hslash}$ which is continuous of $\hslash$ for any $f \in B$. 
	\item For any $f, g \in B$ we have $\left\|\left(f\times_\hslash g - g\times_\hslash f \right) /\hslash - i\left\{f, g\right\}_J \right\|_{\hslash}$ AS $\hslash\to 0$, where $\left\{ ,\right\}_J$ is the Poisson bracket defined by $J$ (with a factor of $\pi$).
\end{itemize}
The details of the above construction are explained in  \cite{rieffel:c_quant}.
\paragraph{}
The given by the Definition \ref{blowing_defn}  Hausdorff blowing-up equation
\be\label{blowing_r_eqn}
\mathfrak {Blowing}_{\sX-A }:C_0\left( \sX\right) \hookto M\left(A \right)
\ee
contains both a commutative $C^*$-algebra $C_0\left( \sX\right)$ and  a noncommutative one $A$. Elements of  $C_0\left( \sX\right)$  are limits of elements of $A$ with respect to the strict topology of $M\left(A \right)$. This duality can have rich structure. 

Sometimes this duality comes from the given by Alain Connes  algebraic counterpart of geometric quotient (cf. Appendix \ref{quotions_a_sec}) which makes a non Hausdorff spectrum from a Hausdorff space. In this case we say that there is  {\it Connes quantization}.

Let us describe an  analogy  between algebraic counterpart of geometric quotient and Connes quantization. If $A$ is a $C^*$-algebra and   $C_0\left( \sX\right)\subset M\left(A \right)$ is natural Hausdorff blowing-up (cf. Definition \ref{blowing_defn}). Suppose that the equation \eqref{to_spec_eqn}  yields the following diagram of commutative maps
\be\label{gen_foli_eqn}
\begin{tikzcd}
	\mathfrak{Gelfand}\left(A\right) \arrow[rr] \arrow[rd, "\hat \phi"]&& \sX\arrow[ld, "\phi"]\\
	& \hat A &
\end{tikzcd}
\ee
\begin{definition}\label{generalized_foliation_defn}
	Under the above hypothesis there is the equivalence relation $\sim$ on $\sX$ such that
	\bean
	\forall x', x'' \in \sX \quad x'\sim x'' \quad \Leftrightarrow \quad \phi\left(x' \right)= \phi\left(x'' \right),\\
	\hat{A}\cong \sX/\sim
	\eean 
	The decomposition 
	$$
	\sX = \bigsqcup_{x \in \hat{A}} \phi^{-1}\left(x \right) 
	$$
	and the diagram \eqref{generalized_foliation_defn}
	is a {\it generalized foliation}.
\end{definition}
\begin{remark}
	The described above situation is a specialization of algebraic counterpart of geometric quotient (cf. Appendix \ref{quotions_a_sec}) which makes a non Hausdorff spectrum from a Hausdorff space. However the $C^*$-algebra $A$ contains information about both non Hausdorff and Hausdorff spaces. The Hausdorff space $\sX$ can be naturally restored from the $C^*$-algebra $A$. 
\end{remark}

\section{Algebraic topology of the noncommutative torus}
\paragraph*{}
From  the Theorem \ref{k_0_tor_thm} (resp \ref{k_1_tor_thm})  element of $K_0\left(C\left(\T^2_\th \right)  \right)$  (resp. $K_1\left(C\left(\T^2_\th \right)  \right)$) is represented by an idempotent (resp. invertible element) of $C\left(\T^2_\th \right)$. From this circumstance the given by the Theorem \ref{k_gelfand_thm} equations 
\bean
K_0\left(C\left(\T^2_\th \right) \right) =H^0_c\left(\mathfrak{Gelfand}\left(C\left(\T^2_\th \right)\otimes \K \right); \Z \right)/U\left(\left(C\left(\T^2_\th \right) \otimes \K \right) ^+ \right),\\
K_{1}\left(C\left(\T^2_\th \right) \right) \cong H^1_c\left(\mathfrak{Gelfand}\left(C\left(\T^2_\th \right) \otimes \K\right); \Z \right) 
\eean
can be replaced with following ones,
\bea\label{k_0_nt_eqn}
K_0\left( C\left(\T^2_\th \right) \right) =H^0\left(\mathfrak{Gelfand}\left(C\left(\T^2_\th \right) \right); \Z \right)/U\left(C\left(\T^2_\th \right)  \right),\\ \label{k_1_nt_eqn}
K_{1}\left(C\left(\T^2_\th \right) \right) \cong H^1\left(\mathfrak{Gelfand}\left(C\left(\T^2_\th \right)\right); \Z \right). 
\eea
It is known that $C\left(\T^2_\th \right)$ is an universal $C^*$-algebra generated by two unitaries $u, v \in C\left(\T^2_\th \right)$ such that $uv = e^{2\pi i \th}vu$. It is also known 
that $C\left(\T^2_\th \right) \cong C\left(u_{S^1} \right)\rtimes_{\a_u} \Z$, where $u_{S^1} \in C\left( u_{S^1}\right) \cong C(S^1)$ is unitary,  $\a^{u_{S^1}}\in \Aut\left(C\left( S^1\right)  \right) \cong \Aut\left(C\left( u_{S^1}\right)  \right)$ is a $*$- automorphism given by $u_{S^1} \mapsto e^{2\pi i \th} u_{S^1}$.
The Theorem \ref{p_v_theorem} yields  a cyclic six-term exact sequence
\bean
\begin{tikzcd}
	K_0\left( C\left(u_{S^1} \right)\right)\arrow[r, "1-\a^{u_{S^1}}_*"] & K_0\left( C\left(u \right)\right)\arrow[r, "\iota^{u_{S^1}}_*"] & K_0\left( C\left({u_{S^1}} \right) \rtimes_\a \Z \right) \arrow[d, "\partial"]\\
	K_1\left(  C\left(u \right)\rtimes \Z\right)\arrow[u, "\partial"] & \arrow[l, "\iota^{u_{S^1}}_*"] K_1\left( C\left(u \right) \right) & \arrow[l, "1-\a^{u_{S^1}}_*"]K_1\left( C\left(u \right)  \right)
\end{tikzcd}
\eean
where the map $\iota^{u_{S^1}}_*$ corresponds to the injective $*$-homomorphism
\be\label{iota_torus}
\begin{split}
	\iota:	C\left(u_{S^1} \right) \hookto C\left(u \right) \rtimes_\a \Z,\\
	u_{S^1} \mapsto \left( u , 1_C\left(S^1 \right)\right)= \left( u ,  \a^0_u\right)\in C\left(u \right) \rtimes_\a \Z,\\
	\text{equivalently} \quad u_{S^1} \mapsto u \in C\left(\T^2_\th \right).
\end{split}
\ee
From the isomorphisms 	
$$
K_0\left( C\left(u \right)\right)\cong 	K_1\left( C\left(u \right)\right)\cong \Z
$$	
it follow that $\a^{u_{S^1}}_*$ is identical, $1-\a^{u_{S^1}}_*= 0$, so there are the exact sequences
\bea\label{kk_0_nt_eqn}
\{0\}\hookto	K_0\left( C\left({u_{S^1}} \right)\right) \xrightarrow{\iota^u_*} K_0\left( C\left(u \right)  \rtimes_\a \Z \right) \xrightarrow{\partial} K_1\left( C\left({u_{S^1}} \right)\right) \onto \{0\}, \\ \label{kk_1_nt_eqn}
\{0\}\hookto	K_1\left( C\left(u \right)\right) \xrightarrow{\iota^{u_{S^1}}_*} K_1\left( C\left({u_{S^1}} \right)  \rtimes_\a \Z \right) \xrightarrow{\partial} K_0\left( C\left({u_{S^1}} \right)\right) \onto \{0\}.	
\eea
Moreover one has
\be\label{tor_gen_eqn}
\begin{split}
	K_0\left( C\left({u_{S^1}} \right)\right) = \Z \left[1_{u_{S^1}}\right],\\
	K_1\left( C\left({u_{S^1}} \right)\right) = \Z \left[{u_{S^1}}\right],\\
	H^0\left( \mathfrak{Gelfand}\left( C\left({u_{S^1}} \right)\right); \Z\right)\cong 	H^0\left( S^1; \Z\right)  = \Z \left[1_{u_{S^1}}\right],\\
	H^1\left(\mathfrak{Gelfand}\left(  C\left({u_{S^1}} \right)\right);\Z\right)  \cong  H^1\left( S^1; \Z\right)= \Z \left[{u_{S^1}}\right],\\
\end{split}
\ee	
If we consider the cup product (cf. Theorem \ref{cup_sheaf_thm}) then one has 
\be\label{prod_u_tor_eqn}
\left[1_{u_{S^1}}\right]\smile \left[{u_{S^1}}\right]=\left[{u_{S^1}}\right]\in 		H^1\left(\mathfrak{Gelfand}\left(  C\left({u_{S^1}} \right)\right);\Z\right). 
\ee
From  the Appendix \ref{rieffel_sec} it turns out  that
\be\label{iprod_u_tor_eqn}
\iota_*\left( \left[1_{u_{S^1}}\right]\right) = \left[p_\th\right]\in K_0\left( C\left({u_{S^1}} \right)  \rtimes_\a \Z \right)
\ee
where $p_\th$ is  the Powers - Rieffel projector. If $\mathfrak{Image}\left( p_\th\right)\in C\left( \mathfrak{Gelfand}\left(  C\left(u \right)\rtimes_\a \Z\right)\right) $ is given by the equation \eqref{global_map_eqn} then the projector  $\mathfrak{Image}\left( p_\th\right)$  represents 
\bean
\left[\mathfrak{Image}\left( p_\th\right)\right]\in H^0\left(\mathfrak{Gelfand}\left(  C\left(u \right)\rtimes_\a \Z\right);\Z \right). 
\eean
From \eqref{prod_u_tor_eqn} and  \eqref{iprod_u_tor_eqn} it turns out that
\be
\left[\mathfrak{Image}\left( p_\th\right)\right]\smile \iota^{{u_{S^1}}}_* \left[{u_{S^1}}\right] \in  H^1\left(\mathfrak{Gelfand}\left(  C\left({u_{S^1}} \right)\rtimes_\a \Z\right);\Z \right).
\ee
Using the Theorem \ref{cup_prod_thm} and the Remark \ref{cup_prodem} we have 
\be
\left[ p_\th\right]\smile \left[{u}\right] = \left[{u}\right]\in  K_1\left( C\left({u_{S^1}} \right)  \rtimes_\a \Z \right)=  K_1\left( C\left(\T^2_\th \right)\right)
\ee
where $\left[ p_\th\right]\in K_0\left( C\left(\T^2_\th \right)\right)$ and $\left[ {u}\right]\in K_1\left( C\left(\T^2_\th \right)\right)$. From Lemma \ref{nt_gen_lem} it follows that $p_\th \in   C\left(\T^2_\th \right)$ so one has the disconnected union
\be\label{tor_prodw_eqn}
\begin{split}
	\mathfrak{Gelfand}\left(C\left(\T^2_\th \right)\right) = \mathfrak{Gelfand}\left(C\left(\T^2_\th \right)\right) _{C\left(\T^2_\th \right)p_\th}\bigsqcup	\mathfrak{Gelfand}\left(C\left(\T^2_\th \right)\right) _{C\left(\T^2_\th \right)\left(1- p_\th\right) }
\end{split}
\ee
where the notation \eqref{topology_eqn} is used. Ir turns out that there is the following direct sum of rings.
\be\label{torh_prodw_eqn}
\begin{split}
	H^*_c\left( 	\mathfrak{Gelfand}\left(C\left(\T^2_\th \right)\right);\Z\right)=\\= H^*_c\left( \mathfrak{Gelfand}\left(C\left(\T^2_\th \right)\right) _{C\left(\T^2_\th \right)p_\th}; \Z\right) \oplus H^*_c\left( 	\mathfrak{Gelfand}\left(C\left(\T^2_\th \right)\right) _{C\left(\T^2_\th \right)\left(1- p_\th\right) }\right) 
\end{split}
\ee
Taking into account the Theorem \ref{pimsner_voiculesky_thm} one has the following direct sum of rings
\be\label{torhk_prodw_eqn}
\begin{split}
	H^*_c\left( \mathfrak{Gelfand}\left(C\left(\T^2_\th \right)\right) _{C\left(\T^2_\th \right)p_\th}; \Z\right) \cong H^*_c\left( C\left(S^1\right);Z\right) ,\\
	H^*_c\left( \mathfrak{Gelfand}\left(C\left(\T^2_\th \right)\right) _{C\left(\T^2_\th \right)\left( 1-p_\th\right) }; \Z\right) \cong H^*_c\left( SC\left(S^1\right);Z\right) ,\\
	K_*\left( C\left(\T^2_\th \right)\right) =K_*\left( C\left(S^1\right)\right) \oplus 
	K_*\left( SC\left(S^1 \right)\right) 
\end{split}
\ee
From the equation \eqref{torhk_prodw_eqn} one has 
\be\label{torhk_hprodw_eqn}
\begin{split}
	K_0\left(C\left(\T^2_\th \right) \right) \cong 	H^0_c\left( 	\mathfrak{Gelfand}\left( C\left(\T^2_\th \right) \right) ; \Z \right) \cong \Z\left[p_\th\right]\oplus \Z\left[1- p_\th\right],\\
	K_1\left(C\left(\T^2_\th \right) \right) \cong 	H^1_c\left( 	\mathfrak{Gelfand}\left(C\left(\T^2_\th \right) \right) ; \Z \right) \cong \Z\left[u\right]\oplus \Z\left[v\right]
\end{split}
\ee
and  following conditions hold.

\be\label{tor_prod_eqn}
\begin{split}
	\left[ p_\th\right]\smile \left[1 - p_\th\right] = 0,\\
	\left[ p_\th\right]\smile \left[ p_\th\right] = \left[ p_\th\right],\\
	\left[1 - p_\th\right] \smile \left[1 - p_\th\right]= \left[1 - p_\th\right],\\
	\left[ p_\th\right]\smile \left[u\right]= \left[u\right],\\
	\left[1- p_\th\right]\smile \left[u\right]=0,\\
	\left[1- p_\th\right]\smile \left[v\right]=\left[v\right],\\
	\left[ p_\th\right]\smile \left[v\right]= 0,\\
	\left[u\right]\smile \left[u\right]= 	\left[v\right]\smile \left[v\right]= \left[u\right]\smile \left[v\right]=0.
\end{split}
\ee

\begin{definition}\label{good_torus_defn}
	Let  $n = 2k \in \N$ be ab even natural number. We say that the noncommutative torus $C\left(\T^n_\th \right)$ is \textit{good} if one has 
	\be\label{tor_prodg_eqn}
	\begin{split}
		K_0\left( C\left(\T^n_\th\right)\right)  \cong	H^0\left(\mathfrak{Gelfand}\left(C\left(\T^n_\th \right) \right); \Z \right)\cong   \bigoplus_{j = 1}^{2^{n-1}}\Z s_j,\\
		K_1\left( C\left(\T^n_\th\right)\right)  \cong	H^1\left(\mathfrak{Gelfand}\left(C\left(\T^n_\th \right) \right); \Z \right)\cong   \bigoplus_{j = 1}^{2^{n-1}}\Z t_j
	\end{split}
	\ee
	where for each $j, k \in \left\{1,..., 2^{n-1}\right\}$ the cup product satisfies to the following equations	
	\be 
	\begin{split}
		s_j \smile s_k = \begin{cases}
			s_j & j = k\\
			0 & j \neq k
		\end{cases},\\
		s_j \smile t_k = t_k \smile s_j =\begin{cases}
			t_k & j = k\\
			0 & j \neq k
		\end{cases},\\
		t_j \smile t_k = 0.
	\end{split}	
	\ee
	
\end{definition}
Suppose that the noncommutative torus $C\left(\T^n_\th \right)$ is {good}. From the Lemma  \ref{product_lem} it follows that 
\be 
\begin{split}
	\mathfrak{Gelfand}\left(C_0\left(S^1 \right)\otimes  C\left(\T^n_\th \right)\right)   \cong S^1 \times \mathfrak{Gelfand}\left(C\left(\T^n_\th \right) \right)
\end{split}
\ee
From \eqref{long_pair_eqn} it follows that 

\be\label{tor_long_pair_eqn} 
\begin{split}
	0 \hookto  H^0_c\left(S^1 \times \mathfrak{Gelfand}\left(C\left(\T^n_\th \right) \right), \{x_0\}\times \mathfrak{Gelfand}\left(C\left(\T^n_\th \right)\right) ;\Z \right)\to \\ \to H^0_c\left(S^1 \times \mathfrak{Gelfand}\left(C\left(\T^n_\th \right) \right);\Z \right) \to\\ \to  H^0_{c\cap \{x_0\}\times \mathfrak{Gelfand}\left(C\left(\T^n_\th \right)\right)}\left(\{x_0\}\times \mathfrak{Gelfand}\left(C\left(\T^n_\th \right)\right);\Z \right)\to \\ \to 
	H^1_c\left(S^1 \times \mathfrak{Gelfand}\left(C\left(\T^n_\th \right) \right), \{x_0\}\times \mathfrak{Gelfand}\left(C\left(\T^n_\th \right)\right) ;\Z \right)\to \\ \to H^1_c\left(S^1 \times \mathfrak{Gelfand}\left(C\left(\T^n_\th \right) \right);\Z \right) \to\\ \to  H^1_{c\cap \{x_0\}\times \mathfrak{Gelfand}\left(C\left(\T^n_\th \right)\right)}\left(\{x_0\}\times \mathfrak{Gelfand}\left(C\left(\T^n_\th \right)\right);\Z \right)
\end{split}
\ee
and taking into account \eqref{setmunus_h_eqn} one has
\be\label{t_exact_eqn}
\begin{split}
	\{0\} \hookto  H^0_c\left(\R \times \mathfrak{Gelfand}\left(C\left(\T^n_\th \right) \right) ;\Z \right)\hookto  H^0_c\left(S^1 \times \mathfrak{Gelfand}\left(C\left(\T^n_\th \right) \right);\Z \right) \xrightarrow{p^0_*}\\   H^0_c\left( \mathfrak{Gelfand}\left(C\left(\T^n_\th \right)\right);\Z \right)\xrightarrow{\dl}\\ 
	H^1_c\left(\R \times \mathfrak{Gelfand}\left(C\left(\T^n_\th \right) \right) ;\Z \right)\xrightarrow{\iota^1_*} H^1_c\left(S^1 \times \mathfrak{Gelfand}\left(C\left(\T^n_\th \right) \right);\Z \right)\xrightarrow{p^1_*}\\ \to  H^1_c\left( \mathfrak{Gelfand}\left(C\left(\T^n_\th \right)\right);\Z \right)\to...
\end{split}
\ee
From the equations \eqref{k0_s_eqn} and \eqref{k1_s_eqn} one has 
\be\label{t1_eqn}
\begin{split}
	H^1_c\left(\R \times \mathfrak{Gelfand}\left(C\left(\T^n_\th \right) \right) ;\Z \right)\cong K_0\left( C\left(\T^n_\th \right)\right)\cong   \Z^{2n-1},\\
	H^1_c\left( \mathfrak{Gelfand}\left(C\left(\T^n_\th \right) \right) ;\Z \right)\cong K_1\left( C\left(\T^n_\th \right)\right)\cong   \Z^{2n-1}
\end{split}
\ee
From the Exercise \ref{cs1_exer} it follows that
\be\label{t2_eqn}
\begin{split}
	K_0\left( C\left(\T^n_\th \right) \otimes C\left(S^1 \right) \right) \cong K_1\left(  C\left(\T^n_\th \right) \otimes C\left(S^1 \right) \right)\cong\\ K_0\left(  C\left(\T^n_\th \right)\right) \oplus K_1\left( C\left(\T^n_\th \right) \right)\cong \Z^{2n}
\end{split}
\ee
From the exact sequence \eqref{t_exact_eqn} it turns out that 
both equations \eqref{t1_eqn} and \eqref{t2_eqn} are compatible if and only if the homomorphism $\iota^1_*$ is injective and the homomorphism $\p^1_*$ is surjective. It follows that $\dl$ is trivial, $p^0_*$ is a surjective homomorphism and
\be\label{t22_eqn}
\begin{split}
	H^0_c\left(S^1 \times \mathfrak{Gelfand}\left(C\left(\T^n_\th \right) \right);\Z \right)\cong \\ H^0_c\left(\R \times \mathfrak{Gelfand}\left(C\left(\T^n_\th \right) \right) ;\Z \right)\oplus H^0_c\left( \mathfrak{Gelfand}\left(C\left(\T^n_\th \right)\right);\Z \right) \cong  \Z^{2n}
\end{split}
\ee
From the isomorphism $H^0_c\left(S^1 \times \mathfrak{Gelfand}\left(C\left(\T^n_\th \right) \right);\Z \right) \cong  \Z^{2n}\cong 	K_0\left( C\left(\T^n_\th \right) \otimes C\left(S^1 \right) \right)$ and the equation \eqref{k_00_eqn} it follows that the natural action
$$
U\left(C\left(\T^n_\th \right)\otimes C\left(S^1 \right) \right)\times K_0\left( C\left(\T^n_\th \right) \otimes C\left(S^1 \right) \right)\to K_0\left( C\left(\T^n_\th \right) \otimes C\left(S^1 \right) \right)
$$
is trivial. If $\a_{u_{S^1}}\in \Aut\left(C\left( S^1\right)  \right) \cong \Aut\left(C\left( u_{S^1}\right)  \right)$ is a $*$- automorphism given by $u_{S^1} \mapsto e^{2\pi i \th} u_{S^1}$ then one construct the natural  $*$- automorphism $\a \in \Aut\left( C\left(\T^n_\th \right) \otimes C\left(S^1 \right)\right)$ such that 
\bean
C\left(\T^{n + 1}_\th \right)\cong  C\left(\T^n_\th \right)\rtimes_\a \Z
\eean .
The Theorem \ref{p_v_theorem} yields  a cyclic six-term exact sequence
\bean
\begin{tikzcd}
	K_0\left(C\left(\T^n_\th \right)\right)\arrow[r, "1-\a_*"] & K_0\left(C\left(\T^n_\th \right)\right)\arrow[r, "\iota^u_*"] & K_0\left( C\left(\T^n_\th \right) \rtimes_\a \Z \right) \arrow[d]\\
	K_1\left(  C\left(\T^n_\th \right)\rtimes \Z\right)\arrow[u] & \arrow[l, "\iota^u_*"] K_1\left( C\left(\T^n_\th \right) \right) & \arrow[l, "1-\a_*"]K_1\left(C\left(\T^n_\th \right) \right)
\end{tikzcd}
\eean
Taking into account $1-\a_*$ we have the following short exact sequences
\bea\label{kkt_0_nt_eqn}
\{0\}\hookto	K_0\left( C\left(\T^n_\th \right)\right) \xrightarrow{\iota_*} K_0\left( C\left(\T^n_\th \right)  \rtimes_\a \Z \right) \xrightarrow{p_*} K_1\left( C\left(\T^n_\th \right)\right) \onto \{0\}, \\ \label{kkt_1_nt_eqn}
\{0\}\hookto	K_1\left( C\left(\T^n_\th \right)\right) \xrightarrow{\iota_*} K_1\left( C\left(\T^n_\th \right)  \rtimes_\a \Z \right) \xrightarrow{p_*} K_0\left( C\left(\T^n_\th \right)\right) \onto \{0\}
\eea
From the Theorem \ref{nt_gen_lem} it follows that the image $\iota_*\left( 	K_0\left( C\left(\T^n_\th \right)\right)\right)\subset  K_0\left( C\left(\T^n_\th \right)  \rtimes_\a \Z \right)$ yields an open and closed subset  $\sX' \subset  \mathfrak{Gelfand}\left(C\left(\T^{n+1}_\th \right)\right)$, such that there is 
the disconnected union
\be\label{tor_prodwx_eqn}
\begin{split}
	\mathfrak{Gelfand}\left(C\left(\T^{n + 1}_\th \right)\right) =\sX'\bigsqcup	\sX''
\end{split}
\ee
$\sX'' \bydef \mathfrak{Gelfand}\left(C\left(\T^{n + 1}_\th \right)\right)\setminus\sX'$ Ir turns out that there is the following direct sum of rings.
\be\label{torh_prodwx_eqn}
\begin{split}
	H^*_c\left( 	\mathfrak{Gelfand}\left(C\left(\T^{n+1}_\th \right)\right);\Z\right)=\\= H^*_c\left(\sX'; \Z\right) \oplus H^*_c\left(\sX''; \Z\right)
\end{split}
\ee
Taking into account the Theorem \ref{pimsner_voiculesky_thm} one has the following direct sum of rings
\be\label{torhkx_prodw_eqn}
\begin{split}
	H^*_c\left( \sX'; \Z\right) \cong H^*_c\left( 	\mathfrak{Gelfand}\left(C\left(\T^{n}_\th \right)\right)\times S^1;\Z\right) ,\\
	H^*_c\left( \sX''; \Z\right) \cong H^*_c\left( 	\mathfrak{Gelfand}\left(C\left(\T^{n}_\th \right)\right)\times S^1\times \R;\Z\right)  ,\\
	K_*\left( C\left(\T^{n + 1}_\th \right)\right) =	K_*\left( C\left(\T^{n}_\th \right)\otimes C\left(S^1\right)\right) \oplus 
	K_*\left( C\left(\T^{n}_\th \right)\otimes C\left(S^1\right)\otimes C_0\left(\R \right) \right)
\end{split}
\ee
From the equation \eqref{torhkx_prodw_eqn} one can deduce is good (cf. Definition \ref{good_torus_defn}. Using the induction one has the following theorem
\begin{theorem}
	For any $n \in \N$ the noncommutative torus $C\left(\T^n_\th \right)$ is good.
\end{theorem}

\section{Hausdorff  blowing-up}\label{blowing_sec}
\subsection{Basic construction}	
\begin{definition}\label{blowing_defn}
	If $A$ is a $C^*$-algebra then an inclusion $\mathfrak {Blowing}_{\sX-A }:C_0\left( \sX\right) \hookto M\left(A \right)$  (or equivalently $C_0\left( \sX\right) \subset M\left(A \right)$ ) is \textit{Hausdorff  blowing-up} of $A$ if  both sets
	\be\label{blowing_eqn}
	\begin{split}
		\mathfrak {Blowing}_{\sX-A }\left( 	C_c\left( \sX\right)\right)\cdot A \bydef \left\{fa| f \in 	\mathfrak {Blowing}_{\sX-A }\left( 	C_c\left( \sX\right)\right)\quad a \in A \right\},\\
		A\cdot	\mathfrak {Blowing}_{\sX-A }\left( 	C_c\left( \sX\right)\right) \bydef \left\{af| f \in 	\mathfrak {Blowing}_{\sX-A }\left( 	C_c\left( \sX\right)\right)\quad a \in A \right\}
	\end{split}
	\ee
	are dense in $A$.
\end{definition}

\begin{remark}\label{blowingem}
	$	\mathfrak {Blowing}_{\sX-A }\left( 	C_c\left( \sX\right)\right)\cdot A$ is dense in $A$ if and only if $A\cdot 	\mathfrak {Blowing}_{\sX-A }\left( 	C_c\left( \sX\right)\right)$ is dense in $A$ (cf. equations \eqref{blowing_eqn}), i.e. both equations inclusions  \eqref{blowing_eqn} are equivalent.
\end{remark}
\begin{remark}
	The inclusion $
	\mathfrak{Blowing}_{\sX-A }:C_0\left( \sX\right) \hookto M\left(A \right)$ can be uniquely extended up to
	\be\label{blowing_p_eqn}
	\mathfrak{Blowing}^+_{\sX-A }:C\left( \sX^+\right) \hookto M\left(A \right)
	\ee
	such that $\sX^+$ is the one-point compactification  of $\sX$ (cf. Definition \ref{top_compactification_defn}) and $$\mathfrak{Blowing}^+_{\sX-A }\left( 1_{C\left( \sX^+\right)}\right) \bydef 1_{M\left(A \right) }.$$
\end{remark}

\begin{definition}\label{blowing_ideals_au_ua_defn}
	Let  $ C_0\left( \sX\right)\subset  M\left( A\right) $ be  Hausdorff blowing-up of $A$ (cf. Definition \ref{blowing_defn}), and let $\sU \subset \sX$ be an open subset. Both   left and right  closed ideals $A_\sU$  and $_\sU A$ of $A$ generated by sets 	$AC_0\left( \sU\right)$ and $C_0\left( \sU\right)A$ are the \textit{left} $\sU$-\textit{ideal} and the \textit{right} $\sU$-\textit{ideal} respectively. A hereditary $C^*$-subalgebra of $A$
	\be\label{blowing_hereditary_u_eqn} 
	\begin{split}
		_\sU A_\sU \bydef	~	_\sU A\cap  A_\sU = A^*	_\sU \cap  A_\sU\\ 
	\end{split}
	\ee	
	is the $ \sU$-\textit{subalgebra}.
	
\end{definition}

\begin{lemma}\label{blowing_compact_lem}
	If $C\cong C_0\left( \sY\right) \subset M\left(A \right)$ is    {Hausdorff blowing-up} of $A$  (cf. Definition \ref{blowing_defn}	for any $a \in A$ and $\eps > 0$ following conditions hold:
	\begin{enumerate}
		\item[(i)] there is a positive $f \in C_c\left( \sY\right)_+$ with 
		\be\label{blowing_compact_eqn}
		\begin{split}
			\left\| f  \right\| \le 1,\\
			\left\| a - af  \right\|< \eps,\\
			\left\| a - fa \right\|< \eps,\\
			\left\| a - faf  \right\|< \eps,
		\end{split}
		\ee
		\item[(ii)] there are an open subset $\sU \subset \sY$ with compact closure and $b \in ~_\sU A_\sU$ such that
		\be\label{blowing_compact_b_eqn}
		\begin{split}
			b \in  ~_\sU A_\sU,\\
			\left\| a - b \right\|< \eps.
		\end{split}
		\ee
	\end{enumerate}
\end{lemma}
\begin{proof}(i)
	If $g \in C_c\left( \sY\right)$ then $g = fg = gf$ for any positive $f \in C_c\left( \sY\right)_+$ such that
	\bean
	\left\| f \right\|= 1,\\
	f\left(\supp g \right) = 1.
	\eean 
	If $f' \in C_c\left( \sY\right)$ and $c' \in A$ such that
	$\left\|a-f'c' \right\| < \eps/4$ (cf. equation \eqref{blowing_eqn}) then there for any positive $f_1$ such that $\left\| f_1 \right\|= 1$, $~f_1\left(\supp f' \right) = 1$ one has
	\bean
	\left\|f_1\left( a-f'c \right) \right\|\le \left\|f_1 \right\|\left\|a-f'c \right\|\le \frac{\eps}{4},\\
	\left\|a-f_1a \right\| < \left\|f_1a - f_1f'c\right\|+ \left\|a - f_1f'c\right\|\le \frac{\eps}{2}
	\eean
	Similarly 	If $f'' \in C_c\left( \sY\right)$ and $c''/ \in A$ such that
	$\left\|a-f'c'' \right\| < \eps/4$ then for any positive for any positive $f_2$ such that $\left\| f_2 \right\|= 1$, $~f_1\left(\supp f' \right) = 1$ one has
	\bean
	\left\|a-a f_2 \right\| <  \frac{\eps}{2}
	\eean
	If $f = \max\left(f_1, f_2 \right)$ then $f\left(\supp f' \cup \supp f'' \right)= \{1\}$ 
	and
	\bean
	\begin{split}
		\left\| f  \right\| \le 1,\\
		\left\| a - af  \right\|< \frac{\eps}{2},\\
		\left\| a - fa \right\|< \frac{\eps}{2},\\
	\end{split}
	\eean
	On the other hand
	\bean
	\left\| a - faf \right\|\le \left\| a - fa \right\|+ \left\| fa - faf \right\|< \frac{\eps}{2}+ 	\left\| f\right\|\left\|a - fa\right\|< \eps.
	\eean
	(ii) The set $\sU \bydef \left\{y \in \sY | f\left(y\right)\neq 0\right\}\subset \sY$ is open and closure of $\sU$ is the compact set $\supp f$ (cf. Definition \ref{top_compact_defn}). Moreover  $f a f \in ~_\sU A_\sU$ and  $\left\|a - faf\right\|< \eps$.
	
\end{proof}

\begin{empt} 
	
	We leave to the reader a proof of following equations:
	\be\label{blowing_su_eqn}
	\begin{split}
		_\sU A \bydef \left\{ a \in A~ |~ \forall f \in C_0\left( \sY\right)\quad f\left(\sU \right)= \{0\}  \quad \Rightarrow \quad fa = 0\right\} ,\\
		A_\sU \bydef \left\{ a \in A~ |~ \forall f \in C_0\left( \sY\right)\quad  f\left(\sU \right)= \{0\} \quad \Rightarrow \quad af = 0\right\}
	\end{split}
	\ee

	\bea\label{blowing_sue1_eqn}
	\sU'\cap \sU'' =\emptyset\quad \Rightarrow\quad A_{\sU'}~_{\sU''}A= \{0\}.
	\eea
	
	From \eqref{blowing_su_eqn} it follows that
	\be\label{blowing_su_inc_eqn}
	\sU'\subset \sU'' \quad \Rightarrow\quad _{\sU'}A \subset~  _{\sU''}A ~\text{ AND } ~A_{\sU'}\subset A_{\sU''}~\text{ AND }~  _{\sU'}A _{\sU'}\subset _{\sU''}A _{\sU''}.
	\ee
\end{empt}

\begin{definition}\label{blowing_support_defn}
	If $C_0\left( \sX\right) \subset M\left(A \right)$ is    {Hausdorff blowing-up} of $A$  (cf. Definition \ref{blowing_defn}),  $a \in A$ and
	$
	\sU_a \bydef\bigcap 
	\left\{\left.{\sU} \subset \sX\right| a\in~_\sU A_{\sU} \right\}
	$
	then the  closure $\sV_a$  of $\sU_a$ is said to be the \textit{support} of $a$. We write $\supp a \bydef \sV_a$.
\end{definition}

\begin{lemma}\label{pedersen_eps_lem}
	Let $\eps > 0$, and let 	 $f_\eps: \R \to \R$ be a continuous function given by 
	\begin{equation}\label{f_eps_eqn}
		f_\eps\left( x\right)  \bydef\left\{
		\begin{array}{c l}
			0 &x \le \eps \\
			x - \eps & x > \eps
		\end{array}\right.
	\end{equation}
	If $A$ is a $C^*$-algebra 
	then one has
	\bea\label{k0_ped_e_eqn}
	K\left( A \right)_0 = \left\{f_\eps \left(a\right) \left|~a \in A_+, \quad \eps > 0 \right.\right\},~~~\\
	\label{kp_ped_e_eqn}
	K\left( A \right)_+ = \left\{a \in A_+ \left|a \le \sum_{j = 1}^n f_{\eps_j}\left(  a_j\right) \quad a_j \in  	K\left( A \right)_0\quad \eps_j > 0\quad j=1,...,n\right.\right\}~~
	\eea
	where both $	K\left( A \right)_0$ and 	$K\left( A \right)_+$ are given by equations \eqref{pedersen_k0_eqn} and \eqref{pedersen_k_plus_eqn} respectively
\end{lemma}
\begin{proof}
	If  $a\in A_+$ and $\eps > 0$ then $f_\eps \left(a\right)= \phi_\eps\left(a \right)$ where  $\phi_\eps \in K(]0, \infty [$ is given by 
	\bean
	\phi_\eps\left( x\right)  =\left\{
	\begin{array}{c l}
		0 &x \le \eps \\
		x - \eps &  \eps \le x \le \left\| a\right\|\\
		2	\left\| a\right\| - \eps - x & \left\| a\right\| \le x \le 2\left\| a\right\|-\eps\\
		0 & x \ge 2 \left\| a\right\| -\eps\\
	\end{array}\right.
	\eean
	It follows that $f_\eps\left(a\right)	\in K\left( A \right)_0$. Conversely  if $a \in	K\left( A \right)_0$ then from \eqref{pedersen_k0_eqn} it turns out that there is $b \in A_+$ and $\varphi \in   K(]0, \infty [$ such that $a = \varphi\left(b\right)$. If  $\supp \varphi \subset \left[\eps, c\right]$ and  $\psi \in C_c\left( \R\right)_+$ is given by
	\bean
	\psi\left( x\right)  =\left\{
	\begin{array}{c l}
		0 &x \le 0 \\
		x &0 \le x \le \eps \\
		\varphi\left(x\right) + \eps &\eps \le x \le c \\
		\eps + c - x&  c \le x \le c+ \eps\\
		0 & x \ge c + \eps
	\end{array}\right.
	\eean
	then $ \varphi = f_\eps  \circ\psi$. It follows that $a = f_\eps\left(b' \right)$ where $b' \bydef \psi\left(b\right)$. So the equation \eqref{k0_ped_e_eqn} is proven. The equation \eqref{kp_ped_e_eqn} is a direct consequence of \eqref{k0_ped_e_eqn} and \eqref{pedersen_k_plus_eqn} ones.
	
\end{proof}

\begin{theorem}\label{blowing_pedersen_compact_thm}
	If  $C_0\left( \sX\right)\hookto M\left( A\right)$ is Hausdorff  blowing-up and $a\in A$ belongs to the Pedersen's ideal $K\left(A \right)$ (cf. Definition \ref{pedersen_ideal_defn}) then the support of $a$ (cf. Definition \ref{blowing_support_defn}) is compact (cf. Definition \ref{top_compact_defn}).
\end{theorem}
\begin{proof}
	If $a \in K\left(A\right)_0$ (cf \eqref{pedersen_k0_eqn}) then from the Lemma  \eqref{pedersen_eps_lem} it follows that there is $\eps > 0$ and $b \in A_+$ such that $a = f_\eps \left( b\right)$ where $f_\eps$ is given by \eqref{f_eps_eqn}. On the other hand  there is a positive element   $c\in  A_+$  such that $\left\|c - b \right\|  < \eps/2$ and $\supp c$ is compact (cf. Definition  \ref{blowing_defn}).
	If $a \le c$ does not hold and $\rho: A \hookto B\left(\H \right)$ is a faithful  nondegenerate representation  then there is $\xi \in \H$ such that
	\bean\label{blowing_ac_eqn}
	\forall \xi \in \H \quad \left( \xi, \rho\left( a \right) \xi\right) > \left( \xi, \rho\left( c \right) \xi\right)
	\eean
	From $\left\|c - b \right\|  < \eps/2$ it follows that
	\be\label{blowing_acc_eqn}
	\forall \xi \in \H  \quad \left\| \xi \right\| = 1\quad \Rightarrow\quad  \left|\left( \xi, \rho\left( b \right) \xi\right)-\left( \xi, \rho\left( c \right) \xi\right) \right| < \eps/2
	\ee
	On the other hand from $a =f_\eps \left( b\right)> 0$ it follows that there is $\xi \in \H$ and $\la \in \R_+$ such that
	\bean
	\left\| \xi \right\| = 1,\\
	\rho\left(a \right) \xi = \la \xi,\\
	\rho\left(b \right)\xi = \left( \la+\eps \right) \xi,\\
	\rho\left(a \right)\xi = \la  \xi,\\
	\left( \xi, \rho\left( b \right) \xi\right) - \left( \xi, \rho\left( a \right) \xi\right)= \eps
	\eean
	and taking into account \eqref{blowing_acc_eqn} one has
	$$
	\left( \xi, \rho\left( c \right) \xi\right)- \left( \xi, \rho\left( a \right) \xi\right)> \frac{\eps}{2}.
	$$
	Above condition contradicts with \eqref{blowing_ac_eqn} so $a \le c$.
	If $\supp a \subsetneqq \supp c$ then there is a nonempty  open set  $\sU \subset \supp a\setminus \supp c$. For any $f \in C_0\left(\sU \right)\setminus \{0\}$ one has
	\bean
	faf^* > 0,\\
	fcf^* = 0.
	\eean
	However it is impossible since $a \le c$, so $\supp a \subsetneqq \supp c$ is not true and $\supp a \subset\supp c$. Thus the set $\supp a$ is a closed subset of the compact set $\supp c$ therefore $\supp a$ is compact. Using this fact and the Definition \ref{pedersen_ideal_defn} we conclude that $\supp a$ is compact for any $a \in K\left( A\right)$. 
\end{proof}
\begin{remark}
	The Theorem \ref{blowing_pedersen_compact_thm} can be regarded as a generalization of the equation  \eqref{peder_c0_eqn}.
\end{remark}

\begin{remark}\label{blowing_gelfandem}
	The Lemma means that any Hausdorff blowing-up is a good $*$-homomorphism  $\mathfrak {Blowing}_{\sX-A }:C_0\left( \sX\right)\hookto M\left( A\right)$ (cf. Definition \ref{good_hom_defn}), so there is the natural surjective continuous map
	\be\label{blowing_gelfand_eqn}
p_{\mathfrak{Gelfand}\left(A \right) - \sX}	:\mathfrak{Gelfand}\left(A \right) \onto \sX
	\ee
	(cf. Theorem \ref{good_thm}).
\end{remark}

\begin{definition}\label{blowing_compact_supp_defn}
	The given by
	\be\label{blowing_compact_supp_eqn}
	A_c \bydef \left\{a \in A | \supp\left(a \right) \quad \text{is compact}\quad \right\}\subset A
	\ee
dense $*$-subalgebra of $A$ is said to be the \textit{algebra of compactly supported elements}.
\end{definition}

\begin{lemma}\label{blowing_ajk_lem}
	Let $A$ be a $C^*$-algebra with Hausdorff blowing-up $C_0\left(\sY\right)\to M\left(A \right)$. If a family $\left\{\sU_\a\right\}_{\a \in \mathscr A}$ of open subsets of $\sY$ is such that $\sY = \bigcup_{\a \in \mathscr A}\sU_\a$ then any element $a \in A_c$ of  {algebra of compactly supported elements} (cf. Definition \ref{blowing_compact_supp_defn}) is given by
	\be\label{blowing_ajk_eqn}
	a = \sum_{\substack{j=1\\ k = 1}}^n f_ja f_k \quad \forall j = 1,..., n\quad \exists \a \in \mathscr A\quad \supp f_j \subset \sU_\a
	\ee
\end{lemma}
\begin{proof} it follows that the support $\supp a$ (cf. Definition \ref{blowing_support_defn}) of $a$ is compact. If $f \bydef \sum_{j=1}^n f_j$ is a dominated to the family $\left\{\sU_\a\right\}_{\a \in \mathscr A}$ covering sum for $\supp a$ (cf. Definition \ref{top_covering_sum_defn}) then one has
	$$
	a = faf = \sum_{\substack{j=1\\ k = 1}}^n f_ja f_k.
	$$
\end{proof}

\begin{remark}
	The "blowing-up" word is inspired by following reasons:
	\begin{itemize}
		\item the existence  of the  surjective  map $\mathfrak{Blowing}$ from the  Hausdorff  space  to the spectrum, 
		\item  in the algebraic geometry   "blowing-up" means  from a map $\sX \to \sY$ such that the variety $\sX$ has no singular points (cf. \cite{hartshorne:ag}).
	\end{itemize}
\end{remark}

\subsection{Lift and descent}
\begin{definition}\label{top_lift_desc_defn} Let $\widetilde{\sU}\subset \widetilde{\sX}$ be an open subset such that the restriction $\left.p\right|_{\widetilde{\sU}}$ is injective. Let $\sU\bydef p\left( \widetilde{\sU}\right)$. If $a \in C_b\left(\sX \right)$ is such that $\supp a \subset \sU$ (cf. Definition \ref{top_support_defn}) then there is $\widetilde a \in C_b\left( \widetilde \sX\right)$ such that 
	\be\label{comm_lift_desc_eqn}
	\widetilde{a}\left( \widetilde{x}\right) =\left\{\begin{array}{c l}
		a \left(p\left(\widetilde x \right)  \right)     &{\widetilde{x}}\in \widetilde{\sU} \\
		0 &{\widetilde{x}}\notin \widetilde{\sU}\\
	\end{array}\right. \quad 
	\ee 
	is said to be the $p$-$\widetilde{\sU}$-\textit{lift} or simply the $\widetilde{\sU}$-\textit{lift} of $a$. Otherwise we say that $a$ is the $p$-\textit{descent} of $ \widetilde{a}$. We write
	\be\label{top_lift_desc_eqn}
	\begin{split}
		\widetilde{a}\stackrel{\text{def}}{=} \lift^p_{\widetilde{\sU}}\left(a \right)  \text{ or simply } \widetilde{a}\stackrel{\text{def}}{=} \lift_{\widetilde{\sU}}\left(a \right), \\
		a \stackrel{\text{def}}{=} \desc_{p} \left(\widetilde{a}  \right)  \text{ or simply } a \stackrel{\text{def}}{=} \desc \left(\widetilde{a}  \right).
	\end{split}
	\ee
	If $\widetilde{\sV}$ is the closure of $\widetilde{\sU}$ then we define  the \textit{closed}-$p$-$\widetilde{\sV}$-\textit{lift}   of $a$ by the following way
	\be\label{top_lift_desc_closed_eqn}
	\begin{split}
		\lift^p_{\widetilde{\sV}}\left(a \right)  \bydef \lift^p_{\widetilde{\sU}}\left(a \right).
	\end{split}
	\ee
	
\end{definition}
\begin{remark}
	if $a = \desc_{p} \left(\widetilde{a}  \right)$ then $a$ is represented by the family $\left\{a_x\right\}_{x \in \sX}$ such that
	\be\label{top_desc_eqn}
	a_x = \begin{cases}
		\widetilde{a}\left( {\widetilde{x}}\right)  & x \in \sU ~\text{AND} ~ \widetilde x \in \widetilde\sU~\text{AND} ~p\left( \widetilde{x}\right) = x \\
		0 &  x \notin \sU
	\end{cases}	\quad 
	\ee 
\end{remark}
\begin{remark}\label{comm_lift_desc_l_rem}
	From the equation \eqref{top_lift_desc_eqn} it turns out 
	\begin{equation}\label{comm_lift_desc_l_eqn}
		\begin{split}
			a = \mathfrak{desc}_{p}\left( \mathfrak{lift}^p_{\widetilde{\mathcal U}}\left(a \right)\right) ,\\
			\widetilde{a} = \mathfrak{lift}^p_{\widetilde{\mathcal U}}\left( \mathfrak{desc}_{p}\left(\widetilde{a} \right)\right). 
		\end{split}
	\end{equation}	
\end{remark}

\begin{definition}\label{top_coveing_triple_defn}
	Let $p:\widetilde\sX \to \sX$ be a covering. Suppose that there are connected open subsets $\widetilde \sU, \widetilde \sV \subset \widetilde \sX$, and a continuous map $\widetilde{s}: \widetilde \sX \to \left[0,1\right]$ such that
	\begin{enumerate}
		\item[(a)] the closure of $\widetilde \sV$ is compact,
		\item[(b)] $\widetilde{   \mathcal V}$ is mapped homeomorphically onto $\sV \bydef p\left(\widetilde{   \mathcal V} \right)$,
		\item[(c)] $\supp \widetilde s \subset \widetilde{   \mathcal V}$,
		\item[(d)] $\widetilde s\left(\widetilde \sU \right) = \{1\}$.
	\end{enumerate}
	We say that $\left(\widetilde \sU, \widetilde \sV, \widetilde s\right)$ is a \textit{covering triple for} $p$.
\end{definition}
\begin{definition}\label{top_coveing_quadruple_defn}
	Let $p:\widetilde\sX \to \sX$ be a covering. Suppose that there are connected open subsets $\widetilde \sU, \widetilde \sV, \widetilde {\mathcal W} \subset \widetilde \sX$, and two continuous map $\widetilde{s}: \widetilde \sX \to \left[0,1\right]$, $\widetilde{t}: \widetilde \sX \to \left[0,1\right]$ such that both
	$\left(\widetilde \sU, \widetilde \sV, \widetilde s\right)$ and $\left(\widetilde \sV,  \widetilde {\mathcal W}, \widetilde s\right)$ are covering triples for $p$ then a tuple 
	$\left(\widetilde \sU, \widetilde \sV, \widetilde {\mathcal W},\widetilde s, \widetilde t\right)$ is a \textit{covering quadruple for} $p$.
\end{definition}

\begin{remark}\label{top_covering_sum_subordinated_rem}
	From the Lemma \ref{top_a_u_lem} it follows that if the space $\sX$ is Hausdorff, locally compact, and locally connected then for any  $\widetilde x \in \widetilde\sX$ there is a covering triple $\left(\widetilde \sU, \widetilde \sV, \widetilde s\right)$  and a covering quadruple $\left(\widetilde \sU, \widetilde \sV, \widetilde {\mathcal W},\widetilde s, \widetilde t\right)$  for $p$.
\end{remark} 
\begin{empt}\label{top_cov_empt}
	Let $p: \widetilde{\sX}\to \sX$ be a   covering (cf. Definition \ref{top_covering_defn}). Suppose that the space $\sX$ is locally connected (cf. Definition \ref{top_locally_connected_defn}). Let
	$$
	\left\{\left(\sU_\a, \sV_\a, s_\a\right)\right\}_{\a \in \mathscr A}
	$$
	be a family of triples such that 
	\begin{itemize}
		\item for any $\a \in \mathscr A$ there is  a covering  triple  $\left(\widetilde \sU, \widetilde \sV, \widetilde s\right)$ for $p$ (cf. Definition \ref{top_coveing_triple_defn}) with
		\bean
		\sU \bydef p\left( \widetilde \sU \right),\\
		\sV \bydef p\left( \widetilde \sV \right),\\
		s_\a = \desc_p\left(\widetilde s \right), 
		\eean 
		where $\desc_p$ is the $p$-descent (cf. Definition \ref{top_lift_desc_defn}),
		\item  $\sX = \bigcup_{\a \in \mathscr A} \sU_{{\a}}$.
	\end{itemize}
	The existence of the family follows from the Remark \ref{top_coveing_triple_defn}.
	For any $\a\in\mathscr A$ we select  a connected open set $\widetilde \sV_\a\subset \widetilde \sX$ which is mapped homeomorphically onto $\sV_\a$.  Let $\left(\widetilde \sU_\a, \widetilde \sV_\a, \widetilde s_\a\right)$ be a triple  such that  $\widetilde \sU_\a \bydef \widetilde\sV_\a \cap p^{-1}\left(\sU_\a \right)$ and $\widetilde s_\a \bydef \lift^p_{\widetilde{\sV}_\a}\left( s_\a\right)$ where $\lift^p_{\widetilde{\sV}_\a}$ is $p$-$\widetilde{\sV}_\a$-lift (cf. Definition \ref{top_desc_eqn}). If $\widetilde{\mathscr A}\bydef G\times  \mathscr A$ then for any $\widetilde \a = \left(g, \a \right)\in \widetilde{\mathscr A}$. If  we denote by
	$$
	\left(\widetilde \sU_{\widetilde\a}, \widetilde \sV_{\widetilde\a}, \widetilde s_{\widetilde\a}\right)\bydef \left(g\widetilde \sU_{\a},g \widetilde \sV_\a, g\widetilde s_\a\right). 
	$$ 
	There are following objects:
	\begin{itemize}
		\item  the natural action $G\times \widetilde{\mathscr A}\to \widetilde{\mathscr A}$,
		\item the natural projection $p_{\widetilde{\mathscr A}}: \widetilde{\mathscr A} \to {\mathscr A}$.
	\end{itemize}
	Moreover  one has $\widetilde \sX = \bigcup_{\widetilde \a \in \widetilde{\mathscr A}} \widetilde \sU_{\widetilde \a}$.
	
\end{empt}

\begin{definition}\label{top_cov_defn}
	Under the hypotheses \ref{top_cov_empt} we say that a pair $$\left(\	\left\{\left(\sU_\a, \sV_\a, s_\a\right)\right\}_{\a \in \mathscr A}, \left\{\left(\widetilde \sU_{\widetilde\a}, \widetilde \sV_{\widetilde\a}, \widetilde s_{\widetilde\a}\right)\right\}_{\widetilde \a \in \widetilde{\mathscr A}} \right)$$ is a $p$-\textit{covering}. Both the action $G\times \widetilde{\mathscr A}\to \widetilde{\mathscr A}$ and the map  $p_{\widetilde{\mathscr A}}: \widetilde{\mathscr A} \to {\mathscr A}$ are $p$-\textit{action} and $p$-\textit{projection} respectively.
\end{definition}

 \subsection{Weak fundamental group}
 \paragraph{}
 The notion of the fundamental group (cf. \cite{spanier:at}) is relevant to path connected spaces (cf. Definition \ref{top_path_connected_defn}). Here we consider a generalized notion which can be applied for more general connected spaces. 
 
 \begin{lemma}\label{top_compact_nigthbohood_lem}
 	If $\sY\subset\sX$ is a compact subset of a locally compact,  Hausdorff space then there is an open subset $\sU \subset\sX$ with compact closure, such that $\sY \subset \sU$.
 \end{lemma}
 \begin{proof}
 	If  $f \bydef \sum_{j = 1}^n f_j\left( x\right)$ is a covering sum of $\sY$  then one has
 	\begin{itemize}
 		\item a set $\sU \bydef \left\{x \in \sX | f\left( x\right)> 0 \right\}$ is open,
 		\item the closure of $\sU$ is compact because it equals to  a finite union $\bigcup_{j=1}^n \supp f_j$ of compact sets.
 	\end{itemize}
 \end{proof}
  \begin{remark}
 	Sometimes instead of \eqref{top_cfs_eqn} we use the following notation for a covering  sum
 	\be\label{top_cfsa_eqn}
 	\sum_{\a\in \mathscr A} a_\a
 	\ee
 	where $\mathscr A$ is a finite set.
 \end{remark}
 
 \begin{definition}\label{top_covering_sum_subordinated_defn}
 	Let $p:\widetilde\sX \to \sX$ be a covering. Let both $\sY\subset\sX$ and $\widetilde\sY\subset\widetilde\sX$ are compact subsets. A covering sum $\sum_{j = 1}^n f_j$ for $\sY$ (cf. Definition \ref{top_covering_sum_defn}) is said to be \textit{subordinated  to} $p$ if for all $j=1,..., n$ the set $\left\{\left.x\in\sX\right|f_j\left( x\right)>  0 \right\}$ is evenly covered  by $p$ (cf. Definition \ref{top_covering_defn}). A covering sum $\sum_{j = 1}^n \widetilde f_j$ for $\widetilde \sY$ is said to be \textit{subordinated to} $p$ if for all $j=1,..., n$ the set $\widetilde\sU_j\bydef \left\{\left.\widetilde x\in\widetilde\sX\right|\widetilde f_j\left(\widetilde x\right)>  0 \right\}$ is homeomorphically mapped onto $p\left(\widetilde\sU_j\right)$. 
 \end{definition}

 \begin{lemma}\label{top_tietze_ext_lem}
 	If $\sX$ is a Hausdorff,  locally compact space, and $\sY \subset \sX$ is a compact subspace then the restriction map
 	\bean
 	C_c\left(\sX \right) \to C\left(\sY \right) ,\\
 	f \mapsto f|_{\widetilde\sY}
 	\eean
 	is surjective.
 \end{lemma}
 \begin{proof}
 	Let $f \in C\left(\sY \right)$.
 	If $f' \in C_c\left( \sX\right)$ is a covering sum of $\sY$ (cf. Definition \ref{top_covering_sum_defn}) then from the Theorems \ref{comp_normal_thm} and \ref{tietze_ext_thm} it follows that there if $f'' \in C\left(\supp f'\right)$ such that $ff''|_\sY = f$. However $ff''f' \in C_c\left(\sX \right)$ and  $ff''f'|_\sY = f$.
 \end{proof}
 
 \begin{definition}\label{top_weakly_simply_connected_defn} 
 	A connected space $\sX$ is said to be \textit{weakly simply connected} if for any covering $p: \widetilde \sX\to \sX$ the space $\sX$ is evenly covered by $p$ (cf. Definition \ref{top_covering_defn})
 \end{definition}
 \begin{remark}
 	Any connected, path connected, simply connected  space is weakly simply connected.
 \end{remark}
 \begin{remark}
 	The Definition \ref{top_weakly_simply_connected_defn} can be regarded as a generalization of \ref{top_weakly_simply_connected_defn} one.
 \end{remark}
 
 \begin{defn}\label{top_weakly_semi1_defn}\cite{spanier:at}
 	An open connected subset $\sU\subset \sX$ of a topological space is said to be \textit{semilocally proper} if it is evenly covered by any covering $\widetilde{\sX}\to\sX$ (cf. Definition \ref{top_covering_defn}).
 	A space $\sX$ is said to be \textit{weakly semilocally 1-connected} if for  every point $x$ and there is a semilocally proper neighbourhood.
 \end{defn} 
  
 \begin{definition}\label{top_weak_path_defn}
 	Let $\sX$ be a topological space.	A \textit{weak path on} $\sX$ is a finite sequence $\left(\sU_1, ..., \sU_n\right)$ of open  connected subsets of $\sX$ such that $\sU_j \cap \sU_{j+1} \neq \emptyset$ for any $j = 1, ..., n - 1$. A weak path $\left(\sU_1, ..., \sU_n\right)$ is said to be \textit{weakly semilocally 1-connected} if $\sU_j$ is semilocally proper for all $j = 1,..., n$.
 \end{definition}

 \begin{definition}\label{top_weak_path_conn_defn}
 	A  topological space $\sX$ is \textit{weakly path-connected} if for any $x'. x'' \in \sX$ there is {weakly semilocally 1-connected} path $\left(\sU_1, ..., \sU_n\right)$ with $x'\in \sU_1$ and $x'' \in\sU_n$.
 \end{definition}

 \begin{empt}\label{top_path_inv_empt}
 	There is an involution * on the set of paths such that
 	\be\label{top_path_inv_eqn}
 	\left(\sU_{1},...,\sU_{n}\right)^*\bydef \left(\sU_{n},...,\sU_{1}\right).
 	\ee
 \end{empt}
 \begin{definition}\label{top_path_comp_defn}
 	If both $\mathfrak p'=\left(\sU'_{1},...,\sU'_{k}\right)$ and $\mathfrak p''=\left(\sU''_{1},...,\sU_{l}''\right)$ are paths such that $\sU'_{k} \cap \sU''_{1}\neq\emptyset$ then we say that a pair $\left( \mathfrak p', \mathfrak p''\right)$ is \textit{composable}, and the path $\mathfrak p=\left(\sU'_{1},...,\sU'_{k},\sU''_{_1},...,\sU_{l}'' \right)$  is said to be the \textit{composition} of $\mathfrak p'$ and $\mathfrak p''$. We write
 	\be\label{top_path_comp_eqn}
 	\mathfrak p'\circ \mathfrak p''\bydef\mathfrak p.
 	\ee
 \end{definition}
 
 \begin{remark}\label{top_path_ass_rem}
 	The given by the Definition \ref{top_path_comp_defn} composition is associative, i.e.  if both pairs $\left( \mathfrak p', \mathfrak p''\right)$ and $\left( \mathfrak p'', \mathfrak p'''\right)$ are composable then one has
 	\be\label{top_path_ass_eqn}
 	\left( \mathfrak p'\circ \mathfrak p''\right)  \circ \mathfrak p''' = \mathfrak p'\circ \left( \mathfrak p'' \circ \mathfrak p'''\right) .
 	\ee 
 \end{remark}
 
 If $p: \widetilde{\sX} \to \sX$ is a covering and $\widetilde{\mathfrak p}\bydef \left(\widetilde\sU_1, ..., \widetilde\sU_n\right)$ is weak path on $\widetilde\sX$ then $\mathfrak p\bydef \left(p\left( \widetilde\sU_1\right) , ..., p\left( \widetilde\sU_n\right) \right)$ is a weak path on $\sX$.
 \begin{defn}\label{top_path_desc_defn}
 	In the above situation we say that $\mathfrak p$ is a $p$-\textit{descent} of $\widetilde{\mathfrak p}$. We write
 	\be\label{top_path_desc_eqn}
 	\mathfrak p\bydef \desc_p\left( \widetilde{\mathfrak p}\right). 
 	\ee
 \end{defn}
 \begin{definition}\label{top_gen_path_defn}
 	Let $\sX$ be a topological space.
 	Let $\mathfrak{A}\bydef\left\{\sU_\a\right\}_{\a \in \mathscr A}$ be a family of connected open subsets of $\sX$, such that $\sX = \cup~ \sU_\a$. A weak path $\left(\sU_{1},...,\sU_{n}\right)$ (cf. Definition \ref{top_weak_path_defn}) on $\sX$  is said to be an  $\mathfrak{A}$-\textit{path} if  $\sU_j \in \mathfrak{A}$ for all $j = 1,...,n$. 
 \end{definition}

 \begin{lemma}\label{top_gen_path_lem}
 	If $\sX$ is a locally compact, Hausdorff space then under the hypothesis 
 	of the Definition \ref{top_gen_path_defn}. For any $\sU_{\a'}, \sU_{\a''} \in \mathfrak{A}$ there is a $\left\{\sU_\a\right\}$-{path} $\left(\sU_{\a_1},...,\sU_{\a_n}\right)$ such that $\sU_{\a_1}=\sU_{\a'}$ and $\sU_{\a_n}=\sU_{\a''}$ .
 \end{lemma}
 
 \begin{proof}
 	Denote by
 	\bean
 	\mathfrak{A}_{\a'}\bydef \left\{\sU_{\a''}\in \mathfrak A\left|\exists \text{ path }\left(\sU_{\a_1},...,\sU_{\a_n}\right)\quad\a_1 = \a',\quad \a_n = \a'' \right.\right\},\\
 	\mathscr A'\bydef \left\{\left.\a \in \mathscr A\right| \sU_{\a} \in\mathfrak{A}_{\a'} \right\}.
 	\eean
 	If $\mathscr A'=\mathscr A$ then this lemma is proven. If not then  and 
 	$$
 	\left( \bigcup_{\bt \in\mathscr A'}\sU_{\bt}\right) \bigcap \left( \bigcup_{\a \in\mathscr A\setminus\mathscr A'}\sU_{\a}\right)\neq \emptyset,
 	$$
 	then there is  $\bt \in\mathscr A'$ and $\a \in\mathscr A\setminus\mathscr A'$ such that $\sU_{\bt}\cap \sU_{\a}\neq \emptyset$. According to the definition of $\mathfrak{A}_{\a'}$ there is a path $\left(\sU_{\a_1},...,\sU_{\a_n}\right)$ such that $\a_1 = \a',\quad \a_n = \bt$. So one has a path $\left(\sU_{\a_1},...,\sU_{\a_n}, \sU_{\a} \right)$, it follows that $\a \in \mathscr A'$ and one has a contradiction. It turns out that 
 	\bean
 	\left( \bigcup_{\bt \in\mathscr A'}\sU_{\bt}\right) \bigcap \left( \bigcup_{\a \in\mathscr A\setminus\mathscr A'}\sU_{\a}\right)= \emptyset.
 	\eean
 	Both unions $\bigcup_{\bt \in\mathscr A'}\sU_{\bt}$ and  $\bigcup_{\a \in\mathscr A\setminus\mathscr A'}\sU_{\a}$ of open sets are open, so one has a disjoint union
 	$$
 	\sX = \left( \bigcup_{\bt \in\mathscr A'}\sU_{\bt}\right) \bigsqcup \left( \bigcup_{\a \in\mathscr A\setminus\mathscr A'}\sU_{\a}\right),
 	$$
 	i.e.  $\sX$ is not connected. It contradicts with our assumption.	From this contradiction it turns out that $\mathscr A' = \mathscr A$.  
 \end{proof}
 
 \begin{corollary}\label{top_connected_union_cor}
 	If $\sX$ is a connected, locally connected, locally compact, Hausdorff space  then there is a  directed set $\La$ and  family $\left\{\sU_\la\right\}_{\la\in\La}$ of connected open subsets of $\sX$ such that
 	\begin{itemize}
 		\item 
 		$$
 		\sX = \bigcup_{\la\in \La}\sU_\la,
 		$$
 		\item a closure of $\sU_\la$ is compact for each $\la\in \La$,
 		\item
 		\be\label{top_lmit_eqn}
 		\forall \mu, \nu \in \La  \quad \sU_\mu \cap \sU_\nu \neq \emptyset,
 		\ee
 		\item
 		\be\label{top_lmord_eqn}
 		\forall \mu, \nu \in \La \quad \mu \le \nu\quad\Leftrightarrow \quad \sU_\mu \subset \sU_\nu.
 		\ee
 	\end{itemize}
 \end{corollary}
 \begin{proof}
 	Consider a situation of the Lemma \ref{top_gen_path_lem}. Select $\a_0 \in \mathscr A$. For any $\a\in \mathscr A$ we select a path $\left\{\sU_\a\right\}$-{path} $\left(\sU_{\a_1},...,\sU_{\a_n}\right)$ such that $\sU_{\a_1}=\sU_{\a_0}$ and $\sU_{\a_n}=\sU_{\a}$ (cf Lemma \ref{top_gen_path_lem}). For all $\a \in \mathscr A$ a closure of a finite union $\sV_{\a}\bydef \bigcup \sU_{\a_j}$ is compact because a closure of $\sU_{\a_j}$ is compact for every $j = 1,..., n$. From $\sU_{\a_j}\cap \sU_{\a_{j+1}}\neq \emptyset$ for all $j=1,...,n-1$ it turns out that $\sV_{\a}$ is connected, and taking into account $\bigcup_{\a\in \mathscr A}\sU_\a = \sX$ one has 
 	\bean
 	\forall \a\in \mathscr A \quad  \sU_\a\subset \sV_\a\quad \Rightarrow\quad \bigcup_{\a\in \mathscr A}\sU_\a\subset \bigcup_{\a\in \mathscr A}\sV_\a \quad \Rightarrow\quad\bigcup_{\a\in \mathscr A}\sV_\a= \sX.
 	\eean 
 	If $\left\{\sU_\la\right\}_{\la \in \La}$ is a family of all finite unions  $$\sU_\la\bydef \bigcup_{j=1}^n \sV_{\a_j}$$ then from $\sU_{\a_0} \subset \sV_{\a}$ it follows that $\sU_\la$ is connected for all $\la \in \La$ and the condition \ref{top_lmit_eqn} holds.
 	The condition \eqref{top_lmord_eqn} follows from $\sU_\a \subset \sU_\la$ for all $\la\in \La$.
 	A union of two finite unions is also finite union so one has
 	$$
 	\forall \mu, \nu \in \La\quad \exists \la \in \La\quad \sU_\mu \cup \sU_\nu = \sU_\la.
 	$$
 	So if we consider the given by \eqref{top_lmord_eqn} pre-order then $\La$ is a directed set.	
 \end{proof}
 
 \begin{defn}\label{top_closed_path_defn}
 	If $\left(\sX, x_0\right)$ is a pointed space then a path $\left(\sU_1, ..., \sU_n\right)$ on $\sX$ is said to be $x_0$-\textit{path} if $x_0\in\sU_1$.  The  $x_0$-{path} is said to be \textit{closed} if $x_0\in\sU_n$
 \end{defn}
 \begin{rem}
 	Any pair of closed path is composable.
 \end{rem}
 
 \begin{exercise}\label{top_path_lift_exer}
 	Let $p: \left(\widetilde\sX, \widetilde x_0\right)\to\left(\sX, x_0\right)$ be a pointed covering, and let $\mathfrak{A}\bydef\left\{\sU_\a\right\}_{\a \in \mathscr A}$ 
 	be a family of connected all open subsets of $\sX$ evenly covered by $p$ (cf. Definition \ref{top_covering_defn}). Let  ${\mathfrak p}\bydef \left(\sU_1, ..., \sU_n\right)$ be a $\mathfrak{A}$-{path}, such that $x_0\in\sU_1$. 
 	Prove that there is an unique  $\widetilde{\mathfrak p}\bydef \left(\widetilde \sU_1, ...,\widetilde \sU_n\right)$ such that $\widetilde x_0 \in \widetilde \sU_1$ and ${\mathfrak p}$ is the $p$-descent of $\widetilde{\mathfrak p}$.
 \end{exercise}
 \begin{definition}\label{top_path_lift_defn}
 	The given by the Exercise  \ref{top_path_lift_exer} path $\widetilde{\mathfrak p}$ is said to be the $p$-\textit{lift} of $\mathfrak p$. We write
 	\be\label{top_path_lift_eqn}
 	\lift_p\left( ~\mathfrak p\right) \bydef\widetilde{\mathfrak p}.
 	\ee 
 	or the $p$-$\widetilde x_0$-\textit{lift}.
 \end{definition}

 \begin{definition}\label{top_contactible_path_defn}
 	Let 
 	$\left(\sX, x_0\right)$ be  a pointed space such that  $\sX$ is a connected, locally connected and weakly semilocally 1-connected space.  A closed  semilocally 1-connected path $x_0$-path $\left(\sU_1, ..., \sU_n\right)$ (cf. Definition \ref{top_weak_path_defn} and \ref{top_closed_path_defn}) is said to be \textit{contractible} if for any pointed covering  $p: \left(\widetilde\sX, \widetilde x_0\right)\to\left(\sX, x_0\right)$ a $p$-{lift} of $\mathfrak p$ is closed (cf. Definitions \ref{top_closed_path_defn} and \ref{top_path_lift_defn}).
 \end{definition}
 
 \begin{exercise}\label{top_path_inv_exer} 
 	Prove that 	if $\mathfrak p$ is a closed semilocally 1-connected $x_0$-path then both compositions $\mathfrak p^* \circ \mathfrak p$ and $\mathfrak p \circ \mathfrak p^*$ are contractible.
 \end{exercise}
 \begin{definition}
 	We say that both closed, weakly semilocally 1-connected  $x_0$-paths $\mathfrak p_1$ and $\mathfrak p_2$ are \textit{homotopically equivalent} if a composition $\mathfrak p_1 \circ \mathfrak p^*_2$ is contractible.
 \end{definition}
 
 \begin{exercise}
 	Prove that $x_0$-paths $\mathfrak p_1$ and $\mathfrak p_2$ are {homotopically equivalent} if and only if at least one of following paths
 	\bean
 	\mathfrak p^*_1 \circ \mathfrak p_2,\\
 	\mathfrak p_2 \circ \mathfrak p^*_1,\\
 	\mathfrak p^*_2 \circ \mathfrak p_1
 	\eean 
 	is contractible.
 \end{exercise}
 \begin{empt}\label{top_weak_fundamental_group_empt}
 	From the Remark \ref{top_path_ass_rem} and the Exercise \ref{top_path_inv_exer} it follows that the set of classes of closed weakly semilocally 1-connected $x_0$-paths modulo homotopical equivalence  is a group with  a composition  and an inversion given by $\circ$ and $*$ operations respectively.
 \end{empt}
 \begin{definition}\label{top_weak_fundamental_group_defn}
 	If
 	$\left(\sX, x_0\right)$ is  a pointed space, such that $\sX$ is a connected, locally connected and weakly semilocally 1-connected space
 	then the described above group is said to be a \textit{weak fundamental group} of $\sX$. We denote it by $\pi_1^{\mathrm{w}}\left(\sX, x_0 \right) $.
 \end{definition}

 Let $\left(\sX, x_0\right)$ be  a pointed space, such that $\sX$ is a connected, locally connected and weakly semilocally 1-connected space. 
 For any $x \in \sX$ denote by $\mathfrak{Paths} \left( x_0, x\right)$ a set of all semilocally 1-connected $x_0$-paths such that $$
 \left( \sU_1, ..., \sU_n\right)\in \mathfrak{Paths} \left( x_0, x\right) \quad \Leftrightarrow \quad x_0 \in \sU_1 \quad x \in \sU_n.
 $$
 There is an equivalence relation $\sim_{\left( x_0, x\right)}$ on $\mathfrak{Paths} \left( x_0, x\right)$ such that
 $$
 \mathfrak p_1, \mathfrak p_2 \in  \mathfrak{Paths} \left( x_0, x\right)~~ \mathfrak p_1\sim_{\left( x_0, x\right)} \mathfrak p_2 ~~\Leftrightarrow ~~ \mathfrak p^*_1\circ  \mathfrak p_2 \text{ is contractible (cf. Definition \ref{top_contactible_path_defn})}.
 $$
 A class of equivalence of $\mathfrak p \in \mathfrak{Paths} \left( x_0, x\right)$ will be denoted by $$\left[\mathfrak p\right]_{\left( x_0, x\right)}\in \mathfrak{Paths} \left( x_0, x\right)/\sim_{\left( x_0, x\right)}.$$
 Let $\widetilde \sX$ be a set of pairs $\left(x, \left[\mathfrak p\right]_{\left( x_0, x\right)}\right)$ where 
 \be\label{top_paths_eqn}
 x\in \sX,\\
 \left[\mathfrak p\right]_{\left( x_0, x\right)}\in  \mathfrak{Paths} \left( x_0, x\right)/\sim_{\left( x_0, x\right)}.
 \ee
 There is a natural map 
 \be
 \begin{split}
 	\widetilde p: \widetilde \sX\to \sX,\\
 	\left(x, \left[\mathfrak p\right]_{\left( x_0, x\right)}\right)\mapsto x.
 \end{split}
 \ee
 For any $\mathfrak p = \left( \sU_1, ..., \sU_n\right)\in \mathfrak{Paths} \left( x_0, x\right)$ we define a set 
 \be\label{top_u_path_eqn}
 \widetilde \sU_{\mathfrak p} \bydef \left\{\left.  \left( x',  \left[\mathfrak p\right]_{\left( x_0, x'\right)}\right)\in \widetilde{\sX} \right|x' \in \sU_n\right\}.
 \ee
 \begin{lemma}\label{top_uni_top_lem}
 	If $\sX$ is a connected,  weakly semilocally 1-connected (cf. Definition \ref{top_weakly_semi1_defn}) space then a family of sets given by the equation \eqref{top_u_path_eqn} a {basis} for a topology on $\widetilde\sX$. 
 \end{lemma}
 \begin{proof}
 	One needs prove conditions (a) and (b) of the Definition \ref{top_base_defn}.\\
 	(a) From the Lemma \ref{top_gen_path_lem} it follows that the set $\mathfrak{Paths} \left( x_0, x\right)$ (cf. \eqref{top_paths_eqn}) is not empty. For any $\left(x, \left[\mathfrak p\right]_{\left( x_0, x\right)}\right)\in \widetilde \sX$ the given by \eqref{top_u_path_eqn} set $ \widetilde \sU_{\mathfrak p}$ is a neighbourhood of $\left(x, \left[\mathfrak p\right]_{\left( x_0, x\right)}\right)$,
 	\\
 	(b). Let both ${\mathfrak p'}\bydef \left( \sU'_1, ..., \sU'_{n'}\right)$ and ${\mathfrak p''}\bydef \left( \sU''_1, ..., \sU''_{n''}\right)$ be paths  such that there is $\widetilde x \in \widetilde \sU_{\mathfrak p'}\cap \widetilde \sU_{\mathfrak p''}$. It follows   that $\widetilde p\left(\widetilde x \right) \in \sU'_{n'}\cap \sU''_{n''}$. From 
 	\eqref{top_u_path_eqn} it turns out that
 	\be\label{top_up_path_eqn}
 	\left[\mathfrak p'\right]_{\left( x_0, \widetilde p\left(\widetilde x \right)\right)}= \left[\mathfrak p''\right]_{\left( x_0, \widetilde p\left(\widetilde x \right)\right)}
 	\ee
 	From the Definition \ref{top_weakly_semi1_defn} it turns out that there is a weakly 1-connected neighbourhood $\sU$ of $p\left(\widetilde x \right)$ which is a subset of $\sU'_{n'}\cap \sU''_{n''}$. If ${\mathfrak p}\bydef \left( \sU'_1, ..., \sU'_{n'}, \sU\right)$ then from \eqref{top_up_path_eqn} one has
 	\be\label{top_upp_path_eqn}
 	\left[\mathfrak p\right]_{\left( x_0, \widetilde p\left(\widetilde x \right)\right)}=\left[\mathfrak p'\right]_{\left( x_0, \widetilde p\left(\widetilde x \right)\right)}= \left[\mathfrak p''\right]_{\left( x_0, \widetilde p\left(\widetilde x \right)\right)}.
 	\ee
 	From \eqref{top_upp_path_eqn} it follows that $\widetilde x \in \widetilde \sU_{\mathfrak p}$ and  $\widetilde\sU_{\mathfrak p}\subset \widetilde\sU_{\mathfrak p'}\cap \widetilde\sU_{\mathfrak p''}$.
 \end{proof}
 \begin{lemma}
 	If we consider the given by the Lemma \ref{top_uni_top_lem} topology on $\widetilde \sX$ then the map $p: \widetilde \sX \to \sX$ is a covering.
 \end{lemma}
 \begin{proof}
 	For any  $x \in \sX$  there is an open weakly 1-connected neighbourhood $\sU$.  If $\widetilde x \in p^{-1}\left(x\right)$ then $\widetilde{x}= \left( x,  \left[\mathfrak p\right]_{\left( x_0, x\right)}\right)$. If $\mathfrak p= \left( \sU_1, ..., \sU_{n}\right)$ and $\mathfrak p'= \left( \sU_1, ..., \sU_{n}, \sU\right)$ and $\widetilde \sU_{\mathfrak p'}$ is given by  \eqref{top_u_path_eqn} then $p\left( \widetilde \sU_{\mathfrak p'}\right)= \sU$ and $\widetilde \sU_{\mathfrak p'}$ is mapped homeomorphically onto $\sU$. So for all $\widetilde x \in p^{-1}\left(x\right)$ one can find an open neighbourhood $\widetilde \sU_{\widetilde x}$ which is is mapped homeomorphically onto $\sU$, i.e. $\sU$ is evenly covered by $p$.
 \end{proof}
 Let $p': \widetilde \sX'\to \sX$ is a covering then we define a map $\widetilde p': \widetilde \sX \to \widetilde \sX'$ by a following way. For any $\widetilde{x}= \left( x,  \left[\mathfrak p\right]_{\left( x_0, x\right)}\right)$ we define a unique $p'$-lift  $\left( \widetilde\sU'_1, ..., \widetilde\sU'_{n}\right)$ of $\mathfrak p = \left( \sU_1, ..., \sU_{n}\right)$. There is a unique $\widetilde x' \in \widetilde \sX'$ such that $\left\{\widetilde x'\right\}= p^{-1}\left( x\right) \cap \widetilde\sU'_{n}$. So one can define a map 
 \be\label{top_from_uni_defn}
 \begin{split}
 	\widetilde p': \widetilde \sX \to \widetilde \sX',\\
 	\widetilde x \mapsto  \widetilde x'.
 \end{split}
 \ee
 \begin{exercise}\label{top_uni_cov_exer}
 	Prove that the given by \eqref{top_from_uni_defn} map is a covering.
 \end{exercise}
 From the Exercise \ref{top_uni_cov_exer} it follows that $\widetilde \sX$ is a universal covering space of $\sX$ (cf. Definition \ref{top_universal_covering_defn}).
 So one has the following Theorem.
 \begin{theorem}\label{top_simply_con_cov_thm}If $\sX$ is a connected, locally connected, weakly semilocally 1-connected space then there is an universal covering space for $\sX$.
 \end{theorem}
 
 \begin{remark}
 	The Theorem \ref{top_simply_con_cov_thm} is a generalization of the Lemma \ref{top_simply_con_cov_lem}.
 \end{remark}
 
 \begin{lemma}\label{top_weak_covering_iso_lem}
 	Let $\left(\sX, x_0\right)$ be   a pointed space, such that $\sX$ is a connected, locally connected and weakly semilocally 1-connected space, and let $p: \widetilde\sX\to\sX$ is an universal covering. Then  there is a following natural isomorphism
 	$$
 	\pi^{\mathrm{w}}_1\left(\sX, x_0\right)\cong G\left(\left. \widetilde \sX\right|\sX \right)  
 	$$
 	where $\pi^{{\mathrm{w}}}_1\left(\sX, x_0\right)$ is a weak fundamental group (cf. Definition \ref{top_weak_fundamental_group_defn}) and $G\left(\left. \widetilde \sX\right|\sX \right)$ is a group of covering transformations (cf. Definition \ref{top_covering_transformations_geoup_defn}). 
 \end{lemma}
 \begin{proof}
 	Select $\widetilde x_0\in \widetilde \sX$. For any $g \in G\left(\left. \widetilde \sX\right|\sX \right)$ we select a {weakly semilocally 1-connected} $\widetilde x_0$-path $\widetilde {\mathfrak{p}}\bydef\left( \widetilde \sU_1, ...,  \widetilde \sU_n\right)$ with $g\widetilde x_0 \in \widetilde \sU_n$ (cf. Definitions \ref{top_weak_path_defn} and \ref{top_closed_path_defn}) . Let $\mathfrak p = \left(   \sU_1, ...,   \sU_n\right)\bydef \desc_p \left(\widetilde {\mathfrak{p}} \right) $ be the descent of $\widetilde {\mathfrak{p}}$ (cf Definition \ref{top_path_desc_defn}). If $\widetilde x_1 \in \widetilde \sX$ then we select an $\widetilde x_1$-path $\widetilde {\mathfrak{p}}'\bydef\left( \widetilde \sU'_1, ...,  \widetilde \sU'_m\right)$ with $\widetilde x_0 \in \widetilde \sU'_m$. If $\widetilde {\mathfrak{p}}'' \bydef \lift_p\left(\mathfrak p'^*\circ \mathfrak p\circ \mathfrak p' \right) = \left(\widetilde{\sU}''_1, ... , \widetilde{\sU}''_r \right)$ the $p$-$\widetilde x_1$ -lift (cf. Definition \ref{top_path_lift_defn}) then there is the unique $\widetilde{x}'_1 \in  \widetilde{\sU}''_r$. with $p\left( \widetilde{x}'_1\right)= \widetilde{x}_1$. We leave to the reader prove of following statements
 	\begin{itemize}
 		\item the map 
 		\bean
 		\phi: \widetilde \sX \to \widetilde \sX,\\
 		\widetilde{x}_1 \mapsto \widetilde{x}'_1
 		\eean 
 		is uniquely defined homeomorphism,
 		\item $\phi = g$.
 	\end{itemize} 
 \end{proof}
 \begin{corollary}\label{top_weak_equivalence_cor}
 	Let $\sX$ be a connected, locally path-connected  and  semilocally 1-connected (cf. Definition \ref{top_semi1_defn}) Hausdorff space then  there is a natural isomorphism
 	\be
 	\pi_1^{\mathrm{w}}\left(\sX, x_0 \right)\cong \pi_1\left(\sX, x_0 \right)
 	\ee
 	between the weak fundamental group, and the explained in the  \cite{spanier:at} one.
 \end{corollary}
 \begin{proof}
 	From the Lemma \ref{top_simply_con_cov_lem} it follows that there is the universal covering $\widetilde \sX \to \sX$. This lemma becomes  a consequence of the Lemma \ref{top_weak_covering_iso_lem} and the Remark \ref{top_simply_con_cov_rem}.
 \end{proof}
 \begin{empt}\label{top_xg_empt}
 	If $\sX$ is  a connected, locally connected and weakly semilocally 1-connected space (cf. Definition \ref{top_weakly_semi1_defn}) and $G \subset \pi^{\mathrm{w}}_1\left(\sX, x_0\right)$ is a subgroup then one can naturally construct a topological space ${\sX}\left[G\right]$. As a set ${\sX}\left[G\right]$ equals to $G \times \sX$, and the topology of ${\sX}\left[G\right]$ is generated by sets
 	\be\label{top_xg_eqn}
 	\left( g, \sU\right) \bydef \left\{\left. \left(g, x  \right)\in G \times \sX \right| x \in \sU\right\}
 	\ee
 	where $\sU\subset \sX$ is a connected, semilocally proper space (cf. Definition \ref{top_weakly_semi1_defn}). We leave to the reader the prove of the following statements: 
 	\begin{itemize}
 		\item[(i)] the given by \eqref{top_xg_eqn} family satisfies to the  Definition \ref{top_base_defn}),
 		\item[(ii)] there is the natural   covering $p_G : {\sX}\left[G\right] \to \sX$ such that $G\left(\left. {\sX}\left[G\right] \right| \sX\right)\cong G$.
 	\end{itemize}
 \end{empt}
 \begin{definition}\label{top_xg_defn}
 	Under the hypothesis \ref{top_xg_empt} we say that ${\sX}\left[G\right]$ is the $\sX$-$G$-\textit{space}, and the map $p_G : {\sX}\left[G\right] \to \sX$ is the $\sX$-$G$-\textit{covering}.
 \end{definition}
 \begin{empt}\label{top_xgp_empt}
 	Any $\widetilde x \in {\sX}\left[G\right]$ can be obtained by the following algorithm. We define a weak $x_0$-path $\mathfrak p = \left(\sU_0, ..., \sU_n \right)$ such that $p_G\left( \widetilde x\right)= x\in \sU_n$. Then we find a closed $x_0$ path $\mathfrak p'$ which is a representative of $g \in G$. There are the $p_G$-$\widetilde x_0$-lift $\widetilde{\mathfrak p}= \left(\widetilde \sU_1, ..., \widetilde \sU_m\right)$ of the composition $\mathfrak p' \circ \mathfrak p$, and  $\widetilde x$ is a unique point such that $\widetilde x \in \widetilde \sU_m$ and $p_G\left( \widetilde x\right)= x$. Any closed $\widetilde x_0$-path is contractible (cf. Definition \ref{top_contactible_path_defn}),
 \end{empt}
 
 \begin{lemma}\label{top_lex_square_lem}
 	The described in \ref{top_lex_square_empt} ordered square $I^2_o$  is weakly simply connected (cf. Definition \ref{top_weakly_simply_connected_defn}).
 \end{lemma}

 \begin{proof}
 	Let $p:\widetilde\sX \to  I^2_o$ be a covering.
 	It is known (cf. \cite{counter_topology,munkres:topology}) that the space  $I^2_o$ is compact and locally connected.  If $\mathfrak A$ is a family of connected open subsets of $\widetilde \sX$ which are homeomorphically  mapped onto open  evenly covered by $p$ subsets of $I^2_o$ then  from the Lemma \ref{top_gen_path_lem} it follows that there is an  $\mathfrak A$-path $\left( \sU_1, ..., \sU_{n}\right)$ (cf. Definition \ref{top_gen_path_defn}) such that $\left(0,0 \right)  \in \sU_1$ and  $\left(1,1 \right)  \in \sU_n$. If there is $x \in I^2_o\setminus\cup_{j = 1}^n\sU_j$ then according to \ref{top_lex_square_empt} both sets $I^2_o\setminus \{x\}$ and $\cup_{j = 1}^n\sU_j$ are not connected. 
 	However the set $\cup_{j = 1}^n\sU_j$ is connected it follows that $ I^2_o\setminus\cup_{j = 1}^n\sU_j = \emptyset$, so $I^2_o=\cup_{j = 1}^n\sU_j$. Since a set $\sU_j$ is connected for any $j\in \left\{1,...,n\right\}$ there is a set $\left\{x'_1, x''_1, x'_2, x''_2, ..., x'_n, x''_n\right\}$ such that
 	\be\label{top_op_eqn}
 	\sU_j = \begin{cases}
 		\left\{\left.x \in I^2_o\right| x'_j < x < x''_j \right\} & x'_j > \left(0, 0 \right)\text{ AND } x''_j < \left(1, 1 \right)\\
 		\left\{\left.x \in I^2_o\right| x'_j \le x < x''_j \right\} & x'_j = \left(0, 0 \right)\text{ AND } x''_j < \left(1, 1 \right)\\
 		\left\{\left.x \in I^2_o\right| x'_j < x \le x''_j \right\} & x'_j > \left(0, 0 \right)\text{ AND } x''_j = \left(1, 1 \right),\\
 		I^2_o & x'_j = \left(0, 0 \right)\text{ AND } x''_j = \left(1, 1 \right)
 	\end{cases}	
 	\ee
 	where the order on $I^2_o$ is given by the equation \eqref{top_order_top_s_eqn}. Using \eqref{top_op_eqn} one can prove that there is a set $\left\{y'_1\bydef \left(0, 0 \right), y''_1, y'_2, y''_2, ..., y'_m, y''_m\bydef \left(1, 1 \right)\right\} \subset  I^2_o$ such that if
 	$$
 	\sV_j \bydef \begin{cases}
 		\left\{\left.y \in I^2_o\right| y'_j < y < y''_j \right\} & 1 < j < n\\
 		\left\{\left.y \in I^2_o\right| y'_j \le y < y''_j \right\} & j = 0\text{ AND } j < n\\
 		\left\{\left.y \in I^2_o\right| y'_j < y \le y''_j \right\} &j> 1\text{ AND } j= n,\\
 		I^2_o & j = 1\text{ AND } j = n.
 	\end{cases}
 	$$
 	then $\sV_j$ is evenly covered by $p$ for any $j \in \left\{1,..., m\right\}$. Let $\left\{\widetilde x_1^\la \right\}_{\la \in \La}\bydef p^{-1}\left( \left(0, 0 \right)\right)$ and for all $\la \in \La$, and let  $\widetilde\sV^\la_1 \subset \widetilde\sX$ be the unique connected component of $p^{-1}\left(\sV_1 \right)$ such that $\widetilde x^\la_1 \in \widetilde\sV^\la_1$. For every $\la\in \La$ we define a  set  $\widetilde\sV^\la_2 \subset \widetilde\sX$ as  the unique connected component of $p^{-1}\left(\sV_2 \right)$ such that  $\widetilde\sV^\la_2 \cap \widetilde\sV^\la_1 \neq \emptyset$. Using an induction for each $j \in   \left\{3,..., m\right\}$ and $\la\in \La$ one can define $ \widetilde\sV^\la_j$ which is a  unique connected component of $p^{-1}\left(\sV_j \right)$ with $\widetilde\sV^\la_{j-1} \cap \widetilde\sV^\la_j \neq \emptyset$. For any $\la\in \La$ the set $\widetilde\sV^\la \bydef \widetilde\sV^\la_1 \cup ...\cup \widetilde\sV^\la_m$ is open since it is a finite union of open sets. On the other hand   $\widetilde\sV^\la$ is closed since the set $\widetilde \sX \setminus \widetilde\sV^\la = \bigcup_{\la' \in \La \setminus \left\{\la\right\}}$ is open. So $\widetilde\sV^\la$ is a quasi-component (cf. Definition \ref{top_quasi_component_defn}). We leave to the reader a proof of the following statements:
 	\begin{itemize}
 		\item $\widetilde{\sX }= \bigsqcup_{\la\in \La} \widetilde\sV^\la$,
 		\item for all $\la \in \La$ the restriction $p|_{\widetilde\sV^\la}: \widetilde\sV^\la\to I^2_o$ is a homeomorphism.
 	\end{itemize}
 	According to the  Definition \ref{top_weakly_simply_connected_defn}) the space $I^2_o$  is weakly simply connected.
 \end{proof}
 
 \begin{exercise}\label{top_lex_square_exer}
 	Let  $I^2_o$ be ordered square (cf. \ref{top_lex_square_empt}). Denote by $x_0 \bydef (0,0), x_1 \bydef (1,1)\in I^2_o$. Let $\sX \bydef I^2_o/\sim_\sX$ where $x_0 \sim_\sX x_1$.
 	\begin{enumerate}
 		\item Prove that $\sX$ is a connected, locally connected, compact, Hausdorff space, but $\sX$ is not path connected (cf. \cite{counter_topology}).
 		\item Similarly to the proof of the Lemma \ref{top_lex_square_lem} prove that the space $\sX$ is weakly semilocally 1-connected (cf. Definition \ref{top_weakly_semi1_defn}).
 	\end{enumerate}
 \end{exercise}
 \begin{exercise}\label{top_lex_square1_exer}
 	Let us use the notation of the Exercise \ref{top_lex_square_exer}. Let $\sX_j$ be a homeomorphic to $I^2_o$ space for any $j\in \Z$. Let $x^j_0 , x^j_1 \in \sX_j$ be points which correspond to $x_0 \bydef (0,0), x_1 \bydef (1,1)\in I^2_o$. Let $\sim_{\widetilde \sX}$ be an equivalence relation on $\bigsqcup_{j\in \Z} \sX_j$ such that
 	$$
 	\forall j \in \Z \quad x^j_1 \sim_{\widetilde \sX}  x^{j+1}_0.
 	$$
 	Let $\widetilde{\sX}\bydef \bigsqcup_{j\in \Z} \sX_j/ \sim_{\widetilde \sX}$.
 	\begin{enumerate}
 		\item Prove that $\widetilde\sX$ is a connected, locally connected, locally compact, Hausdorff space, but $\widetilde\sX$ is not path connected (cf. \cite{counter_topology}).
 		\item Similarly to the proof of the Lemma \ref{top_lex_square_lem} prove that the space $\widetilde\sX$ is weakly simply connected (cf. Definition \ref{top_weakly_simply_connected_defn}).
 	\end{enumerate}
 \end{exercise}
 
 \begin{exercise}\label{top_lex_square2_exer}
 	Let us use the notation of the Exercises \ref{top_lex_square_exer}, \ref{top_lex_square1_exer}.   for any $n \in \Z$ define a homeomorphism  $\varphi_n^\sqcup$ of $\sqcup\sX_j$ which homeomorphically  maps $\sX_j$ onto $\sX_{j + n}$. Prove following statements:
 	\begin{enumerate}
 		\item for any $n\in \Z$ the homeomorphism $\varphi_n^\sqcup$ induces a homeomorphism $\varphi_n$ of $\widetilde \sX$,
 		\item a family of $\left\{\varphi_n\right\}$ induces an action $\Z\times \widetilde \sX \to \widetilde{\sX}$ such that $\Z$ is a properly discontinuous group of homeomorphisms (cf. Definition \ref{top_properly_disc_group_defn}),
 		\item there is a natural isomorphism $\sX = \widetilde\sX /\Z$.
 	\end{enumerate}
 \end{exercise}
 \begin{empt}
 	Let us use the notation of the Exercises \ref{top_lex_square_exer}, \ref{top_lex_square1_exer}. From the Theorem \ref{top_g_cov_thm} it follows that there is a covering $p: \widetilde\sX\to \sX$. Since $\widetilde\sX$ is simply 1-connected the covering is universal. If  $n \in \Z$ and $\widetilde x_n \bydef x^n_0 / \sim_{\widetilde \sX}\in \widetilde \sX$ then $p\left(\widetilde x_n \right) = x_0 \bydef p\left(\widetilde x_0 \right)$ for all $n \in \Z$. One can regard $p$ as a pointed covering $\left(\widetilde\sX, \widetilde x_0 \right)\to  \left(\sX, x_0 \right)$. If $\mathfrak p = \left( \sU_1, ..., \sU_k\right)$ is a semilocally 1-connected closed $x_0$-path on $\sX$ and  $\widetilde{\mathfrak p} = \left( \widetilde\sU_1, ..., \widetilde\sU_k\right)$ is its $p$-lift, then there is a unique $n \in \Z$ such that $\widetilde x_n \in \widetilde\sU_k$. One can proof that  the map
 	$$
 	\mathfrak	p\mapsto n
 	$$
 	yields a homomorphism $h: \pi^{\text{w}}_1\left(\sX \right)  \to \Z$. Since the covering $p$ is universal a path $\mathfrak p$ is contractible if and only  $\lift_p \mathfrak p$ is contractible, i.e. $h\left(e \right) = 0$ where $e$ is a unity of $\pi^{\text{w}}_1\left(\sX \right)$. So the map $h$ is injective. 	From the Lemma \ref{top_gen_path_lem} it follows that for all $k \in \Z$ there is a path
 	$\widetilde {\mathfrak p} = \left(\widetilde \sU_1, ..., \widetilde\sU_k\right)$ such that
 	\begin{itemize}
 		\item $\widetilde x_0\in \widetilde \sU_1, \quad \widetilde x_n \subset \widetilde \sU_k$, 
 		\item the set $\widetilde \sU_j$ is weakly 1-connected.
 	\end{itemize}
 	If $\mathfrak p = \desc_p \widetilde {\mathfrak p}$ is a $p$ descent  of $\widetilde {\mathfrak p}$ (cf. Definition \ref{top_path_desc_defn}) then  the path $\mathfrak p$ is closed and
 	$h\left(\left[\mathfrak p\right]\right)= n$. So the homomorphism $h$ is surjective. In result one has an isomorphism $\pi^{\text{w}}_1\left(\sX \right)\cong\Z$. However the  fundamental group $\pi_1\left(\sX \right)$ is trivial since $\sX$ is not path connected.
 \end{empt}
\subsection{Noncommutative coverings}
\paragraph{}
Here we suppose that we have Hausdorff blowing-up and investigate conditions when they enable us to construct given by the Definition \ref{covering_defn} noncommutative coverings.
\begin{theorem}
Let $\varphi : M\left(A \right) \hookto M\left(\widetilde{A} \right)$ be a good injective $*$-homomorphism (cf. Definition \ref{good_hom_defn}). If every of both inclusions  
\bean
\mathfrak {Blowing}_{\sY-A }:C_0\left( \sY\right) \hookto M\left(A \right),\\
\mathfrak {Blowing}_{\widetilde\sY- \widetilde A }:C_0\left( \widetilde\sY\right) \hookto M\left( \widetilde A \right)
\eean 
is a Hausdorff blowing-up (cf. Definition \ref{blowing_defn}), such that one has:
\begin{itemize}
	\item the covering $\pi: \widetilde\sY \to \sY$,
	\item the natural commutative diagram
	\bean
	\begin{tikzcd}
	C_b\left(\sY \right) \arrow[rrr, hook, "\mathfrak {Blowing}_{\sY-A }"] \arrow[d, hook, "C_b\left( \pi\right)"] & & & M\left(A \right)\arrow [d, hook, "\varphi"]\\ 
		C_b\left(\widetilde \sY \right) \arrow[rrr, hook, "\mathfrak {Blowing}_{\widetilde\sY-\widetilde A }"]& & & M\left(\widetilde A \right)
	\end{tikzcd}
	\eean 
\end{itemize}
then $\varphi$ is a covering (cf. Definition \ref{covering_defn})
\end{theorem}
\begin{proof}
	The Theorem \ref{good_thm} yields the natural commutative diagram
	\bean
	\begin{tikzcd}
	\sY   & \arrow[l, two heads, "\phi"]	\mathfrak{Gelfand}\left(A \right)\\ 
\widetilde \sY \arrow[u, two heads, "\pi"] &  \arrow[l, two heads, "\widetilde\phi"]	\mathfrak{Gelfand}\left(\widetilde A \right) \arrow[u, two heads, "\mathfrak{Gelfand}\left( \varphi\right) "]
\end{tikzcd}
\eean 
If $\xi \in \mathfrak{Gelfand}\left(A \right)$ and $y = \phi\left(\xi \right)$ then from the Definition \ref{top_covering_defn} it follows that there is an open neighborhood $\sU\subset \sY$ of $y$  such that there is the disconnected union 
\be\label{pre_u_eqn}
\pi^{-1}\left( \sU \right) = \bigsqcup_{\la \in \La} \widetilde{\sU}_\la 
\ee
such that for every $\la\in \La$ the restriction
\bean
\pi|_{\widetilde{\sU}_\la}: \widetilde{\sU}_\la \cong \sU
\eean 	
is a homeomorphism. The given by the Definition \ref{blowing_ideals_au_ua_defn} closed left ideal $A_\sU$ yields an open projection $p_\sU\in A^{**}$. The set $\mathfrak{Gelfand}\left(A \right)_{p_\sU}$ is a neighborhood of $\xi$. Since $\mathfrak{Gelfand}\left(A \right)_{p_\sU}$ is locally compact there is a compact neighborhood $\sV$ of $\xi$ with $\sV \subset \sU$. One can suppose that $\sV$ corresponds to a compact projection $p \in A^{**}$ (cf. Corollary \ref{gelfand_compact_cor}) such that $p < p_\sU$ and
\bean
 p_\sU A^p p_\sU = A^p
\eean 
where $A^p$ is given by \eqref{algebra_closed_eqn}. For any $\la \in \La$ the  closed left ideal $\widetilde A_{\widetilde \sU_\la}$ yields an open projection $p_{\widetilde\sU_\la}\in A^{**}$ such that one has
\bea\label{lc_1_eqn} 
	\mathfrak{Gelfand}\left(\varphi\right)^{-1}\left(\sV  \right) = \bigsqcup_{\la \in \La} \widetilde{\sV}_\la \quad \text{where} \quad \widetilde{\sV}_\la = \mathfrak{Gelfand}\left(\varphi\right)^{-1}\left(\sV  \right)\cap \widetilde \phi^{-1} \left( \widetilde{\sU}_\la\right).
\eea 
If $p_\la \bydef \varphi^{**}\left(p\right)\wedge p_{\widetilde\sU_\la}$ then for any $\la \in \La$ there is the natural $*$-homomorphism
\bean
A \to \widetilde A^{p_\la},\\
a \mapsto p_\la a p_\la
\eean 
which induces the natural $*$-isomorphism $
A^p \to  \widetilde A^{p_\la}
$.
From the Lemma \ref{ultra_cloded_lem} this $*$-isomorphism yields the natural homeomorphism
\bean
	\mathfrak{Ultrafilers}\left(A^{p}\right) \cong 	\mathfrak{Ultrafilers}\left(\widetilde A^{p_\la} \right)
\eean 
and taking into account that  both $A^{p}$ and $\widetilde A^{p_\la}$ are unital one has the natural homeomorphisms
\bean
\mathfrak{PreGelfand}\left(A^{p}\right) \cong 	\mathfrak{PreGelfand}\left(\widetilde A^{p_\la} \right),\\
\mathfrak{Gelfand}\left(A^{p}\right) \cong 	\mathfrak{Gelfand}\left(\widetilde A^{p_\la}\right) ,\\
\mathfrak{Gelfand}\left(A\right)^{p} \cong 	\mathfrak{Gelfand}\left(\widetilde A \right)^{p_\la}
\eean 
or equivalently 
\be\label{lc_2_eqn}
\forall \la \in \La \quad \sV \cong \widetilde \sV_\la.
\ee
This lemma is a consequence of both \eqref{lc_1_eqn} and \eqref{lc_2_eqn}.

\end{proof}

\begin{empt}\label{desc__empt}
	Let  $\pi: \widetilde{\sX}\to \sX$ is a covering (cf. Definition \ref{top_covering_defn}) where both $\sX$ and $\widetilde{\sX}$ are locally compact Hausdorff spaces. If $\widetilde f \in C_c\left(\widetilde{\sX} \right)$ then for each $x \in \sX$ the set $$\left\{\left.\widetilde x \in \widetilde \sX\right| \widetilde f\left( \widetilde x\right)= x, \quad \widetilde f \left( \widetilde x\right) \neq 0 \right\}$$ finite.  
	We leave to the reader proof of that  if
	\be\label{cov_comp_eqn}
	\begin{split}
		f : \sX \to \C,\\
		f \bydef \sum_{\substack{\widetilde x\in \widetilde \sX\\ \widetilde f\left( \widetilde x\right)= x}} \widetilde f\left( \widetilde x\right) 
	\end{split}
	\ee 
	then $f\in C_c\left(\sX \right)$. 
\end{empt}
\begin{defn}\label{top_compactly_supported_descent_defn}
	Under the hypothesis \ref{desc__empt} the $\C$-linear homomorphism 	\be\label{desc_eqn}
	\begin{split}
		\desc_\pi : C_c\left(\widetilde\sX \right)\to C_c\left(\sX \right),\\
		\widetilde f \mapsto f
	\end{split}
	\ee 
	is the \textit{compactly supported descent}.
\end{defn}
\begin{lemma}\label{top_a_u_lem}
	If $\mathcal X$  is a locally compact, 
	Hausdorff space then for any $x_0 \in  \mathcal X$ and any open neighbourhood $\mathcal U\subset\mathcal X$ of $x_0$ with compact closure   there is a continuous function $a: \mathcal X \to \left[0,1 \right]$  such that following conditions hold:
	\begin{itemize}
		\item $a\left(\sX\right) = \left[0,1\right]$,
		\item 	$a\left( x_0\right) = 1$,
		\item $\supp a \subset \mathcal U$,
		\item	there is an open neighbourhood $\mathcal V\subset\mathcal U$ of $x_0$ which satisfies to the following condition
		\be\label{com_a_u_eqn}
		a\left(\mathcal V \right)= \{1\}.
		\ee 
	\end{itemize}
	If $\sX$ is locally connected (cf. Definition \ref{top_locally_connected_defn}) then one $a$ can be selected such that the support of $a$ (cf. Definition \ref{top_support_defn}) is connected.
\end{lemma}
\begin{proof}
	From the Exercise \ref{top_completelyegular_exer} it turns out that that $\sX$ is completely regular (cf. Definition \ref{top_completelyegular_defn}), i.e. there is a  continuous function $b: \mathcal X \to \left[0,1 \right]$  such that $b\left( x_0\right) = 1$ and $b\left(\mathcal X \setminus \mathcal U \right)= \{0\}$. A set 
	$\sU' \bydef \left\{x \in \sX | b\left(x \right) \neq 0\right\}$ is open. 
	The set $\mathcal V = \left\{ x \in \mathcal X~|~b\left(x \right)> \frac{2}{3}  \right\}$ is open.
	If $f: \R \to \R$ is a continuous function given by
	\be\nonumber
	f\left(t\right)=	\left\{\begin{array}{c l}
		0 & t \le \frac{1}{3}\\
		3t - 1 & \frac{1}{3} < t \le \frac{2}{3}\\
		1 & t > \frac{2}{3}
	\end{array}
	\right.
	\ee
	then $a \bydef f\left(b\right): \mathcal X \to \left[0,1 \right]$ satisfies to all conditions of this lemma. If $\sX$ is locally connected and $\sU''$ is a component of $x_0$ is $\sU'$ (cf. Definition \ref{top_connected_component_defn}) then   from the Exercise \ref{top_loc_conn_exer} it follows that $\sU''$ is a quasi-component of $\sU'$ (cf. Definition \ref{top_quasi_component_defn}). If we define $a' \in C_c\left(\sX \right)$ by following way
	$$
	a'\left(x \right) \bydef \begin{cases}
		a\left( x\right)& x \in \sU''\\
		0& x \notin \sU''	\end{cases}
	$$ 
\end{proof}
then the support $\supp a'$ is the closure of $\sU''$. From the Theorem \ref{top_connected_closure_thm} it follows that the set $\supp a'$ is connected.
\begin{definition}\label{top_stump_defn}
	Let $\mathcal X$ be a locally compact,  
	Hausdorff space, $x_0\in \sX$ and $f$ is given by the Lemma \ref{top_a_u_lem} then we denote by
	\be\label{top_fx_eqn}
	{f}_{{x}_0} \stackrel{\mathrm{def}}{=} 	{f}.
	\ee
	We say that ${f}_{{x}_0}$ is an $x_0$-\textit{stump}. Denote by $\mathfrak{Stumps}_{x_0}$ the set of $x_0$-stumps. 
\end{definition}
\begin{remark}\label{top_a_uem}
	If ${f}_{{x}_0}$ is an $x_0$-{stump} then $\supp f_{x_0}$ is compact, i.e. ${f}_{{x}_0}\in C_c\left(\sX\right)$, since it was assumed that the closure of $\sU$ is compact (cf. Lemma \ref{top_a_u_lem}).
\end{remark}
\begin{corollary}\label{com_a_u_cor}
	Let $\mathcal X$ be a locally compact,   
	Hausdorff space. For any $x_0 \in  \mathcal X$, and any open neighbourhood $\mathcal U$ of $x_0$ there is an open neighbourhood  $\mathcal V$ of $x_0$ such that the closure of $\mathcal V$ is a subset of $\mathcal U$.
\end{corollary}
\begin{proof}
	If $a$ satisfies to the Lemma \ref{top_a_u_lem} then a set $\mathcal V = \left\{x \in \mathcal X~|~ a\left(x \right)>0 \right\}$ is open and the closure of $\mathcal V$ is a subset of $\mathcal U$.
\end{proof}

\begin{definition}\label{top_stump_cov_defn}
	Let $p: \widetilde{\sX} \to {\sX}$ be a covering. Let $\widetilde{x}_0\in   \widetilde{\sX}$, and let $\widetilde{\sU}$ be an open neighbourhood of $\widetilde{x}_0$ such that the restriction $\left.p\right|_{\widetilde{\sU}}:\widetilde{\sU}\xrightarrow{\approx}\sU \bydef p\left(\widetilde{\sU} \right)$ is a homeomorphism. Since $\widetilde{\sX}$ is locally compact and Hausdorff there is $	\widetilde{f} \in C_c\left(\widetilde{\sX} \right)$ and open subset  $\widetilde{\sV}$ such that $\widetilde{x}_0\in  \widetilde{\sV} \subset {\widetilde{\sU}}$, $~\widetilde{f}\left(\widetilde{\sV} \right)= 1$,  $~\widetilde{f}\left(\widetilde{\sX} \right)= \left[0,1\right]$ and $\supp 	\widetilde{f} \subset \widetilde{\sU}$. We write 
	\be\label{top_tfx_eqn}
	\widetilde{f}_{\widetilde{x}_0} \stackrel{\mathrm{def}}{=} 	\widetilde{f},
	\ee
	and we say that $\widetilde{f}_{\widetilde{x}_0}$ is a $p$-$\widetilde x_0$-\textit{stump}.
\end{definition}
\begin{empt}\label{top_fx_empt} 	
	Let $\sX$ be a locally compact, Hausdorff space, and let $\sY \subset\sX$ be a compact subset. For any $x \in \sY$ we select an $x$-stump $f_x\in C_c\left(\sX\right)$ (cf. Definition \ref{top_stump_defn}) such that
	\begin{itemize}
		\item a set $\supp f_x$ is compact for all $x \in \sY$,
		\item there is an open neighbourhood $\sV_x$ of $x$ such that $f_x\left(\sV_x\right)=\{1\}$.
	\end{itemize}
	A family 
	\be\label{top_vx_eqn}
	\left\{\sV_x\right\}_{x\in\sX}.
	\ee
	is such that $\sY \subset \cup_{x \in \sY} \sV_x$.
	The set $
	\sU_x \bydef \left\{\left. x\in\sX\right| f_x\left( x\right)> 0 \right\}
	$ is dense in $\supp f_x$ so $\supp f_x$ is connected (cf. Theorem \ref{top_connected_closure_thm}).  
	Moreover  $\sV_x \subset \sU_x$ and there is a finite set $\left\{x_1,..., x_n \right\}\subset \sX$ such that $\sY \subset \bigcup_{j=1}^n \sV_{x_j}$, because the set $\sY$ is compact.
	If $f \bydef f_{x_1} + ... +  f_{x_n}$ then $f\left(x \right) \ge 1$ for all $x \in \sY$. If 
	$$
	f_j \bydef \frac{f_{x_j}}{\max\left(1, f\right)}\in C_c\left(\sX \right)_+ 
	$$
	then
	\be\label{top_cfs_eqn}
	\sum_{j = 1}^n f_j\left( x\right) = 1\quad\forall x \in \sY.
	\ee
\end{empt}
\begin{defn}\label{top_covering_sum_defn}
	Let $\sY\subset\sX$ is a compact subset of a locally compact,  Hausdorff space $\sX$. The  $f \bydef \sum_{j = 1}^n f_j$.
	is said to be a \textit{covering sum for} $\sY$. We also say that the covering sum \eqref{top_cfs_eqn} is \textit{dominated} by the family 
	$\left\{\sV_x\right\}_{x\in\sX}$ (cf. equation \ref{top_vx_eqn}). 
\end{defn}
\begin{exercise}\label{top_u_net_exer}
	Let $\sX$ be a locally compact, Hausdorff space.\\ Denote by $\Xi\left( \sX\right)\bydef \left\{u_\la\right\}_{\la \in \Xi\left( \sX\right)}\subset C_c\left( \sX\right)_+$ a net of all positive continuous maps such that 
	\bean
	\forall	\la \in \Xi\left( \sX\right)\quad u_\la\left(\sX\right)\subset \left[0,1\right],\\
	\forall	\mu, \nu \in \Xi\left( \sX\right)\quad	\mu \le \nu \quad \Leftrightarrow\quad u_\mu \le u_\nu.
	\eean
	Prove that the net $\Xi\left( \sX\right)$ is an approximate unit of $C_c\left( \sX\right)$. 
\end{exercise}

\begin{lemma}\label{lift_mult_lem}
	Let $\pi: A \hookto M\left( \widetilde A\right)$ be an injective $*$-homomorphism of $C^*$-algebras, and let $\left\{u_\a\right\}_{\a \in \mathscr A}\subset A$ be an approximate unit for $A$ (cf. Definition \ref{approximate_unit_defn}). If
	\be\label{lift_mult_norm_eqn}
	\forall \widetilde a \in \widetilde A \quad \lim_{\a \in \mathscr A}\left\| \widetilde a - \pi\left( u_\a\right)  \widetilde a \right\| = \lim_{\a \in \mathscr A}\left\| \widetilde a -  \widetilde a \pi\left( u_\a\right)\right\| = 0
	\ee
	i.e. $\bt\text{-}\lim_{\a \in \mathscr A} \pi\left( u_\a\right)= 1_{M\left( \widetilde A\right)}$,
	then $\pi$ can be naturally  extended up to an injective $*$-homomorphism 
	\be\label{lift_mult_eqn}
	\begin{split}
		M\left(\pi \right) : M\left( A\right)  \hookto M\left( \widetilde A\right),\\
		a \mapsto  \bt\text{-}\lim_{\a \in \mathscr A} \pi\left( u_\a  a \right)= \bt\text{-}\lim_{\a \in \mathscr A} \pi\left( a u_\a \right)= \bt\text{-}\lim_{\a \in \mathscr A} \pi\left( u_\a a u_\a \right)
	\end{split}
	\ee
	where $\bt\text{-}\lim_{\a \in \mathscr A}$ means the limit with respect to the strict topology  of $M\left( \widetilde A\right)$ (cf. Definition \ref{strict_topology_defn}).
\end{lemma}
\begin{proof}
	For any $a \in M\left( A\right)$ we define both  maps:
	\bean 
	L_a : \widetilde A \to \widetilde A,\\
	\widetilde a \mapsto  \lim_{\a \in \mathscr A }\widetilde a~\pi \left( u_\a a \right) = \lim_{\a \in \mathscr A }\widetilde a~\pi \left( u_\a a u_\a\right);\\
	R_a : \widetilde A \to \widetilde A,\\
	\widetilde a \mapsto  \lim_{\a \in \mathscr A }\pi \left( a u_\a\right)  \widetilde a = \lim_{\a \in \mathscr A }\pi \left(u_\a a u_\a\right)  \widetilde a 
	\eean
	where the convergence with respect to $C^*$-norm topology is is implied.
	From 
	$$
	\forall \widetilde{a}, \widetilde{b} \in \widetilde{A}\quad \left( \widetilde a \pi \left( u_\a a \right)\right)  \widetilde b=\widetilde a \left( \pi \left( u_\a a \right) \widetilde b\right) 
	$$
	one can deduce that a pair $\left( L_a, R_a\right)$ satisfies to  the equation \eqref{double_centralizer_eqn}, i.e.  $\left( L_a, R_a\right)$ is  a {double centralizer} (cf. Definition \ref{double_centralizer_defn}). From the Remark \ref{double_centralizerem} it follows that $\left( L_a, R_a\right)$ yields an element of $M\left( \widetilde{A}\right)$, so one has a natural map $	M\left(\widetilde \pi \right) :	M\left(A \right)  \to M\left( \widetilde A\right)$. 
	We leave to the reader the proof of that $M\left(\pi \right)$ is an injective $*$-homomorphisms  and $M\left(\pi \right)$ is an extension of  $\pi$.
\end{proof}

\begin{empt}\label{blowing_lift_empt}
	Let $A$ be a $C^*$-algebra, and let $C_0\left(\sY \right)\subset  M\left(A \right)$ be  Hausdorff blowing-up (cf. Definition \ref{blowing_defn}), and let $q: \widetilde \sY\to \sY$ be a topological covering (cf. Definition \ref{top_covering_defn}). There is $C_0\left(\sY \right)$-valued product
	\bean
	\left\langle\cdot, \cdot  \right\rangle_{C_0\left(\sY \right)}: C_c \left(\widetilde \sY \right) \times C_c \left(\widetilde \sY \right)\to C_0\left(\sY \right),\\
	\left(\widetilde a, \widetilde b \right) \mapsto \desc^c_q\left(\widetilde a^* \widetilde b \right) 
	\eean
	where $\desc^c_q$ is  a {compactly supported} $q$-{descent} \ref{top_compactly_supported_descent_defn}. So $ C_c \left(\widetilde \sY\right)$ is a pre-Hilbert $C_0\left(\sY \right)$-module (cf. Definition \ref{hilbert_module_defn}). If $\mathscr L^2\left(C_0 \left(\widetilde \sY \right) \right)$ is a completion of $C_c \left(\widetilde \sY \right)$ with respect to given by \eqref{hilbert_module_norm_eqn} norm then $\mathscr L^2\left(C_0 \left(\widetilde \sY \right) \right)$ is a $C^*$-Hilbert $C_0\left(\sY \right)$-module. If $\mathscr L^2\left(C_0 \left(\widetilde \sY \right) \right)\otimes_{C_0\left(\sY \right)} A$ is an algebraic tensor product then there is an $A$-valued product.
	\bean
	\left\langle\cdot, \cdot  \right\rangle: \left( \mathscr L^2\left(C_c \left(\widetilde \sY \right) \right)\otimes_{C_0\left(\sY \right)} A\right)\times \left( \mathscr L^2\left(C_c \left(\widetilde \sY \right) \right)\otimes_{C_0\left(\sY \right)} A\right) \to A,\\
	\left(\left(\widetilde f \otimes a \right) , \left(\widetilde g  \otimes b \right)  \right) \mapsto a^*\left\langle\widetilde f, \widetilde g \right\rangle_{C_0\left(\sY \right)} b.
	\eean
	so $\mathscr L^2\left(C_0 \left(\widetilde \sY \right) \right)\otimes_{C_0\left(\sY \right)} A$ becomes a pre-Hilbert $A$-module. We denote by $\mathscr L^2\left(\widetilde \sY \right) _A$ the $C^*$-Hilbert $A$-module which is a completion  of  $\mathscr L^2\left(C_0 \left(\widetilde \sY \right) \right)\otimes_{C_0\left(\sY \right)} A$ with respect to given by \eqref{hilbert_module_norm_eqn} norm. There are both a left  action 	$C_0 \left(\widetilde \sY \right)\times \mathscr L^2\left(\widetilde \sY \right) _A\to \mathscr L^2\left(\widetilde \sY \right) _A$ and the right action $\mathscr L^2\left(\widetilde \sY \right) _A\times A \to \mathscr L^2\left(\widetilde \sY \right) _A$ such that $\mathscr L^2\left(\sY\right)_A$ is $C_0 \left(\widetilde \sY \right)$-$A$-bimodule. These actions induce injective homomorphisms of $C^*$-algebras
	\bea\label{blowing_a_h_eqn}
	\varphi_A : A \hookto \End^*_A\left( \mathscr L^2\left(\widetilde \sY \right)_A\right), \\
	\label{blowing_c_h_eqn}
	\varphi_{\widetilde \sY} : C_0 \left(\widetilde \sY \right) \hookto \End^*_A\left( \mathscr L^2\left(\widetilde \sY \right)_A\right)
	\eea
	where $\End^*_A\left( \mathscr L^2\left(\widetilde \sY \right)_A\right)$ is the $C^*$-algebra of adjointable  endomorphisms of  the $C^*$-Hilbert $A$-module $\mathscr L^2\left(\widetilde \sY \right)_A$ (cf. Definition \ref{adjointable_operator_defn}).
	There is a homomorphism 
	\be\label{blowing_tensor_eqn} 
	\begin{split}
		\varphi^{\widetilde \sY}_A: C_0\left( \widetilde \sY \right)\otimes_{C_0\left(\sY\right) } K\left( A\right) \to \End^*_A\left( \mathscr L^2\left(\widetilde \sY \right)_A\right),\\
		\sum_{j=1}^n 	\widetilde f_j \otimes a_j \mapsto 	\sum_{j=1}^n 	\varphi_{\widetilde \sY}\left( \widetilde f_j\right) \varphi_A\left(a_j \right)
	\end{split}
	\ee
	of right $C_0\left( \widetilde\sY\right)$-$A$-bimodules, where 	$\varphi_A, ~	\varphi_{\widetilde \sY}$ are given by  \eqref{blowing_a_h_eqn}, \eqref{blowing_c_h_eqn} and the algebraic tensor product is implied. Since both $C_c\left(\widetilde \sY \right)$ and   $K\left(A \right)$ are dense in $C_0\left(\widetilde \sY \right)$ and  $A$ respectively the $*$-homomorphism \eqref{blowing_tensor_eqn} can be uniquely extended up to a  $C_0\left( \widetilde\sX\right)$-$A$-bimodule homomorphism 
	\be\label{blowing_tensoar_eqn} 
	\begin{split}
		\psi^{\widetilde \sY}_A: C_0\left( \widetilde \sY \right)\otimes_{C_0\left(\sY\right) } A \to \End^*_A\left( \mathscr L^2\left(\widetilde \sY \right)_A\right)
	\end{split}
	\ee
	such that the $C^*$-closure of $\varphi\left(  C_c\left( \widetilde \sY \right)\otimes_{C_0\left(\sY\right) } K\left( A\right)\right)$ coincides with\\  	$\varphi'\left(  C_c\left( \widetilde \sY \right)\otimes_{C_0\left(\sY\right) } A\right) $ one.
\end{empt}

\begin{definition}\label{blowing_lift_hm_defn}
	Under the hypotheses \ref{blowing_lift_empt}  we say that the $C^*$-Hilbert $A$-module  $\mathscr L^2\left(\widetilde \sY \right)_A$ is the $\widetilde \sY$-$A$-\textit{module}.
\end{definition}
\begin{definition}\label{blowing_aegular_defn} Let $A$ be a $C^*$-algebra, and let $C_0\left(\sY \right)\subset  M\left(A \right)$ be  Hausdorff blowing-up (cf. Definition \ref{blowing_defn}).
	If 	$C_0\left( \widetilde\sY\right)$-$A$-bimodule
	\be\label{blowing_prods_eqn}
	\varphi^{\widetilde \sY}_A\left(C_c\left( \widetilde \sY \right)\otimes_{C_0\left(\sY\right) } K\left( A\right) \right)\subset  \End^*_A\left( \mathscr L^2\left(\widetilde \sY \right)_A\right)
	\ee	
	(cf. equation \eqref{blowing_tensor_eqn}) is a $*$-subalgebra of $\End^*_A\left( \mathscr L^2\left(\widetilde \sY \right)_A\right)$
	then  we say that the covering $q: \widetilde{\sY}\to \sY$ is $A$-\textit{regular}.  
\end{definition}

\begin{lemma}\label{blowing_lift_eq_lem}
	If $C_0\left(\sY \right)\subset  M\left(A \right)$ is  Hausdorff blowing-up (cf. Definition \ref{blowing_defn}) then a covering $q: \widetilde{\sY}\to \sY$ is $A$-{regular} (cf. Definition \ref{blowing_aegular_defn}) if and only if for any $\widetilde a \in \varphi\left(C_c\left( \widetilde \sY \right)\otimes_{C_0\left(\sY\right) } K\left( A\right) \right)$ there is $\widetilde f \in C_c\left(\widetilde\sY\right)$ such that
	\be\label{blowing_lift_eq_eqn}
	\forall \widetilde a \in \varphi\left(C_c\left( \widetilde \sY \right)\otimes_{C_0\left(\sY\right) } K\left( A\right) \right) \quad \exists  \widetilde f \in C_c\left(\widetilde\sY\right) \quad \varphi_A\left(  \widetilde a\right)  =\varphi_A\left(  \widetilde a\right) \varphi_{\widetilde \sY}\left(\widetilde f \right). 
	\ee 
\end{lemma}
\begin{proof}
	If $\widetilde a \in \varphi\left(C_c\left( \widetilde \sY \right)\otimes_{C_0\left(\sY\right) } K\left( A\right) \right)$ then under the hypotheses of the Definition \ref{blowing_aegular_defn} one has $\widetilde a^* \in \varphi\left(C_c\left( \widetilde \sY \right)\otimes_{C_0\left(\sY\right) } K\left( A\right) \right)$ and
	\bean
	\widetilde a^* = \varphi\left(\sum_{j = 1}^n \widetilde f_j \otimes a_j \right),\\
	\widetilde a = \sum_{j = 1}^n \varphi_A \left(a^*_j\right)\varphi_{\widetilde \sY}\left(\widetilde f_j^* \right)   
	\eean
	The finite union $\widetilde \sV \bydef \bigcup_{j = 1}^n \supp \widetilde f_j$ is compact. If $\widetilde f\in C_c\left(\widetilde\sY \right) $ is a covering sum for $\widetilde \sV$ (cf. Definition \ref{top_covering_sum_defn}) then $\varphi_A\left(  \widetilde a\right) = \varphi_A\left(  \widetilde a\right) \varphi_{\widetilde \sY}\left(f \right)$.
\end{proof}
\begin{definition}\label{blowing_lift_defn}
	If 	$C_0\left( \sY\right)\hookto M\left( A\right) $ is  Hausdorff blowing-up (cf. Definition \ref{blowing_defn}), and   $q: \widetilde{\sY}\to \sY$ is an   $A$-{regular} covering (cf. Definition \ref{blowing_aegular_defn}), then the $C^*$-norm completion $A_0\left(\widetilde\sY\right)$ of $\varphi\left(C_c\left( \widetilde \sY \right)\otimes_{C_0\left(\sY\right) } K\left( A\right) \right)$ (cf. equation \eqref{blowing_prods_eqn}) is the $q$-\textit{lift of} $A$.
\end{definition}
\begin{lemma}\label{blowing_lift_constr_lem}
	Let 	$C_0\left( \sY\right)\hookto M\left( A\right) $ be  Hausdorff blowing-up (cf. Definition \ref{blowing_eqn}), and let   $q: \widetilde{\sY}\to \sY$ be a  covering , $A$-{regular} covering (cf. Definition \ref{blowing_aegular_defn}). If $A_0\left(\widetilde\sY\right)$ is the  $q$-\textit{lift of} $A$ (cf. Definition \ref{blowing_lift_defn}) then there is a natural injective $*$-homomorphism
	\be\label{blowing_lift_eqn}
	\begin{split}
		A_b\left(q \right) : A \hookto M\left( A_0\left(\widetilde\sY\right)\right).
	\end{split}
	\ee
\end{lemma}
\begin{proof}
	If  $a \in A$ and $\widetilde{a} \in A_0\left(\widetilde\sY\right)$ then for any $\eps > 0$ there is $\widetilde{a}' \in \varphi\left(C_c\left( \widetilde \sY \right)\otimes_{C_0\left(\sY\right) } K\left( A\right) \right)$ (cf. equation \eqref{blowing_prods_eqn}) such that
	\be\label{blowing_le_eqn}
	\left\| \widetilde{a} - \widetilde{a}'\right\| < \frac{\eps}{\left\|a \right\| }.
	\ee
	From the Lemma \ref{blowing_lift_eq_lem} it follows that there is a positive function $\widetilde f \in C_c\left(\widetilde{\sY} \right)_+$ such that 
	$$
	\widetilde{a}' = \varphi_{\widetilde \sY}\left(\widetilde f \right)\widetilde{a}'\varphi_{\widetilde \sY}\left(\widetilde f \right)
	$$
	If $ \left\{\widetilde u_\la\right\}_{\la \in \Xi\left( \sY\right)}\subset C_c\left( \sY\right)_+$ is a given by the Exercise \ref{top_u_net_exer} approximate unit of $C_0\left(\widetilde{\sY} \right)$ then there is $\la' \in \La$ such that 
	$$\forall \la \in \La	\quad \la \ge \la' \quad \Rightarrow \quad\widetilde u_\la\left(\supp \widetilde f \right)= \{1\}.$$ 
	It turns out that  
	\bean
	\forall \la' \in \La	\quad \la \ge \la' \quad \Rightarrow \quad   \varphi_{\widetilde \sY}\left(  \widetilde{u}_\la\right)   \widetilde a' =   \widetilde a' \quad \Rightarrow \\
	\Rightarrow\quad \varphi_A\left( a\right)  \varphi_{\widetilde \sY}\left(  \widetilde{u}_\la\right) \widetilde{a}'= \varphi_A\left( a\right)  \varphi_{\widetilde \sY}\left(  \widetilde{u}_{\la'}\right) \widetilde{a}'\in  \varphi\left(C_c\left( \widetilde \sY \right)\otimes_{C_0\left(\sY\right) } K\left( A\right) \right).
	\eean
	Taking into account the Lemma \ref{blowing_lift_eq_lem}  one has a positive function $\widetilde f'' \in C_c\left(\widetilde{\sY} \right)_+$ such that 
	$$
	\forall \la' \in \La	\quad \varphi_A\left( a\right)  \varphi_{\widetilde \sY}\left(  \widetilde{u}_{\la'}\right) \widetilde{a}'= \varphi_{\widetilde \sY}\left(  \widetilde{f}''\right) \varphi_A\left( a\right)  \varphi_{\widetilde \sY}\left(  \widetilde{u}_{\la'}\right) \widetilde{a}'
	$$. 
	There is $\la'' \in \La$ such that $$\forall \la \in \La	\quad \la \ge \la' \quad \Rightarrow \quad\widetilde u_\la\left(\supp \widetilde f'' \right)= \{1\}.$$
	So if $\la_0\in \La$ is such that $\la_0 \ge \la'$ and $\la_0\ge \la''$ then
	\bean
	\forall \la \in \La	\quad \la \ge \la_0 \quad \Rightarrow\quad \varphi_A\left( a\right)  \varphi_{\widetilde \sY}\left(  \widetilde{u}_\la\right) \widetilde{a}'= \varphi_{\widetilde \sY}\left(  \widetilde{f}''\right)\varphi_A\left( a\right)  \varphi_{\widetilde \sY}\left(  \widetilde{u}_{\la}\right) \widetilde{a}'=\\=  \varphi_{\widetilde \sY}\left(  \widetilde u_\la\right)\varphi_A\left( a\right)  \varphi_{\widetilde \sY}\left(  \widetilde{u}_{\la}\right) \widetilde{a}'= \varphi_{\widetilde \sY}\left(  \widetilde u_{\la_0}\right)\varphi_A\left( a\right)  \varphi_{\widetilde \sY}\left(  \widetilde{u}_{\la_0}\right) \widetilde{a}'.
	\eean
	Taking into account $\left\|\widetilde u_{\la_0} \right\|= 1$ and the inequality \eqref{blowing_le_eqn} on has
	$$
	\forall \la \in \La	\quad \la \ge \la_0 \quad \Rightarrow\quad \left\| \varphi_{\widetilde \sY}\left(  \widetilde u_{\la}\right)\varphi_A\left( a\right)  \varphi_{\widetilde \sY}\left(  \widetilde{u}_{\la}\right) \widetilde{a}- \varphi_{\widetilde \sY}\left(  \widetilde u_{\la_0}\right)\varphi_A\left( a\right)  \varphi_{\widetilde \sY}\left(  \widetilde{u}_{\la_0}\right) \widetilde{a} \right\|< \eps
	$$
	Since the number $\eps$ can be arbitrary small we conclude that the net
	$$\left\{\varphi_{\widetilde \sY}\left(  \widetilde u_{\la}\right)\varphi_A\left( a\right)  \varphi_{\widetilde \sY}\left(  \widetilde{u}_{\la}\right) \widetilde{a}\right\}_{\la\in \La} $$ is $C^*$-norm convergent. We define a map
	\bean
	L_a :  A_0\left(\widetilde\sY\right)\to A_0\left(\widetilde\sY\right),\\
	\widetilde{a}\mapsto \lim_{\la\in \La} \varphi_{\widetilde \sY}\left(  \widetilde u_{\la}\right)\varphi_A\left( a\right)  \varphi_{\widetilde \sY}\left(  \widetilde{u}_{\la}\right) \widetilde{a}
	\eean
	Similarly there is a map
	\bean
	R_a :  A_0\left(\widetilde\sY\right)\to A_0\left(\widetilde\sY\right),\\
	\widetilde{a}\mapsto \lim_{\la\in \La}\widetilde{a} \varphi_{\widetilde \sY}\left(  \widetilde u_{\la}\right)\varphi_A\left( a\right)  \varphi_{\widetilde \sY}\left(  \widetilde{u}_{\la}\right). 
	\eean
	From the equality
	$$
	\forall \la \in \La\quad \forall \widetilde b',  \widetilde b'' \quad \widetilde b'\left( \varphi_{\widetilde \sY}\left(  \widetilde u_{\la}\right) \varphi_A\left( a\right)  \varphi_{\widetilde \sY}\left(  \widetilde{u}_{\la}\right) \widetilde b''\right) = \left( \widetilde b'\varphi_{\widetilde \sY}\left(  \widetilde u_{\la}\right)\varphi_A\left( a\right)  \varphi_{\widetilde \sY}\left(  \widetilde{u}_{\la}\right)\right)  \widetilde b''
	$$	
	we conclude that a pair  $\left(L_a, R_a\right)$ of maps  is a {double centralizer} (cf. Definition \ref{double_centralizer_defn}).From the Remark \ref{double_centralizerem} it follows that there is a map 
	\bean
	A \to M\left( A_0\left(\widetilde\sY\right)\right),\\
	a \mapsto \text{ the corresponding to }\left(L_a, R_a\right)\text{ multiplier}.
	\eean 
	We leave to the reader an elementary proof of that the map is an injective $*$-homomorphism.	
\end{proof}

\begin{definition}\label{blowing_lift_hom_defn}  If	$C_0\left( \sY\right)\hookto M\left( A\right) $ is  Hausdorff blowing-up (cf. Definition \ref{blowing_defn}), and    $q: \widetilde{\sY}\to \sY$ is an   $A$-{regular} covering (cf. Definition \ref{blowing_aegular_defn}) then the given by the Lemma \eqref{blowing_lift_constr_lem} $*$-homomorphism $	A_b\left(q \right) : A \hookto M\left( A_0\left(\widetilde\sY\right)\right)$
	is the  $q$-\textit{lift} of $A$. The difference between both  Definitions  \ref{blowing_lift_defn} and \ref{blowing_lift_hom_defn} of $q$-lift will be explained as we go along.
\end{definition}
\begin{lemma}\label{blowing_mult_iclusion_lem}
	If	$C_0\left( \sY\right)\hookto M\left( A\right) $ is  Hausdorff blowing-up (cf. Definition \ref{blowing_defn}), and    $q: \widetilde{\sY}\to \sY$ is an   $A$-{regular} covering (cf. Definition \ref{blowing_aegular_defn}) then the  $q$-{lift} $	A_b\left(q \right) : A \hookto M\left( A_0\left(\widetilde\sY\right)\right)$ of $A$ (cf. Definition \ref{blowing_lift_hom_defn}) can be uniquely extended up to an injective $*$-homomorphism
	\be\label{blowing_lift_m_eqn}
	M\left( A_b\left(q \right)\right)  : M\left( A\right)  \hookto M\left( A_0\left(\widetilde\sY\right)\right).
	\ee 
\end{lemma}
\begin{proof}
	Let $\left\{u_\la\right\}_{\la \in \La}$ be an approximate unit of $A$ (cf. Definition \ref{approximate_unit_defn}), and let $\widetilde a \in A_0\left(\widetilde\sY\right)$. Form the Definition \ref{blowing_lift_defn} it follows that for any $\eps > 0$ there is $\widetilde a' \in  \varphi\left(C_c\left( \widetilde \sY \right)\otimes_{C_0\left(\sY\right) } K\left( A\right) \right)$  such that
	$$
	\left\| \widetilde a - \widetilde a'\right\| 	< \frac{\eps}{3}.
	$$
	On the other hand there are two  families  $\left\{\widetilde f_1, ..., \widetilde f_n\right\}\subset C_c\left(\widetilde \sY \right)$ and $\left\{a_1, ..., a_n\right\}$ such that 
	$$
	\widetilde a' = \varphi\left( \sum_{j=1}^n \widetilde f_j \otimes a_j\right) \sum_{j=1}^n \phi_C\left( \widetilde f_j\right) \phi_A\left(a_j \right). 	
	$$
	On the other hand for any $j=1,...,n$ there is $\la_j \in \la$ such that
	$$
	\forall \la\in \La \quad \la \ge \la_j \quad \Rightarrow \quad 	\left\|a_j - a_ju_\la \right\| < \frac{\eps}{3n \left\|\widetilde f_j\right\|}
	$$
	If $\la_0\in \La$ is such that $\la_O \ge \la_j$ for every $j=1,...,n$ then from the triangle identity it follows that
	\bean
	\forall \la\in \La \quad \la \ge \la_0 \quad  \left\|\widetilde a - \widetilde a\varphi_A\left( u_\la\right) \right\|\le \\\le  \left\| \widetilde a - \widetilde a'\right\|+\left\|\left( \widetilde a - \widetilde a'\right) \varphi_A\left( u_\la\right)\right\|+\left\|\widetilde a' - \widetilde a'\varphi_A\left( u_\la\right)\right\|<\\< \frac{2\eps}{3} + \sum_{j = 1}^n \left\|\widetilde f_j\right\|\left\|a_j - a_ju_\la \right\|< \eps,
	\eean 
	so one has
	$
	\lim_{{\la \in \La}}\left\|\widetilde a - \widetilde a\varphi_A\left( u_\la\right) \right\|= 0
	$. Similarly one can prove that \\ $\lim_{{\la \in \La}}\left\|\widetilde a - \varphi_A\left( u_\la\right)\widetilde a \right\|= 0$. This Lemma becomes  a consequence of the \ref{lift_mult_lem} one.
\end{proof}

\begin{lemma}\label{blowing_blowing_lem}
	Let	$C_0\left( \sY\right)\hookto M\left( A\right) $ be  Hausdorff blowing-up (cf. Definition \ref{blowing_defn}), and let  $q: \widetilde{\sY}\to \sY$ be an   $A$-{regular} covering (cf. Definition \ref{blowing_aegular_defn}).
	If   $A_0\left(\widetilde\sY\right)$  is the $q$-{lift of} $A$ then the natural inclusion $C_0\left( \widetilde\sY\right)\hookto M\left( A_0\left(\widetilde\sY\right)\right)$  is Hausdorff  blowing-up (cf. Definition \ref{blowing_defn}). 
\end{lemma}
\begin{proof}
	Form the Definition \ref{blowing_lift_defn} it follows that for any $\eps > 0$ there is\\ $\widetilde a' \in  \varphi\left(C_c\left( \widetilde \sY \right)\otimes_{C_0\left(\sY\right) } K\left( A\right) \right)$  such that
	$$
	\left\| \widetilde a - \widetilde a'\right\| 	< \eps .
	$$
	On the other hand there are two  families  $\left\{\widetilde f_1, ..., \widetilde f_n\right\}\subset C_c\left(\widetilde \sY \right)$ and $\left\{a_1, ..., a_n\right\}$ such that 
	$$
	\widetilde a' = \varphi\left( \sum_{j=1}^n \widetilde f_j \otimes a_j\right) \sum_{j=1}^n \phi_C\left( \widetilde f_j\right) \phi_A\left(a_j \right). 	
	$$
	A finite union $\widetilde \sV \bigcup_{j=1}^n  \supp \widetilde f_j$ of compact sets is compact. If $\widetilde f \in C_c\left( \widetilde \sY\right)$ is a covering sum for  $\widetilde \sV$ (cf. Definition \ref{top_covering_sum_defn}) then $\widetilde a' = \widetilde f \widetilde a'$ and 
	\bean
	\left\| \widetilde a -\widetilde f \widetilde a'\right\| 	< \eps, 
	\eean
	i.e. a set $C_c\left( \widetilde\sY\right)A_0\left(\widetilde\sY\right) \bydef \left\{\widetilde f\widetilde a \left| \widetilde f \in C_c\left( \widetilde\sY\right)\quad\widetilde a \in A_0\left(\widetilde\sY\right)\right. \right\}$ is dense in $A_0\left(\widetilde\sY\right)$. Taking into account the Remark \ref{blowingem} we conclude that $C_0\left( \widetilde\sY\right)\hookto M\left( A_0\left(\widetilde\sY\right)\right)$  is Hausdorff  blowing-up.
\end{proof}

\begin{theorem}\label{hausdorff_covering_theorem}
	The given by the equation \eqref{blowing_lift_m_eqn} $*$-homomorphism 
	\bean
	M\left( A_b\left(q \right)\right)  : M\left( A\right)  \hookto M\left( A_0\left(\widetilde\sY\right)\right).
	\eean
	is a noncommutative covering (cf. Definition \ref{covering_defn}).
\end{theorem}
\begin{proof}
	The Theorem \ref{good_thm} yields the natural continuous map $\mathfrak{Gelfand}\left(A \right)\onto \sY$. Taking into account the Lemma \ref{blowing_blowing_lem} one has the natural continuous map $\mathfrak{Gelfand}\left(A_0\left(\widetilde\sY\right) \right)\onto \widetilde \sY$. Applying the  Theorem \ref{good_thm} once again one has the following diagram
	\bean
	\begin{tikzcd}
	\mathfrak{Gelfand}\left(A_0\left(\widetilde\sY\right) \right) \arrow[r, "\widetilde\phi"] \arrow[d, "\widetilde\varphi"] & \widetilde \sY \arrow[d, "\varphi"] \\
		\mathfrak{Gelfand}\left(A \right) \arrow[r, "\phi"] &  \sY
		\end{tikzcd}
	\eean
	For ant $x \in 	\mathfrak{Gelfand}\left(A \right)$, there is an open neighborhood $\sU \subset \sY$ of $\phi\left(x \right)$ such that
\bean
	\varphi^{-1}\left( \sU\right) = \bigsqcup_{\la\in ]La}\sU_\la,\\
	\forall \la \in \La \quad \sU_\la \cong \sU.
\eean
It turns out that 
\be\label{h_c_eqn}
\begin{split}
\widetilde	\varphi^{-1}\left( \sU\right) = \bigsqcup_{\la\in ]La}\phi^{-1}\left( \sU_\la\right) ,\\
\end{split}
\ee	
The equation \eqref{h_c_eqn} means that the nap $\varphi$ is the topological covering.
\end{proof}
Under the hypotheses \ref{blowing_lift_empt} let  $\mathscr L^2\left(\widetilde \sY \right)_A$ be the $\widetilde \sY$-$A$-module (cf. Definition \ref{blowing_lift_hm_defn}). If $\rho : A \to B\left(\H \right)$ is a nondegenerate representation there is the natural nondegenerate induced representation (cf. Proposition \ref{induced_prop} and Definition \ref{induced_representation_defn})  given by \eqref{induced_representation_eqn} representation
\be\label{blowing_induced_eqn}
\widetilde{\rho}\bydef\mathscr L^2\left(\widetilde \sY \right)_A\text{-}\Ind^A_{{A}_0\left(\widetilde\sY \right) }\rho: A_0\left(\widetilde\sY \right)\to B\left(\widetilde{\H} \right).
\ee
\begin{lemma}\label{blowing_inducedepresentation_lem} If the representation $\rho : A \to B\left(\H \right)$ is faithful and nondegenerate (cf. Definition \ref{faithfulrepr_defn} and \ref{nondegeneraterepr_defn}) then the given by \eqref{blowing_induced_eqn} representation  is faithful and nondegenerate.
\end{lemma}

\begin{proof}
Follows from the Proposition \ref{induced_prop} and the Corollary \ref{induced_cor}.
\end{proof}
\section{Groupoids, foliations and their $C^*$-algebras}\label{foliations_chap}
\subsection{Geometric construction of coverings}
\paragraph{}
In this section we consider locally-compact groupoids $\G$ (cf. Definition \ref{groupoid_topological_defn}) such that $\G^0$ is Hausdorff.

\begin{definition}\label{groupoid_lift_defn}
	Let $q: \widetilde \G^0 \to \G^0$ be a covering (cf. Definition \ref{top_covering_defn})
	and let $\widetilde \G^0\times_{\G^0}\G$ be pullback where the map $s : \G \to \G^0$ is implied. We say that the covering  $q$ is $\G$-\textit{regular} if there is  the groupoid $\widetilde \G^0$ with homeomorphism $\widetilde \G\cong \widetilde \G^0\times_{\G^0}\G$ such that the natural continuous map 
	\be\label{groupoid_lift_eqn}
	q_\G: \widetilde \G \to \G
	\ee
	is a covering (cf. Definition \ref{groupoid_covering_defn}). The groupoid $\widetilde \G$ is the $q$-\textit{lift} of $\G$. 	A $p$-\textit{topology} on $\widetilde\G$ is such that both maps $q_\G$ and $s:  \widetilde\G \to \widetilde\G^0$ are continuous.
\end{definition}

\begin{remark}\label{groupoid_liftem}
	From the  Theorem \ref{covering_groupoid_thm} it follows that the given by the  the Definition \ref{groupoid_lift_defn} groupoid is unique.
\end{remark}
\begin{remark}\label{groupoid_liftppem}
	Any element $\left(\widetilde y, x \right) \in \widetilde \G^0\times_{\G^0}\G$ uniquely defines a point $r\left(\widetilde{ x} \right)\in \widetilde{\G^0}$ where  $\widetilde{ x} \in \widetilde{\G}$ comes from  $\left(\widetilde y, x \right)$.
\end{remark}
\begin{example}\label{groupoid_lift_exm}
	According to the Theorem \ref{covering_groupoid_thm}) the existence of covering is relevant to the lifting problem
	\newline
	\begin{tikzcd}
		& \widetilde \G\arrow[d, ""]\\
		\G^u_u \arrow[ru,  dashrightarrow]\arrow[r]	& \G
	\end{tikzcd}
	\newline
	One can find obstructions of fundamental group $\pi_1\left( \G^u_u, u \right)$. However if the space  $ \G^u_u$ is not locally path connected  then 	usage of $\pi_1\left( \G^u_u, u \right)$ does not sufficient. Application of the weak fundamental group $\pi^{\text{w}}_1\left( \G^u_u, u \right)$ (cf. Definition \ref{top_weak_fundamental_group_defn}) is more universal.
	
\end{example}
\begin{empt}\label{groupoid_haar_lift_empt}
	Let $\G$ be a locally compact groupoid, and let   $p: \widetilde \G^0 \to \G^0$ is a $\G$-{regular} covering. If  $\widetilde\G$ is  the $p$-{lift} of $\G$  then $\widetilde\G$ is locally compact (cf. Exercise \ref{groupoid_loc_comp_lift_exer}). Let $\left\{\la^u \left| u \in \G^0\right.\right\}$ be a left  Haar system (cf. Definition \ref{groupoid_haar_defn}) for $\G$. For any $\widetilde u \in \widetilde\G^0$ there is a homeomorphism $\widetilde \G^{\widetilde u} \cong \G^{p\left(\widetilde u \right) }$. Using it  for any $\widetilde u$ one can  naturally construct a measure $\widetilde  \la^{\widetilde u}$ on $\widetilde \G$. From the conditions (a)-(c) of the Definition \ref{groupoid_haar_defn} it follows that
	\begin{enumerate}
		\item [(a)] the support $\supp \widetilde  \la^{\widetilde u}$ of the measure $\widetilde  \la^{\widetilde u}$ is $\widetilde \G^{\widetilde u}$,
		\item [(b)]  (continuity) for any $\widetilde f \in C_c\left(\widetilde \G\right)$, $\widetilde u \mapsto \widetilde\la(\widetilde f)\left( \widetilde u\right)  = \int \widetilde f \widetilde d\la^{\widetilde u}$ is continuous, and
		\item [(c)]  (left invariance) for any $\widetilde x\in \widetilde \G$ and any $\widetilde f \in  C_c(\widetilde \G )$, $\int \widetilde  f \left(  \widetilde x \widetilde y \right)  d\la^{s(\widetilde x)}\left( \widetilde y\right)  =
		\int \widetilde f\left( \widetilde y\right) d\la^{d\left( \widetilde x\right) }\left( \widetilde y\right) $.
	\end{enumerate}
	From these circumstances it turns out that $\left\{\widetilde  \la^{\widetilde u}\left| \widetilde u \in \widetilde \G^0\right.\right\}$ is a left  Haar system for $\widetilde{\G}$.
\end{empt}
\begin{definition}\label{groupoid_haar_lift_defn}
	Under the hypotheses \ref{groupoid_haar_lift_empt} we say that the left  Haar system $\left\{\widetilde  \la^{\widetilde u}\left| \widetilde u \in \widetilde \G^0\right.\right\}$ is the $p$-\textit{lift} of $\left\{\la^u \left| u \in \G^0\right.\right\}$.
\end{definition}

\begin{empt}
	Let both $\G^n$ and $\widetilde \G^n$ be explained in \ref{groupoig_gn_empt}. If $\left(\widetilde x_0,...,  \widetilde x_{n-1}\right) \in \widetilde \G^n$ then from the Definition \ref{groupoid_lift_defn}  it follows that $\left(q_\G\left( \widetilde x_0\right) ,...,  q_\G\left( \widetilde x_{n-1}\right)\right)\in G^n$ ($q_\G$ is given by the equation \eqref{groupoid_lift_eqn}). If $\sigma$ is a continuous 2-cocycle in $Z^2\left(\G, \T\right)$ (cf. Definition \ref{groupoid_cocycle_defn}) then there is a map
	\be\label{groupoid_lift_cocyle_eqn}
	\begin{split}
		\widetilde \sigma: \widetilde \G^2 \to \T,\\
		\left(\widetilde x_0, \widetilde x_1 \right) \mapsto \left(q_\G\left( \widetilde x_0\right), q_\G\left( \widetilde x_1\right)\right). 
	\end{split}
	\ee
	Moreover  from \eqref{groupoid_b_eqn} it follows that 
	$$
	\forall \left(\widetilde x_0, \widetilde x_1 \right) \in \widetilde \G^2\quad \dl^2 \widetilde \sigma \left(\widetilde x_0, \widetilde x_1 \right)=\dl^2  \sigma \left(q_\G\left( \widetilde x_0\right), q_\G\left( \widetilde x_1\right)\right)= 1,
	$$
	i.e. $\widetilde \sigma$ is  a continuous 2-cocycle (cf. Definition \ref{groupoid_cocycle_defn}).	
\end{empt}
\begin{definition}\label{groupoid_cocycle_lift_defn}
	We say that the continuous 2-cocycle $\widetilde \sigma\in Z^2\left(\G, \T\right)$ is the $q$-lift of $\sigma$.
\end{definition}
\begin{exercise}\label{groupoid_loc_comp_lift_exer}
	Prove that if a groupoid $\G$ is locally compact then any $q$-lift of $\G$ is locally compact with respect to $p$-topology (cf. Definition \ref{groupoid_lift_defn}).
\end{exercise}

\subsection{Algebraic construction of coverings}
\begin{empt}\label{groupoidho_empt}
	If $\G$ is locally compact groupoid with a continuous 2-cocycle in $Z^2\left(\G, \T\right)$ (cf. Definition \ref{groupoid_cocycle_defn}) and a left  Haar system $\left\{\la^u\left| u \in \G^0\right.\right\}$  (cf. Definition \ref{groupoid_haar_defn}) then there is a $*$-algebra $C_c\left(\G, \sigma \right)$ with given by \eqref{groupoid_*_c_eqn} operations.
	If  $\rho: 	C_c\left(\G , \sigma \right)\to B\left(\H \right)$ is a representation in the sense of the Definition \ref{groupoid_representation_defn} then there is a $C^*$-seminorm 
	\be\label{groupoid_semin_eqn}
	\begin{split}
		\left\|\cdot  \right\|_\rho :   	C_c\left(\G , \sigma \right)\to \R,\\
		a \mapsto \left\|\rho\left( a \right)  \right\|
	\end{split}
	\ee
	We suppose that $\rho$ is \textit{faithful}, i.e. $a \neq 0\quad \Rightarrow \quad \rho_c\left(a \right) \neq 0$. The completion of  $C_c\left(\G , \sigma \right)$ with respect to 	$\left\|\cdot  \right\|_\rho$ is a $C^*$-algebra denoted by $C^*_\rho\left(\G , \sigma \right)$
\end{empt}
\begin{definition}\label{groupoidho_defn}
	Under the hypotheses \ref{groupoidho_empt} we say that the $C^*$-algebra $C^*_\rho\left(\G , \sigma \right)$ is the $\rho$-\textit{completion} of  $C_c\left(\G , \sigma \right)$.
\end{definition}
\begin{remark}\label{groupoid_blowingem}
	From the Lemma \ref{groupoid_mult_repr_lem} it follows that there is  natural Hausdorff blowing-up (cf. Definition \ref{blowing_defn})
	\be\label{groupoid_blowing_eqn}
	C_0\left(\G^0 \right)\hookto M\left( C^*_\rho\left(\G , \sigma \right)\right)  
	\ee
\end{remark}

\begin{empt}\label{groupid_tens_empt}
	Let $\G$ be a locally compact groupoid, and let  $q: \widetilde \G^0 \to \G^0$ be $\G$-{regular} covering (cf. Definition \ref{groupoid_lift_defn}). Let $$\left(\	\left\{\left(\sU_\a, \sV_\a, s_\a\right)\right\}_{\a \in \mathscr A}, \left\{\left(\widetilde \sU_{\widetilde\a}, \widetilde \sV_{\widetilde\a}, \widetilde s_{\widetilde\a}\right)\right\}_{\widetilde \a \in \widetilde{\mathscr A}} \right)$$
	be a $p$-{covering} (cf. Definition \ref{top_cov_defn}).
	One can proof that following conditions hold.
	\begin{itemize}
		\item Similarly to \eqref{blowing_ajk_eqn}  any $a \in C_c\left(\G , \sigma \right)$ can be represented by following way
		\be\label{groupoid_faf_eqn}
		a = \sum^{\substack{j = m\\k=n}}_{\substack{j = 1\\k=1}}  f _{j}af_k
		\ee
		where there are open subsets $\sU'_j, \sU''_k \in \left\{{\sU}_\a\right\}_{\a\in \mathscr A}$ such that $\supp f_j \subset \sU'_j$ and $\supp f_k \subset \sU'_k$.
		\item Similarly to \eqref{groupoid_faf_eqn} any $\widetilde a \in C_c\left(\widetilde \G, \widetilde\sigma \right)$
		\be\label{groupoid_tfaf_eqn}
		\widetilde a = \sum^{\substack{j = m\\k=n}}_{\substack{j = 1\\k=1}}  \widetilde f _{j}\widetilde a\widetilde f_{k} 
		\ee
		where there are open subsets $\widetilde \sU'_j, \widetilde\sU''_k \in \left\{{\widetilde\sU}_{\widetilde\a}\right\}_{{\widetilde\a}\in \widetilde{\mathscr A}}$ such that $\supp \widetilde f_j \subset \widetilde\sU'_j$ and $\supp \widetilde f_k \subset \widetilde\sU'_k$.
	\end{itemize}
	Below 	for any $a \in C_c\left(\G , \sigma \right)$ and $\widetilde f \in C_c\left( \widetilde \G^0\right)$ we will define a product $\widetilde f a \in C_c\left(\widetilde \G, \widetilde\sigma \right)$. 
 There is a subordinated $q$ to the family $\left\{\widetilde \sU_{\widetilde \a}\right\}$ covering sum of $\supp \widetilde f$ (cf. Definition \ref{top_covering_sum_defn}), i.e.
	$$
	\widetilde f = \sum_{j=1}^m \widetilde f_j,
	$$
	where for any $j=1,..., m$ there is $\widetilde\sU_j \in \left\{\widetilde \sU_{\widetilde \a}\right\}$ such that $\supp \widetilde f_j\subset \widetilde\sU_j$. Similarly an application of the covering sum of $\supp \widetilde f$ yields the following 
	$$
	a  = \sum_{k=1}^n f_k a.
	$$
	where for any $k=1,..., n$ there is $\sU_k \in \left\{ \sU_{ \a}\right\}$.
	If $q\left( \widetilde\sU_j \right)\cap  \sU_k = \emptyset$ then we set $\widetilde f_j\left( f_k a\right) = 0$. Otherwise there is $\widetilde\sU'_k \in \left\{\widetilde \sU_{\widetilde \a}\right\}$ such that $q\left(\widetilde\sU'_k  \right) = \sU_k$ and $\widetilde\sU'_k\cap \widetilde\sU_j\neq 0$. Let $\widetilde f'_k \bydef \lift^q_{\widetilde\sU'_k}\left( f_k\right)\in C_c\left(\widetilde \G^0 \right)  $ be the  $q$-${\widetilde\sU'_k}$-lift (cf. Definition \ref{top_lift_desc_defn}). Denote by $q_\G: \widetilde \G \to \G$ the natural continuous map. If $a$ is represented by $f_a \in C_c\left(\G \right)$ and there is an open Hausdorff subspace  $\sV \subset \G$ such that $\supp f_a \subset \sV$  then the set
	$$
	\widetilde\sV \bydef q^{-1}_\G\left(\sV \right) \cap r^{-1}\left( \widetilde\sU'_k\right) \in \widetilde \G
	$$
	is mapped homeomorphically onto $q_\G\left( \widetilde\sV\right)$. If $ \widetilde f'_a \bydef \lift^{q_\G}_{\widetilde\sV}\left( f_a\right) \in C_c\left(\widetilde G\right)$ then $\widetilde f'_a$ uniquely defines $\widetilde a' \in \widetilde \G$. We set
	\be\label{groupoid_fa_eqn}
	\widetilde f_j\left( f_k a\right)\bydef\widetilde f_j  \widetilde a' 
	\ee 
\end{empt}
\begin{lemma}\label{groupoid_tens_lem}
	If a covering  $q: \widetilde \G^0 \to \G^0$ be $\G$-{regular}  (cf. Definition \ref{groupoid_lift_defn}) then one has.
	\begin{enumerate}
		\item [(i)] The equation \eqref{groupoid_fa_eqn} yields a surjective map 
		\be\label{groupoid_tensor_eqn}
		\phi:	C_c\left(\widetilde \G^0 \right)\otimes_{C_0\left(\G^0 \right) }C_c\left( \G, \sigma \right)\xrightarrow{\approx} C_c\left(\widetilde \G, \widetilde\sigma \right)
		\ee
		where an algebraic tensor product is implied.
		\item[(ii)] There are natural  right and left  actions 
		\bean
		C_c\left(\widetilde \G, \widetilde\sigma \right)\times C_c\left( \G, \sigma \right)\to  C_c\left(\widetilde \G, \widetilde\sigma \right),\\
		C_c\left( \G, \sigma \right)\times  C_c\left(\widetilde \G, \widetilde\sigma \right)\to  C_c\left(\widetilde \G, \widetilde\sigma \right).
		\eean
	\end{enumerate}
\end{lemma}
\begin{proof}
	(i)  
	From  \eqref{groupoid_tfaf_eqn}  it follows that one should prove that  any $\widetilde f' \widetilde a \widetilde f''$ with $\widetilde{\sU}', \widetilde{\sU}''\in \left\{{\widetilde\sU}_{\widetilde\a}\right\}$ such that $\supp \widetilde f' \subset \widetilde \sU'$ and $\supp \widetilde f'' \subset \widetilde \sU''$  there is $\widetilde f \otimes a \in C_c\left(\widetilde \G^0 \right)\otimes_{C_0\left(\G^0 \right) }C_c\left( \G, \sigma \right)$ which satisfies to following condition $\phi\left(\widetilde f \otimes a \right) = \widetilde f' \widetilde a \widetilde f''$. The set $\widetilde{\sU}''$ is mapped homeomorphically onto $q\left(\widetilde{\sU}''\right)$ so the set $s^{-1}\left(\widetilde{\sU}''\right)$ is mapped homeomorphically onto $q_\G\left(s^{-1}\left(\widetilde{\sU}''\right)\right)$, so if $\widetilde a \widetilde f''$ corresponds to $\widetilde h \in C_c\left(\widetilde \G \right)$ then one has a $q_\G$ descent $h \bydef \desc_{q_\G}\left( \widetilde h\right)$ (cf. Definition \ref{top_lift_desc_defn}) if $a \in  C_c\left( \G, \sigma \right)$ corresponds to $h$ then from the construction \ref{groupid_tens_empt} it follows that 
	$$
	\widetilde f' \widetilde a \widetilde f''= \widetilde f' a. 
	$$
	(ii) From (i) it follows that 
	$$
	\forall \widetilde a \in C_c\left(\widetilde \G, \widetilde\sigma \right)~~ \exists \sum_{j = 1}^n \widetilde f_j \otimes a_j \in C_c\left(\widetilde \G^0 \right)\otimes_{C_0\left(\G^0 \right) }C_c\left( \G, \sigma \right)~~ \widetilde a= \phi\left( \sum_{j = 1}^n \widetilde f_j \otimes a_j\right).
	$$
	For any $a \in  C_c\left( \G, \sigma \right)$ we define
	$$
	\widetilde a a\bydef \phi\left( \sum_{j = 1}^n \widetilde f_j \otimes a_ja\right)
	$$
	so one has a right  action  $C_c\left(\widetilde \G, \widetilde\sigma \right)\times C_c\left( \G, \sigma \right)\to  C_c\left(\widetilde \G, \widetilde\sigma \right)$. An application of $*$-operation yields the left  action $C_c\left( \G, \sigma \right)\times  C_c\left(\widetilde \G, \widetilde\sigma \right)\to  C_c\left(\widetilde \G, \widetilde\sigma \right)$.
\end{proof}

\begin{corollary}\label{groupoid_tens_cor}
	Use notation of the Definition \ref{blowing_aegular_defn}.
	Let $\rho: {C_c\left(\G , \sigma \right)} \to B\left( \sH\right) $ be a faithful,   representation in the sense of the Definition \ref{groupoid_representation_defn}.
	Let $q: \widetilde \G^0 \to \G^0$ let a $\G$-regular   covering (cf. Definition \ref{groupoid_lift_defn}) and $C^*_\rho\left(\G , \sigma \right)$ is the $\rho$-\textit{completion} of  $C_c\left(\G , \sigma \right)$. If $\mathscr L^2\left(\widetilde \G^0 \right)_{C^*_\rho\left(\G , \sigma \right)}$ is the $\widetilde \G^0$-$C^*_\rho\left(\G , \sigma \right)$-module (cf. Definition \ref{blowing_lift_hm_defn}) then there is the natural homomorphisms 
	\be\label{groupoid_prods_eqn}
	\phi:C_c\left( \widetilde \G^0 \right)\otimes_{C_0\left(\G^0\right) } C_c\left(\G , \sigma \right)\to  \End^*_A\left( \mathscr L^2\left(\widetilde \G^0 \right)_{C^*_\rho\left(\G , \sigma \right)}\right)
	\ee
	of right $C_c\left(\G , \sigma \right)$-modules such that 	$\varphi\left( C_c\left( \widetilde \G^0 \right)\otimes_{C_0\left(\G^0\right) } C_c\left(\G , \sigma \right)\right)$  is a $*$-subalgebra of $$\End^*_{C^*_\rho\left(\G , \sigma \right)}\left( \mathscr L^2\left(\widetilde \G^0 \right)_{C^*_\rho\left(\G , \sigma \right)}\right).$$ 
\end{corollary}
\begin{proof}
	From the Lemma \ref{groupoid_tens_lem} it follows that the generated by elements \ref{groupoid_prods_eqn}  $C_c\left(\G , \sigma \right)$-module is isomorphic to $*$-algebra  $C_c\left(\widetilde \G, \widetilde\sigma \right)$.
\end{proof}
\begin{theorem}\label{groupoid_lift_thm}
	If   $q: \widetilde \G^0\to  \G^0$  is a $\G$-regular   covering (cf. Definition \ref{groupoid_lift_defn}) then  $q: \widetilde \G^0\to  \G^0$ is $C^*_\rho\left(\G , \sigma \right)$-regular (cf. Definition \ref{blowing_aegular_defn}) 
\end{theorem}
\begin{proof}
	Since $C_c\left(\G , \sigma \right)$ is dense in $C^*_\rho\left(\G , \sigma \right)$ the map \eqref{groupoid_prods_eqn} can be uniquely extended up to a map
	$$
	C_c\left( \widetilde \G^0 \right)\otimes_{C_0\left(\G^0\right) } C^*_\rho\left(\G , \sigma \right)\to  \End^*_A\left( \mathscr L^2\left(\widetilde \G^0 \right)_A\right)	
	$$
	so one has a map
	$$
	\varphi : 	C_c\left( \widetilde \G^0 \right)\otimes_{C_0\left(\G^0\right) } K\left( C^*_\rho\left(\G , \sigma \right)\right) \to  \End^*_A\left( \mathscr L^2\left(\widetilde \G^0 \right)_A\right)	
	$$
	On the other hand both $C_c\left(\G , \sigma \right)$ and $K\left( C^*_\rho\left(\G , \sigma \right)\right) $ are dense in $C^*_\rho\left(\G , \sigma \right)$ so the  $C^*$-norm closure  of $\phi\left( C_c\left( \widetilde \G^0 \right)\otimes_{C_0\left(\G^0\right) } C_c\left(\G , \sigma \right)\right)$ coincides with the $C^*$-norm closure of  $\varphi\left( C_c\left( \widetilde \G^0 \right)\otimes_{C_0\left(\G^0\right) } K\left( C^*_\rho\left(\G , \sigma \right)\right)\right)$. However from the Corollary \ref{groupoid_tens_cor} it turns out that the $C^*$-norm completion of $\phi\left( C_c\left( \widetilde \G^0 \right)\otimes_{C_0\left(\G^0\right) } C_c\left(\G , \sigma \right)\right)$ is a $C^*$-subalgebra of $\End^*_A\left( \mathscr L^2\left(\widetilde \G^0 \right)_A\right)$, i.e. one has a specialization of the Definition \ref{blowing_aegular_defn}.
\end{proof}

\begin{corollary}\label{groupoid_lift_cor}
Any    is a $\G$-regular   covering  $q: \widetilde \G^0\to  \G^0$  (cf. Definition \ref{groupoid_lift_defn}) yields a noncommutative covering (cf. Definition \ref{covering_defn}).
\end{corollary}
\begin{proof}
	Follows from the Theorems \ref{hausdorff_covering_theorem} and \ref{groupoid_lift_thm}.
\end{proof}

\section{Gelfand spaces of continuous trace $C^*$-algebras}
\subsection{Basic construction}
\paragraph{}

Let   $A$  be a {continuous-trace} $C^*$-{algebra} with the spectrum $\sX$, and let
\be\label{ctr_comm_eqn}
\mathfrak{Comm}\left( A\right)  \text{ is a set of all commutative } C^*\text{-subalgebras of } A.
\ee
Then there is a presheaf  $\mathscr P_{\sX\text{-}A}$ (cf. Definition \ref{presheaf_defn}) of sets such that for any open $\sU \subset \sX$  one has
\be\label{presh_eqn}
\begin{split}
	\mathscr P_{\sX\text{-}A}\left(\sU \right) \left\{ C \in	\mathfrak{Comm}\left( A\right) \left|~ C \subset ~_\sU A_\sU \text{ cf. \eqref{blowing_hereditary_u_eqn})}\right. \right\};\\
	\rho_{\sU\sV} : \mathscr P_{\sX\text{-}A}\left(\sU \right)\to \mathscr P_{\sX\text{-}A}\left(\sV \right),\\
	C \mapsto C \cap ~_\sV A _\sV.
\end{split}
\ee
and let $\mathrm{Sp\acute{e}}\left(\mathscr P_{\sX\text{-}A} \right)$ be the \'espace etal\'e of the presheaf $\mathrm{Sp\acute{e}}\left(\mathscr P_{\sX\text{-}A} \right)$ (cf. Exercise \ref{sheaf_etale_exer}) with the natural local homeomorphism $p_{\mathrm{Sp\acute{e}}\left(\mathscr P_{\sX\text{-}A} \right)} : \mathrm{Sp\acute{e}}\left(\mathscr P_{\sX\text{-}A} \right) \onto \sX$.
For any hereditary $C^*$-subalgebra $B \subset A$ we define
\be\label{u_b_eqn} 
\mathscr U_B \bydef \left\{\left.C_x \in \mathrm{Sp\acute{e}}\left(\mathscr P_{\sX\text{-}A} \right)\right|\exists C \in \mathfrak{Comm}\left( A\right)\quad C \subset B \right\}\subset \mathrm{Sp\acute{e}}\left(\mathscr P_{\sX\text{-}A} \right)
\ee
where $C_x$ is a stalk of $C$ (cf. Definition \ref{sheaf_stalk_defn}). For any open subset $\mathscr U\subset \mathrm{Sp\acute{e}}\left(\mathscr P_{\sX\text{-}A} \right)$ and any $\eta \in \mathscr U$ we say that $C \in \mathfrak{Comm}\left( A\right)$ is an $\mathscr U$-{\it representative of} $\eta$ if one has
\bea
\exists x \in \sX \quad \eta = C_x,\\
\forall y \in \sX \quad \rep_y\left(C \right) \neq \{0\}\quad \Rightarrow \quad C_y \in \mathscr U.
\eea
Any $\eta \in \mathscr U$ has an an $\mathscr U$-{ representative of} $\eta$ since the set an $\mathscr U$ is open. We define a hereditary $C^*$-subalgebra $B_{\mathscr U } \subset A$ which contains the union
\be
\bigcup \left\{C \in \mathfrak{Comm}\left( A\right)\left|\exists  \eta \in  \mathscr U\quad C \quad \text{ is an} \mathscr U \text{-representative of }\quad \eta \right.\right\}
\ee
\begin{lemma}
	One has.
	\begin{enumerate}
		\item [(i)] For any hereditary $C^*$-subalgebra $B \subset A$  the given by \eqref{u_b_eqn}  set $\mathscr U_B$ is open.
		\item [(ii)] The maps
		\bea \label{u_t_b_eqn}
		\mathscr U \mapsto B_{\mathscr U},\\\label{b_t_u_eqn}
		B \mapsto \mathscr U_B
		\eea	
		yield the  1-1 correspondence between open subsets of $\mathrm{Sp\acute{e}}\left(\mathscr P_{\sX\text{-}A}\right) $ and hereditary $C^*$-subalgebras $B \subset A$ 
	\end{enumerate}
\end{lemma}
\begin{proof}
	(i)
	If $C \in \mathfrak{Comm}\left( A\right)$ and $C \subset B$ then from the Exercise \ref{sheaf_etale_exer}) it turns out that the union 
	$$
	\bigcup_{\substack{x \in\sX\\ \rep_x\left(C \right)\neq 0 }}C_x \subset \mathrm{Sp\acute{e}}\left(\mathscr P_{\sX\text{-}A}\right)
	$$
	is open. So $\mathscr U_B$ is a union of open subsets.
	
	(ii) If both  $\mathscr U', \mathscr U'' \subset \mathrm{Sp\acute{e}}\left(\mathscr P_{\sX\text{-}A}\right)$ are open subsets with $\eta \in \mathscr U'\setminus \mathscr U''\neq \emptyset$ then one has
	\bean
	C \subset B_{\mathscr U'},\\
	C \not\subset B_{\mathscr U''}	
	\eean 
	where $C$ is an $	\mathscr U'$-representative of $\eta$. If follows that
	$$
	\mathscr U' = \mathscr U'' \quad \Rightarrow \quad B_{\mathscr U'}= B_{\mathscr U''}.
	$$
	From the Proposition \ref{continuous_trace_c_a_proposition} it turns out that $B'$ is a $C^*$-algebra of type $I_0$ i.e. It is generated by its Abelian elements. So if $B'\setminus B'' \neq \{0\}$ then there is an Abelian element $a$ in $B'\setminus B''$. The generated by $a$ $C^*$-algebra $C$  is commutative, i.e. $C \in \mathfrak{Comm}\left( A\right)$.  From 
	\bean
	C \subset B',\\
	C \not\subset B''.
	\eean 
	one has $x \in \sX$ such that
	\bean
	C_x \in \mathscr U_{B'},\\
	C_x \notin \mathscr U_{B''}.
	\eean 
	It turns out that
	$$
	B' = B'' \quad \Rightarrow \quad \mathscr U_{B'} = \mathscr U_{B''}.
	$$
\end{proof}

\begin{lemma}\label{ctr_gelfand_exist_lem}
	If $A$ is a {continuous-trace} $C^*$-{algebra}  with the spectrum $\sX$. For any element $\xi \in 	\mathfrak{Gelfand}\left( A\right)$ there is a commutative hereditary $C^*$-subalgebra $C$ such that $C \in \xi$ (cf. Notation \eqref{hered_ideal_eqn}).
\end{lemma}
\begin{proof}
	Let   $x\bydef \hat \phi \left( \xi\right)  \in \sX$ 	where $\hat \phi$ is a specialization of the continuous map \eqref{to_spec_eqn}. According to the Definition \ref{gelfand_space_defn} there is  positive n element $a \in K\left( A\right)_+$ of the Pedersen's ideal such that  $aAa \in \xi$ (cf. notation \eqref{hered_ideal_eqn}). If $\rep_y : A \to B\left( \sH_y\right)$ is  the corresponding to $y \in \sX$  irreducible representation of $A$ (cf. \ref{irred_defn}) then from the equation  \eqref{peder_k_eqn} it follows that there is $N \in \N$ such that $\forall y \in \sX \quad \rep_y \left(aAa \right) \subset  \mathbb{M}_N\left(\C\right)$. Let $\left\{\sU\right\}$ be a basis of neighborhoods of $x$. Denote by  $S \subset \N$ the subset be such that for any $l\in S$ there are $B \in \xi$ and $\sU \in \left\{\sU\right\}$ with
	\bean
	B \subset aAa,\\
	\forall y \in \sU \quad \rep_y\left(B\right) \subset \mathbb{M}_l\left(\C\right)
	\eean 
	From the our construction it follows that $N \in S$. Suppose that   $d \bydef \min_{l \in S}l > 1$, and let $\sU\in  \left\{\sU\right\}$, $~B \in \xi$ be such that
	\bean
	B \subset aAa,\\
	\forall y \in \sU  \quad \rep_y\left(B \right)\subset \mathbb{M}_d\left(\C\right),\\
	\exists y \in \sU  \quad \rep_y\left(B \right)\cong \mathbb{M}_d\left(\C\right)
	\eean 
	From the Proposition \ref{less_n_pi_prop} it follows that
	the set
	$$
	\sV \bydef \left\{y \in \sU \left|\rep_y\left(B\right) \cong  \mathbb{M}_d\left(\C\right) \right.\right\}\subset \sU
	$$
	is open. Moreover there is $\sU' \in \left\{\sU\right\}$ such that there is the the natural inclusion 
	$$
	C_0\left(\sU' \cap \sV \right) \otimes \mathbb{M}_d\left(\C\right)\subset B
	$$
	For any inclusion $\mathbb{M}_{d-1}\left(\C\right)\subset \mathbb{M}_{d}\left(\C\right)$ there is a  $C^*$-subalgebra $B' \subset  B$ which is the image of $C_0\left(\sU' \cap \sV \right) \otimes \mathbb{M}_{d-1}\left(\C\right)$ where the following composition of inclusions
	$$
	C_0\left(\sU' \cap \sV \right) \otimes \mathbb{M}_{d-1}\left(\C\right)\subset C_0\left(\sU' \cap \sV \right) \otimes \mathbb{M}_d\left(\C\right)\subset   B
	$$
	is implied.
	If   
	$$
	B'' \bydef \left\{b \in B \left| \forall y \in \sU' \cap \sV  \quad \rep_y\left(b \right) \in \rep_x\left(B' \right) \right.\right\}
	$$ 
	then 
	\bean
	\forall y \in \sU' \quad  \rep_y\left(B' \right)\subsetneqq \mathbb{M}_d\left(\C\right),\\
	\forall B \in \xi \quad B \cap B'' \neq \{0\}.
	\eean
	If a filter $\xi'$ is generated by the set
	$$
	\left\{B \cap B'' | B \in \xi \right\}
	$$
	then $\xi \subsetneqq \xi'$. It is impossible since $\xi$ is an ultrafilter. From this contradiction it follows that $d= 1$. So there is $B'' \in \xi$ and $\sU'' \in \left\{\sU\right\}$ 
	such that 
	$$
	\forall x \in \sU' \quad \rep_x\left(B'' \right) \cong \C
	$$
	and $C''' \bydef B'' \cap ~_{\sU''} A_{\sU''}\subset A$ is a commutative $C^*$-algebra such that $C''' \in \xi$.
\end{proof}
\begin{theorem}\label{ctr_gelfand_thm}
	If $A$ is a {continuous-trace} $C^*$-{algebra}  with the spectrum $\sX$ then there is the natural bijective set theoretic map $\phi_{\mathscr P}:\mathfrak{Gelfand}\left( A\right) \cong \mathrm{Sp\acute{e}}\left(\mathscr P_{\sX\text{-}A} \right)$  with the following commutative diagram
	\bean
	\begin{tikzcd}
		\mathfrak{Gelfand}\left( A\right)		 \arrow[rr, "\phi_{\mathscr P}"]\arrow[rd, "\hat \phi"]  & ~ & \mathrm{Sp\acute{e}}\left(\mathscr P_{\sX\text{-}A} \right) \arrow[ld, "p_{\mathrm{Sp\acute{e}}\left(\mathscr P_{\sX\text{-}A} \right)}" ]\\
		& \sX &
	\end{tikzcd}
	\eean
	where $\hat \phi$ is a specialization of the continuous map given by the equation \ref{to_spec_eqn}.
\end{theorem}	
\begin{proof}
	Let $\xi \in \mathfrak{Gelfand}\left( A\right)$ and let   $x\bydef \hat \phi \left( \xi\right)  \in \sX$. From the Lemma \ref{ctr_gelfand_exist_lem} it turns out that there is a commutative $C^*$-algebra $C$ such that $C \in \xi$. Consider a set of hereditary $C^*$-subalgebras
	\be\label{ctr_xx_eqn}
	\xi' \bydef \left\{B \subset A \left| C_x \in\mathscr U_B\right.\right\}
	\ee
	where $\mathscr U_B \subset \mathrm{Sp\acute{e}}\left(\mathscr P_{\sX\text{-}A} \right)$ is given by \eqref{u_b_eqn}. If $B',B''\in \xi'$ then there are commutative $C'\subset B'$ and $C'' \subset B''$ such that $C'_x = C''_x = C_x$. From $C'\cap C'' \neq \{0\}$ it follows that $B'\cap B'' \neq \{0\}$, i.e. $\xi'$ is a filter. If there is a filter $\xi''$ with $\xi'\subsetneqq \xi''$ then there is $B''' \in \xi'' \setminus \xi'$, i.e. $C_x \notin \mathscr U_{B'''}$. So there is a commutative $C^*$-algebra $C'''$ with $C'''_x = C_x$ and $C'''\cap B''' = \{0\}$. However it is impossible since $C''' \in \xi' \subset \xi''$. So is no a filter $\xi''$ with $\xi'\subsetneqq \xi''$, i.e. $\xi'$ is an ultrafilter. From \eqref{ctr_xx_eqn} it turns out that $\xi'$ is the unique ultrafilter having following properties:
	\bean
	\hat \phi\left(\xi' \right) = \hat \phi\left(\xi \right),\\
	C \in \xi'
	\eean 
	It follows $\xi = \xi'$ since $\xi$ satisfies to the above properties. From the equation \eqref{ctr_xx_eqn} it follows that both $\xi$ and $\xi'$ uniquely depend on $C_x$. The bijective map $\phi_{\mathscr P}$ is given by $\xi \mapsto C_x$.
\end{proof}
\begin{remark}\label{etale_g_rem}
	
	Denote by $\mathrm{Sp\acute{e}}\left(\mathscr P_{\sX\text{-}A} \right)_{\mathfrak{Etale}}$ the set $\mathrm{Sp\acute{e}}\left(\mathscr P_{\sX\text{-}A} \right)$ supplied with the smallest topology such that the map $\mathrm{Sp\acute{e}}\left(\mathscr P_{\sX\text{-}A} \right)_{\mathfrak{Etale}}\to \sX$ is continuous. If $\mathrm{Sp\acute{e}}\left(\mathscr P_{\sX\text{-}A} \right)_{\mathfrak{Gelfand}}$ is the topological space such that the bijective map $\mathrm{Sp\acute{e}}\left(\mathscr P_{\sX\text{-}A} \right)_{\mathfrak{Gelfand}}\cong \mathfrak{Gelfand}\left( A\right)$ is  homeomorphism then since the map $\mathfrak{Gelfand}\left( A\right)\to \sX$ is continuous the topology of $\mathrm{Sp\acute{e}}\left(\mathscr P_{\sX\text{-}A} \right)_{\mathfrak{Gelfand}}$ is finer than $\mathrm{Sp\acute{e}}\left(\mathscr P_{\sX\text{-}A} \right)_{\mathfrak{Etale}}$ one. 
\end{remark}

\subsection{Generalized Dixmier–Douady obstruction}\label{ctr_dd_sec}
\paragraph{}
The Proposition \ref{ctr_d_prop} yields an obstruction which gives a  difference between $C^*$-algebras. Here we explain the interoperability og Proposition \ref{ctr_d_prop} and the generalized Dixmier–Douady obstruction (cf. Definition \ref{dixmier_douady_defn}).
Let $A\bydef CT\left(\sX, \dl\left(A \right)  \right)$ be a stable separable algebra $A$ of continuous trace over a second-countable locally compact Hausdorff space $\sX$ with  Dixmier–Douady class $\dl\left(A \right) \in \check{H}^3\left(\sX, \Z \right)$ (cf. Proposition \ref{ctr_bundle_prop}).  
From the Proposition \ref{ctr_bundle_prop}
it follows that  there is a locally
trivial bundle $\pi_{\dl\left(A \right)}: \F^{\dl\left(A \right)}\to \sX$ with a fibre $\K(\H)$ and a structure group $\Aut \left( \K(\H)\right)$ such that $CT\left(\sX, {\dl\left(A \right)} \right)$ is $C_0\left( \sX\right)$-isomorphic to the space of sections $\Ga_0\left(\sX, \F^{\dl\left(A \right)} \right)$. The space $\F^{\dl\left(A \right)}$ as a set equals to a union 
$$
\F^{\dl\left(A \right)} = \bigcup_{x\in\sX}\K\left(\H_x \right) 
$$
where $\H_x$ is a space of an irreducible representation $\pi_x: A \to B\left(\H_x\right)$ which corresponds to $x \in \sX$. There is a sheaf $\mathscr F^{\dl\left(A \right)}$ of stalks of local sections of $\F^{\dl\left(A \right)}$.
If
$$
\E^{\dl\left(A \right)}  \bydef \bigcup_{x\in\sX} \left\{e \in \K\left(\H_x \right)| ~ e\quad\text{is a rank-one projector}\right\}\subset \F^{\dl\left(A \right)}
$$
and a topology of $\E^{\dl\left(A \right)}$ is induced by the natural inclusion $\E^{\dl\left(A \right)}\subset \F^{\dl\left(A \right)}$ then one has a locally trivial subbundle $\pi_{\dl\left(A \right)}|_{\E^{\dl\left(A \right)}} \E^{\dl\left(A \right)}\to  \sX$, 
with the structure group $\Aut \left( \K(\H)\right)$.  Any rank one projector $e\in  \K\left(\H_x \right)$ uniquely defines an element of a projective space $P\left(\H_x \right)$ such that 
$$
P\left(\H_x \right)\bydef \C P^\om \bydef \begin{cases}
	\C P^n & \dim \H_x = n\\
	\C P^\infty & \dim \H_x = \infty
\end{cases}
$$
Using this circumstance one can construct the bundle $\G^{\dl\left(A \right)}$ with fibre $\C P^\om$. One can prove that ${\dl\left(A \right)}$ is the obstruction to the global section of $\G^{\dl\left(A \right)}$. Let $x \in \sX$, and let  $s_x \in \left(	\mathscr P_{\sX\text{-}A} \right)_x$ be a stalk of the given by the equation \eqref{presh_eqn} presheaf.
If $s_x$ is represented by commutative subalgebra $C\subset CT\left(\sX, {\dl\left(A \right)} \right)$ then $C$ yields a point of $\C P^\om$.  So there is the natural surjective map $\mathrm{Sp\acute{e}}\left(\mathscr P_{\sX\text{-}A} \right)\onto \G^{\dl\left(A \right)}$ where 
$\mathrm{Sp\acute{e}}\left(\mathscr P_{\sX\text{-}A} \right)$ is the \'espace etal\'e of the presheaf $\mathrm{Sp\acute{e}}\left(\mathscr P_{\sX\text{-}A} \right)$ (cf. Exercise \ref{sheaf_etale_exer}). One can prove that this map yields the continuous map $\mathrm{Sp\acute{e}}\left(\mathscr P_{\sX\text{-}A} \right)_{\mathfrak{Gelfand}}\onto \G^{\dl\left(A \right)}$ (cf. Remark \ref{etale_g_rem}).
There is the natural commutative diagram
\bean
\begin{tikzcd}
\mathrm{Sp\acute{e}}\left(\mathscr P_{\sX\text{-}A} \right)_{\mathfrak{Gelfand}} \arrow[rd, two heads]   \arrow[rr, two heads, "\hat\phi"]&& \sX \\
& \G^{\dl\left(A \right)}  \arrow[ru, two heads, "p_A"] & 
\end{tikzcd}
\eean
where the map $\varphi$ is a specialization of the  \eqref{to_spec_eqn} one. The above diagram yields the diagram 
\bean
\begin{tikzcd}
\check{H}^3\left( 	\mathrm{Sp\acute{e}}\left(\mathscr P_{\sX\text{-}A} \right)_{\mathfrak{Gelfand}}; \Z\right)   && \check{H}^3\left( \sX;\Z\right)\arrow[ld, "\check{H}^3\left(p_A; \Z\right)"] \arrow[ll,  "\check{H}^3\left(\hat\phi; \Z\right)"]  \\
	&\check{H}^3\left( \G^{\dl\left(A \right)}; \Z \right) \arrow[lu]   & 
\end{tikzcd}
\eean
The above diagram can be rewritten as 
\be
\begin{tikzcd}
	\check{H}^3\left( 	\mathfrak{Gelfand}\left(A \right) ; \Z\right)   && \check{H}^3\left( \sX;\Z\right)\arrow[ld, "\check{H}^3\left(p_A; \Z\right)"] \arrow[ll,  "\check{H}^3\left(\hat\phi; \Z\right)"]  \\
	&\check{H}^3\left( \G^{\dl\left(A \right)}; \Z \right) \arrow[lu]   & 
\end{tikzcd}
\ee
The class $\d\left( A\right)$ is an obstruction for the existence of the section $s$ of $p_A$, i.e. $p_A \circ s = \Id_\sX$.  It turns out that $\check{H}^3\left(p_A; \Z\right)\left(\dl\left(A \right)  \right) = 0$.
Let $B\bydef CT\left(\sX, \dl\left(B \right)  \right)$ be a stable separable algebra $A$ of continuous trace over a second-countable locally compact Hausdorff space $\sX$ with  Dixmier–Douady class $\dl\left(B \right) \in \check{H}^3\left(\sX, \Z \right)\setminus \dl\left(A \right)$. Let $\ga = \bydef \dl\left( B\right) - \dl\left( A\right)$ then from the Proposition \ref{ctr_d_prop} it follows that 
\be\label{ctr_dd_eqn}
B \cong A \times_\sX C \cong  C \times_\sX A
\ee
where $C \bydef  CT\left(\sX, \ga  \right)$. If $b \in B$ is represented by a pair $\left(a, c \right) \in A \times_\sX C$ them for any $a' \in A$ one can define a product $a'b \in B$ which is represented by a pair $\left(a'a, c \right) \in A \times_\sX C$. Similarly one can define a product $ba'$. Using these products one can define an inclusion $A \hookto M\left( B\right)$ which can be uniquely extended up to  $\phi : M\left( A\right) \hookto M\left( B\right)$. One can prove that $\phi$ is good (cf. Definition \ref{good_hom_defn}). The Theorem \ref{good_thm} yields a continuous map 
\be\label{ctr_ba_eqn} 
\varphi: \mathfrak{Gelfand}\left( B\right) \to \mathfrak{Gelfand}\left( A\right)
\ee
Summarize above constructions of has the following diagram
\be
\begin{tikzcd}
\check{H}^3\left( \G^{\dl\left(A \right)}; \Z \right)\arrow[r]& \check{H}^3\left(\mathfrak{Gelfand}\left( A\right); \Z \right)\arrow[dd, "\check{H}^3\left( \varphi; \Z \right)"]\\
\check{H}^3\left( \sX; \Z \right)\arrow[u]\arrow[d]& \\
\check{H}^3\left( \G^{\dl\left(B \right)}; \Z \right)\arrow[r]& \check{H}^3\left(\mathfrak{Gelfand}\left( B\right); \Z \right)
\end{tikzcd}
\ee
The image of $\dl\left( B\right)$ in both $\check{H}^3\left( \G^{\dl\left(B \right)}; \Z \right)$ and $\check{H}^3\left(\mathfrak{Gelfand}\left( B\right); \Z \right)$ is always trivial. If the image of $\dl\left( B\right)$ in $\check{H}^3\left(\mathfrak{Gelfand}\left( B\right); \Z \right)$ is not trivial then $\varphi$ is not a homeomorphism, i.e. one has a generalized Dixmier-Douady obstruction (cf. Definition \ref{dixmier_douady_defn}).

\begin{appendices}
	\section{Category theory}
	\subsection{Categories}
	\begin{thm}\label{zorn_thm}\cite{spanier:at} (Zorn's lemma). A partially ordered set in which every  simply ordered set has an upper bound contains maximal elements.
	\end{thm}

	\begin{definition}\label{category_defn}\cite{goldblatt:topoi}
		A \textit{category} $\mathscr C$ comprises: 
		\begin{itemize}
			\item[(1)] 	 a collection of things called $\mathscr C$-\textit{objects}; 
			\item[(2)] a collection of things called $\mathscr C$-\textit{arrows}  or $\mathscr C$-\textit{morphisms}; 
			\item[(3)] operations assigning to each  $\mathscr C$-arrow $f$  $\mathscr C$-object $\mathrm{dom}f$ (the 
			"domain" of $f$) and a $\mathscr C$-object $\mathrm{cod}f$ (the "codomain" of $f$). If $a =\mathrm{dom} f$  
			and $b =\mathrm{cod} f$ we display this as 
			\bean
			f: a \to b\quad \text{or}\quad  a \xrightarrow{f} b;
			\eean
			\item[(4)] an operation assigning to each pair $\left(g, f\right)$ of $\mathscr C$-arrows with $\mathrm{dom}g = \mathrm{cod} f$.
			A $\mathscr{C}$-arrow  $g\circ f$, the \textit{composite of f and g}, having $\mathrm{dom}g\circ f= \mathrm{dom}f$ and  $\mathrm{cod }g\circ f= \mathrm{cod}g$, i.e. $g\circ f:  \mathrm{dom}f \to \mathrm{cod}g$  such that the following condition obtains:
			
			\textit{	Associative Law}: Given the configuration 
			$$
			a \xrightarrow{f}	b \xrightarrow{g}	c \xrightarrow{h}d
			$$ 
			of  $\mathscr{C}$-objects and $\mathscr{C}$-arrows  then $h \circ \left(g \circ f\right)= \left(h \circ g\right)\circ f$;
			\item[(5)] an assignment to each $\mathscr{C}$-object $b$ of a $\mathscr{C}$-arrow $\mathbb{1}_b : b \to b$, called the \textit{identity arrow} on $b$, such that 
			
			\textit{Identity Law}: For any $\mathscr{C}$-arrows $f: a \to b$ and $g : b \to c$ one has
			\bean
			\mathbb{1}_b \circ f = f \quad \text{and}\quad g\circ \mathbb{1}_b= g.
			\eean
		\end{itemize}
		
	\end{definition}
	\begin{defn}\label{sets_cat_defn}\cite{johnstone:topos}
		There is the \textit{category  of sets and functions} $\mathscr{S}$. We use the therm \textit{small category} for a category those morphisms form a set. If $\mathscr C$ is a small category, we write $\mathscr{S}^{\mathscr{C}^{\mathrm{op}}}$ for the category of \textit{presheaves}, i.e. contravariant functors from $\mathscr{C}$ to $\mathscr{S}$; among the objects of $\mathscr{S}^{\mathscr{C}^{\mathrm{op}}}$, we have the \textit{representable} functors $h_U$, where $U$ is an object of $\mathscr C$, defined by $h_U\left( V\right) \bydef \Hom_{\mathscr{C}}\left( V, U\right)$; we shall also write $h^U$ for the covariant representable  functor $\Hom_{\mathscr{C}}\left( U, -\right)$.
	\end{defn} 
	\begin{notation}\label{category_not}
		If $a, b$ are  $\mathscr{C}$-objects then we denote by  $\mathscr{C}\left(a, b \right)$  a family of all arrows from $a$ to $b$.
	\end{notation}

	\begin{definition}\label{initial_ob_defn}\cite{goldblatt:topoi}
			An object $0_{\mathscr C}$ is \textit{initial} in category $\mathscr C$ if for every $\mathscr C$-object $a$ 
		there is one and only one arrow from $0_{\mathscr C}$ to $a$ in ${\mathscr C}$. 
	\end{definition}
	\begin{definition}\label{terminal_ob_defn}\cite{goldblatt:topoi}
		An object $1_{\mathscr C}$ is \textit{terminal} or \textit{final} in category $\mathscr C$ if for every $\mathscr C$-object $a$ 
		there is one and only one arrow to $0_{\mathscr C}$ to $a$ in ${\mathscr C}$. 
	\end{definition}
	
	\subsection{Limits and colimits}\label{limit_sec}
	\paragraph{} Here I follow to \cite{goldblatt:topoi}. The notion of \textit{commutative diagram}, is a 
	very important aid to understanding used in category theory.
	By a 
	diagram we simply mean a display of some objects, together with some 
	arrows (here representing functions) linking the objects. The "triangle" of 
	arrows $f$, $g$, $h$ as shown is another diagram. 
	\newline
	\begin{tikzpicture}
		\matrix (m) [matrix of math nodes,row sep=3em,column sep=4em,minimum width=2em]
		{
			A & B  \\ 
			& C\\};
		\path[-stealth]
		(m-1-1) edge node [above] {$f$} (m-1-2)
		(m-1-1) edge node [left]  {$h~~$} (m-2-2)
		(m-1-2) edge node [right] {$~~g$} (m-2-2);
	\end{tikzpicture}
	\\ 	
	It will be said to \textit{commute} if $h = g\circ f$. The point is that the diagram offers 
	two paths from $A$ to $C$, either by composing to follow $f$ and then $g$, or by 
	following h directly. Commutativity means that the two paths amount to 
	the same thing. A more complex diagram, like the previous one, is said to 
	be commutative when all possible triangles that are parts of the diagram 
	are themselves commutative. This means that any two paths of arrows in 
	the diagram that start at the same object and end at the same object 
	compose to give the same overall arrow. By a \textit{diagram} $D$ in a 
	category $\mathscr C$ we simply mean a collection of $\mathscr C$-objects $d_j, d_k,...$ together 
	with some $\mathscr C$-arrows $g: d_j \to d_k$ between certain of the objects in the 
	diagram. (Possibly more than one arrow between a given pair of objects, 
	possibly none).
	\begin{definition}\label{lim_defn}
		A \textit{cone} for diagram $D$ consists of a $\mathscr C$-object $c$ together with a $\mathscr C$-arrow 
		$c\to d_j$ for each object $f_j$ in $D$, such that
		\newline
		\begin{tikzpicture}
			\matrix (m) [matrix of math nodes,row sep=3em,column sep=4em,minimum width=2em]
			{
				d_j  & & d_k \\ 
				& c  & \\};
			\path[-stealth]
			(m-1-1) edge node [above] {$g$} (m-1-3)
			(m-2-2) edge node [left]  {$f_j~~$} (m-1-1)
			(m-2-2) edge node [right] {$~~f_k$} (m-1-3);
		\end{tikzpicture}
		\\ 	
		commutes, whenever g is an arrow in the diagram $D$. We use the 
		symbolism $\left\{f_j: c\to d_j\right\}$ to denote a cone for $D$.

		A \textit{limit} for a diagram $D$ is a $D$-cone $\left\{f_j: c\to d_j\right\}$ with the property that for 
		any other $D$-cone $\left\{f'_j: c'\to d_j\right\}$ there is exactly one arrow $f:c' \to c$ such 
		\newline
		\begin{tikzpicture}
			\matrix (m) [matrix of math nodes,row sep=3em,column sep=4em,minimum width=2em]
			{
				& d_j &  \\ 
				c'	&  & c \\};
			\path[-stealth]
			(m-2-1) edge node [above] {$g$} (m-2-3)
			(m-2-1) edge node [left]  {$f'_j~~$} (m-1-2)
			(m-2-3) edge node [right] {$~~f_j$} (m-1-2);
		\end{tikzpicture}
		\\ 	
		commutes for every object $d_j$ in $D$. 
		This limiting cone, when it exists, is said to have the \textit{universal property} 
		with respect to $D$ -cones.
		A limit for 
		diagram $D$ is unique up to isomorphism.

			\end{definition} 
		
	\begin{definition}\label{colim_defn}
By duality we define {\it co}-{\it cone} $\left\{f_\iota : d_\iota \to c\right\}$ for diagram $D$ to consist of an 
object $c$, and arrows $f_\iota : d_\iota \to c$ for each object $d_\iota$ in $D$. A {\it colimit} for $D$ is 
then a co-cone $\left\{f_\iota : d_\iota \to c\right\}$ with the {\it co-universal property} that for any 
other co-cone $\left\{f'_\iota : d_\iota \to c'\right\}$ there is exactly one arrow $f:c\to c'$ such 
	\bean
	\begin{tikzcd}
	& d_\iota \arrow[dl, "f_\iota"] \arrow[dr, "f'_\iota"]&\\
	c\arrow[rr, dashed, "f"] && c'
	\end{tikzcd}
	\eean
commutes for every $d_\iota$ in $D$. 
	\end{definition}		
			\begin{definition}\label{pull_back_defn}\cite{goldblatt:topoi}
	A \textit{pullback} of a pair $a \xrightarrow{f}c \xleftarrow{g} b$ of $\mathscr C$-arrows with a common codomain 
	is a limit in $\mathscr C$ for the diagram 
	\newline
	\begin{tikzpicture}
		\matrix (m) [matrix of math nodes,row sep=3em,column sep=4em,minimum width=2em]
		{
			& b  \\ 
			a	&  c \\};
		\path[-stealth]
		(m-1-2) edge node [right]  {$g$} (m-2-2)
		(m-2-1) edge node [above] {$f$} (m-2-2);
	\end{tikzpicture}
	\\ 	
	
\end{definition}

\begin{definition}\label{pullback_defn}\cite{milne:lec}
	For $X$ an object of a category $\mathscr C$, $\mathscr C/X$ denotes the category whose objects are the morphisms $U\to X$ in $\mathscr C$
	and whose arrows are the commutative diagrams 
	\newline
	\begin{tikzcd}
		U \arrow[rr]\arrow[rd] && U'\arrow[ld]\\
		& X &
	\end{tikzcd}
	\\
	A morphism $U \to U'$ making this diagram commute is also referred to as an $X$-\textit{morphism}. The \textit{fibre product} or the \textit{pullback} $U_1\times_X U_2$ of morphisms $\varphi_1: U_1\to X$,  $\varphi_2: U_2\to X$ in $\mathscr C$ is their product in the category $\mathscr C/X$. Thus, there is a commutative diagram 
	\newline
	\begin{tikzcd}
		U_1\times_X U_2 \arrow[d] \arrow[r]
		& U_2 \arrow[d, "\varphi_2"]\\
		U_1\arrow[r, "\varphi_1"]	& X
	\end{tikzcd}
	\\
	\\
	having the obvious universal property. 
	
\end{definition} 			
\subsection{Functors}
\begin{definition}\label{functor_defn}\cite{goldblatt:topoi}
	A \textit{functor} $F$ from category $\mathscr{C}$ to category $\mathscr{D}$ is a function that assigns 
	\begin{enumerate}
		\item [(i)]
		to each $\mathscr{C}$-object $a$, a $\mathscr{D}$-object $F(a)$; 
		\item[(ii)] to each $\mathscr{C}$-arrow $f:a \to b$ a $\mathscr{D}$-arrow $F(f): F(a) \to F(b)$, 
		such that 
		\begin{enumerate}
			\item[(a)]  $F\left(\mathbb 1_a\right) = \mathbb 1_{F\left(a\right)}$ for all  $\mathscr{C}$-objects $a$, i.e. the identity arrow on $a$ is assigned 
			the identity on $F\left(a\right)$,
			\item[(b)]  $F\left(g\circ f\right)=F\left(g\right)\circ F\left( f\right) $, whenever $g \circ f$ is defined. 
			This last condition states that the $F$-image of a composite of two arrows 
			is the composite of their $F$-images.
		\end{enumerate}

	\end{enumerate}
	We write $F:\mathscr{C}\to \mathscr{D}$ or  $\mathscr{C}\xrightarrow{F} \mathscr{D}$ to indicate that $F$ is a 
	functor from $
	\mathscr{C}$ to $\mathscr{D}$. Briefly then a functor is a transformation that 
	"preserves" dom's, cod's, identities and composites. 
\end{definition}
If $a$ and $b$ are  $\mathscr{C}$-objects then a functor $\mathscr{C}\xrightarrow{F} \mathscr{D}$ yields a map
\be\label{f_ab_funct_eqn}
F_{a,b}:\mathscr{C}\left(a, b \right)  \to \mathscr{D}\left( F\left(a\right), F\left(b\right)\right)  
\ee
(cf. Notation \ref{category_not}).

\subsection{Limits of functors}\label{limits_of_functors_sec}	
\paragraph{} Here I follow to \cite{cat:work}. Given categories $C, J$,
and the diagonal functor $\Dl : C \to C^J$ , a {\it limit} for a functor $F: J \to C$.
is a universal arrow $\left\langle r, \nu\right\rangle$  from $\Dl$ to $F$. It consists of an object $r$ of $C$,
usually written $r = \varinjlim F$ or $\lim~F$ and called the {\it limit} object (the
"inverse limit" or "projective limit") of the functor $F$, together with a
natural transformation $\nu : \Dl r \xrightarrow{\cdot } F$ which is universal among natural
transformations $\tau : \Dl c \xrightarrow{\cdot } F$, for objects $c$ of $C$. Since $\Dl c \xrightarrow{\cdot } F$ is the
functor constantly $c$, this natural transformation $\tau$ consists of one arrow
$\tau_j : c\to F_j$ of C for each object $j$ of $J$ such that for every arrow $u: i \to j$ 
of $J$ one has $T_j  = Fu \circ \tau_i$. We may call $\tau : c \xrightarrow{\cdot } F$ a {\it cone} to the base $F$ from
the vertex $c$. (We say "cone to the base $F$" rather than "cocone"). The
universal property of $\nu$ is this: It is a cone to the base $F$ from the vertex
$\varprojlim F$; for any cone $\tau$ to $F$ from an object $c$, there is a unique arrow $t : c \to \varprojlim F$ such that $\tau_i = v_i t$ for all $i$. The situation may be pictured as
\bean
\begin{tikzcd}
	c \arrow[r, "t"] \arrow[d, "\tau_i"] \arrow[rd, "\tau_j"]&  \varprojlim F\arrow[d, "\nu_j"] \arrow[ld, "\nu_i"]\\
	F_i \arrow[r, "Fu"]  & F_j
\end{tikzcd}.
\eean

\subsection{Equivalence relations}

\begin{definition}\label{equivalenceelation_defn}\cite{spanier:at}
	An \textit{equivalence relation} in a set $A$ is a relation $\sim$ between elements of $A$ 
	which is \textit{reflexive} (that is, $a\sim a$ for all $a \in A$), \textit{symmetric} (that is, $a\sim a'$ 
	implies $a'\sim a$ for $a, a'\in A$), and \textit{transitive} (that is, $a\sim a'$ and $a'\sim a''$
	imply $a \sim a''$ for $a, a', a'' \in A$). The \textit{equivalence class} of $a \in A$ with 
	respect to $\sim$ is the subset $\left\{\left. a'\in A\right| a'\sim a \right\}$. The set of all equivalence classes 
	of elements of $A$ with respect to ~ is denoted by $A/\sim$ and is called a 	\textit{quotient set} of $A$. There is a \textit{projection map} $A \to A/\sim$ which sends $a\in A$ to
	its equivalence class. 
\end{definition}

\subsection{Directed sets}

\begin{definition}\label{directed_set_defn}\cite{engelking:general_topology}
	Let $\La$ be a set and $\le$ is relation on $\La$. We say that $\le$ \textit{directs} $\La$ or $\La$ is \textit{directed} by $\le$, if $\le$ has following properties:
	\begin{enumerate}
		\item [(a)] If $\la \le \mu$ and $\mu \le \nu$, then $\la \le \nu$,
		\item[(b)] For every $\la \in \La$, $\la \le \la$,
		\item[(c)] For any $\mu,\nu \in \La$ there exists a $\la \in \La$ such that $\mu \le \la$ and $\nu \le \la$.
	\end{enumerate}
\end{definition}
\begin{definition}\label{cofinal_defn}\cite{engelking:general_topology}
	A subset $\Xi \subset \La$ is said to be \textit{cofinal} in $\La$ if for every $\la \in \La$ there is $\chi \in \Xi$ such that $\la \le \chi$. 
\end{definition}

\subsection{Natural transformations}\label{natural_transformation_sec}
\paragraph{}
Here I follow to \cite{goldblatt:topoi}.
Given two categories $\mathscr C$ and $\mathscr D$ we are going to construct a category, 
denoted $\mathrm{Funct}\left(\mathscr C, \mathscr D\right)$, or $\mathscr D^{\mathscr C}$, whose objects are the functors from $\mathscr C$ to $\mathscr D$. 
We need a definition of arrow from one functor to another. Let 
$F: \mathscr C\to \mathscr D$ and $G: \mathscr C\to \mathscr D$ be two functors. For any $\mathscr C$-object $a$ we define a $\mathscr D$-arrow $\tau_a : F\left(a\right)\to  G\left(a\right)$. We require that each  $\mathscr C$-arrow $f: a \to b$  gives rise to a diagram 
\begin{tikzpicture}
	\matrix (m) [matrix of math nodes,row sep=3em,column sep=4em,minimum width=2em]
	{
		a	& F\left(a \right)  &   G\left(a \right)\\ 
		b	& F\left(b \right)  &   G\left(b \right) \\};
	\path[-stealth]
	(m-1-1) edge node [left]  {$f$} (m-2-1)
	(m-1-2) edge node [above] {$\tau_a$} (m-1-3)
	(m-2-2) edge node [above]  {$\tau_b$} (m-2-3)
	(m-1-2) edge node [left]  {$F\left( f \right)$} (m-2-2)
	(m-1-3) edge node [left]  {$G\left( f \right)$} (m-2-3);
\end{tikzpicture}
\\ 	
that commutes.
In summary then, a \textit{natural transformation} from functor $F: \mathscr C\to \mathscr D$ and $G: \mathscr C\to \mathscr D$  to functor $F: \mathscr C\to \mathscr D$ and $G: \mathscr C\to \mathscr D$  is an assignment $\tau$ that provides, for each $\mathscr C$-object $\mathscr D$-arrow $\tau_a :F(a) \to G(a)$, such that for any $\mathscr C$-arrow $f:a\to b$, the above diagram commutes in $\mathscr D$, i.e. $\tau_b \circ F(f)= G(f)\circ \tau_a$. We use the symbolism $\tau: F\to G$, or $F \xrightarrow{\tau}G$, to denote that $\tau$ is a natural transformation from $F$ to $G$. The arrows $\tau_a$ are called the \textit{components} of $a$. Now if each component $\tau_a$ of $a$ is an iso arrow in $\mathscr D$ then  case we call $\tau$ a \textit{natural isomorphism}. Each $\tau_a: F(a)\to G(a)$ then has an inverse $\tau_a^{-1}: G(a) \to F(a)$, and these $\tau^{-1}_a$'s form the components of a natural isomorphism $\tau^{-1}: G \to F$. We denote natural isomorphism by $\tau: F \cong G$. 
\begin{example}
	The identity natural transformation $\mathbb{1}_F :F\to F$ assigns to each object $a$, the identity arrow $\mathbb{1}_{F(a)}:F(a)\to F(a)$. This is clearly a natural isomorphism. 
\end{example}

\begin{definition}\label{category_equivalence_definition}\cite{goldblatt:topoi}
	A functor $F: \mathscr C\to  \mathscr D$ is called an \textit{equivalence of categories} if there 
	is a functor $G: \mathscr D\to  \mathscr C$ such that there are natural isomorphisms $\tau : 1_{\mathscr C} \cong G \circ F$, and $\sigma : 1_{\mathscr D} \cong F \circ G$, from the identity functor on ${\mathscr C}$ to $ G \circ F$, and 
	from the identity functor on ${\mathscr D}$ to $ F \circ G$.
	
	Categories $\mathscr C$ and $\mathscr D$ are \textit{equivalent},  $\mathscr C$ and $\mathscr D$ when there exists an equivalence $F: \mathscr C\to  \mathscr D$ .
\end{definition}
\subsection{Algebraic $K$-theory}\label{algebraic_k_theory}
\paragraph{} Let $R$ be a ring with unit. To define $K_0\left(R \right)$  we consider the following equivalence
relation. We say that two finitely projective $R$-modules $P$ and $Q$ are equivalent
if and only if they are isomorphic, i.e. if there is an isomorphism of
$A$-modules $P$ and $Q$ This is clearly an equivalence relation.
We note $\overline{P}$ for the equivalence class of the projective $A$-module $P$ and
$Proj(A)$ for the set of all the equivalence classes.
\begin{definition}\cite{bass,isely:akt}
(Milnor) The projective module group $K_0\left(R \right)$ is the group
defined by generators and relations as follows. For each elements $P$ of
$Proj(A)$ we take a generator $\left[P\right]$ and for each pair $\left[P\right]$, $\left[Q\right]$ of generators
we define the relation
\bean
\left[P\right]+\left[Q\right]\bydef \left[P \oplus Q\right]
\eean
\end{definition}
\begin{proposition}\cite{bass,isely:akt}
Every element of $K_0\left(R\right)$  can be expressed by the formal
difference $\left[P\right]-\left[Q\right]$ of two generators.

\end{proposition}
\begin{definition}\cite{bass,isely:akt}
 Let $n \in \N$: A matrix in $GL_n\left(R \right) $ is called {\it elementary} if
it coincides with the identity matrix except for a single off-diagonal entry.
We note $E_n\left(A \right)$  the subgroup of $GL_n\left(A \right) $ generated by all the elementary
matrices.
\end{definition}
\begin{remark}\cite{bass,isely:akt}
Since $i_n : E_n\left( A\right) \subset E_{n+1}\left(R\right)$ we can embed $E_{n}\left( R\right)$ in $E_{n+1}\left( R\right)$ for all $n \in \N$.

\end{remark}

\begin{definition}\cite{bass,isely:akt}
We define $E(R) \bydef \bigcup_{n \in \N} E_n\left(R\right)$ 

\end{definition}
\begin{definition}\cite{bass,isely:akt}
(Whitehead) We define $K_1\left(A \right)$  by the quotient
\bean
K_1\left( R\right)\bydef GL\left(R \right)/E\left(R \right)
\eean
\end{definition}
\begin{remark}\label{ho_topem}
If $A$ is a $C^*$-algebra  then $K_0\left(A \right)\cong \pi_1\left(GL^{\mathrm{top}}\left(A \right)  \right)$. Using this fact one can prove that the functor $A \mapsto K_0\left(A \otimes \K \right)$ is homotopy invariant.  
\end{remark}
\begin{empt}\label{excision_empt}
	 Here I follow to \cite{suslin:exc}. The algebraic  $K$-theory groups $K_j\left(R \right)$ of a ring with unit $R$ are defined as a homotopy groups of $BGL\left(R \right)^+$, the Quillen plus construction applied to $BGL\left(R \right)$:
	 \be\label{kr_pi_eqn}
	 \begin{split}
	 K_j\left(R \right) \cong \pi_j\left( BGL\left(R \right)^+\right) \quad j \ge 1
	 \end{split}
	 \ee
	 If $A\subset R$ is a two-sided ideal, let $F\left(R, A \right)$ denote the homotopy of the map $ BGL\left(R \right)^+\to \overline{BGL}\left(R/A \right)^+$, where $ \overline {GL}\left(R \right)\bydef \im \left(GL\left(R \right)\to GL\left(R /A\right) \right)$. The relative $K$-groups $K_j\left(R, A \right), \quad j \ge 1$, are defined as homotopy groups of $F\left(R, A \right)$ so that one gets, with respect to a pair $\left(R. A \right)$, a functorial long exact sequence  
	 \be\label{kr_spi_eqn}
	 \begin{split}
	 	K_j\left(R \right)\to 	K_j\left(R, A \right)\to 	K_j\left(R/A \right)\to 	K_{j-1}\left(R, A \right) \to ...\\
	 	\to K_1\left(R \right)\to 	K_1\left(R, A \right)\to 	K_1\left(R/A \right).
	 \end{split}
	 \ee
	 Algebraic $K$-theory is extended to the category of rings by setting
	 \be\label{kr_npi_eqn}
	 \begin{split}
	 K_j\left( A\right) = K_j\left(\widetilde A, A \right) \quad j \ge 1,
	 \end{split}
	 \ee
where $	 \widetilde A= \Z \rtimes A$ denotes the ring obtained from $A$ by adjoining the unit.

A ring $A$ is said to satisfy {\it excision}  in algebraic $K$-theory if for every ring with unit $R$, which contains $A$ as a two-sided ideal, the canonical map $K_*\left( A\right) \to K_*\left(R, A\right)$ is an isomorphism.
\end{empt}

	\begin{theorem}\label{fin_gen_thm}\cite{kasch:mr}
	The module $M$ over an (unital) ring $R$ is finitely generated if and only if for every set $\left\{A_\la\right\}_{\la\in\La}$ of submodules $A_\la\hookto  M$ with 
	$$
	M=\sum_{\la \in \La} A_\la
	$$
	there is a finite subset $\La_0\subset \La$  such that 
	$$
	M = \sum_{\la \in \La_0} A_\la.
	$$
\end{theorem}
\subsection{Adjoint functors}\label{adjoint_functor_sec}
\paragraph*{}

If both $\mathscr C$ and $\mathscr D$ are categories then a pair of functors $\mathscr C \substack{ \xrightarrow{F}\\\xleftarrow[G]{}}\mathscr D$ are \textit{adjoint functors} if for any $\mathscr C$-object $A$ and $\mathscr D$-object $B$ there is a natural bijection
$$
\Hom_{\mathscr D}\left(F\left(A \right), B  \right) \cong \Hom_{\mathscr C}\left(A, G\left(B \right)  \right).
$$
A pair of {adjoint functors} can be defined via counit–unit adjunction, i.e. there are natural transformations
\bean
\eps : FG \to 1_{\mathscr C},\\
\eta : 1_{\mathscr  D} \to  GF 
\eean
respectively called the \textit{counit} and the \textit{unit} of the adjunction (terminology from universal algebra), such that the compositions
\be\label{id_thens_eqn}
\begin{split}
	F{\xrightarrow {\;F\eta \;}}FGF{\xrightarrow {\;\varepsilon F\,}}F,\\
	G{\xrightarrow {\;\eta G\;}}GFG{\xrightarrow {\;G\varepsilon \,}}G
\end{split}
\ee
are the identity transformations $1_F$ and $1_G$ on $F$ and $G$ respectively. $F$ is said to be \textit{left adjoint} $G$, and $G$ is \textit{right adjoint} $F$.
 
\subsection{Toposes}

\begin{definition}\label{topos_defn}\cite{johnstone:topos}
	A category $\mathcal E$ is called an (elementary) \textit{topos} if
	\begin{enumerate}
		\item [(a)]  $\mathcal E$ has all finite limits (Equivalently,  $\mathcal E$ has pullbacks and a terminal object).
		\item[(b)]  $\mathcal E$ is \textit{cartesian closed}, i.e. for each object $X$ we have an \textit{exponential functor} $\left(-\right)^X$ which is right adjoint to the functor $\left(-\right)\times X$.
		\item[(c)] $\mathcal E$ has a \textit{subobject classifier}, i.e. an object $\Om$ and a morphism $1 \xrightarrow{t} \Om$ (called "true") such that for each monomorphism $Y \xhookrightarrow{\sigma} X$ there is a unique $\phi_\sigma: X \to \Om$ (the \textit{classifying map} of $\sigma$) making 
		\newline
		\begin{tikzcd}
			Y \arrow[d, hook, "\sigma"] \arrow[r]
			& 1 \arrow[d, "t"]\\
			X\arrow[r, "\phi_\sigma"]	& \Om
		\end{tikzcd}
		\\
		a pullback diagram.
	\end{enumerate}
\end{definition}
\begin{example}\label{sets_exm}\cite{johnstone:topos}
	A  category  of sets and functions $\mathscr S$ (cf. Definition \ref{sets_cat_defn}) of sets is a topos.
\end{example}
\begin{definition}\label{geometric_morphism_defn}\cite{johnstone:topos}
	If both $\mathscr E$, $\mathscr F$ are toposes then \textit{geometric morphism} $\mathscr F\xrightarrow{f}\mathscr E$ consists of a pair of functors $\mathscr F\xrightarrow{f_*}\mathscr E$  and $\mathscr E\xrightarrow{f^*}\mathscr F$ (called the \textit{direct} and \textit{inverse images}) such that $f^*$ is left adjoint to $f_*$ and $f^*$ is left exact. 
\end{definition}

	\section{Grothendieck topology and sheaves}

\subsection{Sites and sheaves}\label{appendix_etale_sec}

\begin{definition}\label{pretopology_defn}\cite{johnstone:topos,milne:lec}
	Let  $\mathscr C$  be a category with pullbacks. Suppose that  for each object $U$ of $\mathscr C$ there are  distinguished sets of families of morphisms $\left\{U_\iota \to U\right\}_{\iota\in I}$, called the \textit{coverings} of $U$, 
	satisfying the following axioms: 
	\begin{enumerate}
		\item[(a)] for any covering $\left\{U_\iota \to U\right\}_{\iota\in I}$ and any morphism $U \to V$ in $\mathscr C$, the fibre products 
		$U_\iota\times_U V$ exist, and $\left\{U_\iota\times_U V \to V\right\}_{\iota\in I}$ is a covering of $V$; 
		\item[(b)] if $\left\{U_\iota \to U \right\}_{\iota\in I}$ is a covering of $U$, and if for each $\iota \in I$, $\left\{V_{\iota j} \to U_\iota  \right\}_{j \in I_\iota}$	is a 
		covering of $U_\iota$, then the family $\left\{V_{\iota j }\to U\right\}_{\iota j}$ is a covering of $U$; 
		\item[(c)] for any $U$ in $\mathscr C$, the family $\left\{U\xrightarrow{\Id}U \right\}$ consisting of a single map is a covering of $U$. 
	\end{enumerate}
	The system of coverings is then called a \textit{Grothendieck pretopology}.
\end{definition}
\begin{example}\label{top_gro_exm}
	If $\sX$ is a topological space then one has a category of open subsets and their inclusions. If both $\sU_1, \sU_2 \subset \sX$ are open subsets then one can define a fibre product (cf. Definition \ref{pullback_defn}) by the following way
	$$
	\sU_1\times_\sX \sU_2 \bydef \sU_1 \cap \sU_2.	
	$$
	If $\sU \subset \sX$ is an open subset then we assume that a family  $\left\{\sU_\iota \subset \sU\right\}_{\iota\in I}$, is a covering of $\sU$ (cf. Definition \ref{pretopology_defn}) if and only if  $\sU = \cup_{\iota\in I}\sU_\iota$. This system of coverings is a specialization of Grothendieck topology. 
\end{example}
\begin{definition}\label{top_gro_definition}
	In the situation of the Example \ref{top_gro_exm} one has a site $\mathbf T_{\sX}$ \textit{arising} from $\sX$.
\end{definition}

\begin{definition}\label{etale_presheaf_defn}\cite{milne:lec}
	A \textit{presheaf of sets} on a category $\mathscr C$
	is a contravariant functor $\mathscr F$ from $\mathscr C$ to the category of sets. Thus,  $\mathscr F$
	to each object $U$ in $\mathscr C$ 
	attaches a set $\mathscr F\left(U \right)$ , and to each morphism $\varphi: U \to V$
	in $\mathscr C$, a map $\mathscr F\left(\varphi\right):\mathscr F\left(V \right)\to \mathscr F\left(U \right)$. Note that the notion of a presheaf on $\mathscr C$ 
	does not depend on the 
	coverings. We sometimes denote $\mathscr F\left(\varphi\right)$ by $a \mapsto a|_U$.
	
\end{definition}
Similarly, a presheaf of (Abelian) groups or rings for the Grothendieck pretopology is a contravariant functor from
$\mathscr C$ to the category of (Abelian) groups or rings.
\begin{definition}\label{etale_sheaf_defn}\cite{milne:lec}
	A \textit{sheaf} for the Grothendieck pretopology
	is a presheaf $\mathscr F$ 
	that satisfies the sheaf condition, that is a sequence 
	\be\label{etale_sheaf_eqn}
	\mathscr F \left(U\right) \to \prod_{\iota \in I} \mathscr F \left(U_\iota \right)\rightrightarrows  \prod_{\iota, j \in I\times I} \mathscr F \left(U_\iota \times_U U_j\right)
	\ee
	is exact for every covering $\left\{U_\iota \to U\right\}$. Thus $\mathscr F$ 
	is a sheaf if the map 
	$$
	f \mapsto \left\{f|_{U_\iota }\right\}:\mathscr F \left(U\right) \to \prod_{\iota \in I} \mathscr F \left(U_\iota \right)
	$$
	identifies $\mathscr F\left( U\right)$ with the subset of the product consisting of families $\left\{f|_{U_\iota }\right\}$ such that 
	$$
	f_\iota|_{U_\iota \times_U U_j}= f_j|_{U_\iota \times_U U_j}
	$$
	for all $\iota, j \in I$. 
	When $\mathbf T$ 
	is the site arising from a topological space then these definitions coincide with the 
	topological ones. 
	
\end{definition}
\begin{empt}\cite{johnstone:topos}
	Two different Grothendieck pretopologies can give the same sheaves. To remove this disadvantage we consider those families which are "saturated" in the sense that $\left(V \xrightarrow{\a}U\right) \in R$ implies $\left(W\xrightarrow{\a\bt}U\right)\in \R$ for any $W \xrightarrow{\bt} V$. Such a family is called a \textit{sieve}.
\end{empt}
\begin{defn}\label{grotendieck_topolgy_defn}\cite{johnstone:topos}
	Let $\mathscr C$ be a small category. A \textit{Grothendieck topology} on $\mathscr C$ is defined by specifying, for each object $U$ of $\mathscr C$ a set $J\left(U\right)$ of {sieves} called \textit{covering sieves} of the topology such that
	\begin{enumerate}
		\item [(a)] For any $U$, the maximal sieve $\left\{\a | \text{codomain}(a)= U\right\}$ is in $J\left(U \right)$.
		\item[(b)] if $R \in J\left( U\right)$ and $V\xrightarrow{f} U$ is a morphism of $\mathscr C$ then the sieve
		$$
		f^*\left( R\right) \bydef \left\{W\xrightarrow{a}V| f\a \in R \right\}
		$$
		is in $J\left(V\right)$.
		\item[(c)] If $R \in J\left(U \right)$ is sieve on $U$ such that for each $V\xrightarrow{f}U$ in $R$ we have $f^*\left( S\right)\in J\left( V\right)$, then $S\in J\left( U\right)$.    
		
	\end{enumerate}
	A category $\mathscr C$  together with 
	the Grothendieck  topology is called a \textit{site} $\mathbf T\bydef \left( \mathscr C, J\right)$. Then $\mathrm{Cat}\left( \mathbf T\right)\bydef \mathscr C$ denotes the underlying category. 
	
\end{defn}
\begin{rem}\label{pretopology_to_topology_rem}\cite{johnstone:topos}
	If we are given a pretopology $P$ on  $\mathscr C$, we can replace it be a topology having same sheaves, by defining a sieve to be $J$-covering if and only if it contains $P$-covering family.
\end{rem}
\begin{definition}\label{forget_sheaf_defn}
	There are categories $\mathbf{PreSh}\left(\mathbf T \right)$, $ \mathbf{Sh}\left(\mathbf T\right)$ of presheaves an sheaves. Moreover there is a \textit{forgetful functor} $\mathfrak{Forget} : \mathbf{Sh}\left(\mathbf T \right)\to \mathbf{PreSh}\left(\mathbf T \right)$.
\end{definition}

\begin{statement}\label{associated_sheaf_stmnt}\cite{johnstone:topos}
	There is an adjoint $\mathfrak{Ass} : \mathbf{PreSh}\left(\mathbf T  \right)\to \mathbf{Sh}\left(\mathbf T \right)$ to the  forgetful functor $\mathfrak{Forget} : \mathbf{Sh}\left(\mathbf T  \right)\to \mathbf{PreSh}\left(\mathbf T \right)$ (cf. Definition \ref{forget_sheaf_defn}).
\end{statement}	
\begin{defn}\label{associated_sheaf_defn}\cite{johnstone:topos}
	The given by the Statement  \ref{associated_sheaf_stmnt} functor $\mathfrak{Ass} : \mathbf{PreSh}\left(\mathbf T  \right)\to \mathbf{Sh}\left(\mathbf T \right)$  is said to be an  \textit{associated sheaf functor}.
\end{defn}
\begin{statement}\label{site_enough_stmt}\cite{johnstone:topos}
	If $\mathbf T$ is a site then a category $\mathbf {Ab}\left(\mathbf T \right)$ of sheaves of Abelian groups has enough injectives.
\end{statement}
\begin{definition}\label{grothenidieck_topos_defn}\cite{johnstone:topos}
	If $\mathbf T$ is a site then  a category  $\mathbf{Sh}\left( \mathbf T \right)$ of sheaves of sets is a \textit{Grothendieck topos}.
\end{definition}
\begin{definition}\label{grothendiek_topos_defn}\cite{johnstone:topos}
	A category $\mathscr E$ is said to be a \textit{Grothendieck topos} if there exist a site $\mathbf T$ such that $\mathscr E$ is equivalent to a category  $\mathbf{Sh}\left(\mathbf T  \right)$ of sheaves of sets on $\mathbf T$.
\end{definition}

\begin{proposition}\label{grothenidieck_topos_prop}\cite{johnstone:topos}
	Any {Grothendieck topos} is a topos in the sense of the Definition \ref{topos_defn}.
\end{proposition}

\begin{definition}\label{constant_presheaf_defn}\cite{johnstone:topos}
	If $X$ is an object of $\mathrm{Cat}\left( \mathbf T\right)$ (cf. Definition \ref{pretopology_defn}    then for any set $F$ there is a \textit{constant presheaf} $\mathscr P$ such that $\mathscr P\left( U\right) \bydef F$ and $\mathscr P\left( f\right) \bydef \Id_F$ for all $\mathrm{Cat}\left( \mathbf T\right)/X$-morphism $f: U\to V$.
	We use a following notation
	\be\label{constant_presheaf_eqn}
	F_X \bydef \mathscr P.
	\ee
\end{definition}
\subsection{Localic topoi}
\begin{definition}\label{frame_defn}\cite{topos:intro}
	A \textit{frame} is partially  ordered set with all joins and all finite meets which satisfies the infinite distributive law:
	$$
	x \wedge \left( \bigvee_\iota y_\iota \right) = \bigvee_\iota\left(x\wedge   y_\iota \right) 
	$$
	A \textit{frame homomorphism} 
	is a function which preserves finite meets and arbitrary joins. Frames and frame homomorphisms form a category $\mathbf{Frm}$.
\end{definition}
\begin{definition}\label{locale_defn}\cite{topos:intro}
	The category $\mathbf{ Locale}$ of \textit{locales} is the opposite of the category of frames
	$$
	\mathbf{ Locale} \bydef \mathbf{Frm}^\text{op}.
	$$
	If $X$ is  a locale  denote by $\mathcal O\left( X\right)$ the corresponding  frame.
\end{definition}
The definition  of sheaves on a topological space depends
only on the lattice of open sets of that space, and so extends at once to
define sheaves on a locale $X$. Thus an "open" object $U$ is underlying category is said to be
covered by a family $\left\{U_\iota \right\}$ of "opens" of $X$ with each $U_\iota \le U$  if and only if $U= \bigvee_\iota V_\iota$ thus
\be\label{localic_site}
\left\{U_\iota\subset U \right\} \text{ covers } U \text{ if and only if } U= \bigvee_\iota V_\iota
\ee
The equation \eqref{localic_site} yields a Grothendieck pretopology, so one obtains a Grothendieck topology. In result we obtain a site and a Grothendieck topos.

\begin{definition}\label{localic_topos_defn}\cite{topos:intro}
	A Grothendieck topos obtained from a locale is said to be \textit{localic}.
\end{definition}
\begin{theorem}\label{locale_thm}\cite{topos:intro}
	For a Grothendieck topos $\mathscr E$ the following are equivalent:
	\begin{enumerate}
		\item [(i)] $\mathscr E$ is localic,
		\item[(ii)] there exists a site for $\mathscr E$ with a patrially ordered set as underlying category,
		\item[(iii)] $\mathcal{E}$ is generated by the subobjects of its terminal object 1.
	\end{enumerate}
\end{theorem}

\begin{proposition}\label{locale_funct_prop}\cite{topos:intro}
	The natural  functor $X \mapsto \mathbf{ Sh}(X)$ from locales to topoi induces
	for any two locales $X$ and $Y$ an equivalence of categories
	$$
	\mathbf{Maps}(X, Y) \xrightarrow{\cong} \Hom\left( \mathbf{Sh}(X), \mathbf{Sh}(Y)\right) 
	$$
	of locales and localic toposes.
	
\end{proposition}

\begin{definition}\label{loc_emp_sup_defn}\cite{topos:intro}
	A map $X \to Y$ of locales is an \textit{embedding}
	(respectively a \textit{surjection}) if and only if the corresponding morphism of frames
	$\mathcal O \left(Y \right)  \to \mathcal O \left(X \right)$ is surjective (resp. is injective).
\end{definition}

\begin{definition}\label{open_locale_mor_defn}\cite{topos:intro}
	A map $f : X \to Y$ of locales is \textit{open} if and only if the corresponding
	map $f^{-1}: \mathcal O\left( Y\right) \to  \mathcal O\left( X\right)$ of posets has a left adjoint $f_!$ which
	satisfies the Frobenius identity
	\be\label{open_locale_mor_eqn}
	f_!(V \wedge  f^{-1} U) = f_!(V) \wedge U\quad  \text{ for all }\quad U\in O\left( X\right), \quad V \in O\left( Y\right).
	\ee
\end{definition}
\subsection{Spatial locales}\label{spatial_toposes_sec}
\paragraph*{}	
Any topological space $\sX$ yields a Grothendieck topology, so 	a Grothendieck topology is a generalization of the classical topology. A topos is \textit{spatial} if it corresponds to a classical topology.

\begin{definition}\label{locale_point_defn}\cite{topos:intro}
	A	\textit{point} $p: 1\to X$ of a locale $X$ is the same thing
	as a frame morphism to the initial frame
	\be\label{locale_point_eqn}
	p^{-1}:\mathcal{O}\left( X \right) \to \left\{0, 1\right\}.
	\ee
\end{definition}
The map $p^{-1}$ can be conveniently presented in terms of its \textit{kernel} $K \bydef
\left\{U \left.| p^{-1} \left( U \right)  = 0\right.\right\}$, a subset of $X$ with the following properties
\be\label{k_p_eqn}
\begin{split}
	1 \notin K,\\
	U \wedge V \in K\quad  \Leftrightarrow \quad U \in K \text{ or } V \in K,\\
	\bigvee V_\iota \in K \quad  \Leftrightarrow \quad \forall \iota \quad V_\iota \in K
\end{split}
\ee
If
\be\label{p_p_eqn}
P \bydef  \bigvee K \bydef \left( \bigvee_{U\in K} U\right) \in \mathcal O\left( X\right) 
\ee
then
\bea\label{prime_eqn}
U\wedge V = P  \quad  \Leftrightarrow \quad U = P \text{ or } V = P,\\
\label{pprime_eqn}
1 \neq P
\eea
A satisfying to \eqref{prime_eqn} element $P\in \mathcal O\left(X \right)$ is said to be \textit{prime}, if it also satisfies to  \eqref{pprime_eqn} then we say that $P$ is. \textit{properly prime}.
\begin{lemma}\label{point_lem}\cite{topos:intro}
	The points of a locale $X$ may be described in any of the
	following three equivalent ways:
	\begin{enumerate}
		\item[(i)] as maps of locales $p: 1 \to X$; i.e., as frame morphisms
		$p^{-1}:\mathcal{O}\left( X \right) \to \left\{0, 1\right\}$;
		\item[(ii)] as subsets $K \subset \mathcal{O}\left( X \right)$ satisfying the three conditions of \eqref{k_p_eqn};
		\item[(iii)] as proper prime elements $P \in  \mathcal{O}\left( X \right)$); i.e., as elements $P\in \mathcal{O}\left( X \right)$ satisfying \eqref{prime_eqn} and \eqref{pprime_eqn}.
	\end{enumerate}
\end{lemma}



This set of points carries a natural
topology for which the open sets are the sets of the form
$$
\mathrm{pt}(U) = \left\lbrace p \in \mathrm{pt}(X) \left| p^{-1}\left( U\right)  = 1\right.\right\rbrace \subset \mathrm{pt}(X)\quad \text{ for some }\quad U \in \mathcal O\left(X\right)
$$
Similarly to a map $X\to Y$ of locales one can construct a continuous map $\mathrm{pt}(X)\to \mathrm{pt}(Y)$ of topological spaces. Thus there is a functor 
\be\label{point_f_eqn}
\mathrm{pt}: \mathbf{Locales}\to \mathbf{Spaces}
\ee
from the category of locales to the category of topological spaces (cf. \cite{topos:intro} for details).
\begin{theorem}
The functor $\mathrm{pt}: \mathbf{Locales}\to \mathbf{Spaces}$ is right adjoint
to the functor $\mathrm{Loc} : \mathbf{Spaces} \to \mathbf{Locales}$.
\end{theorem}

\begin{definition}\label{sober_defn}\cite{topos:intro}
A topological space $\sX$ is said to be \textit{sober} if and only is for any
open subset $\sU\subset \sX$ such that
\begin{enumerate}
	\item[(a)] $\sU \neq \sX$,
	\item[(b)] $\sV \cap \mathcal W \subset \sU \Rightarrow \sV \subset \sU$ or $\mathcal W\subset \sU$ (all open $\sV , \mathcal W\subset \sU$,
\end{enumerate}
there is a unique point $x \in \sX$ with   $\sU = \sX \setminus \overline{\{x\}}$ where  $\overline{\{x\}}$ is the  closure of the one-point set $\{x\}$.
\end{definition}
\begin{remark}\label{sober_rem}\cite{topos:intro}
A closed	subset $F \subset \sX$ is called \textit{irreducible} if it can not be written as the union of two smaller closed subsets; that is, whenever $F_1$ and $F_2$ are closed sets
with $F= F_1 \cup F_2$, then $F_1 = F$ or $F_2 = F$. Clearly, if $x$ is a point of $\sX$,
then a closure of $\{x\}$ is an irreducible closed set. Thus $\sX$ is sober if and only if every nonempty
irreducible closed set is the closure of a unique point.
\end{remark}

\begin{proposition}\label{sober_prop}\cite{topos:intro}
For any topological space $\sX$, the following are equivalent: 
\begin{enumerate}
	\item[(i)] $\sX$ is sober;
	\item[(ii)] the unit of the adjunction $\sX \mapsto \mathrm{pt}\circ \mathrm{Loc}\left( \sX\right)$  is a homeomorphism;
	\item[(iii)] there is a homeomorphism $\sX \cong \mathrm{pt}(X)$ for some locale $X$.
\end{enumerate}	
\end{proposition}
\begin{remark}\label{pointless_loc}\cite{topos:intro}
The functor $\mathrm{pt}$ from locales to spaces is by no means faithful. In
fact, there are many locales which have no points at all.
\end{remark}
\begin{definition}\label{enough_points_defn}\cite{topos:intro}
A locale $X$ is said to have \textit{enough points} (or to be \textit{spatial}) if elements of the lattice $\mathcal O(X)$ can be distinguished by points of $X$. Specifically, this means
that for any two distinct elements $U, V \in  \mathcal O(X)$, there exists a point
$p : 1 \to X$ such that $p^{-1}\left( U\right)\neq p^{-1}\left(V \right)$, i.e.
\be\label{enough_points_p_eqn}
\forall U, V\in \mathcal O\left(X \right) \quad U \neq V \quad \Rightarrow\quad \exists p \in \mathrm{pt} (X)\quad p^{-1}\left(U \right) \neq p^{-1}\left(V \right) 
\ee

Equivalently, $X$ has enough points if and only if, for any  $U, V \in  \mathcal O(X)$.
\be\label{enough_points_eqn}
\mathrm{pt}(U) = \mathrm{pt}(V) \quad \Rightarrow \quad U = V.
\ee
\end{definition}

\begin{proposition}\label{enough_points_prop}\cite{topos:intro}
For any locale X, the following are equivalent:
\begin{enumerate}
	\item[(i)] $X$ has enough points;
	\item[(ii)] the counit $\mathrm{Loc}\circ\mathrm{pt} X \to X$ is an isomorphism of locales;
	\item[(iii)] $X \cong \mathrm{Loc}\left(\sX \right)$  for some topological space $\sX$.
\end{enumerate}

\end{proposition}

\subsection{Cohomology of sheaves}\label{grothendieck_cohomology_section}
\paragraph*{}
Here I follow to \cite{milne:lec}. Let $\mathbf T$ be a site, and let $X$ be an object of $\mathrm{Cat}\left( \mathbf T\right)$. The functor
\bean
\Ga\left(X, \cdot \right) : \mathbf {Ab} \left(\mathbf T \right) \to \mathbf {Ab} ,\\
\mathscr F \mapsto \mathscr F\left(X \right) 
\eean
of global sections is left exact, and we define $H^r\left(X,~ \cdot ~\right)$ to be its $r^{\text{th}}$ right derived functor. Explicitly, for a sheaf $\mathscr F$, choose an injective resolution (cf. Statement \ref{site_enough_stmt})
$$
0 \to \mathscr F \to \mathscr I^0\to \mathscr I^1\to \mathscr I^2 \to ...
$$
and apply the functor $\Ga\left(X, \cdot \right)$ to obtain a complex
\be\label{complex_eqn}
0 \hookto\Ga\left(X,\ \mathscr I^0  \right) \xrightarrow{\dl_0} \Ga\left(X, \mathscr I^1  \right) \xrightarrow{\dl_1} \Ga\left(X, \mathscr I^2  \right) \to ... 
\ee
which is no longer exact (in general). For any $r\ge 0$ the complex \eqref{complex_eqn} yields the $r^{\text{th}}$ \textit{cohomology
group}  
\be\label{etale_coh_eqn}
H^r\left(X,\ \mathscr F \right)\bydef\begin{cases}
\ker \dl^0 & r = 0\\
\ker~ \dl_{r} / \im~ \dl_{r - 1}  & r > 0
\end{cases},
\ee
which does not depend on choice of injective resolution up to isomorphism.

\begin{proposition}\label{spectral_sequence_prop}\cite{johnstone:topos}
Let $\mathscr F \xrightarrow{f} \mathscr E$ be a geometric morphism (cf. Definition \ref{geometric_morphism_defn}) between Grothendieck toposes (cf. Definition \ref{grothendiek_topos_defn}). Then
	if $A$ is an Abelian group in $\mathscr E$ then we have a homomorphism $H^q\left(\mathscr E, A \right) \xrightarrow{} H^q \left(\mathscr F, f^*A\right)$ for each $q\ge 0$ which is functorial in $f$ and natural in $A$.
	\end{proposition}
	\begin{notation}\label{const_shef_coh_not}
If  $F$ is an Abelian group and $F_X$ is a {constant presheaf} (cf. Definition \ref{constant_presheaf_defn}) then we use the following notation
\be\label{const_shef_coh_eqn}
\forall r\ge 0\quad H^r\left(X, F \right) \bydef H^r\left(X, \mathfrak{Ass} \left( F_X\right)  \right).
\ee
where $\mathfrak{Ass}$ means the associated sheaf functor (cf. Definition \ref{associated_sheaf_defn}).
\end{notation}

\subsection{\v{C}ech cohomology}\label{presheaf_cohomology_sec}
\paragraph*{} Here I follow to \cite{milne:lec}. Let $\mathbf T$ be a site, and let $X$ be an object of $\mathrm{Cat}\left( \mathbf T\right)$. Let $\mathscr U \bydef \left\{U_\iota \to X\right\}_{\iota \in I}$ be  covering of $X$, and let $\mathscr P$ be a presheaf of Abelian groups. Define
\be\label{cech_eqn}
C^r\left( \mathscr U, \mathscr P\right)\bydef \prod_{\left(\iota_0,.... \iota_r\right) \in I^{r+1}} \mathscr P\left(U_{\iota_0,.... \iota_r} \right)\quad \text{where} \quad U_{\iota_0,.... \iota_r} \bydef U_{\iota_0}\times_X...\times_XU_{\iota_r}. 
\ee
For $s \in C^r\left( \mathscr U, \mathscr P\right)$ define $d^rs \in C^{r+1}\left( \mathscr U, \mathscr P\right)$ by the rule
$$
d^rs_{\iota_0,..., \iota_{r + 1}}\bydef \sum_{j = 0}^{r + 1}\left(-1 \right) \mathrm{res}_j\left(s_{\iota_0,..., \iota_{j-1},...,\iota_{j+1},..., \iota_{r + 1}} \right)  
$$
$\mathrm{res}_j$ is the restriction map corresponding to the projection map
$$
U_{\iota_0,..., \iota_{r + 1}}\mapsto U_{\iota_0,..., \iota_{j-1},...,\iota_{j+1},.., \iota_{r + 1}}.
$$
As in the classical case, one verifies by a straightforward calculation that
$$
C^\bullet\left( \mathscr U, \mathscr P\right) \bydef C^0\left( \mathscr U, \mathscr P\right) \to C^r\left( \mathscr U, \mathscr P\right) \xrightarrow{d_r} C^r\left( \mathscr U, \mathscr P\right)\to ...
$$
is a complex. Define
\be\label{cech_coh_eqn}
\check{H}^r\left( \mathscr U, \mathscr P\right) \bydef H^r\left(C^\bullet\left( \mathscr U, \mathscr P\right)   \right)= \begin{cases}
\ker d_0 & r = 0\\
\ker d_r / \im d_{r-1} & r > 0
\end{cases}. 
\ee
It is called the $r^{\text{th}}$ \textit{\v{C}ech cohomology group} of $\mathscr P$ relative to the covering $\mathscr U$.
Note that
$$
\check{H}^0\left( \mathscr U, \mathscr P\right)= \ker\left(\prod \mathscr P_\iota \rightrightarrows \prod \mathscr P_{\iota,j} \right) 
$$
Therefore, for a sheaf $\mathscr F$ one has
$$
\check{H}^0\left( \mathscr U, \mathscr F\right)=\Ga\left( \mathscr U, \mathscr F\right).
$$
A second covering $\mathscr V \bydef \left\{V_j\to X \right\}_{j \in J}$ of $X$ is called a \textit{refinement} of $\mathscr U$ if there is a
map $\tau : J\to I$ such that $V_j \to X$ factors through $U_{\tau_j}\to X$ for all $j \in J$ . The choice of
a $\tau$ and $X$-morphisms $\varphi_{j}: V_j \to U_{\tau_j}$ for each $j$ determines a map of complexes
\bean
\tau^\bullet: C^\bullet\left( \mathscr U, \mathscr P\right)\to C^\bullet\left( \mathscr V, \mathscr P\right),\\
\left( \tau^r s_{j_0,..., j_{r}} \right) \bydef s_{\tau j_0,..., \tau j_{r}}.
\eean
As in the classical case, one verifies that the map on cohomology groups
$$
\rho \left(\mathscr U, \mathscr V \right): \check{H}^r\left( \mathscr U, \mathscr P\right)\to \check{H}^r \left( \mathscr V, \mathscr P\right)
$$
is independent of all choices. We may pass to the limit over all coverings, and so obtain limits
\be\label{chech_eqn}
\check{H}^r\left(X, \mathscr P\right)\bydef \varinjlim_{\mathscr U}\check{H}^r\left( \mathscr U, \mathscr P\right).
\ee

\begin{defn}\label{chech_defn}\cite{bryl:loop,milne:lec}
The  \textit{\v{C}ech cohomology groups} are given by equation \eqref{chech_eqn}.
\end{defn}
These groups  have the following properties:
\begin{enumerate}
\item [(a)] $\check{H}^0\left(X, \mathscr F\right)= \Ga\left( X, \mathscr F\right)$ for any sheaf $\mathscr F$ on $X$;
\item[(b)] $\check{H}^r\left(X, \mathscr I\right)$, $r > 0$, for all injective sheaves $\mathscr I$
\end{enumerate}
(cf \cite{milne:lec} for details).
\begin{empt}
For any Abelian group $F$ one can define a constant presheaf $F_X$ of Abelian groups (cf. Definition \ref{constant_presheaf_defn}). 
We use a following notation
\be\label{etale_hom_a_eqn}
\check{H}^r\left(X, F\right)\bydef \check{H}^r\left(X, F_X\right).
\ee
\end{empt}

\section{Topology}
			\subsection{General topology}
	
	\begin{definition}\label{top_base_defn}\cite{munkres:topology}
		If $\sX$ is a set, a \textit{basis} for a topology on $\sX$ is a collection $\mathscr B$ of subsets of $\sX$
		(called \textit{basis elements}) such that:
		\begin{enumerate}
			\item [(a)]  for each $x\in \sX$, there is at least one basis element $B$ containing $x$,
			\item [(b)]  if $x$ belongs to the intersection of two basis elements $B_1$ and $B_2$, then there is a
			basis element $B_3$ containing $x$ such that $B_3 \subset B_1\cap B_2$.
		\end{enumerate}
		
		If  $\mathscr B$ satisfies these two conditions, then we define the \textit{topology}  $\mathscr T$ \textit{generated by}  $\mathscr B$ as
		follows: A subset $\sU$ of $\sX$ is said to be \textit{open} in $\sX$ (that is, to be an element of $\mathscr T$) if for
		each $x \in \sU$, there is a basis element $B \in \mathscr B$ such that $x \in B$ and $B \subset \sU$.
	\end{definition}
	\begin{remark}\label{top_baseem}It is proven that given by the Definition \ref{top_base_defn} collection $\mathscr T$ is a topology  (cf. \cite{munkres:topology}).
	\end{remark}
	\begin{definition}\label{top_subbasis_defn}\cite{munkres:topology}	Let $\sX$ be a topological space. A {\it subbasis } $\mathscr S$ for a topology on $\sX$ is a collection of open subsets of $\sX$ such that
		\begin{itemize}
			\item $$\sX = \bigcup_{\sU \in \mathscr S}\sU,$$
			\item the  collection $\mathscr T$ of all unions of finite intersections of elements of $\mathscr S$ is a basis of the topology og $\sX$.
		\end{itemize}
		The {\it topology  generated by the subbasis} $\mathscr S$ is defined to be the  collection $\mathscr T$ of all unions of finite intersections of elements of $\mathscr S$.
	\end{definition}
	
	\begin{theorem}\label{top_continuous_thm}\cite{kelley:gt}
		If $\sX$ and $\sY$ are topological space and $f$ is a function 
		on $\sX$ to $\sY$, then the following statements are equivalent. 
		\begin{enumerate}
			\item[ (a)] The function $f$ is continuous. 
			\item[ (b)] The inverse of each closed set is closed. 
			\item[ (c)] The inverse of each member of a subbasis for the topology for 
			$\sY$ is open. 
			\item[ (d)] For each $x$ in $\sX$ the inverse of every neighborhood of $f(x)$ is 
			a neighborhood of $x$. 
			\item[ (e)] For each $x$ in $\sX$ and each neighborhood $\sU$ of $f(x)$  there is a 
			neighborhood $\sV$ of $x$ such that $f(\sV)\subset \sU$. 
			\item[ (f)]For each net $S$  in $\sX$ which converges to a 
			point $s$, the composition $f \circ S$ converges to 
			f(s)). 
			\item[ (g)] For each subset $A$ of X the image of the closure is a subset of 
			the closure of the image; that is $f\left(\overline A \right)  \subset \overline {f\left( A\right)}$ 
			\item[ (h)] For each subset $B$ of $\sY$, $\overline{f^{-1}\left( B\right) }\subset f^{-1}\left( \overline B\right)$.
		\end{enumerate}
	\end{theorem}
	\begin{definition}\label{top_continuous_defn}
	A satisfying to the Theorem \ref{top_continuous_thm} map is {\it continuous}.
	\end{definition}
		
\begin{proposition}\label{top_final_prop}\cite{bourbaki_sp:gt}
		Let $\sX$ be a set, let $\left\{\sY_\iota\right\}_{\iota\in I}$ be a family of topological spaces, and for each $\iota\in I$ let $f_\iota$ be a mapping of $\sY_\iota$ into $\sX$. Let $\mathfrak{D}$ be a set of subsets $\sU$ of $\sX$ such that $f_\iota^{-1} \left(\sU \right)$ is open in $\sY_\iota$ for each $\iota \in I$; then  $\mathfrak{D}$ is a set of open sets in a topology $\mathfrak{T}$. In particular $\mathfrak{T}$ is the finest topology for which the mappings $f_\iota$ are continuous. In other words, if $g$ is mapping on $\sX$ into a topological space $\sZ$, then $g$ is continuous ($\sX$ carrying the topology $\mathfrak{T}$) if and only if each of the mappings $g\circ f_\iota$ is continuous.
	\end{proposition}
		
		\begin{definition}\label{top_final_defn}\cite{bourbaki_sp:gt}
			Under the hypotheses of Proposition \ref{top_final_prop} we say that the topology $\mathfrak{T}$ is \textit{final} (\textit{with respect to the family of maps} $\left\{f_\iota:\sY_\iota \to \sX\right\}_{\iota\in I}$).
		\end{definition}
		
		\begin{proposition}\label{top_init_prop}\cite{bourbaki_sp:gt}
			Let $\sX$ be a set, let $\left\{\sY_\iota\right\}_{\iota \in I}$ be a family of topological spaces, and for each $\iota \in I$ let $f_\iota$ be a mapping of $\sX$ into $\sY_\iota$.  Let $\mathfrak{S}$ be the set of subsets of $\sX$ of the form $f^{-1}_\iota\left(\sU_\iota \right)$ where $\iota \in I$, $\sU_\iota$ is open in $\sY_\iota$. Let $\mathfrak B$ be the set of finite intersections of sets of $\mathfrak S$. Then $\mathfrak B$ is a base of a topology $\mathfrak T$ and in particular is the coarsest topology on $\sX$ for which the mappings $f_\iota$ are continuous. More precisely of $g$ is a mapping of topological space into $\sY$, then is continuous at a point  $z \in \sZ$ ($\sX$ carrying the topology $\mathfrak T$) if and only if for each of the functions $f_\iota\circ g$ is continuous at $z$.
		\end{proposition}
		\begin{definition}\label{top_init_defn}\cite{bourbaki_sp:gt}
			Under the hypotheses of Proposition \ref{top_init_prop} we say that the topology $\mathfrak{T}$ is \textit{initial} (\textit{with respect to the family of maps} $\left\{f_\iota:\sX \to \sY_\iota\right\}_{\iota\in I}$).
		\end{definition}

\begin{remark}
	The Definition \ref{top_init_defn} is a specialization of initial object (cf. Definition \ref{initial_ob_defn}).
\end{remark}
		

	\begin{definition}\label{top_pointed_defn}
		If $\sX$ is a set (resp. topological space) and $x_0 \in \sX$ is a point then we say that the pair $\left(\sX, x_0 \right)$ is a \textit{pointed  set} (resp. \textit{pointed space}) (cf. \cite{spanier:at,switzer:at}). If both $\left(\sX, x_0 \right)$ and $\left(\sY, y_0 \right)$ are  pointed  spaces and $\varphi: \sX \to \sY$ is such that $\varphi\left(x_0\right)= y_0$ then we say that $\varphi$ is a \textit{pointed map}. We write
		\be\label{top_pointed_eqn}
		\varphi: \left(\sX, x_0 \right)\to\left(\sY, y_0 \right).
		\ee
		We say that $x_0$ is the \textit{base}-\textit{point}.
	\end{definition}
	\begin{definition}\label{top_hausdorff_defn}\cite{munkres:topology}
		A topological space $\sX$ is called a Hausdorff  space if for each pair $x_1, x_2$
		of distinct points of $\sX$, there exist neighbourhoods $\sU_1$, and $\sU_2$ of $x_1$ and , respectively,
		that are disjoint.
	\end{definition}	
	\begin{definition}\label{top_connected_defn}\cite{munkres:topology}
		Let $\sX$ be a topological space. A \textit{separation} of $\sX$ is a pair $\sU$, $\sV$ of disjoint
		nonempty open subsets of $\sX$ whose union is $\sX$. The space  $\sX$ is said to be \textit{connected}
		if there does not exist a separation of  $\sX$.
	\end{definition}
	\begin{remark}\label{top_connectedem}\cite{munkres:topology}
		A space $\sX$ is connected if and only if the only subsets of $\sX$ that are both
		open and closed in $\sX$ are the empty set and $\sX$ itself.
	\end{remark}
	\begin{theorem}\label{top_connected_closure_thm}\cite{munkres:topology}
		Let $A$ be a connected subspace of $\sX$, and let $\overline A$ be the closure of $A$, i.e. intersection of containing $A$ closed sets. If $A \subset B \subset \overline A$, then $B$ is also
		connected.
	\end{theorem}
	\begin{definition}\label{top_path_connected_defn}\cite{munkres:topology}
		Given points $x$ and $y$ of the space $\sX$, a \textit{path} in $\sX$ from $x$ to $y$ is a continuous map $f: \left[a, b\right]\to \sX$ of some closed interval in the real line into $\sX$, such
		that $f(a) = x$ and $f(b) = y$. A space $\sX$ is said to be \textit{path connected} if every pair of
		points of $\sX$ can be joined by a path in $\sX$.
	\end{definition}
	
	\begin{definition}\label{top_connected_component_defn}\cite{munkres:topology}
		Given $\sX$, define an equivalence relation on $\sX$ (cf. Definition \ref{equivalenceelation_defn} by setting $x\sim y$ if there
		is a connected subspace of $\sX$ containing both $x$ and $y$. The equivalence classes are
		called the \textit{components} (or the \textit{connected components}) of $\sX$.
	\end{definition}

	\begin{theorem}\label{top_conn_comp_thm}\cite{munkres:topology}
		The components of $\sX$ are connected disjoint subspaces of $\sX$ whose
		union is  $\sX$, such that each nonempty connected subspace of  $\sX$ intersects only one of
		them.
	\end{theorem}
	\begin{definition}\label{top_locally_connected_defn}\cite{munkres:topology}.
		A space $\mathcal X$ is said to be \textit{locally connected at} $x$ if for every neighbourhood $\mathcal U$ of $x$ there is a connected neighbourhood $\mathcal V$ of $x$  contained in $\mathcal U$. If $\mathcal X$ is locally connected at each of its points, it is said simply to be \textit{locally connected}.
	\end{definition}
	\begin{theorem}\label{top_loc_conn_thm}\cite{munkres:topology}
		A space $\mathcal X$ is locally connected if and only if for every open set $\mathcal U$, each component of $\mathcal U$ is open in $\mathcal X$.
	\end{theorem}
	\begin{proposition}\label{top_quasi_component_prop}\cite{bredon:topology_geometry}
		The statement "$d(p) = d(q)$" for every discrete valued map $d$
		on $\sX$" is an equivalence relation (cf. Definition \ref{equivalenceelation_defn}. 
	\end{proposition}
	\begin{definition}\label{top_quasi_component_defn}\cite{bredon:topology_geometry}
		The equivalence classes of the relation in Proposition \ref{top_quasi_component_prop}
		are called the \textit{quasi}-\textit{components} of $\sX$.
	\end{definition}
	\begin{proposition}\label{top_quasi_component_closed_prop}\cite{bredon:topology_geometry}
		Quasi-components of a space $\sX$ are closed. Each connected
		set is contained in a quasi-component. (In particular, each connected component is contained
		in a quasi-component.) Quasi-components are either equal or disjoint,
		and fill out $\sX$.
	\end{proposition}
	
	\begin{exercise}\label{top_loc_conn_exer}\cite{bredon:topology_geometry}
		Prove that If $\sX$ is locally connected, show that its components are open and equal its
		quasi-components.
	\end{exercise}

	\begin{definition}\label{top_compact_defn}\cite{munkres:topology}
		A space $\sX$ is said to be \textit{compact} if every open covering $\A$ of $\sX$ contains
		a finite subcollection that also covers $\sX$.
	\end{definition}	
	\begin{theorem}\label{top_compact_img_thm}\cite{munkres:topology}
		The image of a compact space under a continuous map is compact.
	\end{theorem}	
	
	\begin{definition}\label{top_locally_compact_defn}
		A space $\sX$ is said to be \textit{locally compact} at $x$ if there is some compact
		subspace $\sV$  of $\sX$ that contains a neighbourhood of $x$. If $\sX$ is locally compact at each of
		its points, $\sX$ is said simply to be \textit{locally compact}.
	\end{definition}
	\begin{empt}\cite{engelking:general_topology}
		Every set of cardinal numbers being well-ordered by $<$, the set of all cardinal numbers
		of the form $\left|B\right|$, where $B$ is a base for a topological space $\sX$, has a smallest element; this cardinal number is called the \textit{weight of the topological space} $\sX$ and is denoted by $w\left(\sX\right)$.
		
	\end{empt}
	
	
	\begin{definition}\label{top_net_defn}\cite{engelking:general_topology}
		A \textit{net in topological space} $\mathcal{ X}$ is an arbitrary function from a non-empty directed set $\La$ (cf. Definition \ref{directed_set_defn}) to the space $\mathcal{ X}$. Nets will be denoted by $S = \left\{x_\la \in \mathcal{ X}\right\}_{\la \in \La}$. 
	\end{definition}
	\begin{definition}\label{top_net_lim_defn}\cite{engelking:general_topology}
		A point $ x \in \mathcal{ X}$ is called a \textit{limit of a net} $S = \left\{x_\la \in \mathcal{ X}\right\}_{\la \in \La}$ if for every neighbourhood $\mathcal U$ of $x$  there exists $\la_0 \in \La$ such that $x_\la \in \mathcal U$ for every $\la \ge \la_0$; we say that $\left\{x_\la \right\}$ \textit{converges} to $x$. The limit will be denoted by $x = \lim S$, or $x = \lim_{\la \in \La} x_\la$, or $\lim x_\la$.
	\end{definition}
	\begin{remark}
		A net can converge to many points, however if $\mathcal{ X}$ is Hausdorff  then the net can have the unique limit.
	\end{remark}
	\begin{definition}\label{haudorff_defn}\cite{munkres:topology}
		A topological space $\sX$ is called a \textit{Hausdorff  space} if for each pair $x_1$, $x_2$ of distinct points of $\sX$, there exist neighbourhoods $\sU_1$, and $\sU_2$ of $x_1$  and $x_2$, respectively, that are disjoint.
	\end{definition}
	\begin{definition}\label{top_normal_defn}\cite{munkres:topology}
		Suppose that one-point sets are closed in $\sX$. Then $\sX$ is said to be  
		\textit{regular} if for each pair consisting of a point $x$ and a closed set $B$ disjoint from $x$, there
		exist disjoint open sets containing $x$ and $B$ respectively. The space $\sX$ is said to be
		\textit{normal} if for each pair $A$, $B$ of disjoint closed sets of $\sX$ there exist disjoint open sets containing $A$ and $B$ respectively.
	\end{definition}
	\begin{definition}\label{top_completelyegular_defn}\cite{munkres:topology}
		A topological space $\mathcal X$ is \textit{completely regular} if one-point sets are closed in  $\mathcal X$ and for each point $x_0$ and each closed $\mathcal \sY \subset \mathcal X$ not containing $x_0$, there is a continuous function $f: \mathcal X \to \left[0,1 \right]$ such that $f\left(x_0 \right)= 1$ and $f\left(\mathcal Y \right)= \left\{0\right\} $.  
	\end{definition}
	\begin{theorem}\label{top_scegular_normal_thm}\cite{munkres:topology}
		Every regular space with a countable basis is normal.
	\end{theorem}
	\begin{exercise}\label{top_completelyegular_exer}\cite{munkres:topology}
		Show that every locally compact, Hausdorff  space is completely regular.
	\end{exercise}
	
	\begin{thm}\label{comp_normal_thm}\cite{munkres:topology}
		Every compact Hausdorff  space is normal.
	\end{thm}
	
	
	\begin{thm}\label{urysohn_lem}\cite{munkres:topology}\textbf{ Urysohn lemma.}
		Let $\mathcal X$ be a normal space, let $\mathcal A$, $\mathcal B$ be disconnected closed subsets of $\mathcal X$. Let $\left[a,b\right]$ be a closed interval in the real line. Then there exist a continuous map $f: \mathcal X \to \left[a, b\right]$ such that $f(\mathcal A)=\{a\}$ and $f(\mathcal B)=\{b\}$.
	\end{thm}
	\begin{theorem}\label{tietze_ext_thm}\cite{munkres:topology} \textbf{Tietze extension theorem.} Let  $\mathcal X$ be a normal space; let  $\mathcal A$ be a closed subspace of  $\mathcal X$.
		\begin{enumerate}
			\item [(a)] Any continuous map of  $\mathcal A$ into the closed interval $[a,b]$ of $\R$ may be extended to a continuous map of all  $\mathcal X$ into $[a,b]$
			\item [(b)] Any continuous map of  $\mathcal A$ into $\R$ may be extended to a continuous map of all  $\mathcal X$ into $\R$.
		\end{enumerate}
		
	\end{theorem}
	\begin{definition}\label{top_onen_map_defn}\cite{munkres:topology}
		A map $f: \sX \to \sY$ is said to be an \textit{open map} if for every open set $\sU$ of $\sX$, the
		set $f\left(\sU \right)$  is open in $\sY$.
	\end{definition}	
	
	\begin{definition}\label{top_separable_defn}\cite{munkres:topology}
		A space having a countable  dense subset is said to be \textit{separable}.
	\end{definition}	
	\begin{defn}\label{top_support_defn}\cite{munkres:topology}
		If $\phi: \mathcal X \to \mathbb{C}$  is continuous then the \textit{support} of $\phi$ is defined to be the closure of the set $\phi^{-1}\left(\mathbb{\C}\setminus \{0\}\right)$ Thus if $x$ lies outside the support, there is some neighbourhood of $x$ on which $\phi$ vanishes. Denote by $\supp \phi$ the support of $\phi$.
	\end{defn}

	There are two equivalent definitions of $C_0\left(\mathcal{X}\right)$ and both of them are used in this book.
\begin{definition}\label{top_ind_lim_defn}\cite{candel:foliII}
		Let $\sX$ be a locally compact Hausdorff  space, and let $C_c\left(\sX\right)$ denote the	space of continuous, compactly supported functions on $\sX$. The natural	topology on $C_c\left(\sX\right)$ is the \textit{inductive limit topology}, defined as follows. A net	$\left\{f_\a\right\}$ in $C_c\left(\sX\right)$ converges to $f$ if there is a compact set $K\subset\sX$ containing	the supports of all $f_\a$ and $f$, and such that $f_\a$ converges uniformly to $f$ on$K$.
	\end{definition}
	\begin{defn}\label{top_loc_fin_defn}\cite{munkres:topology}
		An indexed family of sets $\left\{A_\al\right\}$ of topological space $\sX$ is said to be \textit{locally finite} if each point $x$ in $\sX$  has a neighbourhood that intersects for only finite many values of $\a$.
	\end{defn}
	\begin{defn}\label{top_part_of_unity_defn}\cite{munkres:topology}
		Let $\left\{\mathcal U_\alpha\in \mathcal X\right\}_{\alpha \in \mathscr A}$ be an indexed open covering of $\mathcal{X}$. An indexed family of functions 
		\begin{equation*}
			\phi_\alpha : \mathcal X \to \left[0,1\right]
		\end{equation*}
		is said to be a {\it partition of unity }, dominated by $\left\{\mathcal{U}_\alpha \right\}$, if:
		\begin{enumerate}
			\item $\phi_\alpha\left(\mathcal X \setminus \mathcal U_\alpha\right)= \{0\}$
			\item The family $\supp\phi_\alpha$ is locally finite.
			\item $\sum_{\alpha \in \mathscr A}\phi_\alpha\left(x\right)=1$ for any $x \in \mathcal X$.
		\end{enumerate}
	\end{defn}
	\begin{definition}\label{top_paracompact_defn}\cite{munkres:topology}
		A space $\sX$ is \textit{paracompact} if every open covering $\sX = \cup~\sU_{   \a}$ has a locally finite open refinement $\sX = \cup~\sV_{   \bt}$.
	\end{definition}	
	\begin{thm}\label{top_part_u_thm}\cite{munkres:topology}
		Let $\mathcal X$ be a paracompact Hausdorff  space; let $\left\{\mathcal U_\alpha\in \mathcal X\right\}_{\alpha \in \mathscr A}$ be an indexed open covering of $\mathcal{X}$. Then there exists a partition of unity, dominated by $\left\{\mathcal{U}_\alpha \right\}$.  
	\end{thm}
	
	\begin{definition}\label{top_compactification_defn}\cite{munkres:topology}
	If $\sY$ is a compact Hausdorff  space and $\sX$ is a proper subspace of $\sY$ whose
	closure equals $\sY$, then $\sY$ is said to be a \textit{compactification} of $\sX$. If $\sY\setminus\sX$ equals a single
	point, then $\sY$ is called the \textit{one-point compactification} of $\sX$. The following notation will be used
	\be\label{top_compactification_eqn}
	\sX^{+}\bydef \sY.
	\ee
	\end{definition}
	\begin{remark}
	It is shown  (cf. \cite{munkres:topology}) that $\sX$ has a one-point compactification $\sY$ if and only if $\sX$ is
	a locally compact Hausdorff  space that is not itself compact. We speak of $\sY$ as "the"
	one-point compactification because $\sY$ is uniquely determined up to a homeomorphism.
	\end{remark}
	\begin{theorem}\label{sc_comp_thm}\cite{munkres:topology}
	Let $\sX$ be a completely regular space. There exists a  
	compactification $\sY$ of $\sX$ having the property that every bounded continuous map $f : \sX\to\R$
	extends uniquely to a continuous map of $\sY$ into $\R$.
	\end{theorem}
	
	\begin{definition}\label{top_stone_cech_defn}\cite{munkres:topology}
	For each completely regular space $\sX$, let us choose, once and for all,
	a compactification of $\sX$ satisfying the extension condition of Theorem \ref{sc_comp_thm}. We will
	denote this compactification of  by $\bt\sX$ and call it the \textit{Stone-\v{C}ech compactification}  
	of $\sX$. It is characterized by the fact that any continuous map $\sX\to\sY$ of $\sX$ into a
	compact Hausdorff  space $\sY$ extends uniquely to a continuous map $\bt\sX\to\sY$.
	\end{definition}
	\begin{theorem}\label{top_sc_to_compact_thm}\cite{munkres:topology}
	Let $\sX$ be a completely regular space; let $\bt\sX$ be the  {Stone-\v{C}ech compactification} of $\sX$. Given any continuous map $f: \sX \to \sY$ into a compact Hausdorff  space $\sY$, the map $f$ extends uniquely to a
	continuous map $\bt f: \bt \sX\to\sY$.
	\end{theorem}	
	
	\begin{definition}\label{f_topology_defn}\cite{rudin:fa}
	Suppose next that $\sX$ is a set and $\mathscr F$ is a nonempty family of mappings $f: \sX \to \sY_f$, where each $\sY_f$ is a topological space. (In many important cases, $\sY_f$ is the same for all $f \in \mathscr F$.) Let $\tau$ be the collection of all unions of finite intersections of sets $f^{-1}\left(\sV  \right)$ with $f\in \mathscr F$ and $\sV$ open in $\sY_f$. Then $\tau$ is a topology on $\sX$, and it is in fact the weakest topology on $\sX$ that makes every 
	$f\in \mathscr F$ continuous: If $\tau'$ is any other topology with that property, then $\tau\subset\tau'$. This $\tau$ is called the \textit{weak topology on} $\sX$ \textit{induced by} $\mathscr F$, or, more 
	shortly, the $\mathscr F$-\textit{topology of} $\sX$.
	\end{definition}
	\begin{definition}\label{top_locally_contractible_defn}\cite{spanier:at,switzer:at}
	A topological space $\sX$ is said to be 
	\textit{contractible} if the identity map of $\sX$ is homotopic to some constant map of $\sX$ 
	to itself. $\sX$ is \textit{locally contractible} if
	every neighbourhood $\sU$ of a point $x$ contains a neighbourhood $\sV$ of $x$ deformable to $x$ in $\sV$). 
	
	\end{definition} 
	\begin{definition}\label{top_simply_conn_defn}\cite{spanier:at}
	Let $E^{k+1}, S^k\subset\R^{k+1}$ be a disk ans a sphere, i.e.
	\bean
	E^{k+1} \bydef \left\{\left.\left(x_1, ..., x_{k+1}\right) \in \R^{k+1}~\right| x_1^2+...+x^2_{k+1}\le 1\right\},\\
	S^k \bydef \left\{\left.\left(x_1, ..., x_{k+1}\right) \in \R^{k+1}~\right| x_1^2+...+x^2_{k+1}= 1\right\}
	\eean
	with the natural inclusion $S^k\subset E^{k+1}$. A space $\sX$ is said to be $n$-\textit{connected} if every continuous map $f: S^k \to \sX$ for $k\le n$ has a continuous extension over $E^{k+1}$. A 1-connected space is also said to be \textit{simply connected}.
	\end{definition}
	
	\begin{definition}\label{top_lower_semi_defn}\cite{bourbaki_sp:gt}
	A real-valued function $f$, defined on a topological space $\sX$, is said to be \textit{lower semi-continuous} (resp. \textit{upper semi-continuous} at point $a \in \sX$ if for each $h < f\left(a\right)$ (resp. each $k > f\left(a\right))$ there is an open neighbourhood $\sV$ of $a$ such that $h < f\left(x\right)$ (resp. each $k > f\left(x\right))$ for each $x\in\sV$.\\
	A real-valued function $f$, is said to be \textit{lower semi-continuous} (resp. \textit{upper semi-continuous} on $\sX$ if it is  {lower semi-continuous} (resp. {upper semi-continuous} at every point of $\sX$.
	\end{definition}
	\begin{theorem}\label{top_lower_semi_thm}\cite{bourbaki_sp:gt}
	Let $f$ be a lower semi-continuous function on a non-empty quasi-compact space $\sX$. Then there is at least one point $a \in \sX$ such that $f\left(a \right) = \inf_{x \in \sX}f\left( x\right) $ (in other words, $f$ attains its greatest lower bound in $\sX$).
	\end{theorem}
	\begin{remark}\label{top_lower_semiem}
	There is inverted result about upper semi-continuous functions and supremums.
	\end{remark}

	\begin{empt}\label{top_cauchy_empt}\cite{engelking:general_topology}
	If $\left( \sX, \rho\right) $ is a metric space then wa say that $\left\{x_j\right\}_{j \in \N}\subset \sX$ is a \textit{Cauchy sequence} if for every $\eps > 0$ there exist $k > 0$ such that $\rho\left(x_j, x_k \right) < \eps$ whenever $j> k$. If $\left( \sX, \rho\right)$ is complete then any Cauchy sequence is convergent.
	\end{empt}
	\begin{remark}\label{top_cauchyem}
	Above statements can be generalized to nets (cf. Definition \ref{top_net_lim_defn}) the word \textit{Cauchy net} is also used in this book.
	\end{remark}
	
	\begin{definition}\label{metric_space_defn}\cite{rudin:pa}
		A set $X$, whose elements we shall call \textit{points}, is said to be a 
		\textit{metric space} if with any two points $p$ and $p$ of $X$ there is associated a real 
		number $(p, q)$, called the \textit{distance} from $p$ to $q$, such that
		\begin{enumerate}
			\item [(a)]
			\bean
			p \neq q \quad \Leftrightarrow \quad d\left(p, q\right)> 0;\\
			d(p,p)=0;
			\eean
			\item[(b)] $d(p,q)=d(q,p);$
			\item[(c)] $\forall r \in X \quad d(p, q) < d(p, r) + d(r, q)$.
		\end{enumerate} 
		Any function with these three properties is called a \textit{distance function}, or a \textit{metric}. 
	\end{definition}

	\begin{remark}\label{triangle_ineqem}
	The condition (b) of the Definition \ref{metric_space_defn} is said to be the \textit{triangle inequality}. If $Y$ is a normed vector space then there is a distance function on $Y$ such that
	$$
	d(x, y)\bydef \left\|x - y \right\| \quad \forall x, y \in Y.
	$$
	The triangle identity is given by
	\be\label{triangle_ineq_vec_eqn}
	\forall x, y, z \in Y \quad  \left\|x - y \right\|\le  \left\|x - z \right\|+ \left\|y - z \right\|.
	\ee
	\end{remark}


	\begin{definition}\label{top_susp_defn}\cite{blackadar:ko}
	If $\sX$ is locally compact, the {\it (reduced) suspension} $S\sX$ of $\sX$
	is defined to be $\sX \times \R$.
	\end{definition}
	
\begin{theorem}\label{locally_compact_hausdorff_thm}\cite{munkres:topology}
 Let $\sX$ be a space. Then $\sX$ is locally compact Hausdorff if and only
if there exists a space $\sY$ satisfying the following conditions:
\begin{enumerate}
	\item[(a)] $\sX$ is a subspace of $\sY$.
		\item[(b)] The set $\sY \setminus\sX$ consists of a single point.
		\item[(a)] $\sY$ is a compact Hausdorff space.
	Jf $\sY$ and $\sY'$ are two spaces satisfying these conditions, then there is a homeomorphism of $\sY$ with $\sY'$ that equals the identity map on $\sX$
\end{enumerate}
\end{theorem}

\begin{definition}\label{one_point_compactification_defn}\cite{munkres:topology}
If $\sY$ is a compact Hausdorff space and $\sX$ is a proper subspace of $\sY$  whose
closure equals $\sY$, then $\sY$ is said to be a {\it compactification} of $\sX$. If $\sY \setminus\sX$ equals a single
point, then $\sY$ is called the {\it one-point compactification} of X.
\end{definition} 
\begin{theorem}\label{comp_to_thm}\cite{munkres:topology}
 Let $f : \sX \to \sY$ be a bijective continuous function. If $\sX$ is compact
	and $\sY$ is Hausdorff, then $f$ is a homeomorphism

\end{theorem}	

\begin{empt}\label{top_lex_square_empt}\cite{counter_topology,munkres:topology}
	Let $I^2_o = \left[0,1\right]^2$ be a unit square. We order $I^2_o$ lexicographically
	\be\label{top_order_top_s_eqn}
	\forall (x, y),(u, w) \in \left[0,1\right]^2 = I^2_o \quad (x, y)< (u, w)  = \begin{cases}
		y < w & x = u\\
		x < u & \text{otherwise}
	\end{cases}
	\ee
	and place the order topology on $I^2_o$. It is proven in  \cite{counter_topology,munkres:topology} $I^2_o$ is connected but it is not path connected.
	The space $I^2_o$ is said to be the \textit{the ordered square} (cf. \cite{munkres:topology}).
\end{empt}

	\subsection{Algebraic topology}
	\subsubsection{Coverings and fundamental groups}
	\paragraph*{}
	Results of this section are copied from \cite{spanier:at}. The \textit{covering projection} word is replaced with \textit{covering} outside this section.
	
	\begin{defn}\label{top_covering_defn}\cite{spanier:at}
		Let $\widetilde{\pi}: \widetilde{\mathcal{X}} \to \mathcal{X}$ be a continuous map. An open subset $\mathcal{U} \subset \mathcal{X}$ is said to be {\it evenly covered } by $\widetilde{\pi}$ if $\widetilde{\pi}^{-1}(\mathcal U)$ is the disjoint union of open subsets of $\widetilde{\mathcal{X}}$ each of which is mapped homeomorphically onto $\mathcal{U}$ by $\widetilde{\pi}$. A continuous map $\widetilde{\pi}: \widetilde{\mathcal{X}} \to \mathcal{X}$ is called a {\it covering projection} if each point $x \in \mathcal{X}$ has an open neighbourhood evenly covered by $\widetilde{\pi}$. $\widetilde{\mathcal{X}}$ is called the {
			\it covering space} and $\mathcal{X}$ the {\it base space} of the covering.
	\end{defn}
	\begin{defn}\cite{spanier:at}
		A topological space $\mathcal X$ is said to be \textit{locally path-connected} if the path components of open sets are open.
	\end{defn}
	\begin{defn}\label{top_covering_transformations_geoup_defn}\cite{spanier:at}
		Let $p: \mathcal{\widetilde{X}} \to \mathcal{X}$ be a covering.  A self-equivalence is a homeomorphism $f:\mathcal{\widetilde{X}}\to\mathcal{\widetilde{X}}$ such that $p \circ f = p$. This group of such homeomorphisms is said to be the {\it group of covering transformations} of $p$ or the {\it covering group}. Denote by $G\left( \left.\mathcal{\widetilde{X}}~\right|~\mathcal{X}\right)$ this group.
	\end{defn}
	\begin{theorem}\label{top_fundamental_group_mor_thm}\cite{spanier:at}
		Let $p: \widetilde{\sX}\to \sX$ a fibration with unique path lifting. Let $\sX$ be path connected and let $x_0\in \sX$. Then  there is  a monomorphism $\psi$ of $G\left( \left.\widetilde{\sX} \right|\sX\right) $ to 
		the quotient group $N\left(\pi_1\left( p\right) \left( \pi_1\left(\widetilde\sX,\widetilde x_0 \right)\right)\right)   /\pi_1\left(\widetilde\sX,\widetilde x_0 \right)$ where $N\left(\pi_1\left( p\right) \left( \pi_1\left(\widetilde\sX,\widetilde x_0 \right)\right)\right)$ is a maximal among subgroups $G\subset \pi_1\left(\sX, x_0 \right)$ such that $\pi_1\left(\widetilde\sX,\widetilde x_0 \right)$ is a normal subgroup of $G$.	If $\sX$ is also locally path connected, 
		$\psi$ is an isomorphism. 
	\end{theorem}

	\begin{thm}\label{top_g_cov_thm}\cite{spanier:at}
		Let $G$ be a properly discontinuous group of homeomorphisms of a space $\mathcal X$. Then the projection $\mathcal X \to \mathcal X/G$ of $\mathcal X$ to the orbit space $\mathcal X/G$ is a covering. If $\mathcal X$ is connected then this covering is regular and $G$ is its group of covering transformations.
	\end{thm}
	
	\begin{theorem}\label{top_path_lift_thm}\cite{spanier:at}
		A fibration has unique path lifting if and only is every fiber has no nonconstant paths. 
	\end{theorem}

	\begin{theorem}\label{top_locally_conn_cov_thm}\cite{spanier:at}
		If $\mathcal{ X}$ is locally connected , a continuous map  $p: \widetilde{\mathcal X}\to {\mathcal X}$ is a covering projection if and only if for each component $\mathcal Y$ of $\mathcal X$ the map
		$$
		p|_{p^{-1}\left(\mathcal Y \right) }:p^{-1}\left(\mathcal Y \right) \to \mathcal Y
		$$
		is a covering projection.
	\end{theorem}	
	%
	\begin{corollary}\label{top_cov_cat_cor}\cite{spanier:at}
		Consider a commutative triangle
		\newline
		\begin{tikzpicture}
			\matrix (m) [matrix of math nodes,row sep=3em,column sep=4em,minimum width=2em]
			{
				\widetilde{\mathcal X}_1 & &\widetilde{\mathcal X}_2\\ 
				& {\mathcal X}\\};
			\path[-stealth]
			(m-1-1) edge node [above] {$p$} (m-1-3)
			(m-1-1) edge node [left]  {$p_1~~$} (m-2-2)
			(m-1-3) edge node [right] {$~~p_2$} (m-2-2);
		\end{tikzpicture}
		\\
		where $\sX$ is locally connected (cf. Definition \ref{top_locally_connected_defn}) and $p_1$, $p_2$ are covering projections. If $p$ is a surjection then $p$ is a covering projection.
	\end{corollary}
	\begin{theorem}\label{top_locally_conn_cov_com_thm}\cite{spanier:at}
		If   $p: \widetilde{\mathcal X}\to {\mathcal X}$ is a covering projection onto locally connected base space, then for any component  $\widetilde{\mathcal Y}$ of $\widetilde{\mathcal X}$ the map
		$$
		p|_{\widetilde{\mathcal Y}}:\widetilde{\mathcal Y} \to p\left(\widetilde{\mathcal Y} \right) 
		$$
		is a covering projection onto some component of $\widetilde{\mathcal X}$.
	\end{theorem}
	\begin{theorem}\label{top_conjugate_thm}\cite{spanier:at}
		Let $p_1: 	\widetilde{\sX}_1 \to \sX$, $p_2: 	\widetilde{\sX}_2 \to \sX$ be objects  in the category of connected covering spaces of a connected locally path-connected space $\sX$. The following are equivalent
		\begin{enumerate}
			\item [(a)] There is a covering projection $f:\widetilde{\sX}_1 \to \widetilde{\sX}_2$ such that $p_2 \circ f = p_1$.
			\item[(b)] For all $\widetilde{x}_1\in \widetilde{\sX}_1$ and  $\widetilde{x}_2\in \widetilde{\sX}_2$ such that $p_1\left(\widetilde{x}_1 \right) = p_2\left(\widetilde{x}_2 \right)$, $\pi_1\left(p_1\right)\left( \pi_1 \left( \widetilde{\sX}_1, \widetilde{x}_1\right)  \right)$ is conjugate to a subgroup of  $\pi_1\left(p_2\right)\left( \pi_1 \left( \widetilde{\sX}_2, \widetilde{x}_2\right)  \right)$.
			\item[(c)] There exist $\widetilde{x}_1\in \widetilde{\sX}_1$ and  $\widetilde{x}_2\in \widetilde{\sX}_2$ such that $\pi_1\left(p_1\right)\left( \pi_1 \left( \widetilde{\sX}_1, \widetilde{x}_1\right)  \right)$ is conjugate to a subgroup of  $\pi_1\left(p_2\right)\left( \pi_1 \left( \widetilde{\sX}_2, \widetilde{x}_2\right)  \right)$.
		\end{enumerate} 
	\end{theorem}
	\begin{definition}\label{local_homeomorphism_defn}\cite{spanier:at}
		A continuous map $f: \sX \to \sY$ is called a \textit{local homeomorphism} if each 
		point $y \in \sY$ has an open neighbourhood mapped homeomorphically by $f$ onto 
		an open subset of $\sX$
	\end{definition}
	\begin{lemma}\label{top_cov_lem}\cite{spanier:at}
		A covering projection is a local homeomorphism.
	\end{lemma}
	
	\begin{lemma}\label{top_om_lem}\cite{spanier:at}
		A local homeomorphism is an open map (cf. Definition \ref{top_onen_map_defn}).
	\end{lemma}

	\begin{defn}\label{top_path_lifting_defn}\cite{spanier:at}
		A continuous map $p:\mathcal{E} \to \mathcal{B}$ is said to have the {\it unique path lifting} if, given paths $\omega$ and $\omega'$ in $E$ such that $p \circ \omega = p \circ \omega'$ and $\omega(0)=\omega'(0)$, then $\omega=\omega'$.
	\end{defn}
	\begin{thm}\cite{spanier:at}\label{spanier_thm_un}
		Let $p: \widetilde{\mathcal{X}} \to \mathcal{X}$ be a covering projection and let $f, g: \mathcal{Y} \to \widetilde{\mathcal{X}}$ be liftings of the same map (that is, $p \circ f = p \circ g$). If $\mathcal{Y}$ is connected and $f$ agrees with g for some point of $\mathcal{Y}$ then $f=g$.
	\end{thm}
	\begin{rem}
		From theorem \ref{spanier_thm_un} it follows that a covering projection has unique path lifting.
	\end{rem}
	
	\begin{defn}\label{top_regular_defn}\cite{spanier:at}
		A fibration $p: \mathcal{\widetilde{X}} \to \mathcal{X}$ with unique path lifting is said to be  {\it regular} if, given any closed path $\omega$ in $\mathcal{X}$, either every lifting of $\omega$ is closed or none is closed.
	\end{defn}
	\begin{thm}\label{locally_path_thm}\cite{spanier:at}
		Let $p: \widetilde{\mathcal X} \to \mathcal X$ be a fibration with unique path lifting and assume that a nonempty $\widetilde{\mathcal X}$ is a locally path-connected space. Then $p$ is regular if and only if for some $\widetilde{x}_0 \in  \widetilde{\mathcal X}$, $\pi_1\left(p\right)\pi_1\left(\widetilde{\mathcal X}, \widetilde{x}_0\right)$ is a normal subgroup of $\pi_1\left(\mathcal X, p\left(\widetilde{x}_0\right)\right)$.
	\end{thm}
	\begin{definition}\label{top_properly_disc_group_defn}\cite{spanier:at} 
		A group $G$ of homeomorphisms of a topological space $\sY$  is said to be  
		\textit{discontinuous} if the orbits of $G$ in $\sY$ are discrete subsets of $\sY$. $G$ is \textit{properly discontinuous} if for $y \in \sY$ there is an open neighbourhood $\sU$ of $y$ in $\sY$ such 
		that if $g, g' \in  G$ and $g\sU$ meets $g'\sU$, then $g = g'$. $G$ \textit{acts without fixed points} if the only element of $g\in G$ having fixed points is the identity element.	
	\end{definition}
	\begin{remark}\label{top_properly_disc_rem}\cite{spanier:at}
		A finite group of homeomorphisms acting without fixed points on a 
		Hausdorff space is properly discontinuous.
	\end{remark}
	\begin{theorem}\label{top_group_of_covering_transformations_thm}\cite{spanier:at}
		Let $G$ a properly discontinuous group of homeomorphisms 
		of a space $\sY$. Then the projection of $\sY$ to the orbit space $\sY/G$ is a covering projection. If $\sY$ is connected, this covering projection is regular and $G$  is its	group of covering transformations. 
	\end{theorem}
	\begin{lem}\label{top_cov_from_pi1_cor}\cite{spanier:at}
		Let $p: \widetilde{\mathcal X} \to \mathcal X$ be a fibration with a unique path lifting. If $ \widetilde{\mathcal X}$ is connected and locally path-connected and $\widetilde{x}_0 \in \widetilde{\mathcal X}$ then $p$ is regular if and only if $G\left(\widetilde{\mathcal X}~|~{\mathcal X} \right)$ principal  \ref{top_principal_defn} ly acts on each fiber of $p$, in which case 
		$$
		\psi: G\left(\widetilde{\mathcal X}~|~{\mathcal X} \right) \approx \pi_1\left(\mathcal X, p\left( \widetilde{x}_0\right)  \right) / \pi_1\left( p\right)\pi_1\left(\widetilde{\mathcal X}, \widetilde{x}_0 \right)  
		$$
		where $\pi_1$ denotes fundamental groups.
	\end{lem}
	\begin{remark}\label{top_sim_con_reg_rem}\cite{spanier:at}
		If $\widetilde{   \sX}$ is simply connected, any fibration $p: \widetilde{   \sX} \to \sX$ is regular, and we also have the next result.
	\end{remark}
	
	\begin{theorem}\label{top_gr_cov_thm}\cite{spanier:at}
		Let $\sX$ be a connected locally path-connected space and let 
		$x_0\in \sX$. Let $H$ be a subgroup of $\pi_1\left( \sX, x_0\right)$  and assume  there is an open  
		covering $\mathscr U$ of $\sX$ such that $\pi_1\left(\mathscr U, x_0 \right)\subset H$ . Then there is a covering projection $p:\left(  \widetilde \sX, \widetilde x_0\right) \to \left(\sX, x_0 \right)$ such that  $p_{\#}\left(\widetilde \sX, \widetilde x_0 \right)= H$ . 
	\end{theorem}
	
	\begin{definition}\label{top_universal_covering_defn}\cite{spanier:at}
		A \textit{universal covering space} of a connected space $\sX$ is an object $p: \widetilde{\sX}\to \sX$ of the category of connected covering spaces of $\sX$ such that for any object $p': \widetilde{\sX}'\to \sX$ of this category there is a morphism 
		\newline
		\begin{tikzpicture}
			\matrix (m) [matrix of math nodes,row sep=3em,column sep=4em,minimum width=2em]
			{
				\widetilde{\sX}  &  & \widetilde{\sX}'\\
				& \sX & \\};
			\path[-stealth]
			(m-1-1) edge node [above] {$f$} (m-1-3)
			(m-1-1) edge node [left]  {$p~~$} (m-2-2)
			(m-1-3) edge node [right] {$~~p'$} (m-2-2);
		\end{tikzpicture}
		\\
		in the category.
	\end{definition}
		\begin{defn}\label{top_semi1_defn}\cite{spanier:at}
		A space $\sX$ is said to by \textit{semilocally} 1-\textit{connected} if every point $x_0$ has a neighbourhood $N$ such that $\pi_1\left( N, x_0\right) \to \pi_1\left(\sX, x_0 \right)$ is trivial. 
	\end{defn} 
	
	\begin{lem}\label{top_simply_con_cov_lem}\cite{spanier:at}
		A connected locally path-connected space $\mathcal X$ has a simply connected covering space $\widetilde{\sX}$ if and only if  $\mathcal X$ is semilocally 1-connected (cf. Definition \ref{top_semi1_defn}).
	\end{lem}
	\begin{remark}\label{top_simply_con_cov_rem}
		Under the hypothesis of the Lemma \ref{top_simply_con_cov_lem} there is an isomorphism
		$$
		\pi_1\left(\sX, x_0 \right)\cong G\left(\left.\widetilde \sX \right| \sX\right).  
		$$
	\end{remark}	

\begin{lem}\label{top_uni_exist_spa_lem}\cite{spanier:at}
A simply connected covering space of a connected locally path-connected space $\sX$ is an universal covering space of $\sX$.
\end{lem}

\begin{definition}\label{top_principal_defn}
	A covering $p: \widetilde{\mathcal X}\to {\mathcal X}$  is said to be \textit{principal} if  $\sX$ is connected and for every $x \in \sX$ the group of covering transformations (cf. Definition \ref{top_covering_transformations_geoup_defn}) $G\left( \left.\mathcal{\widetilde{X}}~\right|\mathcal{X}\right)$ transitively  acts on $p^{-1}\left( x\right)$, i.e. 
	\be\label{top_principal_eqn}
	\forall \widetilde x \in \widetilde\sX \quad	G\left( \left.\mathcal{\widetilde{X}}~\right|\mathcal{X}\right)\widetilde x = p^{-1}\left( p\left( \widetilde x\right) \right).
	\ee
\end{definition}

\subsubsection{Elements of $K$-theory}\label{top_k_sec}
\paragraph*{}
Let $k$ be the field of real or complex numbers, and let $\sX$ be a topological space.
\begin{definition}\label{top_vb_fiber_defn}\cite{karoubi:k}
	A \textit{quasi-vector bundle with base} $\mathcal X$ is given by:
	\begin{enumerate}
		\item [(a)] A finite dimensional $k$-vector space $E_x$ for every point $x$ of $\mathcal X$.
		\item[(b)] A topology on the disjoint union $E = \bigsqcup E_x$ which induces the natural topology on each $E_x$, such that the obvious projection $\pi: E \to  \mathcal X$ is continuous.
	\end{enumerate}
	The quasi-vector bundle with base will be denoted by $\xi = \left( E, \pi, \mathcal X\right)$. The space $E$ is the \textit{total space} of $\xi$ and $E_x$ is the \textit{fiber} of $\xi$ at the point $x$.
\end{definition}
\begin{empt}\label{trivial_vb_empt}\cite{karoubi:k}	Let $V$ be a finite dimensional vector space over $k$,  $E_x = V$ and the total space may be identified with $\sX \times V$ with the product topology then the quasi-vector bundle  $\left(\sX \times V, \pi, \mathcal X\right)$ is called a \textit{trivial vector bundle}.
\end{empt}
\begin{empt}\cite{karoubi:k}	Let $\xi = \left( E, \pi, \mathcal X\right)$ be a quasi-vector bundle, and let let $\sX' \subset \sX$ be a subspace of $\sX$. The triple $\xi' = \left(\pi^{-1}\left( \sX'\right) , \pi|_{\pi^{-1}\left( \sX'\right)}, \mathcal X'\right)$ is called the \textit{restriction} of $\xi$ to $\sX'$. The fibers of $\xi'$ are just fibers of $\xi$ over the subspace $\xi$. One has
	\be
	\sX'' \subset \sX' \subset \sX \quad \Rightarrow\quad \left.\left(\xi|_{\sX'} \right) \right|_{\sX''}= \xi|_{\sX''}.
	\ee
\end{empt}
\begin{definition}\label{top_vb_defn}\cite{karoubi:k}
	Let $\xi = \left( E, \pi, \mathcal X\right)$ be quasi-vector bundle. Then $\xi$ is said to be \textit{locally trivial} or a \textit{vector bundle} if for every point $x$ in $\sX$, there exists a neighbourhood of $x$ such that $\xi|_{\sU}$ is isomorphic to a trivial bundle.
\end{definition}
\begin{definition}\label{top_vb_cs_defn}\cite{karoubi:k}	
	Let $\xi \bydef \left( E, \pi, \mathcal X\right)$ be a vector bundle. Then a \textit{section} of $\xi$ is a map $s: \sX \to E$ such that $\pi \times s = \Id_{\sX}$. A section $s$ is called \textit{continuous} if $s$ is a continuous. The space $\Ga\left({\mathcal X}, {E}\right)$, of continuous sections can be regarded as both left  and right $C_b\left(\mathcal{X} \right)$-module.
\end{definition}
\begin{notation}\label{top_sec_notn}
	We denote by $\Ga\left( M, E\right)$ the $\C$-space of continuous sections. Both  $\Ga_0\left( M, E\right)$ and $\Ga_c\left( M, E\right)$ are subspaces of  sections tending to zero at infinity and having compact support.
\end{notation}
\begin{remark}
	In \cite{karoubi:k} the vector bundles with base  fields $\R$ and $\C$ are considered. Here we consider complex vector bundles only.
\end{remark}
\begin{definition}\label{vb_inv_img_funct_defn}\cite{karoubi:k}
	Let $f: \mathcal X' \to \mathcal X$ be a continuous map. For every point $x'$ of $\mathcal X'$, let $E'_{x'}= E_{f\left(x' \right) }$. Then the set $E' = \bigsqcup_{x' \in \mathcal X'}E'_{x'}$ may be identified with the \textit{fiber product} $\mathcal X' \times_{\mathcal X} E$ formed by the pairs $\left(x',e \right)$ such that $f\left(x' \right) = \pi\left( e\right)$. If  $\pi': E' \to \mathcal X'$ is defined by $\pi'\left(x',e \right) = x'$, it is clear that the triple $\xi = \left( E', \pi', \mathcal X'\right)$ defines a quasi-vector bundle  with base  $ \mathcal X'$, when we provide $E'$ with the topology induced by $ \mathcal X' \times E$. We write $\xi' = f^*\left(\xi \right)$ or $f^*\left(E \right)$: this is the \textit{inverse image} of $\xi$ by $f$.  
\end{definition}
\begin{theorem}\label{serre_swan_thm}\cite{karoubi:k}.
(Serre, Swan). Let $A = C_k(\sX)$ be the ring of continuous 
	functions on a compact space $\sX$ with values in $k$. Then the section functor $\Ga$ (cf. Notation \ref{top_sec_notn}) induces an 
	equivalence of categories of vector bundles  with base  $\sX$ and finitely generated projective $A$-modules.
\end{theorem}

\begin{definition}\label{top_herm_bundle_form_defn}\cite{karoubi:k}
	Let $E$ be a vector bundle over $\C$. A \textit{sesquilinear form} on $E$ is a 
	continuous map $\varphi: E \times_\sX E \to \C$ which has the following property. The map 
	$\varphi_x : E_x \times E_x$ induced on each fiber is "sesquilinear" with respect to the $\C$-vector space structure of $E_x$. In other words, ($\varphi_x$ is $\R$-bilinear and $\varphi_x\left(\la e, e' \right)  = \varphi_x\left( e, \overline \la e' \right) = \la \varphi_x\left( e, e' \right)$ f
	for $\la\in \C$, $e \in E_x$, and $e' \in E_x$. 
\end{definition}


\paragraph{The Chern characters}
\begin{proposition}\cite{karoubi:k}\label{chern_class_prop}
	For each vector bundle $\sV$ of rank $n$, with compact base $\sX$, we 
	define unique functions $c^K_j\left(\sV \right)\in K\left( \sX\right) $, for $j=0, .. ., n$ such that $c^K_0\left(\sV \right)=1$ . 
	These characteristic classes satisfy the following axioms:
	\begin{enumerate}
		\item [(i)]
		The $c^K_j\left( \sV\right)$  are "natural", i.e. 	$c^K_j\left( \sV\right)  = f^*\left(c^K_j\left( \sV\right)  \right)$  for any general morphism $\sV \xrightarrow{g}\sV'$ , which induces $f: \sX \to \sX"$ on the bases, $\sX$ and $\sX'$, of $\sV$ and $\sV'$ respectively; 
		and which induces an isomorphism on each fiber. 
		\bean
		\begin{tikzcd}
			\sV \arrow[d]\arrow[r, "g"] & \sV'\arrow[d]\\
			\sX \arrow[r, "f"] & \sX'\\
		\end{tikzcd}
		\eean
		\item[(ii)] If $\sV_1$ and $\sV_2$ are vector bundles on $\sX$, then 
		\bean
		c^K_k\left(\sV_1 \oplus \sV_2 \right) = \sum_{i + j = k} c_i\left( \sV_1\right) c^K_j\left( \sV_2\right) 
		\eean
		\item[(iii)] If the rank of $\sV$ is one, then $c^K_1\left(\sV \right)= \chi\left(\sV \right) \bydef  1 -\left[\sV\right]$, and $c^K_i\left( \sV\right)  = 0$ for $i > 1$.
	\end{enumerate}
	Moreover, the characteristic classes $c^K_j\left( \sV\right)$  are uniquely determined by these axioms. 
\end{proposition}

\begin{theorem}\label{chern_class_thm}\cite{hatcher:kt}
There is a unique sequence of functions $c^H_1, c^H_2,...$ assigning to each
complex vector bundle $\sV \to \sX$ a class $c_j\left( \sV\right)$  depending only on the isomorphism
type of $\sV$ , such that
\begin{enumerate}
	\item [(i)] $c^H_j\left( f^*\left(
	\sV  \right)\right)   = H^{2j}\left(\bullet, \Z \right)\left( f\right)  \left(
c^H_j\left( 	\sV \right)  \right)$ where $f^*$ is the inverse image (cf. Definition \ref{vb_inv_img_funct_defn}).
\item[(ii)] $c\left( \sV_1 \oplus \sV_2\right) = c\left(\sV_1 \right) \smile  c\left(\sV_1 \right)$ for $c \bydef 1 + c^H_1 + c^H_2 + ... \in  H^*\left(\sX, \Z \right)$.
\item[(iii)] $c^H_j(\sV) = 0$ if $j > \dim \sV$.
\item[(iv)]For the canonical line bundle $\sV \to \C P^\infty$  is a generator of $H^2\left(\C P^\infty, \Z \right)$.
\end{enumerate}
\end{theorem}
\begin{definition}\label{chern_class_defn}
The given by the  Proposition \ref{chern_class_prop} (resp. Theorem \ref{chern_class_thm}) functions $c_j^K$ (resp. $c^H_j$) are \textit{($K$-theoretical}) (resp.\textit{cohomological}) \textit{Chern classes}.
\end{definition}
\begin{remark}\label{full_cem}\cite{karoubi:k}
 Let us write 
  $c^H(V)= \sum_{j= 0}^n  c^H_j\left( V\right) \in H^{\mathrm{even}}\left(\sX \right)$ . This is called the "total 	Chern class" of $V$. One has
  \bean
  c^H\left(V_1 \oplus V_3 \right) = c^H\left( V_1\right) \cdot c^H\left( V_1\right)
  \eean
  Similarly there is "total 	Chern class" $c^K\left(V \right) \in K\left( \sX\right)$  of $V$ in $K$-theory.
	\end{remark}
	\begin{empt}
We let $Q_k$, for $k>1$, denote the "Newton polynomials": they express the 
$\sum_{j=1}^k u^k_j$ symmetric functions as unique polynomials of the elementary symmetric 
functions,
$$
\sigma \bydef \sum_{j_1 < j_2 < ... < j} u_{j_1}... u_{j}\quad 1 \le r \le k.
$$
For instance
\be\label{newton_p_eqn}
\begin{split}
Q_1\left(\sigma_1 \right) = \sigma_1,\\
Q_2\left(\sigma_1, \sigma_2 \right) = \left( \sigma_1\right)^2-2\sigma_2,\\
Q_1\left(\sigma_1, \sigma_2, \sigma_3 \right) = \left( \sigma_1\right)^3 - 3 \sigma_1\sigma_2 + 2\sigma_3,\\
...
\end{split}
\ee
There are classes $s_x\left(V \right)\in H^{2k}$ , such that $s_k\left( V_1 \oplus V_2\right) = s_k\left( V_1\right)+ s_k\left( V_1\right)$  by putting $s_k\left( V\right)  = Q_k\left(c_1,..., c_n \right)$  where $cj \bydef c_j\left( V\right)$ , and $Q_k$ is the Newton polynomial. We 
	make the convention that $s_0\left(V \right) \bydef \rank\left(V \right)$. As well known $H^p\left( \sX, \Q\right) \cong H^p\left(\sX, \Z \right)\otimes \Q$. Thus we may define
	\be 
	\begin{split}
Ch_k\left( V\right) \bydef \frac{1}{k!}	s_k\left(V \right) = \frac{1}{k!}	Q_k\left(c_1\left(V \right), ..., c_n\left(V \right) \right)\in H^{2k} \left(\sX, \Q\right),\\
Ch\left( V\right) \bydef \sum_{k=0}^\infty Ch_k\left( V\right) \in H^{\mathrm{even}} \left(\sX, \Q\right),\\
	\end{split} 
	\ee
\end{empt}
\begin{theorem}\cite{karoubi:k}
Let $V_1$ and$V_2$ be vector bundles, with compact base $\sX$. Then we 
have the formulas 
	\be 
\begin{split}
	Ch\left( V_1 + V_2 \right) =Ch\left( V_1 \right)+Ch\left( V_2 \right),\\
	Ch\left( V_1 \otimes  V_2 \right) =Ch\left( V_1 \right)\cdot Ch\left( V_2 \right),\\
\end{split} 
\ee
	Moreover, if $L$ is a line bundle, then $Ch\left(L \right) = \exp\left(\chi\left(L \right)  \right)$. 
\end{theorem}

\begin{definition}\cite{karoubi:k}
We call the ring homomorphism 
\be
Ch: K(X) \to H^{\mathrm{even}} \left(\sX, \Q\right)\bydef \bigoplus_{j=0}^\infty H^{2j}\left(\sX, \Q \right) 
\ee 
defined by $Ch\left(\left[E\right]-\left[F\right]  \right) \bydef  Ch\left(\left[E\right]  \right)-Ch\left(\left[F\right]  \right)$, the {\it Chern character}. 

\end{definition}
\begin{theorem}\label{chern_ch_thm}\cite{karoubi:k}
Let $\sX$ be compact. Then there are isomorphisms
	\be 
\begin{split}
\chi^0 :K^0\left( \sX\right) \otimes \Q \xrightarrow{\cong } \bigoplus_{n ~\mathrm{ even}} H^n\left(\sX, \Q \right),\\ 
\chi^1 :K^1\left( \sX\right) \otimes \Q \xrightarrow{\cong } \bigoplus_{n ~\mathrm{ odd}} H^n\left(\sX, \Q \right) 
\end{split} 
\ee
where $H_n\left(\sX, \Q \right)$  denotes the n-th ordinary (Alexander or  \v{C}ech) cohomology
group of $\sX$ with coefficients  in $\Q$.
\end{theorem}

\subsubsection{Multiplicative structures}\label{mult_k_sec}
\paragraph{} Here I follow to \cite{karoubi:k}. 
	\begin{empt}\label{etp_empt}
	Let $E$ and $F$ be vector bundles with compact bases $\sX$ and $\sY$ respectively. We define the "{\it external Whitney sum}" of $E$ and $F$ as the vector bundle $E\boxplus F$ on $\sX\times \sY$, where $E\boxplus F\bydef \pi_1^*\left(E \right)\oplus \pi_2^*\left(F \right)$  with $\pi_1\left( \sX \times \sY\right)\onto \sX$  and $\pi_2\left( \sX \times \sY\right)\onto \sY$. Obviously we 	have $E\boxplus F=E\times  F $ and $\left( E\boxplus F\right)_{\left(x,y \right) }= E_x \oplus E_y$, In the same way, the "external 	tensor product" of $F$ and $E$ is the vector bundle $E\boxtimes F\bydef \pi_1^*\left(E \right)\otimes \pi_2^*\left(F \right)$  on $\sX \times \sY$, where $\left( E\boxtimes F\right)_{\left(x,y \right) }= E_x \otimes E_y$
\end{empt}
\begin{empt}\label{k_cup_empt}
Let $\sX$ and $\sY$ be compact spaces, and let $E$ and $F$ be vector bundles with bases 
$\sX$ and $\sY$ respectively. Then their "external tensor product" $E\boxtimes F$ (\ref{etp_empt}) is a vector bundle over $\sX \times \sY$. The correspondence $\left(E , F \right)\mapsto E\boxtimes F$  induces a functor $\varphi$ from 
the category $\E\left( \sX\right)\times \E\left( \sX\right)$  to  $\E\left( \sX\times \sY\right)$, such that $\varphi\left( E \oplus E', F\right)\cong  \varphi\left( E , F\right)\oplus \varphi\left(  E', F\right)$ and $\varphi\left( E, F \oplus F\right)\cong  \varphi\left( E , F'\right)\oplus \varphi\left(  E, F'\right)$, with the analogous property for morphisms. From this functor we obtain a bilinear map 
\be 
\begin{split}
\varphi_*:K\left(\sX \right) \times K\left(\sY \right) \to :K\left(\sX \times \sY\right),\\
\varphi_* \left(\left[E\right]-\left[E'\right], \left[F\right]-\left[F'\right] \right) \bydef \left[\varphi\left(E, F \right) \right]+\left[\varphi\left(E', F' \right) \right]-\left[\varphi\left(E, F' \right) \right]-\left[\varphi\left(E', F \right)\right].
\end{split}
\ee
Also we denote the \emph{cup}-\emph{product} $\varphi_*\left(x, y \right)$ for $x \in K\left(\sX \right)$  and $y \in K\left(\sY \right)$, by $x \smile y$.  

If both $\sX$ and $\sY$ are locally compact space then the cup-product
\be\label{kf_prod_eqn}
\begin{split}
	K\left(\sX, \R^n \right) \times 	K\left(\sY, \R^p \right) \to 	K\left(\sX\times \R^n \times \sY, \R^p \right)\cong 	K\left(\sX\times\sY\times \R^{n + p} \right)
\end{split}
\ee
may be interpreted as a bilinear map
 \be\label{mult_k_eqn} 
 \begin{split}
 K^n\left( \sX \right)\times  K^p\left( \sY \right) \to K^{n+p}\left( \sX \times \sY\right)
 \end{split}
 \ee
 The diagonal map $\sX \hookto\sX \times \sX$ yields  homomorphisms $K^{n+p}\left(\sX \times \sX \right)  \to K^{n+p}\left(\sX  \right)$ and from \eqref{kf_prod_eqn} one has a product
   \be \label{kfn_prod_eqn}
  \begin{split}
  	K^n\left( \sX \right)\times  K^p\left( \sX \right) \to K^{n+p~\mathrm{mod} ~2}\left( \sX \right)
  \end{split}
  \ee
  
\end{empt}
\section{Lattices and filters}
			\begin{definition}\label{upper_defn}\cite{johnstone:stone_spaces}
		Let $\A$ be a partially ordered set, $S$ a subset of $\A$. We say an element $a \in A$
		is a \textit{ meet} (or \textit{ greatest lower bound}) for $S$, and write $a = \wedge S$, if 
		\begin{enumerate}
			\item [(a)] $a$ is an lower bound for $S$, i.e. $a \le s$ for all $s \in S$ and, 
			\item [(b)] if $\forall s \in S\quad  b \le s$, then $b \le a$. 
		\end{enumerate}
		\paragraph{}
		The antisymmetry axiom of the  partially ordered set $A$ ensures that the join of $S$, if it exists, is 
		unique. If $S$ is a two-element set $\{s, t\}$, we write $s \wedge t$ for $\wedge \{s, t\}$ and if $S$	is the empty set $\emptyset$, we write $0$ for  $\wedge \emptyset$.
	\end{definition}
	
	\begin{definition}\label{lattice_defn}\cite{johnstone:stone_spaces}
		A \textit{meet-semilattice} is a partially ordered set which supports for any finite set the  greatest lower bound.
	\end{definition}
	\begin{definition}\label{filter_defn}\cite{johnstone:stone_spaces} 
		A subset $\mathfrak F$ of a meet-semilattice $A$ is said to be an {\it filter} if
		\begin{enumerate}
			\item [(a)] $\mathfrak F$ is a sub-meet-semilattice of $A$; i.e. $1\notin \mathfrak F$, and $a, b \in \mathfrak F$  imply  	$a \wedge b \in \mathfrak F$; and
			\item [(b)] $\mathfrak F$ is a lower set; i.e. $a \in \mathfrak F$ and $ a\le b$ imply $b \in \mathfrak F$.  
		\end{enumerate}
	\end{definition}
	\begin{empt}\label{filter_empt}
		Similarly to \cite{bourbaki_sp:gt} one has
		\begin{enumerate}
			\item [(c)] $0 \notin \mathfrak F$.
			
		\end{enumerate}	
		\begin{definition}
			A filter  $\mathfrak F$ is \textit{principal} if there is $b \in A$ such that 
			$$
			\mathfrak F = \left\{a \in A \left|~ b \le a\right.\right\}.
			$$
			
		\end{definition}
	\end{empt}
	\begin{remark}\label{fil_homem}
		Any homomorphism $\phi : L' \to L''$ of {semi}-{lattices} yields a map
		of filers
		\be\label{filtermap_eqn}
		\left\{I_\la\right\}_{\la \in \La}	\mapsto \text{ minimal filter containing } \left\{\phi\left( I_\la\right) \right\}_{\la \in \La}
		\ee
	\end{remark}
	
	\begin{example}\label{space_semi_exm}
		If $\sX$ is a topological space then the $\sX$-\textit{semi}-\textit{lattice} is  a  meet-semilattice $\mathfrak{Lattice}\left( \sX\right)$ such that elements of $\mathfrak{Lattice}\left( \sX\right)$ are open subsets of $\sX$ and one has
		\be\label{x_lat_eqn}
		\begin{split}
			\sU' \wedge \sU''\bydef \sU' \cap \sU'',\\
			\sU' \le \sU'' \quad \Leftrightarrow\quad \sU'' \subset \sU',\\
			0 \bydef \emptyset.
		\end{split}
		\ee
	\end{example}
	
	\begin{definition}\label{ultra_filter_defn}\cite{johnstone:stone_spaces} 
		A maximal filter is an \textit{ultrafilter}.
	\end{definition}
		\begin{remark}\label{ultra_filterem}
		From the Theorem \ref{zorn_thm} it follows that any filter is a subset of an	ultrafilter. (cf. \cite{johnstone:stone_spaces}).
	\end{remark}

	\subsection{Filter base topology}\label{fbt_sec}
	\begin{defn}\label{filted_conv_defn}
	\cite{engelking:general_topology}
	By a {\it filter} (a {\it filter-base}) in a topological space $\sX$ we mean a filter (a filter-base) in the
	family of all subsets of $\sX$. In this section we shall discuss only filters and filter-bases in
	topological spaces; we shall call them simply filters and filter-bases.
	A point $\sX$ is called a limit of a filter $\mathfrak{F}$ if every neighbourhood of $x$ belongs to  $\mathfrak{F}$; we
	then say that the filter  $\mathfrak{F}$ converges to $x$ and we write $x = \lim  \mathfrak{F}$. 
	\end{defn}
	
		\begin{proposition}\label{top_ultra_prop}\cite{johnstone:stone_spaces}
		One has:
		\begin{enumerate}
			\item [(a)] 	A topological space $\sX$ is Hausdorff  if and only if every ultrafilter on $\sX$ has at most one limit.
			\item[(b)]   	A topological space $\sX$ is compact if and only if every ultrafilter has at least one limit.
		\end{enumerate} 
	\end{proposition}
\section{Sheaves and cohomology}
	\subsection{Basic definitions}
	\begin{definition}\label{presheaf_defn}\cite{hartshorne:ag}
		Let $\sX$ be a topological space. A \textit{presheaf} $\mathscr F$ of sets or Abelian groups on  $\sX$ consists of the data
		\begin{itemize}
			\item[(a)] for every open subset $\sU \subseteq \sX$, an Abelian group $\mathscr F\left(\sU\right)$, and 
			\item[(b)] for every inclusion $\sV \subseteq \sU$ of open subsets of $\sX$, a morphism of Abelian groups $\rho_{\sU \sV}:\mathscr F\left(\sU\right) \to \mathscr F\left(\sV\right)$,\\
			subject to conditions
			\begin{itemize}
				\item [(0)] $\mathscr F\left(\sV\right)= 0$, where $\emptyset$ is the empty set,
				\item[(1)] $\rho_{\sU \sU}$ is the identity map, and
				\item[(2)] if $\mathcal W \subseteq \sV \subseteq \sU$ are three open sets, then $\rho_{\sU \mathcal W} = \rho_{\sV \mathcal W }\circ \rho_{\sU \sV}$.
			\end{itemize}
		\end{itemize}
	\end{definition}	
	\begin{definition}\label{sheaf_defn}\cite{hartshorne:ag}
		A \textit{presheaf} $\mathscr F$ on  
		a topological space $\sX$ is a \textit{sheaf}  if it satisfies the following supplementary conditions:
		\begin{itemize}
			\item[(3)] If $\sU$ is an open set, if $\left\{\sV_{\a}\right\}$ is an open covering of $\sU$, and if $s \in \mathscr F\left(\sU\right)$ is an element such that $\left.s\right|_{\sV_{\a}}= 0$ for all $\a$, then $s = 0$;
			\item[(4)] If $\sU$ is an open set, if $\left\{\sV_{\a}\right\}$ is an open covering of $\sU$ (i.e. $\sU = \cup\sV_\a$), and we have elements $s_\a$ for each $\a$, with property that for each $\al, \bt, \left.s_\a\right|_{\sV_{\a}\cap \sV_{\bt}}= \left.s_\bt\right|_{\sV_{\a}\cap \sV_{\bt}}$, then there is an element $s \in \mathscr F\left(\sU\right)$ such that $\left.s\right|_{\sV_\a} = s_\a$ for each $\a$.
		\end{itemize}
		(Note condition (3) implies that $s$ is unique.)
	\end{definition}
	\begin{definition}
		A presheaf \ref{presheaf_defn} satisfying (4) of the Definition \ref{sheaf_defn} is called \textit{conjunctive} (for $\sU$). (cf. \cite{bredon:sheaf})
	\end{definition}
	\begin{definition}\label{sheaf_stalk_defn}\cite{hartshorne:ag}
		If $\mathscr F$ is a {presheaf} on $\sX$, and if $x$ is a point of $\sX$ we define the \textit{stalk} or the \textit{germ} $\mathscr F_x$ 
		\textit{of} $\mathscr F$ \textit{at} $x$ to be the direct limit of groups $\mathscr F\left(\sU\right)$ for all open sets $\sU$ containing $x$, via restriction maps $\rho$.
	\end{definition}
	
	\begin{definition}\label{sheaf_mor_defn}\cite{hartshorne:ag}
		If $\mathscr F$  and $\mathscr G$ are presheaves on $\sX$, a \textit{morphism} $\varphi:\mathscr F\to\mathscr G$  consists 
		of a morphism of Abelian groups $\varphi_\sU:\mathscr F\left( \sU\right) \to\mathscr F\left( \sU\right)$ for each open set 
		$\sU$, such that whenever $\sV\subset\sU$ is an inclusion, the diagram 
		\newline
		\begin{tikzpicture}
			\matrix (m) [matrix of math nodes,row sep=3em,column sep=4em,minimum width=2em]
			{
				\mathscr F\left( \sU\right)   & \mathscr G\left( \sU\right)\\
				\mathscr F\left( \sV\right)    & \mathscr G\left( \sV\right)\\
			};
			\path[-stealth]
			(m-1-1) edge node [above] {$\varphi_\sU$} (m-1-2)
			(m-1-1) edge node [right] {$\rho_{\sU\sV}$} (m-2-1)
			(m-1-2) edge node [right] {$\rho'_{\sU\sV}$} (m-2-2)
			(m-2-1) edge node [above] {$\varphi_\sV$} (m-2-2);
		\end{tikzpicture}
		\\
		is commutative, where $\rho_{\sU\sV}$ and $\rho'_{\sU\sV}$ are the restriction maps in $\mathscr F$  and $\mathscr G$. If $\mathscr F$  and $\mathscr G$ are sheaves on $\sX$, we use the same definition for a morphism 
		of sheaves. An isomorphism is a morphism  which has a two-sided inverse. 
	\end{definition}
	
	\begin{prdf}\label{sheaf_prdf}\cite{hartshorne:ag}
		Given a presheaf $\mathscr F$, there is a sheaf  $\mathscr F^+$ and a morphism $\th: \mathscr F \to \mathscr F^+$, with the property that for any sheaf  $\mathscr G$, and any morphism $\varphi: \mathscr F \to \mathscr G$, there is a unique morphism $\psi:\mathscr F^+\to \mathscr G$ such that $\varphi = \psi \circ \th$. Furthermore the pair $\left(\mathscr F^+, \th\right)$ is unique up to unique isomorphism. $\mathscr F^+$ is called the $\mathrm{sheaf~associated}$ to the presheaf $\mathscr F$. 
	\end{prdf}
	
	\begin{empt}\label{sheaf_empt}
		Following text is the citation of the proof of \ref{sheaf_prdf} (cf. \cite{hartshorne:ag}). For any open set $\sU$, let $\mathscr F^+\left(\sU\right)$ be set of functions $s$ from $\sU$ to the union $\bigcup_{x \in \sU}  \mathscr F_x$ of stalks of $\mathscr F$ over points of $\sU$, such that
		\begin{itemize}
			\item [(1)] for each $x \in \sU$, $s\left(x\right)\in \mathscr F_x$, and
			\item[(2)] for each $x \in \sU$, there is a neighbourhood $\sV$ of $x$ contained in $\sU$ and an element $t \in \mathscr F\left(\sV\right)$, such that for all $y \in \sV$ the stalk (germ) $t_y$ of $t$ at $y$ is equal to $s\left(y\right)$.
		\end{itemize}
	\end{empt}

	\begin{definition}\label{sheaf_inv_im_defn}\cite{hartshorne:ag}
		Let $f: \sX\to \sY$ be a continuous map of topological spaces. For any sheaf  $\mathscr F$ on $\sX$, we define the \textit{direct image} sheaf  $f_*\mathscr F$ on $\sY$ by $\left(f_*\mathscr F\right)\left(\sV\right)= \mathscr F\left(f^{-1}\left(\sV\right)\right)$ for any open set $\sV \subseteq \sY$. For any sheaf  $\mathscr G$ on $\sY$, we define the \textit{inverse image} sheaf  $f*\mathscr G$ on $\sX$ be the sheaf  associated to the presheaf  $\sU \mapsto \lim_{\sV \supseteq f\left(\sU\right)} \mathscr G\left(\sV\right)$, where $\sU$ is any open set in $\sX$, and the limit is taken over  all open sets $\sV$ of $\sV$ containing $f\left(\sU\right)$.
	\end{definition}
	\begin{definition}\label{flabby_sheaf_defn}\label{bredon:sheaf}
	A sheaf $\mathscr A$ on $\sX$ is \textit{flabby} or \textit{flasque} if $\mathscr A\left( \sX\right) \to \mathscr A\left( \sU\right)$ is surjective for every open set  $\sU \subset \sX$
	\end{definition}
	\begin{definition}\label{sheaf_support_defn}\cite{hartshorne:ag}
		For $s \in \mathscr{F}\left(\sX \right)$, 
		\be\label{sheaf_supp_eqn}
		\left|  s\right| = \left\{\left.x \in \sX \right|s\left(x \right)\neq 0  \right\}
		\ee
		denotes the \textit{support} of the section $s$.
	\end{definition}
	
	\begin{exercise}\label{sheaf_etale_exer}\cite{hartshorne:ag}
		\textit{
			\'Espace Etal\'e of a presheaf}. 
		Given a presheaf $\mathscr F$ on $\sX$, we define a topological space $\mathrm{Sp\acute{e}}\left(\mathscr F \right)$ , called the \textit{
			\'espace etal\'e} of a presheaf of $\mathscr F$ as 
		follows. As a set it is the union $\mathrm{Sp\acute{e}}\left(\mathscr F \right)= \bigcup_{x\in\sX} \mathscr F_x$ of sets of germs (cf. Definition \ref{sheaf_stalk_defn}). We define a projection map 
		\be\label{pf_eqn}
		p_{\mathrm{Sp\acute{e}}\left(\mathscr F \right)}: \mathrm{Sp\acute{e}}\left(\mathscr F \right)\to \sX
		\ee
		by sending $s_x\in \mathscr F_x$ to $x$. Consider the initial with respect to the map $	p_{\mathrm{Sp\acute{e}}\left(\mathscr F \right)}$ topology (cf. Definition \ref{top_init_defn} and denote by $\mathrm{Sp\acute{e}}\left(\mathscr F \right)_{\mathfrak{Etale}}$ the corresponding topological space.
	\end{exercise}
	\begin{definition}\label{sheaf_hom_defn}\cite{hartshorne:ag}
		Let $\mathscr F$, $\mathscr G$ be sheaves of Abelian  groups  on $\sX$. For any open set $\sU \subseteq \sX$ the set of morphisms
		$\Hom\left(\left.\mathscr F\right|_{\sU}, \left.\mathscr G\right|_{\sU}\right)$ has the natural structure of Abelian group. 
		It is a sheaf  (cf. \cite{hartshorne:ag}). It is called the \textit{sheaf of local morphisms} of $\mathscr F\to \mathscr G$, "sheaf hom" for short, and is denoted by $\mathscr Hom \left(\mathscr F, \mathscr G\right)$.
	\end{definition}

\subsection{Families of supports}
	
	\begin{proposition}\label{sheaf_flasque_prop}\cite{hartshorne:ag}
		If $\mathscr F$ is a flasque sheaf on a topological space $\sX$, then  $H^p\left(\sX, \mathscr F \right)  = 0$ for all $p > 0$.
	\end{proposition}

		\begin{definition}\label{phi_supp_defn}\cite{bredon:sheaf,godement:sheaf}.
		Let $\sX$ be a topological space. A \textit{family of supports} on $\sX$ is a family $\Phi$ of closed
		subsets of $\sX$  such that:
		\begin{enumerate}
			\item a closed subset of a member of $\Phi$ is a member of $\Phi$;
			\item $\Phi$ is closed under finite unions.
		\end{enumerate}
		
		$\Phi$  is said to be a  \textit{paracompactifying} family of supports if in addition:
		\begin{enumerate}
			\item  each element of $\Phi$ is paracompact;
			\item each element of $\Phi$ has a (closed) neighborhood which is in $\Phi$.
		\end{enumerate}
	\end{definition}\label{sheaf_soft_defn}
	The family of all compact subsets of $\sX$ is denoted by $c$.
	\begin{definition}\label{soft_sheaf_defn}\cite{bredon:sheaf}
		A sheaf  $\mathscr{F}$ on $\sX$ is called $\Phi$-\textit{soft} if the restriction map	$\mathscr{F}\left(\sX \right) \to  \mathscr{F}\left(K \right)$  is surjective for all $K \in \Phi$. If $\Phi$ is a set of all closed sets  then $\Phi$ is simply
		called \textit{soft}. If $\Phi= c$ is a set of all compact  sets  then $\mathscr{F}$ is called $c$-\textit{soft}.
	\end{definition}
	\
Now  $\mathscr P$ is a presheaf  on $\sX$ such and $s \in \mathscr P\left(\sX \right)$ we put
	$\left|s\right| = \left|\th\left( s\right) \right|$, where $\th: \mathscr{P}\left(\sX \right)\to \mathscr{P}^+\left(\sX \right) $ is the canonical map, $\mathscr{P}^+$ being the sheaf  generated by $\mathscr{P}$.
	Note that for $s \in \mathscr{P}$ one has $x \notin \left|s\right|\Leftrightarrow\left.s \right|_{\sU}=0$ for some neighborhood $\sU$ of $x$. If $\mathscr{F}$ is a sheaf, we put
	$$
	\Ga_\Phi = \left\{\left.s\in \mathscr{F}\left( \sX\right)\right| \left|s\right| \in \Phi \right\}.
	$$

	For a presheaf $\mathscr{P}$  on $\sX$ we put  $\mathscr{P}_\Phi\left(\sX \right) = \left\{\left.s\in \mathscr{P}\left( \sX\right)\right| \left|s\right| \in \Phi \right\}$.

\begin{definition}\cite{bredon:sheaf}
		Let $\Phi$ be a family of supports on $\sX$. Then a subspace
		$A \subset \sX$ is said to be $\Phi$-\textit{taut} if for every flabby sheaf $\mathscr F$ on $\sX$, the restriction
		$\Ga_\Phi \to \Ga_{\Phi\cap A}$ is surjective and $\mathscr F|A$ is $\Phi\cap A$-acyclic.
		The following five cases are examples of -b-taut subspaces A of X:
		\begin{enumerate}
			\item[(a)] $\Phi$ arbitrary, $A$ open.
			\item[(b)] $\Phi$ paracompactifying for the pair $\left(\sX, A \right)$ .
			\item[(c)] $\Phi$ paracompactifying, $\sX$ hereditarily paracompact (e.g., metric), $A$
			arbitrary.
			\item[(d)] $\Phi$ paracompactifying, $A$ closed.
			\item[(a)]($\Phi - cld$, $A$ compact and relatively Hausdorff  in $\sX$, e.g., a point.
		\end{enumerate}
		
		\begin{definition}
			Let $\Phi$ be a family of supports on $\sX$. A continuous map $\sX \to \sY$ is $\Phi$-\textit{closed} meaning $f\left( K\right)$ is closed for any $K \in \Phi$.
		\end{definition}
		
	\end{definition}
	\begin{empt}
		Let $\mathscr L^*$ be a differential sheaf on $\sX$. We are concerned with two cases:
		\begin{enumerate}
			\item [(a)] $\mathscr L^*=0$ for $q < q_0$,
			\item[(b)] $\dim_{\Psi, L}\sX < \infty$ \textit{for a given ground ring $L$ and family $\Psi$ of supports}.
		\end{enumerate}
		If we are in case (b) with $\dim_{\Psi, L}\sX < \infty$ , then
		\bean
		0 \to \mathscr A \to \mathscr C^0\left(\sX; \mathscr A \right) \to ... \to C^{m-1}\left(\sX,\; \mathscr A \right)\to \mathscr Z^m\left(\sX; \mathscr A \right)
		\eean
		is a resolution of $\Psi$ by $\Psi$-acyclic sheaves and is an exact functor of $\mathscr A$. In
		the discussion below, $C^*_\Phi\left(\sX, \mathscr A \right)$  should be replaced by $\Ga_\Phi$ of this resolution
		when we are in case (b).
	\end{empt}
	
	\begin{defn}
		Let $f : \sX \to \sY$ with $\Psi$ and $\Phi$ being families of supports on $\sX$ and $\sY$
		respectively.
		Let the \textit{extension $\Phi\left( \Psi\right)$  of $\Psi$ by $\Psi$} be the family of
		supports on $\sX$ defined by
		\bean
		\Phi\left( \Psi\right) \bydef \left\{K \in \Phi\left( \sY\right)\left| \overline {f\left(K \right)}\in \Phi \right. \right\} = \Psi\left( \sY\right) \cap f^{-1}\left(\Phi \right).
		\eean
		
		Intuitively, $\Phi\left( \Psi\right)$ is the result of "spreading $\Psi$ out over $\Phi$" The reason
		for considering this construction is the following:
	\end{defn}
	\begin{proposition}
		For any sheaf $\mathscr A$ on $\sX$ we have
		\bean
		\Ga_\Phi\left(f_\Phi \mathscr A\right) = \Ga_{\Psi\left(\sY \right) }\left(\mathscr A \right) \subset \mathscr A\left(\sX \right). 
		\eean
		under the defining equality $\left( f_\Psi \mathscr A\right) = \Ga_{\Phi\left(\Psi \right) }\left(\mathscr A\right)$.
	\end{proposition}

	\begin{proposition}\label{flabby_soft_prop}\cite{bredon:sheaf}
		If $\sU$ is a compact relatively Hausdorff  subspace of $\sX$,
		then $\left.\mathscr{F}\right|_{\sU} $ is soft for any flabby sheaf  $\mathscr{F}$ on $\sX$.
	\end{proposition}
	
		\subsection{Canonical resolution and cohomology}\label{canonicalesolution_empt}
	Here I follow to \cite{bredon:sheaf, godement:sheaf}.					
	For any sheaf $\mathscr F$ of Abelian groups on $\sX$ and open set $\sU\subset \sX$  we let $C^0(\left(\sU, \mathscr F \right)$ be the
	collection of all functions (not necessarily continuous) $f: \sU\to \mathrm{Sp\acute{e}}\left(\mathscr F \right)$ such 	that $p \circ f$ is the identity on $\sU$, $~~p : \mathrm{Sp\acute{e}}\left(\mathscr F \right)\to \sX$ being the canonical projection (cf. Exercise \ref{sheaf_etale_exer}).
	Such possibly discontinuous sections are called \textit{serrations}. That is
	\be
	C^0(\left(\sU, \mathscr F \right)\bydef \prod_{x \in \sX} \mathscr F_x
	\ee
	$\mathscr F_x$ is the space of stalks (cf. Definition \ref{sheaf_stalk_defn})
	Under point-wise operations, this is an Abelian group, and the functor $\sU \mapsto 	C^0(\left(\sU, \mathscr F \right)$, $\sX$. Indeed it is a sheaf (cf. \cite{bredon:sheaf}) which will be denoted by $\mathscr C^0(\left(\sU, \mathscr F \right)$. Note that if $\sX_d$ denotes the point set of $\sX$ 
	with the discrete topology and if $\phi : \sX_d \to \sX$ is the canonical map, then $\mathscr C^0(\left(\sU, \mathscr F \right)\cong f_*f^* \mathscr F$.
	Inclusion of the collection of sections of $\mathscr F$ in the collection of all serrations
	gives an inclusion $\mathscr F\left(\sU \right)\mapsto   C^0(\left(\sU, \mathscr F \right)= \mathscr C^0(\left(\sU, \mathscr F \right)\left(\sU \right)$  and hence
	provides a natural monomorphism
	\be\label{sheaf_canes}
	\eps : \mathscr F \hookto \mathscr C^0\left(\sX, \mathscr F \right)
	\ee
	(For $\phi : \sX_d \to \sX$ as above, this inclusion coincides with the monomorphism
	$\bt :\mathscr F \hookto f_*f^*\mathscr F$). 
	Let $\mathscr{Z}^{1}\left(\sX, \mathscr F \right) \bydef \coker\left\{\mathscr F \hookto \mathscr C^0\left(\sX, \mathscr F \right)\right\}$. so that the sequence
	$$
	\{0\}\to \mathscr F \to \mathscr{C}^{0}\left(\sX, \mathscr F \right) \to \mathscr{Z}^{1}\left(\sX, \mathscr F \right)\to {0}
	$$
	is exact. We also define, inductively,
	\bean
	\mathscr{C}^{n}\left(\sX, \mathscr F \right) \bydef \mathscr{C}^{0}\left( \mathscr{Z}^{n}\left(\sX, \mathscr F \right)\right),\\
	\mathscr{Z}^{n+1}\left(\sX, \mathscr F \right) \bydef \mathscr{Z}^{1}\left( \mathscr{Z}^{n}\left(\sX, \mathscr F \right)\right)  
	\eean
	so that
	\bean
	\mathscr{Z}^{n}\left(\sX, \mathscr F \right) \xrightarrow{\eps} \mathscr{C}^{n}\left(\sX, \mathscr F \right)\xrightarrow{\partial}
	\mathscr{Z}^{n+1}\left(\sX, \mathscr F \right)
	\eean
	
	is exact. Let $d \bydef \eps \circ \partial$ be the composition
	$$
	\mathscr{C}^{n}\left(\sX; \mathscr F \right) \xrightarrow{\partial }  \mathscr{Z}^{n + 1}\left(\sX; \mathscr F \right)\xrightarrow{\eps}\mathscr{C}^{n+1}\left(\sX; \mathscr F \right) 
	$$
	
	Then the sequence 
	\be\label{canonicalesolution_eqn}
	0 \to \mathscr F \xrightarrow{\eps} 	\mathscr{C}^{0}\left(\sX, \mathscr F \right) \xrightarrow{d } \mathscr{C}^{1}\left(\sX, \mathscr F \right)\xrightarrow{d } \mathscr{C}^{2}\left(\sX, \mathscr F \right)\xrightarrow{d}\dots
	\ee
	is exact. That is, $\mathscr{C}^{\bullet}\left(\sX, \mathscr F \right)$ is a resolution of $\mathscr F$. It is called the \textit{canonical
		resolution} of $\mathscr F$.

	Since  $\mathscr C^0\left(\sX; \mathscr A \right)$ is an exact functor of $\mathscr A$, so is $\mathscr Z^1\left(\sX; \mathscr A \right)$. By induction
	it follows that $\mathscr C^0\left(\sX; \mathscr A \right)$ and $\mathscr Z^1\left(\sX; \mathscr A \right)$ are all exact functors of $\mathscr A$.
	
	For a family $\Phi$ of supports on $\sX$ we put
	\be
	\begin{split}
	C^n_\Phi\left( \sX, \mathcal A\right) \bydef \Ga_\Phi\left( \mathscr C^n\left(\sX; \mathscr A \right)\right) \cong C^0_\Phi\left( \mathscr Z^n\left(\sX; \mathscr A \right)\right)
	\end{split}
	\ee

	Since $C^0_\Phi\left( \bullet; \bullet \right)$ and $\mathscr Z^1\left( \bullet ; \bullet \right)$ are exact functors, it follows that 
	$C^n_\Phi\left( \sX; \bullet \right)$ and $\mathscr Z^n\left( \sX ; \bullet \right)$ are all exact functors it follows that
	\be
\begin{split}
	C^n_\Phi\left( \sX, \bullet\right) \quad \text{is an exact functor}.
\end{split}
\ee
	
		\begin{definition}\label{cohomol_defn}\cite{bredon:sheaf}
For a family $\Phi$ of supports on $\Phi$ and for a sheaf $\mathscr A$ on
$\sX$ we define
	\be
\begin{split}
	H^n_\Phi\left( \sX, \mathscr A\right) \bydef H^n\left( C^*_\Phi\left( \sX, \mathscr A\right)\right).
\end{split}
\ee
\end{definition}

		\begin{definition}\label{cohomomor_defn}\cite{bredon:sheaf}
			If $A$ and $B$ are presheaves on $\sX$ and $\sY$ respectively, then
			an $f$ -{\it cohomomorphism} $k: A  \leadsto B$ is a collection of homomorphisms
			$k_\sU : B(\sU) \to A\left( f^{-1}\left(\sU \right) \right)$ , for $\sU$ open in $\sY$, compatible with restrictions.	
		\end{definition}
			\begin{remark}
		 Let $f: \sX \to\sY$ is a continuous map, 
		let $\Phi$ and $\Psi$ be families of supports
		on $\sX$ and $\sY$ respectively such that $f^{-1}\Psi  \subset \Phi$.
		It is proven in \cite{bredon:sheaf} that if   $\mathscr B$ is a sheaf of Abelian groups on $\sY$  then 
			there is the natural homomorphism
			\be\label{cohom_mor_eqn}
			f^*: H^*_\Psi\left(\sY;\mathscr B  \right) \to  H^*_\Phi\left(\sX;f^*\mathscr B  \right).
			\ee
			If $\{0\}\hookto \mathscr A' \hookto \mathscr A \onto \mathscr A''\onto \{0\}$ is an exact sequence of sheaves then there is a long exact sequence
			\be\label{cohom_seq_eqn}
			... \to H^p_\Phi\left( \sX; \mathscr A'\right) \to H^p_\Phi\left( \sX; \mathscr A\right) \to H^p_\Phi\left( \sX; \mathscr A''\right) \xrightarrow{\dl} H^{p+1}_\Phi\left( \sX; \mathscr A'\right)\to ... 
			\ee
			If a sheaf $ \mathscr F$ is $c$-soft then
			\be\label{cohom_soft_eqn}
			\forall p \ge 1 \quad H^p_c\left( \sX, \mathscr F\right)= 0  
			\ee
			
		\end{remark}

		\begin{corollary}\label{cech_cor}\cite{bredon:sheaf}
			For $\Phi$ paracompactifying and $A$ is a presheaf  on $\sX$ and $\mathscr A$ is the associated to $A$ sheaf then there is
			a natural isomorphism
			$$
			\check{H}_\Phi^\bullet\left(\sX; \mathscr F \right) \cong {H}_\Phi^\bullet\left(\sX; \mathscr F \right). 
			$$
			where $\check{H}$ are \v{C}ech cohomology.
		\end{corollary}
		\subsection{Cup product}
		\begin{theorem}\label{kunnet_sheaf_thm}\cite{bredon:sheaf}
			(\text{K\"{u}nneth.}) If $\sX$ and $\sY$ are locally compact Hausdorff 
			spaces and if $\mathscr A$ and $\mathscr B$ are sheaves on $\sX$ and $\sY$ respectively with $\mathrm{Tor}_L\left(\mathscr A, \mathscr B \right) = \{0\}$
			then there is a natural exact sequence (over the principal ideal domain $L$
			as base ring)
			\be\label{kunnet_sheaf_eqn}
			\begin{split}
				\bigoplus_{p+q = n} H^p_c\left(\sX; \mathscr A \right) \otimes H^q_c\left(\sY; \mathscr A \right)\hookto H^n_c\left(\sX\times \sY; \mathscr A\otimes_L\mathscr B \right)\onto\\ \onto  \bigoplus_{p+q = n+1} \mathrm{Tor}\left( H^p_c\left(\sX; \mathscr A \right),  H^p_c\left(\sX; \mathscr B \right) \right). 
			\end{split}
			\ee
			that splits nonnaturally.	
		\end{theorem}
		
		\begin{theorem}\label{cup_sheaf_thm}\cite{bredon:sheaf}
			Let $\Phi$ and $\Psi$ be families of supports on $\sX$. Then there exists
			a unique natural transformation of functors (on the category of sheaves on
			X to Abelian groups)
			\be\label{cup_p_eqn}
			\begin{split}
				\smile : H^p_\Phi\left(\sX, \mathscr A \right) \otimes  H^p_\Psi\left(\sX, \mathscr B \right)\to H^p_{\Phi\cap\Psi}\left(\sX, \mathscr A \otimes \mathscr B\right)
			\end{split}
			\ee
			called the {\it cup product} and satisfying the following three properties:
			\begin{enumerate}
				\item [(a)] For $p = q = 0, ~ \smile \Ga_\Phi\left( \mathscr A\right) \Ga_\Psi\left( \mathscr B\right)\to \Ga_{\Phi\cap \Psi}\left(\mathscr A \otimes \mathscr B \right)$  is the transformation
				induced by the canonical map $\mathscr A\left(\sX\right)\otimes \mathscr B\left(\sX\right)\to \left(\mathscr A \otimes \mathscr A \right)\left(\sX \right)$.   
				\item [(b)] If $0 \hookto \mathscr A' \hookto \mathscr A \onto \mathscr A'' \onto 0$ and $0 \hookto \mathscr A' \otimes \mathscr B\hookto \mathscr A \otimes \mathscr B\onto \mathscr A'' \otimes \mathscr B \onto 0$  , are exact then for $\a \in H^p_\Phi\left(\sX, \mathscr A \right)$ and $\bt \in H^q_\Psi\left(\sX, \mathscr B \right)$ we have $\dl\left( \a \smile \bt\right)= \dl \a  \smile \bt$.
				\item [(c)] If $0 \hookto \mathscr B' \hookto \mathscr B \onto \mathscr B'' \onto 0$ and $0 \hookto \mathscr A \otimes\mathscr B' \hookto\mathscr A \otimes \mathscr B \onto \mathscr A \otimes\mathscr B'' \onto 0$, are exact then for $\a \in H^p_\Phi\left(\sX, \mathscr A \right)$ and $\bt \in H^q_\Psi\left(\sX, \mathscr B \right)$ we have $\dl\left( \a \smile \bt\right)= \left(-1 \right)^p  \a  \smile \dl\bt$.
				
			\end{enumerate}
		\end{theorem}
		\begin{proposition}\label{cup_ass_prop}\cite{bredon:sheaf}
		The given by the Theorem \ref{cup_sheaf_thm} cup product is associative.
		\end{proposition}
			\begin{remark}\label{pres_cupem}
			It is proven in \cite{bredon:sheaf} that the map \ref{cohom_mor_eqn} preserves the cup product (cf. Theorem \ref{cup_sheaf_thm})
		\end{remark}
\begin{empt}\cite{bredon:sheaf}
 Let $A$ be a locally closed subspace of $\sX$ and let $\mathscr B$ be a sheaf on $A$. It is
easily seen (since $A$ is locally closed) that there is a unique topology on the
point set $\mathscr B\cup \left(\sX \times \{0\} \right) $ such that $\mathscr B$ is a subspace and the projection onto $\sX$
is a local homeomorphism (we identify $A \times \{0\}$ with the zero section of $\mathscr B$).
With this topology and the canonical algebraic structure, $\mathscr B\cup \left(\sX \times \{0\} \right) $) is
a sheaf on X denoted by
\be
\mathscr B^\sX \bydef \mathscr B\cup \left(\sX \times \{0\} \right) 
\ee
Thus $\mathscr B^\sX$ is the unique sheaf on $\sX$ inducing $\mathscr B$ on $A$ and $0$ on $\sX \setminus A$. Clearly,
$\mathscr B\mapsto \mathscr B^\sX$ is an exact functor. The sheaf $\mathscr B^\sX$ is called the {\it extension of $\mathscr B$
by zero}.\end{empt}		
		\begin{theorem}\cite{bredon:sheaf}
Suppose either that $\Phi$ is a paracompactifying family of
supports on $\sX$ and that $A \subset \sX$ is locally closed, or that $\Phi$ is arbitrary and $A \subset \sX$ is closed. Then there is a natural isomorphism
\be
H^*_\Phi\left(\sX, \mathscr B^\sX \right) \cong H^*_{\Phi| A}\left(A, \mathscr B \right)
\ee
of functors of sheaves on $A$, which preserves cup products.	
\end{theorem}	
\subsection{Homotopy invariance}\label{ho_inv_chap} Here I follow to \cite{bredon:sheaf}. If $\Phi$ and $\Psi$ are given families of supports on $\sX$ and $\sY$ respectively, we
shall say that a map $f: \sX \to \sY$ is {\it proper} (with respect to $\Phi$ and $\Psi$) if
$f^{-1}\Psi\subset \Phi$). If that is the case, then $f^*: H^*_\Psi\left(\sY \right)\to  H^*_\Phi\left(\sX \right)$  is defined. A homotopy $\sX \times \left[0,1\right]$ is {\it proper} if it is so with respect to the families $\Phi\times \left[0,1\right]$
and $\Psi$. For locally compact spaces, "proper" means proper with respect to
compact supports unless otherwise indicated	
\begin{theorem}\label{ho_inv_thm}
Any two properly homotopic maps (with respect to $\Phi$
and $\Psi$ of a space $\sX$ into a space $\sY$ induce identical homomorphisms
\be\label{ho_inv_eqn}
 H^*_\Psi\left(\sY; G \right)\to  H^*_\Phi\left(\sX; G \right)
\ee
where $G$ is any constant coefficient group.
\end{theorem}
\begin{rem}\label{ho_invem}\cite{bredon:sheaf}
If both $\sX$, $\sY$ a locally compact and Hausdorff  and $\Phi = c = \Psi$ then the conditions of the Theorem \ref{ho_inv_thm} hold.
\end{rem}
\subsection{Relative cohomology}
\paragraph*{} Here I follow to \cite{bredon:sheaf}
Let $j : \hookto \sX$ and let $\Phi$ be a family of supports on $\sX$. For any sheaf $\mathscr A$ on $\sX$ we have the natural $j$-cohomomorphism 
\bean
\mathscr C^*\left(\sX; \mathscr A\right) \to C^*\left(A; \mathscr A|A\right)
\eean
Equivalently, we have the homomorphism
\bean
i^*: \mathscr C^*\left(\sX; \mathscr A\right) \to iC^*\left(A; \mathscr A|A\right)
\eean
of sheaves on $\sX$.
In order to define relative cohomology, we shall show that $j^*$ is surjective
and has a flabby kernel. We introduce the notation
\bean
\ker j^* \bydef  \mathscr C^*\left(\sX, A ;\mathscr A\right),\\
C^*_\Phi\left(\sX, A; \mathscr A\right)\bydef \Ga_\Phi\left( C^*\left(\sX; A; \mathscr A\right)\right) 
\eean
and
\bean
H^*_\Phi \left(\sX, A; \mathscr A\right) \bydef H^*\left( C^*_\Phi\left(\sX, A; \mathscr A\right)\right) 
\eean 
Since $\ker j^*$ is flabby and since $\Ga_\Phi\left(j \mathscr B \right) = \Ga_{\Phi \cap A}\left(\mathscr B \right)$  
obtain the induced short exact sequence
\be\label{longx_pair_eqn} 
0 \hookto C^*_\Phi\left(\sX, A;\mathscr A  \right) \hookto C^*_\Phi\left(\sX;\mathscr A  \right)\onto C^*_\Phi\left(A;\mathscr A | A \right)\onto 0
\ee
and hence the long exact cohomology sequence
\be\label{long_pair_eqn} 
... \to  H^p_\Phi\left(\sX, A;\mathscr A  \right) \to H^p_\Phi\left(\sX;\mathscr A  \right) \to H^p_{\Phi\cap A}\left(A;\mathscr A | A \right)\to  H^{p+1}_\Phi\left(\sX ;\mathscr A \right)\to ...
\ee
If  $0 \hookto \mathscr A' \hookto \mathscr A \onto \mathscr A'' \onto 0$ is an exact sequence of sheaves then one has
\be
\to  H^p_\Phi\left(\sX, A;\mathscr A'  \right) \to  H^p_\Phi\left(\sX, A;\mathscr A  \right)\to  H^p_\Phi\left(\sX, A;\mathscr A''  \right) \to H^{p+1}_\Phi\left(\sX, A;\mathscr A'  \right)\to ...
\ee
For any closet and $\Phi$-taut subset $F \subset \sX$ there is the natural isomorphism
\be\label{setmunus_h_eqn}
H^*_\Phi\left(\sX, F;\mathscr A  \right)\cong H^*_\Phi\left(\sX\setminus F;\mathscr A  \right)
\ee
\subsection{Reduced cohomology}
\paragraph{} If $\sX$ is as compact space and $L$ is a principal ideal domain then there is the natural homomorphism 
\be \label{red_s_eqn}
L \hookto \Ga\left(\sX, L \right)\cong H^0\left(\sX, L \right)
\ee

\begin{definition}\label{red_defn}\cite{bredon:sheaf}
If $L$ is a principal ideal domain then there it {\it reduced cohomology} $\tilde H^*\left(*, L \right) $ such that
\bean
\tilde H^p\left(\sX, L \right)\bydef \begin{cases}
H^0\left(\sX, L \right)/L & p = 0\\
H^p\left(\sX, L \right)& p > 0
\end{cases}
\eean
\end{definition}
\begin{remark}\cite{bredon:sheaf}
If $L$ is a principal ideal domain and a $\sX$ space is locally compact Hausdorff space  then 
	\be\label{red_eqn}
	H^0_c\left( \sX, L\right) \cong \tilde H^0\left(\sX^+, L \right) 
	\ee
	where $\sX^+$ is a one point compactification.
\end{remark}

\subsection{Sheaf cohomology related to $K^1$-group}\label{sheaf_k1_sec}

\paragraph{} Here I follow to \cite{rordam:kc}. Let $\sX$ be a locally compact Hausdorff   space. Let us consider there sheaves $\sZ$, $\sC$ and $\sC\setminus\{0\}$ on $\sX$ which corresponds to sheaves of continuous maps from $\sX$ to $\C$, $\C \setminus\{0\}$ and $\Z$  respectively.  There is an exact sequence of sheaves

\bean
\begin{split}
	\{0\}  \hookto \sZ \hookto \sC \xrightarrow{\exp}\sC\setminus\{0\} \onto \{0\}
\end{split}
\eean
which yield the following long exact sequence
\bean
\begin{split}
	\{0\}  \to  H^0_c\left(\sX; \Z\right)  \to  H^0_c\left(\sX; \sC \right)\to  H^0_c\left(\sX; \sC \setminus \{0\} \right)\xrightarrow{\dl_\SS }  H^1_c\left(\sX;\Z \right)\to  H^1_c\left(\sX; \sC \right)\to ... 
\end{split}
\eean
(cf. equation \eqref{cohom_seq_eqn}). Taking into account that $\sC$ is $c$-soft one has  $H^1_c\left(\sX, \sR \right)= 0$ (cf. \eqref{cohom_soft_eqn}) the homomorphism  
\be\label{dl_shh_eqn}
\dl_\SS	: H^0_c\left(\sX; \sC\setminus\{0\} \right)\onto H^1_c\left(\sX; \Z \right).
\ee
is surjective. The above equations are explained in \cite{rordam:kc}.
Similarly consider there sheaves $\sR$, $\sU\left( 1\right) $ and $\sZ$ on $\sX$ which corresponds to sheaves of continuous maps from $\sX$ to $\R$, $\sU\left( 1\right) $ and $\Z$  respectively.  There is an exact sequence of sheaves

\bean
\begin{split}
	\{0\}  \hookto \sZ \hookto \sR \xrightarrow{x \mapsto e^{2\pi i x}}\sU\left( 1\right)  \onto \{1\}
\end{split}
\eean
which yield the following long exact sequence
\be\label{sxx_shh_eqn}
\begin{split}
	\{0\}  \to  H^0_c\left(\sX; \Z\right)  \to  H^0_c\left(\sX; \sR \right)\to  H^0_c\left(\sX; \sU\left( 1\right)  \right)\xrightarrow{\dl_\SS }  H^1_c\left(\sX; \Z \right)\to\\ \to  H^1_c\left(\sX; \sR \right)\to ... 
\end{split}
\ee
(cf. equation \eqref{cohom_seq_eqn}). Taking into account that $\sR$ is $c$-soft one has  $H^1_c\left(\sX, \sR \right)= 0$ (cf. \eqref{cohom_soft_eqn}) the homomorphism  
\be\label{dl_sh_eqn}
\dl_\SS	: H^0_c\left(\sX, \sU\left(1 \right)  \right)\onto H^1_c\left(\sX, \Z \right).
\ee

\subsection{Mayer-Vietoris theorems}
\paragraph{}
First, let $\mathscr A$ be a sheaf on $\sX$ and $\Phi$ a family of supports on $\sX$. Let $\sX_1$
and $\sX_2$ be closed subspaces of $\sX$ with $\sX = \sX_1 \cup \sX_2$ and put $A = \sX_1 \cap \sX_2$
We have the surjections
\bean
r_i : \mathscr A \onto \mathscr A_{\sX_i},\\
s_i : \mathscr A_{\sX_i} \onto \mathscr A_{A}
\eean
Considering stalks, we see that the sequence
\bean
0\to \mathscr A \xrightarrow{\a} \mathscr A_{\sX_1} \oplus \mathscr A_{\sX_2}\xrightarrow{\bt} \mathscr A_{A}
\eean
is exact, where $\a \bydef (r_l,r_2)$ and $\bt \bydef s_1 - s_2$. Thus, there is  the
exact Mayer-Vietoris sequence (for $A$, closed and arbitrary $\Phi$)
\be\label{mv_eqn}
\begin{split}
...\to H^p_\Psi\left(\sX; \mathscr A \right)\to  H^p_{\Psi\cap \sX_1}\left(\sX_1; \mathscr A_{\sX_1} \right)\oplus  H^p_{\Psi\cap \sX_2}\left(\sX_2; \mathscr A_{\sX_2} \right)\to  H^p_{\Psi\cap A}\left(A; \mathscr A \right)\to \\\to H^{p+1}_\Psi\left(\sX; \mathscr A \right)\to ...
\end{split}
\ee
\begin{example}
For an arbitrary space $\sX$, the cone $C\sX$ on $\sX$ is the
quotient space $\sX \times \left[0,1\right]/\sX \times \{0\}$. Since this is contractible, it is acyclic for
any constant coefficients and closed supports. The (unreduced) {\it suspension}
of a space $\sX$ is $\Sigma\sX \bydef C\sX /\sX \times \{1\}$. This is the union along $\sX \times \left\{\frac{1}{2}\right\}$ of two
cones. The Mayer-Vietoris sequence \eqref{mv_eqn} applies to show that $\tilde H^k\left(\Sigma\sX \right) \cong \tilde H^{k-1}\left(\sX \right)$ for all $k$. If $*\in \sX$  is a base point, then $\Sigma \sX$ contains the arc $I$,
which is the image of $*\times \left[0,1\right]$ and is a closed subspace of $\Sigma\sX$. The {\it reduced
suspension} of $\sX$ is $S\sX \bydef \Sigma \sX /I$. Now the collapsing map $\Sigma \sX\onto S\sX$ 
is closed and $I$ is connected, acyclic, and taut (since it is compact and
relatively Hausdorff  in $\Sigma \sX$), so the Vietoris mapping theorem  applies
to show that 
\be\label{red_c_eqn}
\tilde H^k (\R\times \sX)\cong \tilde H^k (S\sX) \cong \tilde H^k\left( \Sigma\sX\right) \cong \tilde H^{k-1}(\sX)
\ee
 for all $k> 0$. 
\end{example}

\section{$C^*$-algebras}

\subsection{Basic definitions}

	\begin{defn}
	\label{strong_topology_defn}\cite{pedersen:ca_aut} Let $\H$ be a Hilbert space. The {\it strong} topology on $B\left(\H\right)$ is the locally convex vector space topology associated with the family of seminorms of the form $x \mapsto \|x\xi\|$, $x \in B(\H)$, $\xi \in \H$.
\end{defn}
\begin{defn}\label{weak_topology_defn}\cite{pedersen:ca_aut} Let $\H$ be a Hilbert space. The {\it weak} topology on $B\left(\H\right)$ is the locally convex vector space topology associated with the family of seminorms of the form $x \mapsto \left|\left(x\xi, \eta\right)\right|$, $x \in B(\H)$, $\xi, \eta \in \H$.
\end{defn} 
	\begin{lemma}\label{increasing_convergent_w_lem}\cite{pedersen:ca_aut} Let $\Lambda$ be an increasing in the partial ordering.  Let $\left\{x_\lambda \right\}_{\la \in \La}$ be an increasing of self-adjoint operators in $B\left(\H\right)$, i.e. $\la \le \mu$ implies $x_\la \le x_\mu$. If $\left\|x_\la\right\| \le \ga$ for some $\ga \in \mathbb{R}$ and all $\la$ then $\left\{x_\lambda \right\}$ is strongly convergent to a self-adjoint element $x \in B\left(\H\right)$ with $\left\|x_\la\right\| \le \ga$.
\end{lemma} 
\begin{thm}\label{gelfand-naimark_thm}\cite{arveson:c_alg_invt} (Commutative Gelfand-Na\u{\i}mark theorem). 
	Let $A$ be a commutative $C^*$-algebra and let $\mathcal{X}$ be the spectrum of A. There is the natural $*$-isomorphism $\gamma:A \xrightarrow{\cong} C_0(\mathcal{X})$.
\end{thm}

	\begin{thm}\label{pedersen_ideal_thm}  \cite{pedersen:ca_aut} 
	For each $C^*$-algebra $A$ there is a dense hereditary ideal $K(A)$,
	which is minimal among dense ideals.
	
\end{thm}
\begin{proof}
	Let $K(]0, \infty [)$ denote the set of continuous functions on $]0, \infty [$ with 
	compact support and define 
	\be\label{pedersen_k0_eqn}
	K\left( A \right)_0 \bydef \left\{f\left(x\right) \left|x \in A_+, \quad f \in K(]0, \infty [) \right.\right\}.
	\ee
	Let 
	\be\label{pedersen_k_plus_eqn}
	K\left( A \right)_+ \bydef \left\{x \in A_+ \left|x \le \sum_{j = 1}^nx_j, \quad x_j \in  	K\left( A \right)_0\right.\right\}, 	
	\ee
	so that $	K\left( A \right)_+$ is the smallest hereditary cone  containing $K\left( A \right)_0$. If $K(A)$ 
	is  the algebraic  $\C$-linear span of $K(A)_+$ then $K(A)$,
	which is minimal among dense ideals. The full  proof is  described in \cite{pedersen:ca_aut}.
\end{proof}

\begin{defn}\label{pedersen_ideal_defn}\cite{blackadar:ko}
	The ideal $K\left( A\right) $ from the theorem \ref{pedersen_ideal_thm} is said to be the {\it Pedersen's ideal of $A$}. 
\end{defn}

\begin{remark}\cite{pedersen:ca_aut} 
	One has
	\bea\label{peder_k_eqn}
	K\left( \K\right) = \left\{\left. a \in \K\right| a  \text{ is a finite rank operator}\right\},\\
	\label{peder_c0_eqn}
	K\left(C_0\left(\sX \right)  \right) = C_c\left(\sX \right).
	\eea
	
\end{remark}	
\begin{theorem}\label{left_ideal_thm}\cite{murphy}
	If $L$ is a closed left ideal in a $C^*$-algebra $A$, then there
	is an increasing net $\left\{u_\la\right\}_{\la\in\La}$ of positive elements in the closed unit ball of
	$L$ such that $a = \lim_{\la\in \La}au_\la $ for all $a\in L$.
\end{theorem}
\begin{proof}
	Set $B \bydef L\cap L^*$. Since $B$ is a $C^*$-algebra, it admits an approximate
	unit $\left\{u_\la\right\}_{\la\in \La}\subset B$. If $a \in L$, then $a^*a\in B$, so $0 =
	\lim_{\la\in \La} a^*a\left(1_{A^\sim }- u_\la \right)$ . Hence,\\ $\lim_{\la\in \La}\left\|  a - au_\la \right\|^2= \lim_{\la\in \La}\left\| \left(1-u_\la \right)a^*a \left(1-u_\la \right) \right\|
	\le \lim_{\la\in \La}\left\|a^*a \left(1-u_\la \right) \right\|=0$, and therefore $\lim_{\la\in \La}\left\|a - au_\la \right\|=0 $.
	In the preceding proof we worked in the unitization $A^\sim$ of $A$. 
	frequently do this tacitly.
\end{proof} 
	\begin{lemma}\label{hered_ideal_lem}\cite{murphy}
	Let $A$ be a $C^*$-algebra.
	\begin{enumerate}
		\item[(i)] If $L$ is a closed left ideal in $A$ then $L\cap L^*$ is a hereditary $C^*$-subalgebra of $A$. The map $L \mapsto L\cap L^*$ is the bijection from the set of closed left deals of $A$ onto the the set of hereditary $C^*$-subalgebras of $A$.
		\item[(ii)] If $L_1, L_2$ are closed left ideals, then $L_1 \subseteq L_2$ is and only if $L_1\cap L_1^* \subset L_2\cap L_2^*$.
		\item[(iii)] If $B$ is a hereditary $C^*$-subalgebra of $A$, then the set 
		$$
		L\left(B \right) = \left\{\left.a \in A~\right| a^*a \in B \right\}
		$$
		is the unique closed left ideal of $A$ corresponding to $B$.
	\end{enumerate}
\end{lemma}
\begin{definition}\label{normal_defn}\cite{murphy}
If $A$ is a $C^*$-algebra then $a\in A$ is {\it normal} if $a^*a = aa^*$. In this case the $*$-algebra it generates is commutative and is in fact the $\C$-linear span of all $a^ma^{*n}$, where $m, n \in \N$ and $m + n > 0$.
\end{definition}
\begin{definition}\label{sp_mea}\cite{murphy}
Let $\Om$ be a compact Hausdorff space and $\H$ a Hilbert space. A {\it spectral
measure} $E$ relative to $\left(\Om, \sH\right)$ is a map from the $\sigma$-algebra of all Borel sets
of $\Om$ to the set of projections in $B\left( \sH\right)$  such that
\begin{enumerate}
	\item [(a)] $E\left(\emptyset \right) = 0$;
\item [(b)] $E(S_1 \cap  S_2) = E(S_1 )E(5_2 )$ for all Borel sets $S_1, S_2$ of $\Om$;
	\item [(c)] for all $x, y \in \H$, the function $E_{x,y}: S \mapsto \left(E\left( S\right) x, y  \right)$, is a regular
	Borel complex measure on $\Om$.
\end{enumerate}
\end{definition}
\begin{theorem}\label{spectral_thm}\cite{murphy}
Let  $\Om$ be a compact Hausdorff  space and $\sH$ a Hilbert
space, and suppose that $\varphi: C\left( \Om\right)  \to B\left(\sH \right)$  is a unital $*$-homomorphisms.
Then there is a unique spectral measure $E$ relative to $\left(\Om, \sH \right)$  such that
\bean
\begin{split}
\varphi \left(f \right) = \int f ~dE
\end{split}
\eean
Moreover, if $u \in  B\left(\sH \right)$ , then $u$ commutes with $\varphi \left(f \right)$  for all $f \in C\left(\Om \right)$  if and
only if $u$ commutes with $E(S)$ for all Borel sets $S$ of $\Om$.
\end{theorem}
\begin{proposition}\label{sp_comm_prop}\cite{pedersen:ca_aut}
Let $x$ be a normal element of a $C^*$-algebra $A$, and let $B$ 
denote the smallest $C^*$-subalgebra of $A$ containing $x$. Then $B\cong C_0\left(  \mathfrak{Sp}_A\left(x \right)\right) $ 
and $\mathfrak{Sp}_A\left(x \right)\cong \mathfrak{Sp}_B\left(x \right)$. 
\end{proposition}

	\begin{defn}\label{approximate_unit_defn} \cite{pedersen:ca_aut}
	Let $A$ be a $C^*$-algebra. A net $\left\{u_\la \right\}_{\la \in \La}$ in $A_+$ with $\left\|u_\la \right\| \le 1$ for all $\la \in \La$ is called an \textit{approximate unit} for $A$ if $\la < \mu$ implies $u_\la < u_\mu$ and if $\lim \left\|x- xu_\la \right\| = 0$ for each $x$ in $A$. Then, of course, $\lim \left\|x- u_\la x \right\| = 0$ as well.
\end{defn}
	\begin{defn}\label{strict_topology_defn}\cite{pedersen:ca_aut}
	Let $A$ be a $C^*$-algebra.  The {\it strict topology} on the multiplier algebra $M(A)$ is the topology generated by seminorms 
	\be\label{strict_topology_norm_eqn}
	\vertiii{x}_a\bydef \|ax\| + \|xa\|,\quad a\in A.
	\ee
	If $\La$ is a directed set and $\left\{a_\la\in M\left( A\right) \right\}_{\la\in \La}$ is a net the we denote by $\bt\text{-}\lim_{\la\in\La }a_\la$ the limit of $\left\{a_\la \right\}$ with respect to the strict topology.
	If $x \in M(A)$  and a sequence of partial sums $\sum_{i=1}^{n}a_i$ ($n = 1,2, ...$), ($a_i \in A$) tends to $x$ in the strict topology then we shall write
	\begin{equation}\label{strict_topology_eqn}
		x = \bt\text{-}\sum_{i=1}^{\infty}a_i.
	\end{equation}
\end{defn}
 	\paragraph{}In this section I follow to \cite{murphy,pedersen:ca_aut}.
\begin{defn}\label{strict_topology}\cite{pedersen:ca_aut}
	Let $A$ be a $C^*$-algebra.  The {\it strict topology} on the multiplier algebra $M(A)$ is the topology generated by seminorms $\vertiii{x}_a = \|ax\| + \|xa\|$, ($a\in A$). We denote by $\bt\text{-}\lim$ the limit with respect to the strict topology.
	If $x \in M(A)$  and a sequence of partial sums $\sum_{i=1}^{n}a_i$ ($n = 1,2, ...$), ($a_i \in A$) tends to $x$ in the strict topology then we shall write
	\begin{equation*}
		x = \bt\text{-}\lim\sum_{i=1}^{\infty}a_i.
	\end{equation*}
\end{defn}
\begin{defn}\label{unitization_defn}\cite{rae:ctr_morita}
	A \text{unitization} of a $C^*$-algebra $A$ is  a $C^*$-algebra $B$ with identity and an injective *-homomorphism $\iota: A \hookto B$ such that $\iota\left(A\right)$ is an essential ideal of $B$. 
\end{defn}

\begin{example}\label{unitization_exm}
	Suppose $A$ is  $C^*$-algebra which has no identity. Then $A^+ = A \oplus \C$
	is a $*$-algebra with 
	\be\label{a_plus_eqn}
	\left( a \oplus \la\right)\left( b \oplus \mu\right) = \left(ab + \la b + \mu a\right)  \oplus \la \mu, \quad \left(a \oplus \la \right)^* = a^*\oplus \overline{\la}. 
	\ee
	It is proven in \cite{rae:ctr_morita} that there is the natural unique $C^*$-norm $\left\| \cdot\right\|_{A^+}$  on $A^+$ such that 
	$$
	\left\|a \oplus 0 \right\|_{A^+}=\left\|a \oplus 0 \right\|_{A}
	$$
	where $\left\|\cdot \right\|_{A}$ is the $C^*$-norm on $A$. Thus $A^+$ is an unital $C^*$-algebra, and the natural map $A\hookto A^+$ is a unitization.
\end{example}

\begin{definition}\label{multiplier_min_defn}
	Let $A$ be a $C^*$-algebra. The described in the Example \ref{unitization_exm} unitization   $\iota: A \hookto B$  is called \textit{minimal}.
\end{definition}

\begin{definition}\label{multiplier_max_defn}\cite{rae:ctr_morita}
	A unitization   $\iota: A \hookto B$  is called \textit{maximal} if for every embedding $j$ of $A$ as an essential ideal of a $C^*$-algebra $\phi: C\to B$ such that $\phi \circ j = \iota$. 
\end{definition}

\begin{rem}\label{multiplierem}
	It is proven in \cite{rae:ctr_morita} that for any $C^*$-algebra $A$ there unique maximal unitization
\end{rem}

\begin{definition}\label{multiplier_defn}
	We say that the maximal unitization of $A$ is the \textit{multiplier algebra} of $A$ and denote it by $M\left( A\right)$. 
\end{definition}
\begin{defn}\label{app_unit_defn}\cite{pedersen:ca_aut}
	Let $A$ be a $C^*$-algebra. A net $\left\{u_\la \right\}_{\la \in \La}$ in $A_+$ with $\left\|u_\la \right\| \le 1$ for all $\la \in \La$ is called an \textit{approximate unit} for $A$ if $\la < \mu$ implies $u_\la < u_\mu$ and if $\lim \left\|x- xu_\la \right\| = 0$ for each $x$ in $A$. Then, of course, $\lim \left\|x- u_\la x \right\| = 0$ as well.
\end{defn}
\begin{proposition}\label{mult_str_pos_prop}\cite{apt_mult}
	If $B$ is a $C^*$-subalgebra of $A$ containing an
	approximate unit for $A$, then  $M\left(B \right) \subset M\left( A\right)$  (regarding $B''$ as a subalgebra	of $A''$).
\end{proposition}
\begin{definition}\label{hered_defn}\cite{pedersen:ca_aut}
	A cone in the positive part of $C^*$-algebra $A$ is said to be \textit{hereditary} if $0 \le x \le y$, $y \in M$ implies $x \in M$ for each $x \in A$. A *-subalgebra $B$ of $A$ is \textit{hereditary} if $B_+$ is hereditary in $A_+$.
\end{definition}
\begin{lemma}\label{hered_lem}\cite{murphy}
	Let $B$ be a $C^*$-subalgebra of $C^*$-algebra $A$. Then $B$ is hereditary in $A$ if and only if $bab' \in B$ for all $b, b' \in B$ and $a \in A$.
\end{lemma}
\begin{defn}\label{center_defn}\cite{murphy}
	If $A$ is a $C^*$-algebra, its \textit{center} $C$ is the set of elements of $A$ commuting with every of $A$.
\end{defn}

\begin{thm}\label{vN_thm}\cite{pedersen:ca_aut}
	Let $M$ be a $C^*$-subalgebra of $B(\H)$, containing the identity operator. The following conditions are equivalent:
	\begin{itemize}
		\item $M=M''$ where $M''$ is the bicommutant of $M$;
		\item $M$ is weakly closed;
		\item $M$ is strongly closed.
	\end{itemize}
\end{thm}

\begin{defn}\cite{pedersen:ca_aut}
	Any $C^*$-algebra $M$ is said to be a {\it von Neumann algebra} or a {\it $W^*$- algebra} if $M$ satisfies to the conditions of the Theorem \ref{vN_thm}.
\end{defn}

\begin{remark}\label{irr_funcem}\cite{pedersen:ca_aut}
	From the condition (i) of the Theorem \ref{irred_thm} it turns out that irreducibility is the categorical property, i.e. if there is the equivalence between category of representations of $C^*$-algebras then any irreducible representation is mapped to irreducible one.  The equivalence between category of representations corresponds to the strong Morita equivalence (cf. \ref{strong_morita_sec}).
\end{remark}
\begin{definition}\label{equivepr_defn}\cite{pedersen:ca_aut}
	Let $A$ be a $C^*$-algebra.
	We say that two representations $\pi_1: A \to B\left( \H_1\right)$ and $\pi_2: A \to B\left( \H_2\right)$  are \textit{spatially equivalent} (or \textit{unitary equivalent}) if there is an isometry $u$ of $\H_1$ onto $\H_2$ such that $u\pi_1\left(a \right)u^*=\pi_2\left(a \right)$ for all $a \in A$. By the \textit{spectrum} of $A$ we understand the set $\hat A$ of spatially equivalence classes of irreducible representations. For any $x \in \hat A$ we denote by
	\be\label{rep_x_eqn}
	\begin{split}
		\rep_x : A \to B\left( \H\right) \quad \text{OR} \quad 	\rep^A_x : A \to B\left( \H\right)
	\end{split}
	\ee
	a representation which corresponds to $x$.  Sometimes we use alternative notation of the spectrum $A^\wedge\stackrel{\text{def}}{=}\hat A$.
\end{definition}

\begin{theorem}\label{jtop_thm}\cite{pedersen:ca_aut}
	The class $\left\{\mathrm{hull}\left( I\right)~|~I \subset A\right\}$ form the closed sets for a topology on $\check{A}$. There is a bijective order preserving isomorphism between open sets in this topology and the closed ideals in $A$.
\end{theorem}
\begin{definition}\label{jtop_defn}\cite{pedersen:ca_aut}
	The topology on  $\check{A}$ defined in the Theorem \ref{jtop_thm} is called be the \textit{Jacobson topology}. The topology on $\hat A$ induced by the surjecive map $\hat{A}\to \check{A}$ and the {Jacobson topology} on  $\check{A}$ is also called to be  the \textit{Jacobson topology}.
\end{definition}
\begin{remark}
	In the following text we consider the only Jacobson topology on both $\hat{A}$ and $\check{A}$.
\end{remark}
\begin{cor}\label{hered_closed_cor}\cite{murphy}
	Every closed ideal  of $C^*$-albebra is a hereditary $C^*$-subalbebra.
\end{cor}
\begin{remark}\label{ideal_open_incem}
	From the Theorem \ref{jtop_thm}, the Proposition \ref{hered_spectrum_prop} and the Corollary \ref{hered_closed_cor} any closed ideal $I$ corresponds of $C^*$-algebra $A$ corresponds to  the open subset $\mathcal U \subset \hat A$ such that there is the natural homeomorphism $\hat I \cong  \mathcal U$. Moreover if $I$ is an essential ideal then $\mathcal U$ is dense subset of $\hat A$.
\end{remark}
\begin{proposition}\label{lift_prop}\cite{murphy}
	Let $B$ be a $C^*$-algebra. For each positive functional $\phi$ on $B$ there is a norm preserving extension of $\phi$ to a positive functional on $A$. If $B$ is hereditary this extension is unique.
\end{proposition}
\begin{proposition}\label{state_prop}\cite{pedersen:ca_aut}
	If $B$ is a $C^*$-subalgebra of $A$ then each pure state of $B$ can be extended to a pure state of $A$.
\end{proposition}
\begin{proposition}\label{sur_prop}\cite{pedersen:ca_aut}
	If $B$ is a $C^*$-subalgebra of $A$ then for each irreducible representation $\rho: B \to B\left( \H'\right)$ of $B$ there is an irreducible representation $A \to B\left(\H \right)$ of $A$ with a closed subspace $\H_1 \subset \H$ such that $\rho_B  :B \to B\left(\H_1 \right)$ is spatially equivalent to $\rho: B \to B\left( \H'\right)$.  
\end{proposition}	
\begin{theorem}\label{norm_semi_thm}\cite{rae:ctr_morita}
	Suppose that $A$ is a $C^*$-algebra and that $a \in A$.
	\begin{enumerate}
		\item [(a)] The function $x \mapsto \left\| \rep_x\left(a \right) \right\|$ is lower semi-continuous on $\hat A$; that is $$\left\{x \in \hat A~|~  \left\| \rep_x\left(a \right) \right\|\le k \right\}$$ is closed for all $k \ge 0$.
		\item[(b)] For each $k > 0$, $\left\{x \in \hat A~|~  \left\| \rep_x\left(a \right) \right\|\ge k \right\}$ is compact.
	\end{enumerate}
\end{theorem}
\begin{thm}\label{peder_id_thm}  \cite{pedersen:ca_aut} For each $C^*$-algebra $A$ there is a dense hereditary ideal $K(A)$,
	which is minimal among dense ideals.
	
\end{thm}

\begin{defn}\label{peder_id_defn}
	The ideal $K(A)$ from the theorem \ref{peder_id_thm} is said to be the {\it Pedersen ideal of $A$}. Henceforth Pedersen ideal shall be denoted by $K(A)$.
\end{defn}

\subsection{Representations}
\begin{theorem}\label{irred_thm}\cite{pedersen:ca_aut}
	Let $\pi: A \to B\left(\H \right)$ be a nonzero representation of $C^*$-algebra $A$. The following conditions are equivalent:
	\begin{enumerate}
		\item [(i)] there are no non-trivial $A$-subspaces for $\pi$,
		\item[(ii)] the commutant of $\pi\left(A \right)$ is the scalar multipliers of 1,
		\item[(iii)] $\pi\left(A \right)$ is strongly dense in   $B\left(\H \right)$,
		\item[(iv)] for any two vectors $\xi, \eta \in \H$ with $\xi \neq 0$ there is $a \in A$ such that $\pi\left(a \right)\xi = \eta$,
		\item[(v)] each nonzero vector in $\H$ is cyclic for  $\pi\left(A \right)$,
		\item[(vi)]  $A \to B\left(\H \right)$ is spatially equivalent to a cyclic representation associated with a pure state of $A$.
	\end{enumerate} 
\end{theorem}
\begin{definition}\label{irred_defn}\cite{pedersen:ca_aut}
	Let $A \to B\left(\H \right)$ be a nonzero representation of $C^*$-algebra $A$. The representation is said to be \textit{irreducible} if it satisfies to the Theorem \ref{irred_thm}.
\end{definition}
\begin{definition}\label{spectrum_prime_primtive_defn}\cite{pedersen:ca_aut}
	An ideal $I$ in a $C^*$-algebra $A$ is \textit{prime}  if $xAy \subset I$  implies $x\in I$  or $y\in I$
	for all $x$, $y$ in $A$.  Equivalently, $I$ is prime if $I_1I_2\subset I$ implies  $I_1\subset I$ or 
	$I_2\subset I$ for any two (left, right, or two-sided) ideals $I_1$ and $I_2$ of $A$. 
	We say that $I$ is a \textit{primitive} ideal if $I= \ker\pi$ for some irreducible 
	representation $\pi: A \to B\left(\H \right)$. The set of prime ideals will be denoted by $\check{A}$ or $\mathrm{Prim}\left(A \right)$  and the set of primitive ideals will be denoted by $\hat A$. We say that $\check{A}$ is a \textit{prime spectrum} of $A$. The set  $\hat A$ is said to by a \textit{primitive  spectrum} or simply a \textit{spectrum} of $A$. For any $x \in \hat A$ denote  by $\rep_x: A \to B\left(\H_x\right)$ a corresponding irreducible representation. 
\end{definition}
\begin{empt}\label{jack_empt}\cite{pedersen:ca_aut}
	For each set $F$ in $\check A$ define a closed ideal $\ker (F)$ in $A$ by
	\be\label{ker_eqn}
	\ker (F)\bydef \bigcap_{\substack{t \in \check A \\ t\in F}} t
	\ee
	For each subset $I$ of $A$ define a set 
\be\label{hull_eqn}
	\mathrm{hull}\left( I\right)\bydef \left\{t\in A | I\subset t\right\}
\ee 	
	
	
\end{empt}
\begin{theorem}\cite{pedersen:ca_aut}\label{jack_defn}
The class $\left\{\mathrm{hull}\left( I\right)| I \subset A\right\}$ form the closed sets for a topology on $\check A$ There is a bijective, order preserving isomorphism between the open sets in 
this topology and the closed ideals in $A$. 
\end{theorem}
\begin{definition}\label{jackobseon_defn}
The topology on $A$ defined in \ref{jack_empt} is called the {\it Jacobson topology}. Note that a point $t \in \check A$ is closed if and only if t is a maximal ideal. 

\end{definition}
\begin{proposition}\label{hered_spectrum_prop}\cite{pedersen:ca_aut}
	If $B$ is a hereditary $C^*$-subalgebra of $A$ then the map $t \mapsto t\cap B$ is a homeomorphism between $\check A\setminus \mathrm{hull}\left( B\right)$ and $\check B$, where 
	$$
	\mathrm{hull}\left( B\right) = \left\{\left. x \in \hat A~\right|~ \rep_x\left(B \right)= \{0\} \right\} .
	$$ 
	Moreover  we have a commutative diagram:
	\\
	\begin{tikzcd}
		\hat A \setminus \mathrm{hull}\left(B \right)\arrow[d]\arrow[r, "\approx"] & \hat B\arrow[d]\\
		\check A \setminus \mathrm{hull}\left(B \right)\arrow[r, "\approx"] & \check B
	\end{tikzcd}
	\\ 	
\end{proposition}

\begin{thm}(Dauns Hofmann)\label{dauns_hofmann_thm}\cite{pedersen:ca_aut}
	For each $C^*$-algebra $A$ there is the natural isomorphism 
	\be\label{dauns_hofmann_eqn1}
	\mathrm{Center}\left(M\left( A\right) \right) \cong C_b\left( \check{A}\right) 
	\ee
	from the center of $M\left( A\right)$ onto the class of bounded continuous  functions on $\check{A}$. 
\end{thm}
\begin{empt}\label{low_up_empt}
	According to \cite{ped_semi} for any $C^*$-algebra $A$ the class $A^m$ 
is a class of elements in $A''$ which can be approximated
weakly from below with self-adjoint operators of the form $x + \a$ a with $x$ in $A$
and $\a$ in $\R$. Define $A_m \bydef -A^m$,
\end{empt}
\begin{theorem}\label{low_up_thm} \cite{ped_semi}
The set $A^m \cap A_m$, is equal to the self-adjoint part of $M\left( A\right)$. 
\end{theorem}
\begin{definition}\label{ctr_homo_defn}\cite{fell:operator_fields} 
	A $C^*$-algebra is \textit{homogeneous of order} $n$ if every irreducible $*$-representation is of the same finite dimension $n$.
\end{definition}
	\begin{proposition}\label{less_n_pi_prop}\cite{pedersen:ca_aut}
	The subset $_n\check{A}$ of $\check{A}$ corresponding to irreducible representations of $A$ with finite dimension less or equal to $n$ is closed. The set $_n\check{A}\setminus_{n-1}\check{A}$ of $n$-dimensional representations is a Hausdorff  space in its relative topology.
\end{proposition}

\begin{definition}\label{faithfulrepr_defn}\cite{murphy}
	A representation $\rho : A\to B\left( \H\right)$ is called \textit{faithful} if the *-homomorphism $\rho$ is injective.
\end{definition}

\begin{definition}\label{nondegeneraterepr_defn}\cite{matro:hcm}
	A representation $\rho : A\to B\left( \H\right)$ is called \textit{nondegenerate} if for any $\xi \in \H$  there exists an element $a \in A$ such that $\rho\left(a \right)\xi \neq 0$. 
\end{definition}
\begin{lemma}\label{nondegenerateepr_lem}\cite{matro:hcm}
	A representation $A\to B\left( \H\right)$ is {nondegenerate} if $\rho\left(A\right)\H$ is dense in $\sH$. 
\end{lemma}

\begin{definition}\label{atomicepr_defn}\cite{pedersen:ca_aut}
	Let $A$ be a $C^*$-algebra with the spectrum $\hat A$. We choose for any $t \in \hat A$ a pure state $\phi_t$ and  associated representation $\pi_t: A \to B\left(\H_t\right)$.	The representation 
	\be
	\pi_a = \bigoplus_{t \in \hat A} \pi_t \quad \text{on the closure } \H_a \text{ of an algebraic direct sum}\quad  \bigoplus_{t \in \hat A} \H_t
	\ee
	is called the (reduced) \textit{atomic representation} of $A$. Any two atomic representations are unitary equivalent and any atomic representation of $A$ is faithful and nondegenerate  (cf.  Definitions \ref{faithfulrepr_defn}, \ref{nondegeneraterepr_defn} and \cite{pedersen:ca_aut}).
\end{definition}
\begin{definition}\label{multiplier_el_defn}\cite{matro:hcm}
	Let $\rho: A\hookto B\left( \H\right)$ be a faithful {nondegenerate} (cf. Definitions \ref{faithfulrepr_defn}, \ref{nondegeneraterepr_defn}) representation, so we assume $A \subset B\left( \H\right)$. An operator $x \in B\left(\H\right)$ is called (two-sided) \textit{multiplier} if 
	\be\label{multiplier_el_eqn}
	xa \in A, \quad ax\in A
	\ee
	for each $a\in A$. Denote by $M\left(A\right)$ the set of all multipliers. It is easy to see that $M\left(A\right)$ is an involutive unital algebra.
\end{definition}
\begin{definition}\label{lrm_defn}\cite{matro:hcm}
	Let $A\to B\left(\H \right)$ be a faithful nondegenerate representation. An operator $x\in  B\left(\H \right)$ is said to be a \textit{left multiplier} of $A$ if
	$$
	xa\in A
	$$
	for any $a \in A$. Denote by $\mathbf{LM}(A)$ the set of all left multipliers. Similarly one can define \textit{right multipliers} $\mathbf{RM}(A)$.
\end{definition}

\begin{definition}\label{lrc_defn}\cite{matro:hcm}
	If $A$ is a $C^*$-algebra then a linear map $\la: A\to A$ is said to be a \textit{left centralizer} if
	\be
	\la\left(ab\right)= 	\la\left(a\right) b \quad \forall a, b \in A.
	\ee
	Similarly one defines a \textit{right} centralizer. Denote the spaces of left and right centralizers by $\mathbf{LC}(A)$ and  $\mathbf{RC}(A)$.
\end{definition}
\begin{lem}
	If $\rho\in  \mathbf{RC}(A)$ then $\rho^*\in  \mathbf{LC}(A)$ where $\rho^*\left(a \right)\bydef \left(\rho\left( a^* \right) \right)^*$. 
\end{lem}

\begin{lem}\label{lrc_lem}\cite{matro:hcm}
	Each left centralizer, and each right centralizer is bounded.
\end{lem}
\begin{proposition}\label{lrc_prop}\cite{matro:hcm}
	Let $A\to B\left(\H \right)$ be a faithful nondegenerate representation. Then there exists a one-to-one isometric correspondence between left, right multipliers and  left, right centralizers respectively.
\end{proposition}
\begin{definition}\label{double_centralizer_defn}\cite{matro:hcm}
	A pair $\left(L, R\right)$ of maps
	\be\label{double_centralizer_eqn}
	L: A \to A, \quad R: A \to A\quad R\left(a \right) b = a L\left( b\right) \quad \forall a, b \in A
	\ee
	is called a \textit{double centralizer}. Let us denote $\mathbf{DC}\left(A \right) $ the space of all double  centralizers of $A$.
\end{definition}

\begin{proposition}\label{dc_prop}\cite{matro:hcm}
	If $\left(L, R\right)\in \mathbf{DC}\left(A \right)$ then
	\begin{enumerate}
		\item [(i)] $L\left( ab\right) = L\left( a\right)b$ and $R\left( ab\right) = aR\left( b\right)$;
		\item [(ii)] $L$ and $R$ are linear;
		\item [(iii)] $L$ and $R$ are bounded and $\left\|L \right\|= \left\|R \right\|$;
	\end{enumerate}
	The set $\mathbf{DC}\left(A \right) $ with operations
	\bean
	\left(L_1, R_1\right) + \left(L_2, R_2\right)\bydef \left(L_1 + L_2, R_1+R_2\right),\quad z\left(L,R\right)\bydef \left(zL, zR\right) \quad \forall z \in \C,\\
	\left(L_1, R_1\right)  \left(L_2, R_2\right)\bydef \left(L_1  L_2, R_2R_2\right),\\
	\left(L, R\right)^*\bydef \left(R^*, L^*\right),\\
	L^*\left( a\right)\bydef \left( L\left( a^*\right)\right)^*, \quad R^*\left( a\right)\bydef \left( L\left( a^*\right)\right)^*\quad \forall a \in A.
	\eean 
	is a normed involutive algebra with respect to the norm
	$$
	\left\|\left(L, R\right) \right\|\bydef \left\|L \right\|= \left\|R \right\|.
	$$
\end{proposition}
\begin{remark}\label{double_centralizerem}
	In \cite{matro:hcm} it is proven  there is a natural $*$-isomorphism $\mathbf{DC}\left(A \right)\cong M\left(A\right)$. 
\end{remark}

	\subsection{GNS construction}\label{gns_constr_sec}
\paragraph*{}

Any state $\tau$ of  $C^*$-algebra $A$  induces a GNS representation  \cite{murphy}. There is a $\mathbb{C}$-valued product on $A$ given by
\begin{equation}\label{tau_prod_eqn}
	\left(a, b\right)\bydef\tau\left(a^*b\right).
\end{equation}
This product induces a product on $A/\mathcal{I}_\tau$ where 
\be\label{tau_ideal_eqn}
\mathcal{I}_\tau\bydef\left\{\left.a \in A \ \right| \ \tau(a^*a)=0\right\}
\ee 
So $A/\mathcal{I}_\tau$ is a pre-Hilbert space. Denote by $L^2\left(A, \tau\right)$ the Hilbert  completion of $A/\mathcal{I}_\tau$.  The Hilbert space  $L^2\left(A, \tau\right)$ is a space of a  GNS representation  $A\to B\left(L^2\left(A, \tau\right) \right)$ (or equivalently $A\times L^2\left(A, \tau\right)\to L^2\left(A, \tau\right) $) which comes from the Hilbert norm completion of the natural action $A \times A/\mathcal{I}_\tau \to A/\mathcal{I}_\tau$. The natural map  $A \to A/\mathcal{I}_\tau$ induces the homomorphism of left  $A$-modules
\be\label{from_a_to_l2_eqn}
\begin{split}
	f_\tau : A \to L^2\left(A, \tau\right),\\
	a \mapsto a + \mathcal{I}_\tau
\end{split}
\ee
such that $f_\tau\left(A \right)$ is a dense subspace of $L^2\left(A, \tau\right)$.
\begin{theorem}\label{stateepr_thm}\cite{pedersen:ca_aut} 
	For each positive functional $\tau$ on a $C^*$-algebra $A$ there is a cyclic representation $\pi_{\tau}: A \to B\left( L^2\left(A, \tau\right)\right) $ with a cyclic vector $\xi_{\tau}\in L^2\left(A, \tau\right)$ such that $\left( \pi_{\tau}\left( x\right) \xi_{\tau}, \xi_{\tau} \right)= \tau\left(a \right)$ for all $x \in A$.   
\end{theorem}
\begin{defn}\label{gns_defn}\cite{murphy,pedersen:ca_aut}
	The given by the Theorem \ref{stateepr_thm} representation is said to be a \textit{GNS representation}. We say that the representation 
	$\pi_{\tau}$ given by the Theorem \ref{stateepr_thm} is the cyclic representation \textit{associated with} $\tau$.
\end{defn}
\begin{proposition}\label{stateepr_prop}\cite{pedersen:ca_aut}
	Let $\phi$ be a positive functional on a $C^*$-algebra A and let 
	$\pi_\phi$ its associated representation. For each positive functional $\psi \le \phi$, 
	there is a unique element $a \in \pi\left(A \right)'$  with $0\le a \le 1$ such that 
	$$
	\pi_\psi\left( x\right) = \left(\pi_\phi\left(x \right)  a \xi_\phi, \xi_\phi \right).	
	$$
	for all $x \in A$.
\end{proposition}

\begin{definition}\label{orth_func_defn}(cf. \cite{pedersen:ca_aut})
	Let $A$ be  $C^*$-algebra. Two positive functionals $\phi: A \to \C$ and  $\psi: A \to \C$ are \textit{orthogonal} if
	$$
	\left\|\phi - \psi \right\| = \left\|\phi \right\| +  \left\|\psi \right\|. 
	$$
	We write
	\bean
	\phi \perp \psi.
	\eean
	If $\pi_\phi : A \to B\left( \H_\phi\right)$ and $\pi_\psi : A \to B\left( \H_\psi\right)$ are representations which correspond to $\phi$ and $\psi$ respectively then we write
	\be\label{orth_func_eqn}
	\pi_\phi \perp \pi_\psi.
	\ee
	
\end{definition}

\begin{corollary}\label{dominep_cor}\cite{pedersen:ca_aut}
	If $\varphi$ and $\psi$ are positive functionals on a $C^*$-algebra $A$ and $\psi$  
	is dominated by a multiple of $\varphi$ then the representation $\pi_\psi: A \to B\left( \H_\psi\right)$  is spatially 
	equivalent to a subrepresentation of $\pi_\varphi: A \to B\left( \H_\varphi\right)$. 
\end{corollary}

\subsection{Enveloping von Neumann algebra}
\begin{defn} \cite{pedersen:ca_aut}
	Let $A$ be a $C^*$-algebra, and let $S$ be the state space of $A$. For any $s \in S$ there is an associated representation $\pi_s: A \to B\left( \H_s\right)$. The representation $\bigoplus_{s \in S} \pi_s: A \to \bigoplus_{s \in S} B\left(\H_s \right)$ is said to be the \textit{universal representation}. The universal representation can be regarded as $A \to B\left( \bigoplus_{s \in S}\H_s\right)$.  
\end{defn} 
\begin{defn}\label{env_alg_defn}\cite{pedersen:ca_aut}
	Let   $A$ be a $C^*$-algebra, and let $A \to B\left(\H \right)$ be the universal representation $A \to B\left(\H \right)$. The strong closure of $\pi\left( A\right)$ is said to be   the  {\it enveloping von Neumann algebra} or  the {\it enveloping $W^*$-algebra}  of $A$. The enveloping  von Neumann algebra will be denoted by $A''$.
\end{defn}
\begin{prop}\label{env_alg_sec_dual}\cite{pedersen:ca_aut}
	The enveloping von Neumann algebra $A''$ of a $C^*$-algebra $A$ is isomorphic, as a Banach space, to the second dual of $A$, i.e. $A'' \approx A^{**}$.
\end{prop}

\begin{thm}\label{env_alg_thm}\cite{pedersen:ca_aut}
	For each non-degenerate representation $\pi: A \to B\left(\H \right)$ of a $C^*$-algebra $A$ there is a unique  normal morphism of  $A''$ onto $\pi\left( A\right)''$ which extends $\pi$.   
\end{thm}

\subsection{The Gelfand representation theory for $C^*$-algebra}
\begin{empt}\label{min_empt} Here I follow to \cite{akerman_gelfand}.
	Consider the $W^*$-algebra $A^{**}$, the second
	Banach space dual of $A$. There exists a central projection
	$z \in A^{**}$ which is the supremum of all the minimal projections in $A^{**}$. Set 
	\be\label{min_eqn}
	M \bydef z A^{**}
	\ee.
	 The minimal projections of $M$ are in
	one to one correspondence with the maximal left ideals of $A$, so that we can define a structure on this set of minimal
	projections instead of directly on the maximal left ideals. Naturally
	the first thing we "build" is the algebra $M$. 
	We can consider $A\subset N$ since $A \subset A^{**}$ and $A \hookto z A$ is
	a $*$-isomorphism. (Also we can view $M$ as the direct sum
	of irreducible representations of $A$, one from each equivalence class.)
	A projection $p$ in $M$ is open if there exists a closed left ideal $I$ of
	$A$ such that the weak$*$ closure $\overline{I}$ of $I$ in $M$ is of the form $Mp$. 
	The {\it open} projections are analogous to the open sets of a topology.
	If $A$ were Abelian, $M$ would be the algebra of all bounded complex
	function on its maximal ideal space $K$. The open projections
	would be characteristic functions of open sets of $K$ for the hull-kernel
	topology. A self-adjoint operator $b$ in $M$ actually lies in
	$A$ if and only if the spectral projections of $b$ corresponding to
	open sets of real numbers are open projections in the above sense. This is a restatement of Gelfand's theorem since a function is continuous
	if and only if its inverse images of open sets are open.
	
\end{empt}
\begin{definition}\label{open_close_ideal_defn}\cite{akerman_stone}
	A projection $p \in A^{**}$ (cf. Definition \ref{env_alg_defn}) is {\it open} if there exists a
	net $\left\{a_\la \right\}_{\la \in \La}$ such that $0\ge a_\la \uparrow p$
	$$
	\lim_{\la\in \La} a_\la = p
	$$
	in the weak* topology of $A^{**}$). If $p$ is open, we say $p' \bydef 1=p$ is {\it closed}.
\end{definition}
\begin{remark}
	The net $\left\{a_\la \right\}_{\la \in \La}$ can be chosen by the Theorem \ref{left_ideal_thm}
\end{remark}

\begin{proposition}\label{open_close_ideal_prop}\cite{akerman_stone}
	The minimal closed projections in  $A^{**}$ are the
	minimal projections.
\end{proposition}
\begin{definition}\label{closure_defn}\cite{akerman_stone}
If $p \in A^{**}$ is a projection, we denote by $\overline{p}$ the
{\it closure} of $p$ which is the smallest closed projection majorizing $p$.
\end{definition}

\begin{definition}\label{compact_p_defn}\cite{akerman_gelfand}
 A projection $p\in M$ is {\it compact} if $p$ is {closed}
and there exists $b \in A_+$ with $bp = p$.
\end{definition}
\begin{theorem}\cite{akerman_gelfand}
If $p$ is closed and there exists be A with $b \ge p$,
then $p$ is compact.\end{theorem}
\begin{definition}\label{coninuous_p_defn}\cite{akerman_gelfand}
A self-adjoint operator be $A^{**}$ is {\it continuous}
if each spectral projection of $b$ corresponding to an open subset of
the spectrum of  $p$ is also open.
Now we can state our best Gelfand representation theorem.
\end{definition}

\begin{theorem}\label{coninuous_p_thm}\cite{akerman_gelfand}
The self-adjoint elements of $A$ are exactly those
continuous elements $b$ of $A^{**}$ such that the spectral projections of $b$
corresponding to closed subsets of the spectrum of $b$ which don't contain
$0$ are compact (i.e., b "vanishes at $\infty$").
\end{theorem}
\begin{proposition}\label{close_int_akk}\cite{akerman_stone}
 If $\left\{p_\la \right\}$ is a set of closed projections, then
p = A, p, is closed.
\end{proposition}
\begin{definition}\label{fin_p_defn}\cite{akerman_stone}
	A family $\left\{p_\la\right\}_{\la \in \La}$ family of closed projections is a \textit{family with the finite intersection property} if for any subset $\mathscr L \subset \La$ one has
	\be
\left( \forall \text{ finite set } K\subset \mathscr L \quad 	\bigwedge_{\la \in K} p_\la \neq 0\right)  \quad \Rightarrow \quad \bigwedge_{\la\in  \mathscr L} p_\la = 0,
	\ee
\end{definition}
\begin{proposition}\label{fin_int_prop}\cite{akerman_stone}
Suppose $\left\{p_\a\right\}$ is a family of closed projections
in $A^{**}$ with the finite intersection property (i.e. $\wedge_{\la \in K} p_\la \neq 0$ for each finite set $K$). Then $\wedge_{\a} p_\a \neq 0$,
\end{proposition}
\begin{conjecture}\label{comp_conj}\cite{akerman_gelfand}
A regular  projection $p \in M$ is {\it compact}
if for every family $\left\{p_\a\right\}$ of closed projections such that the
family $\left\{p_\a \wedge p\right\}$ has the finite intersection property (cf. Proposition \ref{fin_int_prop}), then $p \wedge \bigwedge_\a p_\a\neq 0$.
\end{conjecture}
\begin{theorem}\label{gelfand_compact_thm}\cite{akerman_gelfand}
If p is q-compact, then for any family  $\left\{p_\la\right\}_{\la\in \La}$ of closed projections such that the family $\left\{p_\la \wedge p\right\}$ has the finite intersection property (cf. Definition  \ref{fin_p_defn}), then $p \wedge \bigwedge_{\la\in \La} p_\la\neq 0$.
\end{theorem}
\begin{remark}\label{compact_pem}
	It is explained in \cite{akerman_stone} that the Theorem \ref{gelfand_compact_thm} means that the Definition \ref{compact_p_defn} is a generalization of compact subspace of any topological space.
\end{remark}

\begin{lemma}\cite{akerman_gelfand}
Suppose $B$ is a $C^*$-algebra, $b \in B^+$, $p \in B^{**}$ a projection
and $\left\{a_\a\right\}\subset B$ an increasing net of positive elements with
$\left\|b^{1/2} - b^{1/2} a_\a\right\| \xrightarrow[\a]{~}   0$ If $b \ge p$ (considering $B \subset B^{**}$),then $\left\|b^{1/2} - b^{1/2} a_\a\right\| \xrightarrow[\a]{~}  0$.
\end{lemma}
\begin{proposition}\label{open_closed}\cite{akerman_stone}
	 A projection $p \in A^{**}$ is open and closed if and only if
	$p \in A$.
\end{proposition}
		\subsection{Hilbert modules and compact operators}\label{hilbert_modules_chap}
\begin{definition}\label{banach_non_defn}\cite{rae:ctr_morita}
	A left  $A$-module $X$ \textit{Banach} $A$-\textit{module} if $X$ is a Banach space and $\left\| a \cdot x\right\| \le \left\|a \right\| \left\| x\right\|$ for all $a\in A$ and $x\in\sX$.
	A  \textit{Banach} $A$-{module} is \textit{nondegenerate}  
	is nondegenerate if $\text{span}\left\{\left.a \cdot x\right|a\in A\quad x\in \sX\right\}$
	is dense in $X$. We then have \\$a_\la\cdot x \to x$
	whenever $x\in\sX$ and $\left\{a_\la\right\}$
	is a bounded approximate identity for $A$. 
\end{definition}

\begin{proposition}\label{banach_non_prop}\cite{rae:ctr_morita}
	Suppose that $X$ 
	is a nondegenerate Banach $A$-module. Then 
	every element of $X$ 
	is of the form $a\cdot x$
	for some $a\in A$ and $x\in X$.
\end{proposition}

\begin{definition}\label{hilbert_module_defn} Rieffel \cite{Rieffel:74a}
	Let~$B$ be a $C^*$-algebra.  A \emph{pre-Hilbert $B$-module} is a right
	$B$-module~$X$ (with a compatible $\C$-vector space structure),
	equipped with a conjugate-bilinear map (linear in the second variable)
	$\left\langle{\blank},{\blank}\right\rangle_B\colon X\times X\to B$ satisfying
	\begin{enumerate}
		\item[(a)] $\left\langle{x},{x}\right\rangle_B\ge0$ for all $x\in X$;
		\item[(b)] $\left\langle{x},{x}\right\rangle_B=0$ only if $x=0$;
		\item[(c)] $\left\langle{x},{y}\right\rangle_B=\left\langle{y},{x}\right\rangle_B^\ast$ for all $x,y\in X$;
		\item[(d)] $\left\langle{x},{y\cdot a}\right\rangle_B=\left\langle{x},{y}\right\rangle_B\cdot a$ for all $x,y\in X$, $a\in B$.
	\end{enumerate}
	The map $\left\langle{\blank},{\blank}\right\rangle_B$ is called a \emph{$B$-valued inner product
		on~$X$}.
	Following equation
	\be\label{hilbert_module_norm_eqn}
	\|x\|=\|\left\langle{x},{x}\right\rangle_B\|^{\nicefrac{1}{2}}
	\ee defines a norm on~$X$.
	If~$X$ is complete with respect to this norm, it is called a $C^*$-\emph{Hilbert
		$B$-module} or simply a \emph{Hilbert
		$B$-module}.  
\end{definition}
\begin{definition}\label{full_hilb_defn}\cite{rae:ctr_morita}
	A Hilbert  $A$-module $X_A$  is a \textit{full} Hilbert $A$-module if the ideal 
	$$
	I\bydef \text{span}\left\{\left.\left\langle\xi, \eta \right\rangle_A\right| \xi, \eta\in X_A \right\} 
	$$
	is dense in $A$. 
\end{definition}
\begin{remark}\label{full_hilbem}
	Suppose that $A$ is a $C^*$-algebra and that $p$
	is a projection in $A$ 
	(or $M(A)$ ). Following facts are proven in the Example 2.12 of \cite{rae:ctr_morita}.
	Then $Ap \bydef \left\{\left.ap\right| a\in A\right\}$
	is a Hilbert $pAp$ module 
	with inner product 
	\be\label{full_hilb_eqn}
	\left\langle ap, bp \right\rangle_{pAp} \bydef pa^*bp
	\ee
	This Hilbert module is full. Similarly, $pA$ 
	is a Hilbert $A$-module which is full over  the ideal $\overline{ApA}\bydef \overline{\text{span}}\left\{\left.a p b\right| a,b \in A\right\}$
	generated by $p$, and $Ap$ 
	is itself a full left  Hilbert $\overline{ApA}$-module. 
\end{remark}

\begin{remark}\label{polarization_equalityem}
	For any $C^*$-pre-Hilbert $X$ module, or more  precisely, for any sesquilinear form $\left\langle \cdot , \cdot \right\rangle$ the \textit{polarization equality}
	\be\label{polarization_equality_eqn}
	\begin{split}
		4 \left\langle \xi , \eta \right\rangle=\sum_{k = 0}^3i^k\left\langle \xi + i^k\eta, \xi + i^k\eta \right\rangle
	\end{split}
	\ee 
	is obviously satisfied for all $\xi, \eta \in X$.  If $\H$ is a Hilbert space then there is the following analog of the identity \eqref{polarization_equality_eqn}
	\be\label{polarization_hilb_equality_eqn}
	\\
	\left( \xi , \eta \right)_{\H} = \frac{\sum_{k = 0}^3i^k\left\| \xi + i^k\eta \right\|}{4}, \quad\forall \xi, \eta \in X
	\ee
\end{remark}
\begin{definition}\label{adjointable_operator_defn}\cite{matro:hcm}
	Let $X$, $Y$ be Hilbert modules over $C^*$-algebra $A$. A bounded $\C$-linear $A$-homomorphism from $X$ to $Y$ is called an \textit{operator} from $X$ to $Y$. Let $\Hom_A\left(X, Y \right)$ denote the set of all {operators} from $X$ to $Y$. If $Y = X$ then $\End_A\left(X\right)\bydef \Hom_A\left(X, X \right)$ is a Banach algebra. We say that $L\in \Hom_A\left(X, Y \right)$ 
	\textit{adjointable} if there is $L^*\in\Hom_A\left(Y, X\right)$ such that 
	$$
	\left\langle \eta, L \xi\right\rangle_A=  \left\langle L^*\eta,  \xi\right\rangle_A \quad  \forall \xi\in X,~ \eta \in Y.
	$$
	Denote by $\Hom^*_A\left(X, Y \right)\subset \Hom_A\left(X, Y \right)$ of all  {adjointable} operators. The set $\End^*_A\left(X \right)\bydef \Hom^*_A\left(X, X \right)$ is a  $C^*$-algebra.
\end{definition}
\begin{defn}\label{compact_a_operator_defn}\cite{pedersen:ca_aut}
	If $X$ is a $C^*$ Hilbert $A$-rigged module then
	denote by $\theta_{\xi, \zeta} \in \End^*_A\left(X \right)$   such that
	\begin{equation}\label{rank_one_eqn}
		\forall  \xi, \eta, \zeta \in X\quad		\theta_{\xi, \zeta} (\eta) = \zeta \langle\xi, \eta \rangle_X.
	\end{equation}
	The $C^*$-norm closure of  a generated by such endomorphisms ideal is said to be the {\it algebra of compact operators} which we denote by $\mathcal{K}(X)$. The $\mathcal{K}(X)$ is an ideal of  $\End^*_A(X)$.
\end{defn}
\begin{remark}
	Also we shall use a following notation 
	\be\label{rank_one_notation_eqn}
	\begin{split}
		\xi\rangle \langle \zeta \stackrel{\text{def}}{=} \theta_{\xi, \zeta}: X_A \to X_A,\\
		\eta \mapsto \xi \langle\zeta, \eta \rangle_X.
	\end{split}
	\ee
\end{remark}
\begin{theorem}\label{comp_mult_thm}\cite{blackadar:ko}
	Let $X_A$ is a Hilbert $A$-module.
	The $C^*$-algebra of adjointable maps  $\End^*_A\left( X_A\right)\bydef \Hom^*_A\left( X_A, X_A\right)$ is naturally isomorphic to the algebra $M\left(\K\left( X_A\right) \right)$ of multiplies of compact operators $\K\left( X_A\right)$. 
\end{theorem}

\begin{remark}
	The text of the Theorem \ref{comp_mult_thm} differs from the text of the Theorem 13.4.1 \cite{blackadar:ko}. It is made for a compatibility with this article.
\end{remark}
\begin{defn}\label{standard_h_m_defn}(cf. \cite{matro:hcm}) 
	The direct sum of countable  number of copies of a Hilbert module $X$ is denoted by $\ell^2\left(X \right)$. If $A$ is a $C^*$-algebra then the Hilbert module  
	$\ell^2\left( A\right)$ is said to be the \textit{standard Hilbert $A$-module} over $A$. If $A$ is a unital then $\ell^2\left( A\right)$ possesses the standard basis $\left\{\xi_j\right\}_{j \in \N}$.
\end{defn}
\begin{thm}\label{kasparov_stab_thm}\textbf{Kasparov Stabilization or Absorption Theorem.}\cite{blackadar:ko}
	If $X_A$ is a countably generated Hilbert $A$-module, then $X_A \oplus \ell^2\left( A\right) \cong \ell^2\left( A\right)$.
\end{thm}

\begin{empt}\label{hm_dual_empt}\cite{matro:hcm}
	For a Hilbert $C^*$-module $X$ over $C^*$-algebra $A$ let us denote by $X'$ the set of  all bounded $A$-linear maps  from  $X$ to $A$. 
	The formula
	\be
	\left( f \cdot a\right)\left( \xi\right) \bydef a^*f\left( x\right); \quad a \in A, \quad f \in X' , \quad \xi \in X
	\ee
	introduced the structure if right $A$-module on $X'$.
	The elements of $X'$ are called \textit{functionals} on a Hilbert module $X$. Note that there is an obvious isometric inclusion
	\be\label{to_dual_eqn}
	\begin{split}
		X \subset X',\\
		x \mapsto\left\langle x, \cdot \right\rangle.
	\end{split}
	\ee
	The space $X'$ is called the \textit{dual Banach module} of $X$.
\end{empt}	
\begin{definition}\label{hm_selfdual_defn}\cite{matro:hcm}
	A Hilbert module $X$ is called  \textit{self-dual} if $X\cong X'$.
\end{definition}
\begin{definition}\label{hm_MI_defn}\cite{matro:hcm}
	A $C^*$-algebra is said to be \textit{module}-\textit{infinite} (MI) if each countably generated Hilbert $A$-module is projective and finitely generated if and only if it is self-dual.
\end{definition}

\begin{example}\label{fol_tor_exm}\emph{Linear foliation on torus}. Here I follow to \cite{candel:foliI,connes:ncg94}.
	Consider a vector field $\tilde{X}$ on $\R^2$ given by
	\[
	\tilde{X}=\alpha\frac{\partial}{\partial
		x}+\beta\frac{\partial}{\partial y}
	\]
	with constant $\alpha$ and $\beta $. Since $\tilde{X}$ is
	invariant under all translations, it determines a vector field $X$
	on the two-dimensional torus ${\T}^2={\R}^2/{\Z}^2$. The vector
	field $X$ determines a foliation $\mathcal{F}$ on ${\T}^2$. The leaves of
	$\mathcal{F}$ are the images of the parallel lines
	$\tilde{L}=\{(x_0+t\alpha, y_0+t\beta): t\in\R\}$ with the slope
	$\theta=\beta/\alpha $ under the projection $\R^2\to \T^2$.
	In the case when $\theta$ is rational, all leaves of $\mathcal{F}$ are
	closed and are circles, and the foliation $\mathcal{F}$ is determined by
	the fibers of a fibration $\T^2\to S^1$. In the case when $\theta$
	is irrational, all leaves of $\mathcal{F}$ are everywhere dense in $\T^2$. We say that  $\left(\T^2, \mathcal{F}_\th \right)$ is  the \textit{Kronecker foliation} $dy = \th dx$ of the 2-torus $ \T^2 \bydef \R^2/\Z^2$ 
	with natural coordinates $\left((x, y \right)\in \R^2$. Here $\th\in (0,1)$ is an irrational number.
	The graph $\G$ of this foliation is the manifold $\G \cong \T\times \R$ with range and source maps  $\G \to \T^2$ given by
	\bean
	r((x, y), t) = (x + t, y + \th t).\\
	s((x, y), t) = (x, y)
	\eean
	and with composition given by $((x, y), t)((x_0, y_0), t_0) = ((x_0, y_0), t + t_0)$ for any pair of	composable elements.
	Every closed geodesic of the at torus $\T^2$ yields a compact transversal. More precisely,
	for each pair $(p, q)$ of relatively prime integers we let $N_{p,q}$ be the submanifold of $\T^2$
	given by
	\be\label{foli_npq_eqn}
	N_{p,q} = \left\{(ps, qs) | s\in \R/\Z\right\}
	\ee
	The graph $\G$ reduced by $N = N_{p,q}$, i.e. $\G^N_N \bydef \left\{\left.\ga \in \G \right| r(\ga)\in N, \quad s(\ga)\in N\right\}$
	the manifold $\G^N_N=\T\times \Z$ with range and source maps given by:
	\bean
	r(x, n) = x + n\th' ;\\
	s(x, n) = x
	\eean
	where $\th'$ is determined uniquely by any pair $(p_0, q_0)$ of integers such that $pq_0 - p_0q= 1$, $\quad \th' = \frac{p'\th - q'}{p\th - q}$
	If $N = N_{0,1}$ then 
	\bean
	r(x, n) = x + n\th ;\\
	s(x, n) = x
	\eean
	Let $U$ be an element of $C\left(\T\right)$ which comes from
	\bean
	U:\R \to \C,\\
	t \mapsto e^{2\pi i t}.
	\eean
	If both  $u, v\in C^\infty_c\left(\G^N_N \right)$ are such that
	\be\label{foli_uv_eqn}
	\begin{split}
		u\left(x, n \right) = \begin{cases}
			U\left(x \right) & n = 0\\
			0 & n\neq 0
		\end{cases};\\
		v\left(x, n \right) = \begin{cases}
			1 & n = 1\\
			0 & n\neq 1
		\end{cases},\\
	\end{split}
	\ee
	then from  the equations \eqref{foli_pseudo_eqn}
	one has
	\bean
	u^*=u^{-1},\\
	v^*=v^{-1},\\
	vu = e^{2\pi i\th } uv.
	\eean
	Indeed above equations describe a noncommutative 2-torus $C\left(\T^2_\th\right)$ (cf Definition \ref{nt_defn}). Using this fact one deduce that 
	\be\label{foli_nt_eqn}
	C^*\left( \G^N_N\right) \cong C\left(\T^2_\th\right).
	\ee
\end{example}


		\subsection{Hermitian  modules and functors}
\paragraph*{}
In this section we consider an analogue of the $A \otimes_B - : \ _B\mathcal{M}\to _A\mathcal{M}$ functor or an algebraic generalization of continuous maps. Following text is in fact a citation of \cite{rieffel_morita}.

\begin{defn}\label{herm_mod_defn}
	\cite{rieffel_morita} Let $B$ be a $C^*$-algebra. By a (left) {\it Hermitian  $B$-module} we will mean the Hilbert space $\H$ of a nondegenerate $*$-representation $A \rightarrow B(\H)$. Denote by $\mathbf{Herm}(B)$ the category of Hermitian  $B$-modules.
\end{defn}
\begin{empt}
	\cite{rieffel_morita} Let $A$, $B$ be $C^*$-algebras. In this section we will study some general methods for construction of functors from  $\mathbf{Herm}(B)$ to  $\mathbf{Herm}(A)$.
\end{empt}

\begin{defn}\label{corr_defn}\cite{rieffel_morita}
	Let $A$ and $B$ be $C^*$-algebras. By a {\it Hermitian  $B$-rigged $A$-module} we mean a $C^*$-Hilbert $B$ module, which is a left  $A$-module by means of $*$-homomorphism of $A$ into $\mathcal{L}_B(X)$.
\end{defn}

\begin{empt}\label{herm_functor_empt}
	Let $X$ be a Hermitian $B$-rigged $A$-module. If $V\in \mathbf{Herm}(B)$ then we can form the algebraic tensor product $X \otimes_{B_{\mathrm{alg}}} V$, and equip it with an ordinary pre-inner-product which is defined on elementary tensors by
	\begin{equation}\label{inital_hilb_prod_eqn}
		\langle x \otimes v, x' \otimes v' \rangle = \langle \langle x',x \rangle_B v, v' \rangle_V.
	\end{equation}
	Completing the quotient $X \otimes_{B_{\mathrm{alg}}} V$ by subspace of vectors of length zero, we obtain an ordinary Hilbert space, on which $A$ acts by 
	\be\label{comp_hilb_act_eqn}
	a(x \otimes v)=ax\otimes v
	\ee to give a  $*$-representation of $A$. We will denote the corresponding Hermitian  module by $X \otimes_{B} V$. The above construction defines a functor $X \otimes_{B} -: \mathbf{Herm}(B)\to \mathbf{Herm}(A)$ if for $V,W \in \mathbf{Herm}(B)$ and $f\in \mathrm{Hom}_B(V,W)$ we define $f\otimes X \in \mathrm{Hom}_A(V\otimes X, W\otimes X)$ on elementary tensors by $(f \otimes X)(x \otimes v)=x \otimes f(v)$.	We can define action of $B$ on $V\otimes X$ which is defined on elementary tensors by
	\begin{equation}\nonumber\label{comp_hilb_pre_act_eqn}
		b(x \otimes v)= (x \otimes bv) = x b \otimes v.
	\end{equation}
	The complete proof of above facts is contained in the Proposition 2.66 \cite{rae:ctr_morita}
\end{empt}
\begin{proposition}\label{induced_prop}\cite{rae:ctr_morita}
Suppose $A$ acts as adjointable operators on a Hilbert $B$-module
$X_B$, and $\pi: B \to B\left(\H_\pi \right) $ is a nondegenerate representation of $B$. Then the formula
\bean
\Ind \pi\left( a\right) \left( x \otimes_B h\right) \bydef \left(a \dot x \right)\otimes_B h 
\eean
extends to give a representation $\Ind\pi$ of $A$ as bounded operators on the Hilbert space $X_B \otimes_B H_B$ obtained by completing $X_B \odot_B \H_\pi$. If
$X_B$ is nondegenerate as an $A$-module (i.e., if $A \cdot X_B$ is dense in $X_B$), then $\Ind \pi$ is a nondegenerate representation of $A$ on $X_B \otimes_B \H_\pi$.
\end{proposition}
\begin{definition}\label{induced_representation_defn}
	The given by the Proposition \ref{induced_prop} representation is \textit{induced representation}.
We use following notation for it
\be\label{induced_representation_eqn}
X\text{-}\Ind^A_B\pi\bydef \rho\quad \text{ or } \quad \Ind^A_B\pi\bydef \rho\quad \text{ or } \quad	 X\text{-}\Ind\pi\bydef\rho
\ee
\end{definition}
\begin{corollary}\label{induced_cor}\cite{rae:ctr_morita}
Suppose $\pi$ is a faithful representation of $B$, and the action of $A$
on $X_B$ is faithful, so that the map of $A$ into $\mathscr L\left(  X_B\right)$ is injective. Then $	X\text{-}\Ind^A_B\pi$ is a
faithful representation of $A$.
\end{corollary}

\begin{remark}\label{left_right_rem}
	If $A$ is a $C^*$-algebra and $A \to B\left(\H \right)$ is a representation then $\left(a \xi, \eta \right)_{\H}  = \left( \xi, a^* \eta \right)_{\H}$ for each $a \in A$ and $\xi, \eta$. Taking into account (c) of the Definition \ref{hilbert_module_defn}  the equation \eqref{inital_hilb_prod_eqn} can be rewritten by following way
	
	\begin{equation}\label{hilb_prod_eqn}
		\langle x \otimes v, x' \otimes v' \rangle = \langle v,  \langle x,x '\rangle_B v' \rangle_V.
	\end{equation}
	The equation \eqref{hilb_prod_eqn} is more convenient for our purposes than \eqref{inital_hilb_prod_eqn} one. 
	
\end{remark}

		\subsection{Strong Morita equivalence for $C^*$-algebras}\label{strong_morita_sec}

\paragraph*{}
The notion of the strong Morita
equivalence was introduced by Rieffel.
\begin{definition}\label{strong_morita_defn}[Rieffel \cite{Rieffel:74a,rieffel_morita}]
	Let $A$ and~$B$ be $C^*$-algebras.  By an \emph{$A$-$B$-equivalence
		bimodule} (or \emph{$A$-$B$-imprimitivity
		bimodule}) we mean an $\left(B,A\right)$-bimodule which is equipped with $A$- and
	$B$-valued inner products with respect to which~$X$ is a right Hilbert
	$A$-module and a left  Hilbert $B$-module such that
	\begin{enumerate}
		\item[(a)] $\left\langle{x},{y}\right\rangle_B z = x\left\langle{y},{z}\right\rangle_A$ for all $x,y,z\in X$;
		\item[(b)] $\left\langle{X},{X}\right\rangle_A$ spans a dense subset of~$A$ and $\left\langle{X},{X}\right\rangle_B$ spans a dense
		subset of~$B$, i.e. $X$ is a full Hilbert $A$-module and a full Hilbert $B$-module.
	\end{enumerate}
	We call $A$ and~$B$ \emph{strongly Morita equivalent} if there is an
	$A$-$B$-equivalence bimodule.
\end{definition}

\begin{example}\label{imp_p_exm}\cite{rae:ctr_morita}
	Let $p$ be a projection in $M 
	(A)$. We saw in the Remark \ref{full_hilbem} that $Ap$
	is a full right Hilbert $pAp$-module and a full left  Hilbert $\overline{ApA}$-module. (Recall that 
	$\overline{ApA}$
	denotes the ideal generated by $p$, which is the closed span of the set $ApA$.) 
	Thus $Ap$ is an $\overline{ApA}$-$pAp$-imprimitivity 
	bimodule. 
\end{example}

\begin{rem}
	If~$_AX_B$ be an $A$-$B$ imprimitivity bimodule then there is a \textit{dual} $B$-$A$ bimodule $_B\widetilde X_A$ which is a $B$ $A$-imprimitivity bimodule (cf. \cite{rae:ctr_morita}).
\end{rem}

\begin{theorem}\label{rieffel_equiv_thm}\cite{rae:ctr_morita}
	Suppose that $X$ is an $A$ $B$-imprimitivity bimodule, and $\pi$, $\rho$
	are nondegenerate representations of $B$ 
	and $A$, respectively. Then $\widetilde X-\Ind\left(  X-\Ind \pi \right)$ is 
	naturally unitarily equivalent to $\pi$, and $X-\Ind\left(\widetilde  X-\Ind \rho \right)$ is 
	naturally unitarily equivalent to $\rho$.
	
\end{theorem}
\begin{corollary}\cite{rae:ctr_morita}
	If $X$ is an $A$ $B$-imprimitivity bimodule, then the inverse of the Rieffel correspondence $X-\Ind$ is  $\widetilde X-\Ind$.
\end{corollary}
\begin{corollary}\label{morita_irred}\cite{rae:ctr_morita}
	Suppose that $X$ is an $A$ $B$-imprimitivity bimodule, and that $\pi$
	is 
	a nondegenerate representation of $B$. Then $\widetilde  X-\Ind \pi$
	is irreducible if and only if $\pi$ 
	is irreducible. 
\end{corollary}
\begin{corollary}\label{rieffel_homeo_cor}\cite{rae:ctr_morita}
	If $X$ is $A$-$B$-imprimitivity bimodule then the Rieffel correspondence restricts to a homeomorphism $h_X: \mathrm{Prim}~B \xrightarrow{\approx}\mathrm{Prim}~A$.
\end{corollary}
\begin{definition}\label{rieffel_homeo_defn}\cite{rae:ctr_morita}
	Under the hypotheses of the Corollary \ref{rieffel_homeo_cor} the homeomorphism $h_X: \mathrm{Prim}~B \xrightarrow{\approx}\mathrm{Prim}~A$
	is said to be the \textit{Rieffel homeomorphism}.
\end{definition}
\begin{definition}\label{stable_ca_defn}\cite{wegge_olsen}
	A $C^*$-algebra A is said to be \textit{stable} when $A\otimes \K \cong  A$. 
	The \textit{stabilization} of $A$ is $A^s\bydef A\otimes \K$. 
	$C^*$-algebras $A$ and $B$ are said to be \textit{stably equivalent} when $A\otimes \K\cong B\otimes \K$.
\end{definition}
\begin{theorem}\label{stable_morita_thm}\cite{brown:stable}
	Let $B$ and $E$ be $C^*$-algebras. If $B$ and $E$ are 
	stably isomorphic, then they are strongly Morita equivalent. Conversely, 
	if $B$ and $E$ are strongly Morita equivalent and if they both 
	possess strictly positive elements, then they are stably isomorphic (i.e they are stably equivalent).	
\end{theorem}

\begin{proposition}\label{top_admits_morita_prop}\cite{connes:ncg94}
	Let a group $\Ga$  act freely and properly on a topological space  $\widetilde \sX$ and let $\sX\bydef \widetilde \sX/\Ga$. Then
	the $C^*$-algebra $C_0\left(\sX \right)$  is strongly Morita equivalent to the  product $C^*$-algebra
	$C_0\left(\widetilde \sX\right)\rtimes \Ga$.
\end{proposition}
\begin{remark}\label{top_admits_moritaem}\cite{connes:ncg94}
	The equivalence
	$C^*$-bimodule $E$ is easy to describe; it is given by the bundle of Hilbert spaces $\left\{\H_x\right\}_{x\in \sX}$
	with base  $\sX$  whose fiber at $x\in\sX$ is the $\ell^2$-space of the orbit $x\in\sX$. This yields
	the required $\left(C_0\left(\widetilde \sX \right)\rtimes \Ga, C_0\left(\sX \right) \right)$  $C^*$-bimodule.
\end{remark}

	\subsection{$K$-theory of $C^*$-algebras}\label{ca_k_sec}

\paragraph*{} The notions of $K$-theory is explained in \cite{blackadar:ko,wegge_olsen}.
\begin{definition}\label{projection_equivalences_defn}\cite{blackadar:ko}
	Let $e$ and $f$ be idempotents in $A$.
	\begin{enumerate}
		\item [(a)] $e\sim f$ if there are $x, y \in A$ with $xy = e$ and $yx = f$ ({\it algebraic equivalence}),
		\item[(b)] $e \sim_s f$ if there is an invertible $z \in A^+$ with $zez^{-1}= f$ ({\it similarity}),
		\item[(c)] $\sim_h f$ if there is a norm-continuous path of idempotents in $A$ from $a$ to $f$ ({\it homotopy}).
	\end{enumerate}
\end{definition}

\begin{definition}\label{m_inf_defn}\cite{blackadar:ko}
	If $A$ is an algebra then $\mathbb{M}_\infty\left(A \right)$  is the algebraic direct limit of $\mathbb{M}_n\left( A\right)$  under the
	embeddings $a \mapsto \mathrm{diag}\left(a, 0 \right)$.
\end{definition}
\begin{remark}\label{idempotent_equivalencesem} In \cite{blackadar:ko,wegge_olsen}
 it is proven that the conditions (a) - (c) of the Definition \ref{projection_equivalences_defn} are equivalent  $\mathbb{M}_\infty\left(A \right)$.
\end{remark}
\begin{definition}\label{proj_defn}\cite{blackadar:ko}
A projection is a self-adjoint idempotent. A {\it partial isometry}
is an element $u$ with $u^*u$ a projection. If $A$ has a unit, then $U$ is an {\it isometry} if $u^*u = 1$, and $u$ is unitary if $uu^* = u^*u = uu^*=1$.
\end{definition}
\begin{proposition}
 Every idempotent in $A$ is similar to a projection. In fact,
any idempotent is homotopic to a projection.
\end{proposition}
\begin{remark}\label{projection_equivalencesem} In \cite{blackadar:ko,wegge_olsen}
	it is proven that for any two projections the conditions (a) - (c) of the Definition \ref{projection_equivalences_defn} are equivalent  $\mathbb{M}_\infty\left(A \right)$.
\end{remark}
\begin{proposition}
If $p \sim q$, then there is a partial isometry $u$ with $u^*u= p, \quad uu^*=q$.
\end{proposition}

\begin{proposition}\label{projection_u_equivalence_prop}\cite{blackadar:ko}
	If $p \sim_s q$, then there is a unitary $u$ with $upu^* = q$.
\end{proposition}
\begin{definition}\label{local_ca_defn}\cite{blackadar:ko}
	A {\it local Banach algebra} is a normed algebra $A$ which is closed
	under holomorphic functional calculus (i.e. if $x \in A$ and $f$ is an analytic function
	on a neighborhood of the spectrum of $x$ in the completion of $A$, with $f(0) = 0$ if
	$A$ is nonunital, then $f(x) \in A$. For technical reasons we will also require that
	all matrix algebras over A have the same property. If $A$ is a $*$-algebra, it will be
	called a local Banach $*$-algebra; if the norm is a pre-$C^*$-norm, A will be called a
	{\it local $C^*$-algebra}.
\end{definition}

\begin{theorem}\label{ga_proj_thm}\cite{blackadar:ko}
	Let $A$ be a local $C^*$-algebra. Consider the action  of $U\left(A^+ \right)$ 
	on the set $\mathrm{Pr}(A)$ of projections in $A$ given by $\ga\left(u \right) \left( p\right)\bydef upu^*$.Then:
	\begin{enumerate}
		\item[(a)] 	The orbits of  are open and closed; the orbit containing $p$ contains the open	ball $B_p$ of radius 1 around $p$ in $\mathrm{Pr}(A)$.
		\item[(b)] For any $p \in \mathrm{Pr}(A)$, there is a continuous cross section $v_p : B_p \to U\left(A\right)$  for $\ga$ 
		(i .e. $v_p\left(q\right)p v_p\left(q\right)^*= q$ ) with $v_p(p) = 1$.
		\item[(c)] The association $p \mapsto \pi\left( p\right) $ is functorial, i .e. if $\pi : A \to B$ is a $*$-homomorphism, then $v_{\pi\left(p \right) }\left( \pi\left( q\right) \right)  = \pi\left( v_p\left( q\right) \right)$) for $q \in B_p$.
	\end{enumerate}
\end{theorem}

\begin{definition}\label{proj_v_defn}\cite{blackadar:ko}
	$V(A)$ is the set of algebraic equivalence classes of idempotents
	in $\mathbb{M}_\infty\left(A \right)$ (cf. (a) Definition \ref{projection_equivalences_defn}). 
\end{definition}
\begin{definition}\label{k00_defn}\cite{blackadar:ko}
	$K_{00}(A)$ is the Grothendieck group of $V(A)$.
\end{definition}
\begin{remark}\label{proj_vem}\cite{blackadar:ko}
$V (A)$ can also be described as the set of isomorphism classes of finitely generated
projective (left or right) A-modules. The binary operation on
$V (A)$ corresponds to direct sum of modules
\end{remark}
\begin{empt}\cite{blackadar:ko}\label{k0_empt}
	Let $A$  be unital algebra, $J$ a closed two-sided ideal in $A$, and $\pi: A \to A/J$ the quotient map. A group $K_0\left(A, J \right)$  has generators $\left(e, f, z \right)$  , where $e, f$ are idempotents in $\mathbb{M}_n\left( A\right)\subset \mathbb{M}_\infty\left( A\right)$, and $z$ is an invertible element of $\mathbb{M}_n\left( A/J\right)$
	with $z \pi\left(e\right)z^{-1} = \pi\left(f\right)$ ; and relations
	\be\label{k0_1_eqn}
	\begin{split}
		(e_1, f_1, z_1 ) + (e_2, f_2, z_2 ) \bydef  \left(\mathrm{diag}\left( e_1, e_2\right), \mathrm{diag}\left( f_1, f_2\right), \mathrm{diag}\left( z_1, z_2\right) \right) \end{split}
	\ee
	and $(e_1, f_1, z_1 ) = (e_2, f_2, z_2 )$ if there are idempotents $g_1, g_2\in \mathbb{M}_n\left( A\right)$ and invertible
	elements $u, v \in \mathbb{M}_{n+k}\left( A\right)$ with
	\be\label{k0_2_eqn}
	\begin{split}
		u\mathrm{diag}\left( e_1, g_1\right)u^{-1}= \mathrm{diag}\left( e_2, g_2\right),\\
		v\mathrm{diag}\left( f_1, g_1\right)v^{-1}= \mathrm{diag}\left( f_2, g_2\right),\\
		\pi\left(v\right)\mathrm{diag}\left( z_1, 1\right)\pi\left(u\right)^{-1}= \mathrm{diag}\left( z_2, 1\right).
	\end{split}
	\ee
\end{empt}
\begin{definition}\label{k0_defn}\cite{blackadar:ko}
	$K_0\left( A\right) \bydef K_0\left(A^+, A \right)$. 
\end{definition}
\begin{remark}\label{k0em}\cite{blackadar:ko}
	The group
	$K_0\left(A \right)$  may be viewed as formal differences $\left[e\right]-\left[f\right]$, where $e, f \in M_\infty\left(A^+ \right)$  with
	$e \equiv f \mathrm{mod} M_\infty\left(A \right)$, with the usual notion of equivalence of formal differences
	in $K_{00}\left( A^+\right)$. In fact, any element of $K_0\left( A\right)$ may be written  $\left[e\right]- \left[p_n\right]$, where $p_n$ is the rank $n$ projection. This is the {\it Standard Picture} of $K_0\left( A\right)$  for general $A$.
	
	There is a homomorphism from $K_{00}(A^+)$ to $K_0(A)$, given by $\left[e\right] \mapsto \left[e\right]- \left[p_n\right]$, where $\pi\left(e \right)$ is a rank $n$ projection in $\mathbb{M}_\infty\left(A \right)$. Composing this with the canonical map from $K_{00}(A)$ to $K_{00}\left(A^+ \right) $ yields a homomorphism
	\be\label{k0_eqn} 
	\om_A : K_{00}(A) \to 
	K_0\left(A \right).
	\ee
\end{remark}

\begin{theorem}\label{k_0_iso_thm}\cite{blackadar:ko}
	For any $A$ and $J$, the natural homomorphism from $K_0\left(J^+, J \right)$ 
	to $K_0\left(A, J \right)$ is an isomorphism.
\end{theorem}

\begin{rem}
	Theorem \ref{k_0_iso_thm} is called the Strong Excision Theorem of $K$-theory. It is really
	the result which allows $K$-theory to be developed without reference to relative
	$K$-groups.
\end{rem}
\begin{definition}\label{stably_unital_defn}\cite{blackadar:ko}
A local Banach algebra $A$ is {\it stably unital} if $\mathbb{M}_\infty\left(A \right)$  has an
approximate identity of idempotents.
\end{definition}
If $A$ is unital, or more generally if $A$ has an approximate identity of idempotents,
then $A$ is stably unital. If $A$ is a local $C^*$-algebra, then $A$ is stably unital if and
only if $A \otimes \K$ has an approximate identity of projections.
\begin{proposition}\label{stably_unital_prop}\cite{blackadar:ko}
 If $A$ is stably unital, then the map $K_{00}\left( A\right)\to K_0\left( A\right)$   is an
isomorphism.
\end{proposition}

\begin{proposition}\cite{rordam:kc}\label{k_0_stab_prop}
Let $A$ be a $C^*$-algebra and let $n$ be a natural number Then $K_0\left( A\right)$ is isomorphic to $K_0\left(\mathbb{M}_n\left( A\right)  \right)$.

More specially, the homomorphism 
\bean
\begin{split}
\la_{n, A}:  A \hookto \mathbb{M}_n\left( A\right),\\
a \mapsto \begin{pmatrix}
	a & 0\\ 0 & 0
\end{pmatrix}
\end{split}
\eean
induces an isomorphism $K_0\left(\la_{n, A}  \right) : K_0\left( A\right)\cong K_0\left(\mathbb{M}_n\left( A\right)  \right)$.
\end{proposition}
\begin{remark}\label{standard_pictureem}\cite{blackadar:ko}
The group $K_0\left(A \right)$  may be viewed as formal differences $\left[e\right]-\left[f\right]$, where $e, f \in \mathbb{M}_\infty\left( A^+\right)$  with
$e \equiv f \mod  \mathbb{M}_\infty\left( A\right)$, with the usual notion of equivalence of formal differences in $K_{00}\left(A^+ \right)$ . In fact, any element of $K_0\left( A\right)$  may be written $\left[e\right]-\left[p_n\right]$, where
$p_n = \mathrm{diag}\left(1,...,1,0.... \right)$  (with $n$ ones on the diagonal) and $e \equiv p_n  \mod  \mathbb{M}_\infty\left( A\right)$:
if $n$ is large enough, $f \le p_n$, and $\left[e\right]-\left[f\right] = \left[e' +\left(p_n=f \right) \right]-\left[p_n\right]$, where $e'\sim e$
and $e' \perp p_n$. This is the {\it Standard Picture} of $K_0\left( A\right)$  for general $A$.
\end{remark}
\begin{corollary}\label{stable_k_cor}\cite{blackadar:ko}
	If $A$ is any $C^*$-algebra, then $K_n\left(A \right) \cong K_{n-1}\left(Q^s\left(A \right)  \right)$  for
	$n = 0; 1$.
\end{corollary}
\begin{notation}\cite{blackadar:ko}
We use the following definition
\be\label{gl_eqn}
\forall n \in \N\quad \mathrm{GL}_n\left(A \right) \bydef \left\{\left.x \in \mathrm{GL}_n\left(A^+ \right)\right| x\equiv 1_n \mod  \mathbb{M}_n\left( A\right) \right\}
\ee
We   denote by $U_n\left( A\right)$   the group of unitaries in $\\mathbb{M}_n\left( A^+\right)$
which are congruent to $1_n\text{mod}\mathbb{M}_n\left( A\right)$,
\end{notation}
\begin{definition}
	\label{k_1_defn}\cite{blackadar:ko}
\bean
K_1\left(A \right) \bydef \mathrm{GL}_\infty(A)/\mathrm{GL}_\infty(A)_0 = \lim_{n \to \infty}  \mathrm{GL}_n(A)/\mathrm{GL}_nA)_0
\eean
 \end{definition}

\begin{corollary}\label{stab_k_0_cor}\cite{wegge_olsen}
	The morphism $A \to A\otimes \K$ sending $a \mapsto a\otimes p_1$, where $p_1$ is a rank 1 
	projection in $\K$, induces an isomorphism $K_0\left(A\right)\cong K_0\left(A\otimes\K\right)$. 
	In particular, if $A$ and $B$ are stably isomorphic
	then $K_0\left(A\right)\cong K_0\left(B\right)$. 
\end{corollary}
\begin{cor}\label{stab_k_1_cor}\cite{wegge_olsen}
	$K_1\left(A\right)\cong K_1\left(A\otimes\K\right)$) for all $C^*$-algebras $A$. 
	In particular, $K_1(A)\cong K_1(B)$ when $A$ and $B$ are stably isomorphic. 
\end{cor}
\begin{proof}
	As $A \otimes \K$ is the $C^*$-inductive limit of  $\mathbb{M}_n\left( A\right)$ , $K_1\left(A \otimes \K \right) = \varinjlim K_1\left( \mathbb{M}_n\left( A\right)\right) $. 
		
		Consequently, if $A\otimes \K \cong B\otimes \K$,  $\quad K_1\left( A\right) = K_1\left(A \otimes \K \right) =  K_1\left(B \otimes \K \right)= K_1\left(B \right)$. 
	\end{proof}
\begin{remark}\label{k1_u_rem}\cite{blackadar:ko}
	If $A$ is a $C^*$-algebra then
	\be\label{stab_k1_eqn}
	K_1\left(A \right) \cong U_1\left(  A \otimes \K \right) / U_1\left( A \otimes \K \right)_0.
	\ee
	If $A$ is a local $C^*$-algebra, then 
		\be\label{local_eqn}
\forall n \in \N \quad K_1\left(A \right) \quad \mathrm{GL}_n\left( A \right) / \mathrm{GL}_n\left( A \right)_0\cong U_n\left( A \right) / U_n\left( A \right)_0
	\ee
	
\end{remark}
\begin{definition}\label{suspension_defn}\cite{blackadar:ko}
	The \textit{suspension} of $A$, denoted $SA$, is
	\be\label{suspension_eqn}
	\left\{f : \R \to A\left|f \quad \mathrm{continuous}, \quad \lim_{x \to \infty }\left\|f\left( x\right)= 0  \right\| \right.\right\}
	\ee
	with pointwise  operations and sup norm, $SA$ is a local Banach algebra which is
	complete if $A$ is; if $A$ is a $C^*$-algebra, then so is $SA \cong C_0\left( \R\right) \otimes A$. We also say that 	$SA$ is the \textit{algebraic suspension}.
	We have $S\left( \mathbb{M}_n\left( A\right)\right) \cong  S\left( \mathbb{M}_n\left( A\right)\right) $ , and the map $\phi : A \to B$ induces $S\phi: SA \to SB$.
	Moreover,
	\bean
	\left(SA \right)^+ \cong \left\{\left.f:\left[0,1\right|\to A^+\right| f \text{ continuous},~ f(0)=f(1)=\la1, f(t)=\la 1+x_t \text{ for } x_t \in A\right\}
	\eean
	\be\label{suspension_s_eqn}
	\left(SA \right)^+ \cong \left\{\left.f:S^1\to A^+\right| f \text{ continuous}, ~f(z)=\la1 + x_z, ~x_z \in A, ~x_1 = 0 \in A\right\}.
	\ee
\end{definition}
\begin{remark}\cite{blackadar:ko}
If $\sX$ is a locally compact Hausdorff  space then one has
\be
SC_0\left(\sX \right) = C_0\left( S\sX \right) 
\ee
where $SC_0\left(\sX \right)$ and $S\sX$ are given by the Definitions \ref{suspension_defn} and \ref{top_susp_defn}
\end{remark}
\begin{proposition}\label{k_spit_prop}\cite{blackadar:ko}
If $\{0\} \hookto J \hookto A \onto A/J \onto \{0\}$ is a split exact sequence of
local Banach algebras, then $\{0\} \hookto K_n\left( J\right)  \hookto K_n\left( A\right)  \onto K_n\left(  A/J\right)  \onto \{0\}$ is a split
exact sequence for all $n$.
\end{proposition}

\begin{empt}\label{diag_empt} If $A$ is a $C^*$-algebra, $x, y \in A$ and
	\be\label{diag_eqn}
	\begin{split}
		w: \left[0, 1\right] \to \mathbb{M}_2\left(A \right) \\
		t \mapsto  w\left( t\right)  \bydef \begin{pmatrix}
			x & 0\\
			0 & 1
		\end{pmatrix} 
		\begin{pmatrix}
			\cos \frac{\pi}{2} t &  -\sin \frac{\pi}{2} t\\
			\sin \frac{\pi}{2} t &  \cos \frac{\pi}{2} t.
		\end{pmatrix}
		\begin{pmatrix}
			y & 0\\
			0 & 1
		\end{pmatrix}	\begin{pmatrix}
			\cos \frac{\pi}{2} t &  -\sin \frac{\pi}{2} t\\
			\sin \frac{\pi}{2} t &  \cos \frac{\pi}{2} t.
		\end{pmatrix}
	\end{split}
	\ee
	then
	\bean
	w\left(0 \right) = \begin{pmatrix}
		x & 0\\
		0 & y
	\end{pmatrix}, \\
	w\left(1 \right) = \begin{pmatrix}
		xy & 0\\
		0 & 1\end{pmatrix}
	\eean
	(cf \cite{blackadar:ko}).
\end{empt} 

\begin{proposition}\label{inv_path_prop}\cite{blackadar:ko}
	Let $A$ be a unital local Banach algebra. If $x$ and $x$ are invertible in $A$, then there is a path of invertible elements in $\mathbb{M}_2\left( A\right)$  from $\mathrm{diag}(xy; 1)$
	to $\mathrm{diag}(x; y)$. If $A$ is a local $C^*$-algebra and $x$ and $y$ are unitary, the path may
	be chosen to consist of unitaries. So if $z$ is invertible in $A$, then there is a path
	of invertibles in $\mathbb{M}_2\left( A\right)$ ) from 1 to $\mathrm{diag}(z; z^{-1})$.
\end{proposition}
\begin{remark}
	The proof of the Proposition is based on \ref{diag_empt}.
\end{remark}
	
	\begin{theorem}\label{suspension_thm}\cite{blackadar:ko}
		$K_0(A)$ is naturally isomorphic to $K_1(SA)$, i .e. there is an isomorphism $\th_A$ : $K_0(A)\to  K_1\left(SA \right)$ such that, whenever $\phi: A\to B$, the
		following diagram commutes:
		\bean
		\begin{tikzcd}
			K_1\left(A \right) \arrow[r, "K_1\left(\phi \right)" ]  \arrow[d, "\th_A" ]& K_1\left(B \right)\arrow[d, "\th_B" ]\\ 
			K_0\left(SA \right) \arrow[r, "K_0\left(\phi \right)" ] & K_0\left(SB \right)
		\end{tikzcd}
		\eean
		(In the language of category theory, $\th$ gives an invertible natural transformation
		from $K_1$ to $K_0\circ S$).
	\end{theorem}
	\begin{proof} 
		Let $u \in \mathrm{ GL}_n(A)$. Take a path $z_t$ from $1_{2n}$ to $\mathrm{diag}(u; u^{-1})$ in $\mathrm{GL}_{2n}\left(A \right)$ 	(note that the path given in the Proposition \ref{inv_path_prop} lies in $\mathrm{GL}_{2n}\left(A \right)$. Set $e_t \bydef z_t p_n z^{-1}_t$.
		Then $e \bydef (e_t)$ is an idempotent in $\mathbb{M}_{2n}\left(\left( SA\right)^+  \right)$.  Set $\th_A([u])\bydef [e]-[p_n]$ (where $p_n$
		also denotes the corresponding element of $\mathbb{M}_{\infty}\left(\left( SA\right)^+  \right)$ i.e. the constant function $p_n$).
	\end{proof}
\begin{empt}\label{bott_map_empt}\cite{blackadar:ko}
	We have a split exact sequence 
	\bean
	\begin{split}
		\{0\}\hookto S A \hookto \Om A \onto A \onto \{0\},
	\end{split}
	\eean
	where $\Om A\bydef \left[1, A \right)$ 
	which induces a split exact sequence 	\bean
	\begin{split}
		\{0\}\hookto K_1\left( S A \right) \hookto K_1\left( \Om A \right) \onto K_1\left( A \right) \onto \{0\},
	\end{split}
	\eean
	so $K_1\left( SA\right) = \ker \eta_*$  where $\eta : \Om A \to A$ is evaluation at 1. This will be our {\it standard picture} of $K_1\left( SA \right)$. So $K_1\left( SA \right)$ may be viewed as the
	group of homotopy equivalence classes of loops in $GL_n(A)$ with base point 1.
	The group operation is pointwise multiplication, but may alternately be taken
	as the ordinary concatenation multiplication of loops 
	, i.e.
	$K_1\left( SA\right) \cong \pi_1\left( GL_\infty\left( A\right) \right)$. 
	
	If $e$ is an idempotent in $\mathbb{M}_n\left( A^+\right)$, write 
	\be\label{idem_map_eqn} 
	\begin{split}
		f_e\left( z\right) = - ze+(1-e)  \in \mathrm{GL}_n\left(\Om\left( A^+\right)\right)\cong C\left( S^1, \mathrm{GL}_n\left( A^+\right) \right) 
	\end{split}
	\ee
	A loop of this sort is called an {\it idempotent loop}. If $e_1 \pmod e_2$ mod $\mathbb{M}_n\left(A \right)$ , then $f_{e_1}f^{-1}_{e_2}\in \mathrm{GL}_n\left( \Om A\right)$ 
	taking the value 1 at $z = 1$. If $e_1 \sim_h e_2$,
	then $f_{e_1}$ is homotopic to $f_{e_2}$ as elements of $\mathrm{GL}_n\left(\Om\left( A^+\right)  \right)$  taking the value 1 at 1, i.e. as loops in $\mathrm{GL}_n\left(A^+ \right)$  with base point 1. We will write $f_{e_1}\sim_h f_{e_2}$ to denote
	this type of homotopy of elements.
\end{empt}
\begin{definition}\label{bott_map_defn}\cite{blackadar:ko}
The homomorphism $\bt_A: K_0\left(A \right)\to K_1\left(SA \right) $  defined by
$\bt_A \left( \left[e\right]-\left[p_n\right]\right)\bydef \left[f_e f^{-1}_{p_n}\right]$  is called the {\it Bott map} for $A$.
\end{definition}
\begin{empt}
 It is proven in \cite{blackadar:ko} that
	 there is a commutative diagram 
\be\label{ba_eqn}
\begin{tikzcd}
\{{0}\}\arrow[r]& K_0\left(A \right) \arrow[r] \arrow[d, "\bt_A"]& K_0\left(A^+ \right) \arrow[r] \arrow[d, "\bt_A^+"] & K_0\left(\C \right)\cong \Z \arrow[r]\arrow[d, "\bt_\C"]& \{{0}\}\\
\{{0}\}\arrow[r]& K_1\left(SB \right) \arrow[r] & K_1\left(SB^+ \right) \arrow[r] & K_1\left(\C \right) \arrow[r]& \{{0}\}
\end{tikzcd}
\ee
such that vertical arrows are isomorphisms.
\end{empt}

\begin{theorem}\label{bott_map_thm}\cite{blackadar:ko} $\bt_A$ (cf. \eqref{ba_eqn}) is an isomorphism.
\end{theorem}
If $A$ is a $C^*$-algebra and $J \subset A$ is a closed two-sided ideal then there there a $*$-homomorphisms $\iota: J \hookto A$, $~\pi: A \to A / J$ and a long exact sequence
\be\label{long_exact_seq_eqn}
K_1\left(J \right) \xrightarrow{\iota_*}K_1\left(A \right) \xrightarrow{\pi} K_1\left(A/J \right) \xrightarrow{\partial}K_0\left(J \right) \xrightarrow{\iota_*}K_0\left(A \right) \xrightarrow{\pi} K_0\left(A/J \right)
\ee
(cf. \cite{blackadar:ko})
\begin{definition}\label{index_map_defn}\cite{blackadar:ko}
	$~$Let $u \in  \mathrm{GL}_n(A/J)$, and let $w \in \mathrm{GL}_{2n}(A)$ be a lift of $\mathrm{diag}(u,u^{-1})$. Define $\partial([u]) = \left[wp_nw^{-1}\right]- \left[p_n\right]\in K_0(J)$. The map $\partial$ is called the \textit{index map}.
\end{definition}
\begin{theorem}\label{standard_exact_sequence_thm} \cite{blackadar:ko} (Standard Exact Sequence) 
Let $0 \hookto J \xrightarrow{\iota} A \xrightarrow{\pi}  A / J \onto 0$ be an exact sequence of local Banach algebras. Then the following
six-term cyclic sequence is exact:
\be 
\begin{tikzcd}
K_0\left( J\right)  \arrow[r, "\iota_*"] & K_0\left( A\right) \arrow[r, "\pi_*"]& K_0\left( A/J\right) \arrow[d, "\partial"]  \\
K_1\left(A/J \right) \arrow[u, "\partial"] & \arrow[l, "\pi_*"]K_1\left( A\right) &\arrow[l, "\iota_*"] K_1\left(J \right) 
\end{tikzcd}
\ee
The map $\partial :K_0\left( A/J \right) \to K_1\left(J \right)$  is the composition of the suspended index map $\partial: K_1\left(A /J \right) \to K_0\left( J\right)$ with the Bott map.
\end{theorem}
\begin{definition}\cite{blackadar:ko}\label{exp_map_defn}
The connecting map $\partial :K_0\left( A/J \right) \to K_1\left(J \right)$ is called the {\it exponential map}.
An explicit formula for this map is given by
\be\label{exp_map_eqn}
\partial\left(\left[r\right]-\left[p_n\right] \right) \bydef \left[e^{2\pi ix}\right]
\ee
where $e$ is an idempotent in $\mathbb{M}_\infty \left(\left( A/J \right)^+\right)$   with $e \equiv p_n \text{ mod }\mathbb{M}_\infty \left( A/J \right)$ and $x \in \mathbb{M}_\infty \left( A+ \right)$ 
with $\pi\left( x\right) = e$. 
\end{definition}
\begin{exercise}\label{cs1_exer}\cite{blackadar:ko}
	Let $A$ be a $C^*$-algebra. Then there is a split exact sequence
	\be\label{cs1_1_eqn} 
	0 \hookto SA \hookto A \otimes C\left(S^1 \right) \onto A \onto 0
	\ee
	Use \ref{k_spit_prop} and Bott periodicity to compute that
	\be\label{cs1_2_eqn} 
	K_0\left( A \otimes C\left(S^1 \right) \right) \cong K_1\left( A \otimes C\left(S^1 \right) \right)\cong K_0\left( A\right) \oplus K_1\left(A \right).
	\ee
\end{exercise}

\begin{empt}\label{commutative_k_empt}
Here a follow to  \cite{blackadar:ko}.
Let $A$ be a commutative Banach algebra with maximal ideal space $\sX$.
Then there is a natural homomorphism  $\ga : A \to C_0\left(\sX \right)$, the Gelfand transform induces homomorphisms 
\be\label{top_alg_eqn}
\ga_*: K_j\left( A\right)\to K_j \left(C_0\left( \sX\right)  \right) \cong K^{-j}\left(\sX \right)\cong K^{j}\left(\sX \right)
\ee 
These homomorphisms are actually isomorphisms.
The other isomorphism can then be obtained
using Bott periodicity. The {\it Shilov Idempotent Theorem} (cf. \cite{blackadar:ko}), which states that every idempotent in $C_0(\sX)$ is in the image of $\ga$, and hence that the group of integer
linear combinations of idempotents in $A$ is isomorphic to $H^0_c\left(\sX; \Z \right) \cong C_c\left( \sX, \Z\right)$. 
The Arens-Royden Theorem (cf. \cite{royden:fa}), which states that  induces an isomorphism of
	$\mathrm{GL}_1\left( A\right) / \mathrm{GL}_1\left( A\right)_0$ with 
	\be\label{royden_eqn}
\mathrm{GL}_1\left(  C_0\left(\sX \right)\right) / \mathrm{GL}_1\left( C_0\left(\sX \right)\right)_0\cong H^1_c\left( \sX, \Z\right)	
	\ee
\end{empt}
\begin{theorem}\label{c_t_theorem}\cite{blackadar:ko} (Connes' Thom Isomorphism)
	If $\a : \R \to \Aut\left( A\right)$ then	$K_i\left(A\rtimes_\a \R\right) \cong K_{1-i}\left(A \right)$, $\left(i = 0,1 \right)$. 
\end{theorem}
\begin{empt}\cite{blackadar:ko}
	Let $C \bydef C_0\left( \R\cup \{+\infty\}\right)$, and let $\tau$ be the action of $\R$ on $C$ obtained by fixing $+\infty$ and translating $\R$.  Let $CA \bydef C \otimes A$,  $\ga$ the diagonal action $\tau \otimes \a$ $CA$ has an invariant ideal $SA = C_0\left( \R\right) \otimes A$; we also denote the action of $\R$ on $SA$ by $\ga$.
\end{empt}
\begin{lemma}\cite{blackadar:ko}
	$SA \rtimes_\ga \R \cong SA \rtimes_{\tau \otimes 1} \R \cong A\otimes \K$
\end{lemma}
\begin{remark}\label{c_tem}\cite{blackadar:ko}
The six-term exact sequence for the extension
\bean
\{0\}\hookto S A \rtimes_\ga \R\hookto CA \rtimes_\ga \R \onto A \rtimes_\a \R \onto \{0\}
\eean
becomes
\be\label{tho_xex_eqn}
\begin{tikzcd}
K_1\left( A\right) \arrow[r] &K_1\left( \R\hookto CA \rtimes_\ga \R\right)\arrow[r] & K_1\left( A \rtimes_\a \R\right) \arrow[d, "index"] \\
K_0\left(  A \rtimes_\a \R\right) \arrow[u, "exp"] &K_0\left( \R\hookto CA \rtimes_\ga \R\right)\arrow[l] & K_0\left( A \right)\arrow[l]
\end{tikzcd}
\ee
The Thom isomorphism amounts to the statement that exp and index are isomorphisms.
\end{remark}
\begin{proof}
	The first isomorphism is given by sending $g \in C_c\left( \R, C_c\left( \right) \R, A\right)$  to $\tilde g$, where $\left[\tilde g \left(s \right) \right]\left( t\right)= \a_{-t}\left(g\left[s\right] \left( t\right) \right)$ . The second isomorphism comes from Takai duality.
\end{proof}

\begin{empt}\cite{blackadar:ko}
Let $A$ be a $C^*$-algebra, $G$ a locally compact group, and $\a$ continuous homomorphism
from $A$ into $\Aut\left(A \right)$ , the group of $*$-automorphisms of $A$ with the
topology of pointwise norm-convergence. A {\it covariant representation }of the covariant
system $\left( A, G, \a\right)$  is a pair of representations $\left(\pi, \rho \right)$  of $A$ and $G$ on the
same Hilbert space such that $\rho(g)\pi(a)\rho(g)^*$ for all $a \in A$, $g \in G$.
Each covariant representation of $\left( A, G, \a\right)$ gives a representation of the twisted
convolution algebra $C_c\left(G, A \right)$  by integration, and hence a pre-$C^*$-norm on this
$*$-algebra. The supremum of all these norms is a $C^*$-norm, and the completion
of $C_c\left(G, A \right)$  with respect to this norm is called the {\it crossed product} of $A$ by $G$
under the action $\a$, denoted $A \rtimes_\a G$, or sometimes $C^*\left( A, G\right)$ or $C^*\left( A, G, \a\right)$.
The $*$-representations of $A \rtimes_\a G$ are in natural one-one correspondence with the
covariant representations of the system $\left( A, G, \a\right)$. 
If$\a$ is a single automorphism of $A$, we may regard $\a$ as giving an action of $A$ on $A$; the crossed product will be called the crossed product of $A$ by $\a$
\end{empt}

\begin{theorem}\label{takai_theorem}\cite{blackadar:ko}
Let $\left( A, G, \a\right)$ be a covariant system with
$G$ Abelian. Then $\left(A\rtimes_\a G \right)\rtimes_{\hat\a} \hat G \cong K\left( L^2\left(G \right) \right)$. 
So the second crossed product
is stably isomorphic to $A$, and is stable if $G$ is infinite (if $G$ has $n$ elements, it is
$\mathbb{M}_n\left(A \right)$. 
\end{theorem}

\begin{theorem}\label{p_v_theorem}\cite{blackadar:ko}
	Let A be a $C^*$-algebra and $\a \in \Aut\left( A\right)$. Then there is a cyclic six-term exact sequence
\be\label{p_v_eqn}
\begin{tikzcd}
	K_0\left( A\right)\arrow[r, "1-\a_*"] & K_0\left(A \right)\arrow[r, "\iota_*"] & K_0\left(A \rtimes_\a \Z \right) \arrow[d]\\
		K_1\left( A\rtimes \Z\right)\arrow[u] & \arrow[l, "\iota_*"] K_1\left(A \right) & \arrow[l, "1-\a_*"]K_1\left(A \right)
\end{tikzcd}
\ee
\end{theorem}
\begin{definition}\label{mapping_torus_defn}\cite{blackadar:ko}
Let $\a\in \Aut\left( A\right)$ . The {\it mapping torus} of $\a$ is
\bean
M_\a \bydef \left\{f : \R \to A : f\left( x+1\right) = \a f(x)\right\}\cong \left\{f: \left[0,1\right]\to A | f(1)= \a\left( f(0)\right) \right\}
\eean
\end{definition}
\begin{empt}\cite{blackadar:ko}
There is an exact sequence
\be 
\{0\}\hookto S A \hookto M_\a \onto A \onto \{0\}
\ee

We need the following structure theorem for crossed products. If $\th$ is an
action of $\R$ on a $C^*$-algebra $B$, and $\th$ is trivial on $\Z$, then $\bt$ drops to an action
of $\T \cong \R/\Z$, also denoted $\bt$. $\bt$ induces an action $\hat \bt$ of $\hat \T\cong \Z$ on $B \times_\bt\T$ , i.e. an automorphism of $B \times_\bt\T$.
\end{empt}
\begin{proposition}\cite{blackadar:ko}
$B \times_\bt \R$ is isomorphic to the mapping torus of $\hat \bt$ on $B \times_\bt \T$.
\end{proposition}
\begin{empt}\label{p_v_empt}\cite{blackadar:ko}
Suppose $\a \in \Aut\left( A\right)$, and let $B \bydef A\rtimes \Z$. There is an action $\bt \bydef \hat \a$ of $T$ on $B$; we may regard $\bt$  as an action of $\R$ on $B$ with $\Z$ acting trivially. We have $B \rtimes_\bt \T \cong A \otimes \K$ by Takai duality, and $B \rtimes_\bt \R \cong M_{\hat \bt}$. Thus there is a short exact sequence
\be\label{p_v_sec_eqn}
\{0\}\hookto S\left( A \otimes \K\right) \hookto B\rtimes_\bt \R \onto A \otimes \K\onto \{0\}.
\ee
From the Thom isomorphism, we have $K_i\left(B \rtimes \bt \R \right) \cong K_{1 =i}\left( B\right)$; thus there is 
exact sequence is just the six-term exact sequence associated to this extension.
It only remains to show that the connecting maps in this exact sequence are
of the form $1 - \a_*$, which comes from the following two propositions.
\end{empt}
\begin{proposition}\cite{blackadar:ko}
Let $\a \in \Aut\left( A\right)$, and let $M_\a$ be the
 mapping torus. If
$K_i\left(SA \right)$  is identified  with $K_{1-i}\left( A\right)$  via the Bott map, then the connecting maps
in the six-term exact sequence
\be 
\begin{tikzcd}
	K_1\left( A\right)\arrow[r]&K_0\left(M_\a\right)\arrow[r]& K_0\left( A\right)\arrow[d, "\partial"] \\
		K_1\left( A\right)\arrow[u, "\partial"]&\arrow[l]K_1\left(M_\a\right)& \arrow[l]K_0\left( A\right)
\end{tikzcd}
\ee
are of the form $1-\a_*$.
\end{proposition}
\begin{proposition}\label{k_aut_prop}\cite{blackadar:ko}
Let $\a \in \Aut\left(A \right), ~\ga \in \Aut\left(\K \right)$. If $K_i\left( A\right)$  is identified  with
$K_i\left(A \otimes  \K\right)$  in the standard way, then $\a_* = \left(\a\otimes\ga \right)_*$.
\end{proposition}
\begin{proposition}\cite{blackadar:ko}
Then $K_i\left(A \rtimes_\a \Z \right)\cong K_{1-i}\left( M_\a\right)$  for $i = 0, 1$. The isomorphism is natural with respect to covariant homomorphisms.
\end{proposition}
\subsection{Tensor products}\label{tensor_sec}

\begin{empt}\label{cross_empt}\cite{murphy}
If both $A$ and $B$ are algebras, there is a unique multiplication on $A \otimes B$ such
that
$$
(a \otimes b)( a' \otimes b') = aa' \otimes bb'
$$
Moreover if $A$ and $B$ are $*$-algebras, then there is a unique involution on $A\otimes B$
such that $(a \otimes b)^* = a^* \otimes b^*$ for all $a$ and $b$. The $C^*$-norm on  $A\otimes B$ is a $C^*$-{\it cross norm}.
\end{empt}

\begin{theorem}\label{nuclear_defn}\cite{blackadar:ko}
 Let $A$ be a $C^*$-algebra. The following conditions are equivalent:
\begin{enumerate}
	\item [(i)] For every $C^*$-algebra $B$, the algebraic tensor product $A\odot B$ has a unique $C^*$-cross norm (cf. \ref{cross_empt}).
\item [(ii)] The identity map from $A$ to $A$ can be approximated pointwise in norm by
	completely positive finite-rank contractions.
\item [(iii)] $A^{**}$ is an injective ("hyperfinite") von Neumann algebra.
\item [(iv)] A is $C^*$-amenable: every derivation from $A$ into a dual normal Banach $A$-
	module is inner.
	
\end{enumerate}
	A $C^*$-algebra satisfying these conditions is called nuclear.

\end{theorem}
\begin{remark}\label{ab_nuclear_rem}\cite{blackadar:ko}
	Any commutative $C^*$-algebra is nuclear.
\end{remark}
\subsection{Inductive limits of $C^*$-algebras}
\begin{definition}\label{principal_defn}\cite{takeda:inductive}
	An  injective $*$-homomorphism  $\phi: A \hookto B$ of unital $C^*$-algebras is said to be \textit{unital} if and only if $\phi\left(1_A\right)= 1_B$. 
\end{definition}
\begin{remark}
	In the cited text \cite{takeda:inductive} the word "principal" used instead "unital". In this book all entries of the word "principal $*$-homomorphism" are replaced with "unital $*$-homomorphism".
\end{remark}
\begin{definition}\label{inductive_lim_defn}\cite{takeda:inductive}
	Let $\La$ be an increasingly directed set and $A_\la$ be a $C^*$-algebra
	having an identity $1_\la$ associated with $\la \in \La$ If there exists a $C^*$-algebra $A$ with the identity $1$ and a unital isomorphism $f_\la$ of $A_\la$ into $A$ for every $\la \in \La$ such that
	$$
	f_\mu\left(A_\mu\right)\subset f_\nu\left(A_\nu\right) \quad \text{ if } \quad \nu < \mu; \quad \mu, \nu \in \La
	$$
	and that the join of $f_\la\left(A_\la\right)$ ($\la \in \La$) is uniformly dense in $A$, $A$ is called the $C^*$-\textit{inductive limit} of $A_\la$, and is denoted by $C^*\text{-}\varinjlim_\La A_\la$ or  $C^*\text{-}\varinjlim A_\la$
\end{definition}
\begin{theorem}\label{inductive_lim_thm}\cite{takeda:inductive}
	Let $\left\{A_\la\right\}_{\la \in \La}$ be a family of $C^*$-algebras where $\La$  denotes an increasingly directed set. If, for every $\mu, \nu$ with $\mu < \nu$, there exists a unital injective $*$-homomorphism $f_{\mu\nu}: A_\mu \to A_\nu$ satisfying
	$$
	f_{\mu\nu} = f_{\mu\la}\circ f_{\la\nu}\quad~\mathrm{where} \quad \mu < \la < \nu.
	$$
	then there exists the $C^*$-inductive limit of $A_\la$.
\end{theorem}
\begin{remark}\label{inductive_lim_rem}
	A proof of the Theorem is presented in \cite{takeda:inductive}. Roughly speaking the union $\bigcup A_\la$ has the natural structure of $*$-algebra with $C^*$-norm. The $C^*$-algebra $A$ is the $C^*$-norm completion of $\bigcup A_\la$. Moreover  one has
	\be\label{inductive_lim_eqn}
	\forall a \in A \quad \exists \left(a_\la\right)_{\la\in \La} \in \prod A_\la\quad  a = \lim_{\la \in \La}a_\la
	\ee
	where the $C^*$-norm limit is implied.
\end{remark}

\begin{empt}\label{inductive_lim_empt}
	Let $\Om$ and $\Om_{\la}$ be state spaces of $A$ and $A_\la$ respectively. When $A$ is a $C^*$-
	inductive limit of $\left\{A_\la\right\}$, every state $\tau$ of $A$ defines a state $\tau_{\la}$ on $A_\la$.
	Then, for every $\mu, \nu \in \La$ such as $\mu < \nu$, we put $f^*_{\mu\nu}$ the conjugate mapping of 	the principal isomorphism $f_{\mu\nu}$ of $A_\mu$ into $A_\nu$ which maps $\Om_{\nu}$ onto  $\Om_{\nu}$
	has the following properties
	\be\label{ind_7_eqn}
	\tau_\mu = f^*_{\mu\nu}\left(\tau_\nu \right) 
	\ee
	\be\label{ind_8_eqn}
	f^*_{\mu\nu}= f^*_{\mu\la}\circ f^*_{\la\nu}\quad \text{ if }\quad \mu < \la < \nu.
	\ee
	Conversely, a system of states $\left\{\tau_\la \in \Om_{\la} \right\}_{\la \in \La}$ which satisfies the condition \eqref{ind_7_eqn}  defines a state on $A$ since every positive bounded linear functional on the algebraic inductive limit $A^0$ of  $\left\{A_\la\right\}$, is uniquely extended over $A$.
\end{empt}
\begin{prop}\label{inductive_lim_prop}\cite{takeda:inductive}
	Let a $C^*$-algebra $A$ be a $C^*$-inductive limit of $\left\{A_\la\right\}_{\la \in \La}$, and $\La'$ be a cofinal subset in $\La$, then $A$ is the $C^*$-inductive limit of $\left\{A_{\la'})\right\}_{\la' \in \La'}$.
\end{prop}
\begin{theorem}\label{inductive_lim_state_thm}\cite{takeda:inductive}
	If a $C^*$-algebra $A$ is a $C^*$-inductive limit of $A_\la$ ($\la \in \La$), the
	state space $\Om$ of A is homeomorphic to the projective limit of the state spaces $\Om_\la$ of $A_\la$.
\end{theorem}
\begin{corollary}\label{inductive_lim_state_cor}\cite{takeda:inductive}
	If a commutative $C^*$-algebra $A$ is a $C^*$-inductive limit of the
	commutative  $C^*$-algebras $A_\la$ ($\la \in \La$), the spectrum $\mathcal X$ of $A$ is the projective limit of spectra $\mathcal X_\la$ of $A_\la$ ($\la \in \La$).
\end{corollary}

\	
	
		\subsection{Karoubi conjecture}
	
	\begin{theorem}\label{rosen_ak1_thm}\cite{rosen:algebraicK}
		Let $A$ be a Banach algebra (over $\mathbb{F}=\R$ or $\C$). There is a
		functorial "comparison map" of spectra $c : K(A) \to  K^{\mathrm{top}}(A)$ induced by the
		"change of topology" map  $GL\left( A\right)^\dl$ . The induced map $c_* : K_0(A) \to  K_0^{\mathrm{top}}(A)$ is the identity, and the induced map  $c_* : K_0(A) \to  K_0^{\mathrm{top}}(A)$
		is the	quotient map $GL(A/E(A) \to GL(A)/GL(A)_0$. (Here $E(A)$ is the group generated by the elementary matrices, and $GL(A)_0 \supseteq E(A)$ is the identity component of $GL(A)$.)
		Recall also that $K(A)$ is a $K(\mathbb{F})$-module spectrum and that $\mathbb{F}(A)$ is a
		$K_0^{\mathrm{top}}(\mathbb{F})$-module spectrum. The map c is compatible with the product structures,
		in that the diagram
		\bean
		\begin{tikzcd}
			K(\mathbb{F})\times K(A)\arrow[r, "\mu"]\arrow[d, " c_{\mathbb{F}}\times c_A"] & K(A)\arrow[d,"c"]\\
			K^{\mathrm{top}}(\mathbb{F})\times K^{\mathrm{top}}(A)\arrow[r, "\mu_{\mathrm{top}}"]& K^{\mathrm{top}}(A)
		\end{tikzcd}
		\eean
		$\mu$ denoting the multiplication maps, is homotopy commutative.
	\end{theorem}	
	
	\begin{theorem}\label{rosen_ak2_thm}\cite{rosen:algebraicK}
		The Karoubi Conjecture is true. In other words, if $A \cong A \otimes \K$  is a stable $C^*$-algebra, then the comparison map $c : K(A)\to K^{\mathrm{top}}(A)$ of
		Theorem \ref{rosen_ak1_thm} is an equivalence.
	\end{theorem}

		\section{Continuous trace $C^*$-algebras}

	\begin{definition}\label{abelian_element_defn}\cite{pedersen:ca_aut}
		A positive element in $C^*$ - algebra $A$ is {\it Abelian} if subalgebra $xAx \subset A$ is commutative.
	\end{definition}
	\begin{definition}\label{continuous_trace_c_alt_defn}\cite{rae:ctr_morita}
		A \textit{continuous-trace} $C^*$-\textit{algebra} is a $C^*$-algebra $A$ with Hausdorff
		spectrum $\sX$ such that, for each $x_0\in\sX$ there are a neighbourhood $\sU$ of $x_0$ and $a\in A$ such that $\rep_{ x}\left( a\right) $ is a rank-one projection for all $x \in \sU$.
	\end{definition}
	\begin{lemma}\label{hausdorff_spectrum_lem}\cite{rae:ctr_morita}
		Suppose $A$ is a $C^*$-algebra with Hausdorff spectrum $\mathcal{X}$.
		\begin{itemize}
			\item [(a)] If $a, b \in A$ and $\mathfrak{rep}_x\left(a \right)=  \mathfrak{rep}_x\left(b \right)$ for every $x \in  \mathcal{X}$, then $a = b$.
			\item[(b)] For each $a \in A$ the function $x \mapsto \left\|\mathfrak{rep}_x\left(a \right) \right\|$ is continuous on  $\mathcal{X}$, vanishes at infinity and has sup-norm equal to $\left\| a\right\|$. 
		\end{itemize}
	\end{lemma}
	\begin{defn}\label{type_I_defn}\cite{pedersen:ca_aut}
		We say that a $C^*$-algebra $A$ is \textit{of type} $I$ if each non-zero quotient of $A$ contains a non-zero
		Abelian element. If $A$ is even generated (as $C^*$-algebra) by its Abelian elements we say
		that it is \textit{of type} $I_0$.
	\end{defn}
	\begin{prop}\label{continuous_trace_c_a_proposition}\cite{pedersen:ca_aut}
		Let $A$ be a $C^*$ - algebra with continuous trace . Then
		\begin{enumerate}
			\item[(i)] $A$ is of type $I_0$;
			\item[(ii)] $\hat A$ is a locally compact Hausdorff space;
			\item[(iii)] For each $t \in \hat A$ there is an Abelian element $x \in A$ such that $\hat x \in K(\hat A)$ and $\hat x(t) = 1$.
		\end{enumerate}
		The last condition is sufficient for $A$ to have continuous trace .
	\end{prop}


	\subsection{Dixmier Douady theory}
	\paragraph*{} Here I follow to \cite{cuntz_meyer_ros:bivariant}
	
	\begin{proposition}\label{ctr_bundle_prop}\cite{cuntz_meyer_ros:bivariant}
		(Dixmier–Douady). 
		Any stable separable algebra $A$ of continuous
		trace over a second-countable locally compact Hausdorff space $\sX$ is isomorphic to
		$\Ga_0\left( \sX, \A\right)$ , the sections vanishing at infinity of a locally trivial bundle of algebras
		over $\sX$, with fibres $\K$ and structure group $\Aut(\K) = PU = U/\T$. Classes of
		such bundles are in natural bijection with the \v{C}ech cohomology group $\check{H}^3\left(\sX, \Z \right)$.
		The 3-cohomology class $\dl\left( A\right)$  attached to (the stabilization of) a continuous-trace
		algebra $A$ is called its Dixmier–Douady class.
	\end{proposition}
	\begin{proof}
		Principal $PU$-bundles over $\sU$ are thus classified by
		$\left[\sX, BPU\right] = \left[\sX, K\left( \Z, 3\right) \right] = H^3\left(\sX, \Z \right)$. 
		The details of the proof are presented in \cite{cuntz_meyer_ros:bivariant}.
	\end{proof}
	
	\begin{proposition}\label{ctr_d_prop}\cite{cuntz_meyer_ros:bivariant}
		Let $\sX$ be a second-countable locally compact
		Hausdorff space, and let $A$ and $B$ be stable algebras of continuous trace over $\sX$.
		Then $A \times_\sX B$ is also a stable continuous-trace algebra over $\sX$, and the Dixmier–Douady class $\dl \left(A \otimes_\sX B \right)$  of $A \times_\sX B$ is given by $\dl(A) + \dl(B)$. The Dixmier–Douady class of the opposite algebra $A^{\mathrm{op}}$ is given by $\dl\left( A^{\mathrm{op}}\right)=$ - $\dl\left( A\right)$ , so that
		$A\times_\sX  A^{\mathrm{op}}= C_0\left(\sX, \K\right)$.
	\end{proposition}
	\begin{proof}
		Here a fragment of the proof is presented.
		From the proposition it follows that that if $A =\Ga_0\left( \sX, \A\right)$ and $B = \Ga_0\left( \sX, \B\right)$, where $\A$ and B are locally trivial bundles of $C^*$-algebras over $\sX$ with fibres isomorphic	to $\K$ then $A\times_\sX B = \Ga_0\left( \sX, \A\otimes B\right)$ where $\A\otimes \B$, where  is the locally trivial bundle with
		fibre $\A_x\widehat \otimes \B_x$ over $x\in \sX$.  It is proven in \ref{ctr_d_prop} that $\dl \left(A \otimes_\sX B \right)= \dl  \left(A  \right)+\left(B \right)$.
		
	\end{proof}
	\begin{notation}\label{ctr_not}
		If $\sX$ is a locally compact Hausdorff space then we denote by $CT\left(\sX, \dl \right)$ the stable
		continuous-trace algebra  with Dixmier–Douady class $\delta \in \check{H}^3\left(\sX,\Z \right)$. From the Proposition \ref{ctr_d_prop} it follows that
		\be\label{bundle_prod_iso}
		\varphi_{\dl \rho}	:CT\left(\sX, \dl \right)\times_\sX CT\left(\sX, \rho \right)\cong CT\left(\sX, \dl +\rho\right).
		\ee
	\end{notation}
	\begin{proposition}\label{ctr_cup_prop}\cite{cuntz_meyer_ros:bivariant}
		There are natural homomorphisms
		\bea\label{ka_eqn}
		K_\bullet\left( CT\left(\sX, \dl \right)\right)\otimes_\Z K_\bullet\left( CT\left(\sX, \rho \right)\right) \to K_\bullet\left( CT\left(\sX, \dl + \rho\right)\right),\\
		\label{kx_eqn}
		K_\bullet\left( CT\left(\sX, 0 \right)\right)\cong K^\bullet\left( \sX\right),
		\eea
		where $K_\bullet\left( CT\left(\sX, \dl \right)\right)$ means $K$-theoretic groups of $C^*$-algebra $ CT\left(\sX, \dl \right)$ and $K^\bullet\left( \sX\right)$ means $K$-theoretic groups of the space $\sX$.
	\end{proposition}

	\section{$C^*$-algebras of groupoids}
	\subsection{Groupoids}
	\paragraph*{}
	A groupoid is a small category with inverses, or more explicitly:
	\begin{definition}\label{groupoid_defn}\cite{connes:ncg94}
		A \textit{groupoid} consists of a set $\G$, a distinguished subset $\G^0\subset\G$, two maps
		$r, s : \G\to \G^0$ and a law of composition
		$$
		\circ: \G^2\bydef\left\{\left.\left(\ga_1,\ga_2 \right) \in \G\times\G~\right| s\left(\ga_1\right)= r\left(\ga_2\right)\right\}\to \G
		$$
		such that
		\begin{enumerate}
			\item $s\left(\ga_1\circ\ga_2\right)=s\left(\ga_2\right), \quad r\left(\ga_1\circ\ga_2\right)=r\left(\ga_1\right)\quad \forall\left(\ga_1, \ga_2 \right) \in \G$
			\item $s\left(x\right)=r\left(x\right)=x \quad\forall x\in\G^0$
			\item $\ga\circ s\left(\ga\right)= r\left(\ga\right)\circ\ga = \ga\quad \forall\ga\in\G$
			\item $\left( \ga_1\circ\ga_2\right) \circ\ga_3=\ga_1\circ\left( \ga_2\circ\ga_3\right) $
			\item Each $\ga \in\G$ has a two-sided inverse $\ga^{-1}$, with $\ga\circ\ga^{-1}=r\left(\ga\right)$, $\ga^{-1}\circ\ga=r\left(\ga\right)$.
		\end{enumerate}
		The maps $r$, $s$ are called the \textit{range} and \textit{source} maps.
	\end{definition}
	\begin{definition}\label{groupoid_sets_defn}
		If $A$ and $B$ are subsets of $\G$, one may form the following subsets o f $\G$ :
		\bean
		A^{-1} \bydef \left\{x \in \G \left| x^{-1}\in A\right.\right\},\\
		AB \bydef \left\{z \in \G | x \in A, ~y \in B \quad  z = x y \right\}.
		\eean
		A groupoid $\G$ is said to be \textit{principal} if the map $( r , s )$ from $\G$ into $\G^0\times \G^0$ is one-to-one.
		For $u, v, \in \G^0$ , $G^u \bydef r^{-1}(u),\quad G_v \bydef s^{-1}(v), \quad G^u_v \bydef G^u\cap G_v $ and
		$G(u) = G^u_u$ which is a group, is called  the \textit{isotropy group} at $u$.
		The relation  $u \sim v$ if and only if $G^u\cap G_v \neq \emptyset$ is an equivalence relation on the unit space $\G^0$. its equivalence classes are called \textit{orbits} and the \textit{orbit} of  $u$ is denoted $[ u ]$ . 
		$G^0/G$
		denotes the \textit{orbit space}.
	\end{definition} 
	\begin{example}
		The set $\G^2$ of composable elements may be given the following groupoid structure:
		$( x , y )$ and $( y ' , z )$ are composable if and only if  $y' = xy, ~ ( x , y ) ( x y , z ) = ( x , y z )$, and $( x , y )^{-1} =( x y , y ^{-1} )$.
		Then $r^2 ( x , y ) = ( x , r ( y ) ) = ( x , d ( x ) )$ and $d^2 ( x , y ) = ( x y , d ( x y ) )$. The map
		$x\mapsto ( x , d ( x ) )$ identified the unit space of $G^2$ with $G$. The groupoid $G^2$ is principal.
		One may notice that it  comes from the action of $\G$ G on itself. 
	\end{example}
	
	\begin{definition}\label{groupoid_hom_defn}\cite{renault:gropoid_ca}
		Let $\G$ and $\mathcal H$ be groupoids a map $\phi: \G \to \H$ is \textit{homomorphism} if one has:
		\begin{itemize}
			\item
			\bean
			\left(x, y \right) \in \G^2 \quad \Rightarrow \quad \left(\phi\left( x\right), \phi\left(y\right)  \right) \in \H^2,\\
			\phi\left(\G^0 \right) \subset \H^0, 
			\eean
			\item the  map
			\bean
			\phi^2 : \G^2 \to \H^2,\\
			\left(x, y\right)\mapsto \left(\phi(x), \phi(y) \right) 
			\eean 
			is a homeomorphism. 
		\end{itemize}
		Two homomorphism are \textit{similar} (write $\phi\sim \psi$) if there exists a function $\th: \G^0 \to \H$ such that $\left(\th\circ r\right)(x)\phi\left(x\right)= \psi\left(x\right)\left( \th\circ s\right)(x)$. Groupoids $\G$ and $\H$ are called \textit{similar} (write $\G\sim \H$) if there exists homomorphisms $\phi: \G \to \H$ and $\psi : \H\to \G$  such that $\phi \circ \psi$ and $\psi \circ \phi$ are similar to identity isomorphisms. 
	\end{definition}
	If $\G$ is groupoid then for any $A \subset 
	G^0$ and $B 
	\in G^0$ denote by
	\be\label{gab_eqn} 
	\G^A_B \bydef \left\{\left.\ga \in \G \right| r\left(\ga  \right)\in A \quad s\left(\ga \right)  \in B\right\}
	\ee
	\begin{definition}\label{groupoideduction_defn}\cite{renault:gropoid_ca}
		Let $\G$ be a groupoid, and let $E$ be a subset of $\G^0$.
		A subgroupoid 
		$$
		\G^E_E \bydef \left\{x \in G | r(x), s(x)\in E\right\}
		$$
		with unit space $E$ is said to be  the \textit{reduction} 
		of $\G$ by $E$.
	\end{definition}
	\begin{defn}\label{groupoid_isotropy_defn}\cite{renault:gropoid_ca}
		For $u, v\in \G^0$, $~\G^u\bydef r^{-1}\left( u\right)$,  $~\G_v\bydef s^{-1}\left( v\right)$  $~\G^u_v\bydef \G^u\cap \G_v$ and
		$\G(u) = G^u_u$ which is a group, is called  the \textit{isotropy  group} at $u$.
	\end{defn}
	\begin{prop}\label{groupoideduction_prop}\cite{renault:gropoid_ca}
		Let $\G$ be a groupoid, $E$ a subset of $\G^0$ which meets each orbit in $\G$; then  	$\G^E_E$ and $\G$ are similar (cf. Definition \ref{groupoid_hom_defn}).
	\end{prop}
	\begin{definition}\cite{renault:gropoid_ca}
		Let $\G$ be a groupoid, $A$ a group and $c: \G\to A$ a homomorphism, the
		\textit{skew-product} $\G(c)$ is the groupoid $\G\times A$ where : $( x , a )$ and $(y,b)$ are composable if and only if  $x$ and $y$ are composable and $b = a c ( x )$, $( x , a ) ( y , a c ( x ) )\bydef ( x y , a )$, and $( x , a )^{-1} =
		\left(  x^{-1}  , a c ( x )\right)$; $\quad ( x , a )  \bydef( r ( x ) , a )$ , $\quad s ( x , a )\bydef ( d ( x ) , a c ( x ) )$ . Its unit space is $\G^0\times A$.
	\end{definition}
	
	\begin{definition}\cite{renault:gropoid_ca}
		Let $\G$ be a groupoid, let $A$ be a group and let $\a: A \to \Aut\left(\G\right)$ be a
		homomorphism. We write $x*a \bydef \left[\a\left(a^{-1}\right)\right]$ for $a\in A$ and $x\in \G$. The \textit{semi-direct product} $G\rtimes_\a A$ is the groupoid $G\times A$ where $( x , a )$ and $( z , b )$ are composable if and only if $z = y*a$ with  $x$ and $y$ composable,$ ( x , a ) ( y * a,b) = ( x y , a b )$ , and $(x,a)^{ -1}\bydef \left( x ^{-1} * a, a^{ - l} \right)$ .
		Then, $r ( x , a ) = ( r ( x ) , e )$ and $s ( x , a ) = (d(x)   a , e )$ . The unit  space may be identified 
		with $\G^0$.
	\end{definition}
	
	\begin{proposition}
		With above notation,
		\begin{enumerate}
			\item[(i)] $\G(c) \rtimes_\a A$  is similar to $\G$ and
			\item[(ii)] $\left( \G\rtimes_a A\right) (c)$ is similar to $\G$.
		\end{enumerate}
		
		.\end{proposition}
	\begin{definition}
		An \textit{inverse semi}-\textit{group} is a set $\mathscr G$  endowed with an associative binary operation , noted multiplication, and an inverse map
		\bean
		\mathscr G\to  \mathscr G,\\
		s \mapsto s^{-1}
		\eean
		such that the following relations  are satisfied  $ss^{-1}s = s$ and $s^{-1}ss^{-1}=s^{-1}$ .
	\end{definition}
	\begin{definition}
		Let $\G$ be a groupoid. A subset $s$ of $\G$ will be called a $\G$-\textit{set} if 
		the restriction of $r$ and $s$ to it are one-to-one . Equivalently, $s$ is a $\G$-set if and only is $s^{-1}s$
		and $ss^{-1}$  are contained in $\G^0$.
		
	\end{definition}

	\begin{definition}
		Suppose that $\mathscr C$ is some category. A map $p$ from a set $A$ onto a set $A^0$ such that
		each fiber $p^{-1} ( u )$ is an object of  $\mathscr C$ will be called a  $\mathscr C$-\textit{bundle} map and $A$ will be called  $\mathscr C$-bundle.  Let $A$ be
		a  $\mathscr C$ - bundle with the bundle map $p: A \to A_0$. Write $A_u \bydef p^{-1} ( u )$.
		$$
		\text{Iso}(A) = \left\{\text{isomorphisms }\phi_{u,v}| A_u\to A_v\quad u,v \in A^0\right\} 
		$$	
		has a natural structure of groupoid. 
	\end{definition}
	\begin{definition}
		Let $\G$ be a groupoid. A $\G$-\textit{bundle} $(A,L)$ is a $\mathscr C$ - bundle $A$ together
		with a homomorphism $L : G \to \text{Iso} (A)$ such that $L^0: \G^0\to A^0$ is a bijection . (We will often identify $\G^0$ and $A^0$). When  $\mathscr C$  is the category of Abelian groups, one speaks of a
		$\G$-\textit{module bundle}.
	\end{definition}
	\begin{empt}\label{groupoig_gn_empt}
		Given a $\G$-module bundle $( A , L )$, one can form the following cochain complex. Let
		us first define $\G^n$ for any $n\in\N$. The sets $\H^0$ , $\G^1\bydef \G$ and $\G^2$ have already been defined. For $n> 2$, $\quad \G^n$ is the set of $n$-tuples $(x_0 . . . . . x_{n-1}) \in \G\times...\times \G$ such that for $
		j = 1 , . . . , n - 1$ , $\quad x_j$ is composable with its left  neighbor. A $n$-cochain is a function from $G^n$ to $A$ which satisfies the conditions
		\begin{enumerate}
			\item[(i)] $p\circ f(x_0 . . . . . x_{n-1}) = d(x_0)$ and
			\item[(ii)] if $n > 0$ and for some $j = 0, . . . , n - 1$ , $\quad x_0 \in \G^0$, then $f ( x_0, ..., x_j , ..., x_{n-1})\in A^0$.
		\end{enumerate}
		The set $C^n\left(\G, A\right)$ of $n$-cochains is an Abelian group under point-wise addition. The
		sequence 
		\bean
		0 \to C^0\left( \G, A\right)\to C^0\left( \G, A\right)\to C^1\left( \G, A\right)\to...\to  C^n\left( \G, A\right) \xrightarrow{\delta^n} C^{n + 1}\left( \G, A\right)\to ...
		\eean
		where 
		\be\label{groupoid_b_eqn}
		\begin{split}
			\delta^0f(x)\bydef L(x)\quad  f\circ s(x) - f \circ r ( x ),\\
			\delta^n(f(x_0 ,..., x_n) = L ( x_0 ) f ( x_1,..., x_n) +\\+ \sum_{j=1}^n (-1)^j
			f ( x_0 ,..., x_{j-1}x_j . . . . . x_{n-1})+(-1)^{n-1} f ( x_0 , . . . , x_{n-1}) \quad  n > 0,
		\end{split}
		\ee
		
		is a cochain complex.
	\end{empt}	
	\begin{definition}\label{groupoid_cocycle_defn}
		The group of $n$-cocycles of this complex will be denoted by $Z^n(\G,A)$,
		the group of $n$-coboundaries will be denoted by $B^n(\G,A)$ and the $n$-th cohomology group
		$Z^n(\G,A)/B^n(\G,A)$ will be denoted by $H^n(\G,A)$.
	\end{definition}
	
	\begin{definition}\label{groupoid_topological_defn}\cite{renault:gropoid_ca}
		A \textit{topological groupoid} consists of a groupoid $\G$ and a topology compatible with the groupoid structure:
		\begin{enumerate}
			\item [(a)] $\G \to \G \quad x \mapsto x^{-1}$ is continuous,
			\item [(b)] $\G^2\to \G\quad \left(x,y\right)\mapsto xy$ is continuous where $\G^2$ has the induced topology from $\G \times \G$.
		\end{enumerate}
	\end{definition}
	
	\begin{remark}\cite{renault:gropoid_ca}
		One has:
		\begin{itemize}
			\item the map $x \mapsto x^{-1}$ is a homeomorphism,
			\item if $\G$ is Hausdorff  then $\G^0$ is closed in $\G$,
			\item if $\G^0$ is Hausdorff  then  $\G^2$ is closed in $\G \times \G$, $\G^0$ is both a subspace of $\G$ and a quotient of $\G$ (by the map $r$), the induced and the quotient topology coincide.
		\end{itemize}
	\end{remark}
	\begin{definition}\label{groupoid_haar_defn}\cite{renault:gropoid_ca}
		Let $\G$ be a locally compact groupoid. A \textit{left  Haar system} for $\G$
		consists of measures $\left\{\left.\la^u \right| u \in \G^0\right\}$ on $\G$ such that
		\begin{enumerate}
			\item [(a)] the support $\supp\la^u$ of the measure $\la^u$ is $\G^u$,
			\item [(b)]  (continuity) for any $f \in C_c\left(\G\right)$, $u \mapsto \la(f)(u) = \int f d\la^u$ is continuous, and
			\item [(c)]  (left invariance) for any $x\in \G$ and any $f \in  C_c(\G )$, $\int  f ( x y ) d\la^{s(x)}(y) =
			\int f(y)d\la^{r(x)}(y)$.
			
		\end{enumerate}
	\end{definition}
	\subsection{Covering groupoids}\label{covering_groupoid_sec}
	\paragraph{}
	The theory of covering groupoids is explained in \cite{zhi:cov_group}. Given $u \in \G^0$, a \textit{sieve} on $\G$ is a set $S\subset G^u$ (cf. Definition \ref{groupoid_sets_defn}) $\G$ such that
	\bean
	f \in S \text{ and the composite } fh \text{ is defined implies} fh \in S.
	\eean
	Since $x \in \G$ is invertible, then every nonempty sieve on $\G$ coincides with $\G^u$.
	\begin{definition}\label{groupoid_covering_defn}\cite{zhi:cov_group}
		Let $p : \widetilde{\G}\to G$ be a morphism of groupoids. An ordered pair
		$\left( \widetilde{\G}, p\right)$ is a \text{covering groupoid} if for each object $\widetilde u \in \widetilde{\G}^0$ the restriction of $p$
		$$
		\widetilde{\G}^{\widetilde{u}}\to {\G}^u
		$$
		is bijection where both $\widetilde{\G}^{\widetilde{u}}$ and ${\G}^u$ are given by the Definition \ref{groupoid_isotropy_defn}. The morphism $p$ is called the \textit{covering}. 
	\end{definition}
	
	\begin{theorem}\label{covering_groupoid_thm}\cite{zhi:cov_group}
		Let $\left(\widetilde \G, p \right) $ be a covering groupoid of $\G$
		with $p\left(\widetilde{u} \right) = u$  where $\widetilde u \in \widetilde \G$ and $u \in \G$, and $f:  \F \to \G$ a groupoid
		morphism with $f\left( v\right)= u$  such that$\F$ is connected. Then $f$ lifts
		to a morphism $\widetilde f:  \F \to \widetilde\G$ with $\widetilde f\left(v \right)  =\widetilde u$ if and only if $f^*\F^v_v \subset p^*\widetilde\G^{\widetilde u}_{\widetilde u}$
		and if this lifting exists, then it is unique.
	\end{theorem}
	
	There are following motivations this definition:
	\begin{itemize}
		\item analog of unique lifting theorem,
		\item  analog of fundamental group,
		\item analog of Galois theory.
	\end{itemize} 
	
	\subsection{Groupoid $C^*$-algebras}
	\paragraph{}
	Let $\G$ be a locally compact Hausdorff groupoid with left  Haar system $\left\{\la^u\right\}$ and let $\sigma$ be a continuous 2-cocycle in $Z^2\left(\G, \T\right)$. For $f ,g \in C_c(\G, \sigma )$, let us define
	\be\label{groupoid_*_c_eqn}
	\begin{split}
		f * g \left(x\right)\bydef 
		\int f ( x y ) g \left( y^{-1}\right)\sigma\left(xy, y^{-1} \right) d\la^{d(x)}(y),\\
		f^* ( x ) \bydef \overline{f ( x^{ -1})}~\overline{\sigma\left(x, x^{-1} \right)}, 	
	\end{split}
	\ee	
	In particular is $\sigma$ is trivial then one has a $*$-algebra
	\be\label{groupoid_*_eqn}
	\begin{split}
		f * g \left(x\right)\bydef 
		\int f ( x y ) g \left( y^{-1}\right) d\la^{d(x)}(y),\\
		f^* ( x ) \bydef\overline{f ( x^{ -1})} 	
	\end{split}
	\ee	
	which is a specialization of \eqref{groupoid_*_c_eqn}
	
	\begin{empt} 
		Let $\G$ be a locally compact groupoid with left  Haar system $\left\{\la^u\right\}$ (cf. Definition \ref{groupoid_haar_defn}).  For $f$ and $g\in C_c\left(\G\right)$, let  us define
		\be\label{groupoid_*__defn}
		\begin{split}
			f * g \left(x\right)\bydef 
			\int  f ( x y ) g \left( y^{-1}\right) d\la^{d(x)}(y),\\
			f^* ( x ) \bydef\overline{f ( x^{ -1})} 	
		\end{split}
		\ee
		It is proven in \cite{renault:gropoid_ca} that the equations yield a $*$-algebra. 
	\end{empt}
	\begin{definition}\label{groupoid_representation_defn}\cite{renault:gropoid_ca}
		A \textit{representation} of $C_c\left(\G, \sigma\right)$ on a Hilbert space $\H$ is a $*$-homomorphism $L : C_c\left(\G, \sigma\right) \to B\left(\H \right)$ which is continuous when $C_c\left(\G, \sigma\right)$ has the inductive limit 
		topology (cf. Definition \ref{top_ind_lim_defn}) and $B\left(\H \right)$  the weak operator topology (cf. Definition \ref{weak_topology_defn}), and is such that the linear  span of
		$$
		L\left(f \right) \xi , \quad f \in C_c\left(\G, \sigma\right), \quad \xi \in \H
		$$
		is dense in $\H$.
	\end{definition}
	\begin{lemma}\label{groupoid_mult_repr_lem}\cite{renault:gropoid_ca}
		If $L$ is a representation of $C_c\left(\G, \sigma\right)$, there exists a unique representation
		$M$ of $C_c\left(\G^0\right)$ such that for every $h \in C_c\left(\G^0\right)$ and every $f\in C_c\left(\G, \sigma\right)$, $L\left(h f \right)= M\left(h \right)L\left(f \right)$  and
		$L\left(fh \right)= L\left(f \right)M\left(h \right)$. 
	\end{lemma}

	\begin{definition}\label{foli_groupoid_red_defn}\cite{renault:gropoid_ca}
		If $\Pi$ is the set of irreducible representations of $\C\left[\G \right]$ then the completion of $\C\left[\G \right]$ with respect to $C^*$-norm
		\be\label{groupoid_red_norm}
		\left\| a\right\|_r = \sup_{\pi\in \Pi}\left\|\pi\left( a\right)\right\| 
		\ee
		is said to be the \textit{reduced algebra} of $\G$. It will be denoted by $C^*_r\left(\G\right)$.
	\end{definition}
	\begin{empt}\label{groupoid_reg_empt}\cite{renault:gropoid_ca}
		Let $\sigma$ be a 2-cocycle and $\mu$ a quasi-invariant measure. Consider the measurable
		field of Hilbert space $\left\{L^2\left(\G, \la^u \right) \right\}_{u \in \G^0}$ with square integrable sections
		$\int^\oplus L^2\left(\G, \la^u \right) d\mu\left( u\right)= L^2\left(u \right)$. For $x\in \G$, define $L_u\left( x\right)$  mapping $L^2\left( \G, \la^{d\left( x\right) }\right)$  to
		$L^2\left( \G, \la^{r\left( x\right) }\right)$ by $L_u\left( x\right)\xi\left( y\right) \bydef \sigma\left( x, x^{-1}\right) \xi\left( x^{-1}y\right)$. This yields a given by 
		\be\label{groupoid_reg_eqn}
		L_u :  C_c\left(\G, \sigma\right)\to B\left(L^2\left(u \right)  \right) 
		\ee
		$\sigma$-representation of $\G$,
	\end{empt}
	\begin{definition}\label{groupoid_reg_defn}\cite{renault:gropoid_ca} 
		The above $\sigma$-representation of $\G$ will be called the $\sigma$-\textit{regular representation
			of $\G$ on $\mu$}. Its integrated form is the \textit{regular representation on $\mu$ of
			$C_c\left(\G, \sigma \right)$}. 
	\end{definition}
	\begin{proposition}\label{groupoid_reg_prop}\cite{renault:gropoid_ca}
		$C_c\left(\G, \sigma \right)$ has a faithful family of bounded representations, consisting
		of regular representations.
	\end{proposition}
	It results from Proposition \ref{groupoid_reg_prop} that the function defined by 
	\be\label{groupoid_red_norm_eqn}
	\begin{split}
		\left\|\cdot  \right\|_r : C_c\left(\G, \sigma \right)\to \R,\\
		f \mapsto \sup_{u \in \G^0} \left\|L_u\left( f\right)   \right\|,
	\end{split}
	\ee
	where $L_u$ ranges over all representations induced from the unit space, is
	a $C^*$ -norm on $ C_c\left(\G, \sigma \right)$ dominated by the $C^*$ -norm $\left\|f  \right\|$.
	\begin{definition}\label{groupoid_red_defn}	\cite{renault:gropoid_ca}
		The \textit{reduced $C^*$ -algebra} $C^*_r\left(\G, \sigma \right)$ of $\G$ is the completion of
		$C^*_r\left(\G, \sigma \right)$ for the reduced norm $\left\|\cdot  \right\|_r$.
	\end{definition}
	\subsection{$C^*$-algebras of non-Hausdorff \'etale groupoids}
	\paragraph{} Here I follow to \cite{neshv:non_haudorff}. 
	\begin{definition}
		Assume $\G$ is a locally compact, not necessarily Hausdorff, \'etale groupoid.
		By this, we mean that $\G$ is a groupoid endowed with a locally compact topology
		such that
		\begin{itemize}
			\item  the groupoid operations are continuous;
			\item the unit space $\G^0$ is a locally compact Hausdorff space in the relative
			topology;
			\item the range map $r : \G \to \G^0$ and the source map $r : \G \to \G^0$ are local
			homeomorphisms.
		\end{itemize}
	\end{definition}

	For an open Hausdorff subset $\sV \subset \G$, consider the usual space $C_c\left( \sV\right)$  of continuous
	compactly supported functions on $\sV$ . Every such function can be extended
	by zero to $\G$; in general, this extension is not a continuous function on $\G$.
	This way, we can view  $C_c\left( \sV\right)$ as a subspace of the space of functions $\text{Func}\left(\G \right)$ 
	on $\G$. For arbitrary open subsets $\sU\subset G$ we denote by $C_c\left( \sU\right)$  the
	linear span of the subspaces $C_c\left( \sV\right) \subset \text{Func}\left(\G \right)$ for all open Hausdorff subsets
	$\sV \subset \sU$ Instead of all possible $\sV$ , it suffices to take a collection of open bisections
	covering $\sU$.

	\subsection{Foliations and pseudogroups}
	\begin{definition}\cite{connes:ncg94}		
		Let $M$ be a smooth manifold and $TM$ its tangent bundle, so that
		for each $x \in M$, $T_x M$ is the tangent space of $M$ at $x$. A
		smooth subbundle $\mathcal{F}$ of $TM$ is called {\it integrable} if and only if one of
		the following equivalent conditions is satisfied:
		
		\smallskip
		
		\begin{enumerate}
			
			\item[(a)] Every $x \in M$ is contained in a submanifold $W$ of $M$ such that
			$$
			T_y (W) = \mathcal{F}_y \qquad \forall \, y \in W \, ,
			$$
			
			\smallskip
			
			\item[(b)] Every $x \in M$ is in the domain $U \subset M$ of a
			submersion $p : U \to {\mathbb R}^q$ ($q = {\rm codim} \, \mathcal{F}$) with
			$$
			\mathcal{F}_y = {\rm Ker} (p_*)_y \qquad \forall \, y \in U \, ,
			$$
			
			\smallskip
			\item[(c)] $C^{\infty} \left( \mathcal{F}\right)  = \{ X \in C^{\infty} \left(TM\right) \, , \ X_x \in
			\mathcal{F}_x \quad \forall \, x \in M \}$ is a Lie algebra,
			
			\smallskip
			
			\item[(d)] The ideal $J\left( \mathcal{F}\right) $ of smooth exterior differential forms which
			vanish on $\mathcal{F}$ is stable by exterior differentiation.
		\end{enumerate}
		
	\end{definition}
	
	\begin{empt}\cite{connes:ncg94}
		A foliation of $M$ is given by an integrable subbundle $\mathcal{F}$ of $TM$.
		The leaves of the foliation $\left(M , \mathcal F\right)$ are the maximal connected
		submanifolds $L$ of $M$ with $T_x (L) = \mathcal{F}_x $, $\forall \, x \in L$,
		and the partition of $M$ in leaves $$M = \cup
		L_{\alpha}\,,\quad\alpha \in X$$ is characterized geometrically by
		its ``local triviality'': every point $x \in M$ has a neighborhood
		$\mathcal U$ and a system of local coordinates
		$(x^j)_{j = 1 , \ldots , \dim V}$ called
		{\it foliation charts}, so
		that the partition of $\mathcal U$ in connected components of
		leaves corresponds to the partition of 
		\begin{equation*}
			{\mathbb
				R}^{\dim M} = {\mathbb R}^{\dim \mathcal F} \times {\mathbb R}^{\text{codim}
				\, \mathcal F}
		\end{equation*}
		in the parallel affine subspaces 
		$
		{\mathbb R}^{\dim \mathcal F}
		\times {\rm pt}$.
		The corresponding foliation will be denoted by
		\begin{equation}\label{fol_chart_eqn}
			\left(\R^n, \mathcal{F}_p \right) 
		\end{equation}
		where $p = \dim   \mathcal{F}_p$.
		To each foliation $\left(M, \mathcal{F}\right)$ is canonically associated a $C^*$- algebra
		$C^* (M, ~\mathcal{F})$ which encodes the topology of the space of leaves.  To
		take this into account one first constructs a manifold $\mathcal G$, $\dim
		\, \mathcal G = \dim \,M + \dim \,\mathcal F$. 
	\end{empt}
	\begin{definition}\label{foli_trans_defn}\cite{candel:foliI}
		Let $N \subset M$ be a smooth submanifold. We say that $\sF$ is \textit{transverse} to $N$ (and write $\sF\pitchfork N$) if, for each leaf $L$ of $\sF$ and each point $x \in L\cap N$, $T_x\left(L \right)$ ans $T_x\left(N \right)$ together span $T_x\left( M\right)$. At the other extreme At the other extreme, we say that $\sF$ is tangent to $N$ if, for each leaf $L$ of $\sF$, either $L \cap N = \emptyset$ or $L \subset N$.
	\end{definition}
	The symbol $\mathbb{F}^p$ denotes either the full Euclidean space $\R^p$ or Euclidean half space $\mathbb{H}^p = \left\{\left.\left(x_1,..., x_n \right) \in \mathbb{R}^p\right| x_1 \le 0 \right\}$.
	\begin{definition}\label{foliect_defn}\cite{candel:foliI}
		A rectangular neighborhood in $\mathbb{F}^n$ is an open subset of the form $B = J_1\times...\times J_n$, where each $J_j$ is a (possibly unbounded) relatively open interval in the $j^{\text{th}}$ coordinate axis. If $J_1$ is of the form $\left( a,0\right]$, we say that $B$ has boundary $\partial B\left\{\left(0, x_2,..., x_n \right)\right\}\subset B$.	
	\end{definition}
	\begin{definition}\label{foli_chart_defn}\cite{candel:foliI}
		Let $M$ be an $n$-manifold. A \textit{foliated chart} on $M$ of codimension $q$ is a pair $\left(\sU, \varphi)\right)$, where $\sU\subset M$ is open and $\varphi : \sU \xrightarrow{\approx} B_\tau\times B_\pitchfork$ is a diffeomorphism, $B_\pitchfork$ being a rectangular neighborhood in $\mathbb{F}^q$ and $B_\tau$ a rectangular neighborhood in $\mathbb{F}^{n-q}$. The set $P_y = \varphi^{-1}\left(B_\tau \times \left\{y\right\} \right)$ , where $y \in B_\pitchfork$, is called a \textit{plaque} of this foliated chart. For each $x \in B_\tau$, the set  $S_x=\varphi^{-1}\left(\left\{x\right\} \times B_\pitchfork \right)$  is called a \textit{transversal} of the foliated chart. The set $\partial_{\tau}\sU = \varphi^{-1}\left(B_\tau \times \left(\partial B_\pitchfork \right)  \right)$ is called the \textit{tangential boundary} of $\sU$ and $\partial_{\pitchfork}\sU = \varphi^{-1}\left(\partial \left( B_\tau\right)  \times \partial B_\pitchfork \right)$ is called the \textit{transverse boundary} of $\sU$.
	\end{definition}
	
	\begin{definition}\label{foli_defn}
		Let $M$ be an $n$-manifold, possibly with boundary and corners, and let $\sF= \left\{L_\la\right\}_{\la \in \La}$ be a decomposition  of $M$ into connected, topologically immersed submanifolds of dimension $k=n-q$. Suppose that $M$ admits an atlas $\left\{\sU_\a \right\}_{\a \in \mathfrak A}$ of foliated charts of codimension $q$ such that, for each $\a \in \mathscr A$ and each $\la \in \La$, $L_\la \cap \sU_\a$ is a union of plaques. Then $\sF$ is said to be a \textit{foliation} of $M$ of codimension $q$ (and dimension $k$) and $\left\{\sU_\a \right\}_{\a \in \mathscr A}$ l is called a \textit{foliated atlas} associated to $\sF$. Each $L_x$ is called a leaf of the foliation and the pair $\left(M, \sF \right)$  is called a \textit{foliated manifold}. If the foliated atlas is of class $C^r$ ($0 \le r \le \infty$ or $r=\om$), then the foliation $\sF$ and the foliated manifold $\left(M, \sF \right)$. is said to be \textit{of class} $C^r$.
	\end{definition}
	\begin{definition}\label{foles_defn}
		If $\left(M,\mathcal F \right)$ is a foliation and $\mathcal{U} \subset M$ be an open subset $\mathcal F|_{\mathcal{U}}$ is the restriction of $\mathcal F$ on ${\mathcal{U}}$ then we say that  $\left(\mathcal{U},\mathcal F|_{\mathcal{U}}\right) $ is the \textit{restriction}   of $\left(M,\mathcal F \right)$ \textit{to}  $\mathcal{U}$. (cf. \cite{connes:ncg94})
	\end{definition}
	\begin{definition}\label{foli_atlas_defn}
		A \textit{foliated atlas} of codimension $q$ and class $C^r$ on the $n$-manifold $M$ is a $C^r$-atlas $\mathfrak{A}\bydef\left\{\sU_\a \right\}_{\a \in \mathscr A}$  of foliated charts of codimension $q$ which are \textit{coherently foliated} in the sense that, whenever $P$ and $Q$ are plaques in distinct charts of $\mathfrak{A}$, then $P\cap Q$ is open both in $P$ and $Q$. 
	\end{definition}
	\begin{definition}\label{foli_coh_atlas_defn}
		Two foliated atlases lt and $\mathfrak{A}$ on $\mathfrak{A}'$ of the same codimension and smoothness class $C^r$ are coherent ($\mathfrak{A}\approx\mathfrak{A}'$) if $\mathfrak{A}\cup\mathfrak{A}'$ is a foliated $C^*$-atlas.
	\end{definition}
	\begin{lemma}\label{foli_coh_atlas_eq_lem}
		Coherence of foliated atlases is an equivalence relation.
	\end{lemma}
	\begin{lemma}\label{foli_coh_atlas_ass_lem}
		Let  $\mathfrak{A}$ and $\mathfrak{A}'$ be foliated atlases on $M$ and suppose that $\mathfrak{A}$ is associated to a foliation $\sF$. Then $\mathfrak{A}$ and $\mathfrak{A}'$ are coherent if and only if $\mathfrak{A}'$ is also associated to $\sF$.
	\end{lemma}
	\begin{definition}\label{folieg_atlas_defn}
		A foliated atlas $\mathfrak{A}\bydef\left\{\sU_\a \right\}_{\a \in \mathscr A}$ of class $C^r$ is said to be \textit{regular} if
		\begin{enumerate}
			\item [(a)] For each $\al \in \mathscr A$, the closure $\overline{\sU}_\al$ of $\sU_\al$ is a compact subset of a foliated chart  $\left\{\sV_\a \right\}$ and $\varphi_\a = \psi|_{\sU_\a }$.
			\item[(b)] The cover $\left\{\sU_\a \right\}$ is locally finite.
			\item[(c)] if $\sU_\a$ and $\sU_\bt$ are elements of $\mathfrak{A}$, then the interior of each closed plaque $P \in \overline \sU_\a$ meets at most one plaque in $\overline \sU_\bt$.
		\end{enumerate}
	\end{definition}
	\begin{lemma}\label{folieg_atlasef_lem}
		Every foliated atlas has a coherent refinement that is regular.
	\end{lemma}
	\begin{thm}\label{foli_thm}
		The correspondence between foliations on $M$ and their associated foliated atlases induces a one-to-one correspondence between the set of foliations on $M$.
	\end{thm}
	
	We now have an alternative definition of the term "foliation". 
	\begin{defn}\label{foli_alt_defn}
		A \textit{foliation} $\sF$ of codimension $q$ and class $C^r$ on $M$ is a coherence class of foliated atlases of codimension q and class $C^r$ on $M$.
	\end{defn}
	By Zorn's lemma, it is obvious that every coherence class of foliated atlases contains a unique maximal foliated atlas. 
	\begin{defn}\label{foli_max_defn}
		A \textit{foliation of codimension} $q$ and class $C^r$ on $M$ is a maximal foliated $C^r$-atlas of codimension $q$ on $M$.
	\end{defn}
	
	\begin{empt}\label{foli_graph_empt}
		Let  $\Pi\left( M,\sF\right)$ be the space of paths on leaves, that is, maps $\a : [0,1] \to M$ that are continuous with respect to the leaf topology on $M$. For such a path  let $s\left(\a \right) = \a\left( 0\right)$  be its source or initial point and let  $r\left(\a \right) = \a\left( 1\right)$ be its range or terminal point. The space $\Pi\left( M,\sF\right)$ has a partially defined multiplication: the product $\a\cdot \bt$ of two elements $\a$ and $\bt$ is defined if the terminal point of $\bt$ is the initial point of $\a$, and the result is the path $\bt$ followed by the path $\a$. (Note that this is the opposite to the usual composition of paths  $\al\#\bt = \bt \cdot \a$ used in defining the fundamental group of a space.)
	\end{empt}
	\begin{definition}\label{foli_path_space_defn}
		In the situation of \ref{foli_graph_empt} we say that the topological space $\Pi\left( M,\sF\right)$ is the \textit{space of path on leaves}.
	\end{definition}
	\begin{definition}\label{foli_groupoid_defn}\cite{candel:foliI}
		A groupoid $\G$ on a set $\sX$ is a category with inverses, having $\sX$ as its set of objects. For $y,z \in \sX$ the set of morphisms of $\G$ from $y$ to $z$ is denoted by $\G^z_y$.
	\end{definition}
	\begin{defn}\label{foli_graph_defn}\cite{candel:foliII}
		The \textit{graph}, or \textit{my groupoid}, of the foliated space $\left( M,\sF\right)$  is the quotient space of $\Pi\left( M,\sF\right)$ by the equivalence relation that identifies two paths $\a$ and $\bt$ if they have the same initial and terminal points, and the loop $\a \cdot \bt$ has trivial germinal my.
		The graph of $\left( M,\sF\right)$ will be denoted by $\G\left(M, \sF\right)$, or simply by  $\G\left( M\right)$ or by $\G$ when all other variables are understood.
	\end{defn}
	\begin{remark}
		There is the natural surjective continuous map
		\be\label{foli_cov_map_eqn}
		\Phi : \Pi\left( M,\sF\right)\to \G\left( M,\sF\right)
		\ee
		from the space of path on leaves to the foliation graph.
	\end{remark}
	\begin{proposition}\label{foli_chart_prop}
		Let $\mathfrak{A}= \left\{\sU_\iota\right\}$ be a regular foliated atlas of $M$. For each finite sequence of indices $\left\{\a_0 ,...,\a_k\right\}$, the product
		\be	\label{foli_chart_eqn}
		\sV_{   \a} =	 \G\left(\sU_{\iota_0} \right)~...~\G\left(\sU_{\iota_k} \right) \in \G\left(M, \sF\right) \quad \a = \left({\iota_0},...,{\iota_k}\right)
		\ee
		is either empty or a foliated chart for the graph $\G$. The collection of all such finite products is a covering of $\G$ by foliated charts. 
	\end{proposition}
	\begin{theorem}\label{foli_graph_thm}
		The graph $\G$ of $\left( M,\sF\right)$ is a groupoid with unit space $\G_0 = M$, and this algebraic structure is compatible with a foliated structure on $\G$ and $M$. Furthermore, the following properties hold.
		\begin{enumerate}
			\item [(i)] The range and source maps $r, s : \G \to M$ are topological submersions. 
			\item[(ii)] The inclusion of the unit space $M \to\G$ is a smooth map. 
			\item[(iii)] The product map $\G\times_M \G \to G$, given by $\left( \ga_1 , \ga_2\right) \mapsto\ga_1 \cdot \ga_2$, is smooth.
			\item[(iv)]  There is an involution $j: \G \to \G$, given by $j\left( \ga\right) = \ga^{-1}$, which is a diffeomorphism of $\G$, sends each leaf to itself, and exchanges the foliations given by the range. 
		\end{enumerate}
		
	\end{theorem}
	
	Above definitions refines the equivalence relation coming from
	the partition of $M$ in leaves $M = \cup L_{\alpha}$. 
	An element $\gamma$ of $\mathcal G$ is given by two points $x = s(\gamma)$,
	$y = r(\gamma)$ of $M$ together with an equivalence class of smooth
	paths: $\gamma (t)\in M$, $t \in [0,1]$; $\gamma (0) = x$, $\gamma
	(1) = y$, tangent to the bundle $\mathcal{F}$ ( i.e. with $\dot\gamma (t)
	\in \mathcal{F}_{\gamma (t)}$, $\forall \, t \in {\mathbb R}$) up to the
	following equivalence: $\gamma_1$ and $\gamma_2$ are equivalent if and only if
	the {\it my} of the path $\gamma_2 \circ \gamma_1^{-1}$ at the
	point $x$ is the {\it identity}. The graph $\mathcal G$ has an obvious
	composition law. For $\gamma , \gamma' \in G$, the composition
	$\gamma \circ \gamma'$ makes sense if $s(\gamma) = r(\gamma')$. If
	the leaf $L$ which contains both $x$ and $y$ has no my, then
	the class in $\mathcal G$ of the path $\gamma (t)$ only depends on the pair
	$(y,x)$. In general, if one fixes $x = s(\gamma)$, the map from $\mathcal G_x
	= \{ \gamma , s(\gamma) = x \}$ to the leaf $L$ through $x$, given
	by $\gamma \in \mathcal G_x \mapsto y = r(\gamma)$, is the my covering
	of $L$.
	Both maps $r$ and $s$ from the manifold $\mathcal G$ to $M$ are smooth
	submersions and the map $(r,s)$ to $M \times M$ is an immersion
	whose image in $M \times M$ is the (often singular) subset
	\begin{equation*}\label{subset}
		\{ (y,x)\in M \times M: \, \text{ $y$ and $x$ are on the same leaf}\}.
	\end{equation*}
	For
	$x\in M$ one lets $\Omega_x^{1/2}$ be the one dimensional complex
	vector space of maps from the exterior power $\wedge^k \,  \mathcal{F}_x$, $k =
	\dim F$, to ${\mathbb C}$ such that
	$$
	\rho \, (\lambda \, v) = \vert \lambda \vert^{1/2} \, \rho \, (v)
	\qquad \forall \, v \in \wedge^k \,  \mathcal{F}_x \, , \quad \forall \,
	\lambda \in {\mathbb R} \, .
	$$
	Then, for $\gamma \in\mathcal G$, one can identify $\Omega_{\gamma}^{1/2}$ with the one
	dimensional complex vector space $\Omega_y^{1/2} \otimes
	\Omega_x^{1/2}$, where $\gamma : x \to y$. In other words
	\be\label{foli_om_g_eqn}
	\Omega_{\mathcal G}^{1/2}=\, r^*(\Omega_M^{1/2})\otimes s^*(\Omega_M^{1/2})\,.
	\ee
	
	
	\begin{empt}\label{foli_sc_haus_empt}\cite{candel:foliII}
		The  groupoid of a foliated space all leaves of which are simply connected is Hausdorff .
	\end{empt}
	
	\begin{exercise}\label{foli_haus_exer}\cite{candel:foliII}
		Prove or decide the following.
		\begin{enumerate}
			\item The graph of a foliated space all leaves of which are simply connected is Hausdorff .
			\item The graph of a foliated space all leaves of which have trivial holo nomy is Hausdorff .
		\end{enumerate}	
	\end{exercise}

	\begin{definition}\label{foli_pseudo_defn}\cite{candel:foliI}.
		Let $N$ be a $q$-manifold. A $C^r$ pseudogroup $\Ga$ on $N$ is a collection of $C^r$ diffeomorphisms $h : D(h) \xrightarrow{\approx} R(h)$ between open subsets of $N$ satisfying the following axioms. 
		\begin{enumerate}
			\item  If $g, h \in \Ga$  and $R(h) \subset G(g)$, then $g \circ h \in \Ga$
			\item	 If $h \in \Ga$, then $h^{-1} \in \Ga$. 
			\item $\Id_N \in \Ga$. 
			\item If $h \in \Ga$ and $W \subset  D(h)$ is an open subset, then $\left.h\right|_W \in \Ga$. 
			\item If $h:D(h) \xrightarrow{\approx} R(h)$ is a $C^r$ diffeomorphism between open subsets of $N$ and if, for each $w \in  D(h)$, there is a neighborhood $W$ of $w\in  D(h)$ such that $\left.h\right|_W \in \Ga$, then $h \in \Ga$. 
		\end{enumerate}
		If $\Ga' \subset \Ga$ is also a pseudogroup, it is called a subpseudogroup of $\Ga$.
	\end{definition}
	
	\begin{remark}
		Any pseudogroup is a groupoid (cf. Definition \ref{groupoid_defn}).
	\end{remark}
	\begin{remark}\cite{candel:foliI}\label{foli_pseudoem}
		In the case of general foliations, the total my group of a foliated bun dle must be replaced by a local analogue called the \textit{my pseudogroup}.
	\end{remark}
	\begin{remark}\label{foli_groupoid_ned_defn}
		If $\left(M, \sF\right)$ is a foliated manifold and $N$ is a tranversal then
		\be\label{foli_gnn_eqn}
		\G^N_N \bydef \left\{\left. \ga \in \G\left(M, \sF\right)\right| s\left(\ga\right), r\left(\ga\right)\in N\right\}
		\ee
		is	a pseudogroup.

	\end{remark}
	
	\subsection{Operator algebras of foliations}\label{foli_alg_subsec}
	\paragraph*{}
	Here I follow to \cite{candel:foliII,connes:ncg94}.  Since the bundle $\Om^{1/2}$ is trivial (because $\G\left(M, \sF\right)$ admits partitions of unity), a choice of an everywhere positive density $\nu$ allows us to identify $\Ga_c\left(\G\left(M, \sF\right),\Om^{1/2}  \right)$  with $\Coo_c\left( \G\left(M, \sF\right)\right)$. 
	The definition of foliated space makes sense even when the underlying topological space fails to satisfy the Hausdorff  separation axiom. Non-Hausdorff  spaces appear naturally in the theory of foliations. In a graph of a foliated space is not necessary Hausdorff . It will be necessary to use functions with compact support on such spaces. However, a non-Hausdorff  space may not have sufficiently many such functions, the basic reason being that compact subsets of a Hausdorff  space are not necessarily closed. The non-Hausdorff  spaces that will appear here have a particularly simple local structure, and even when it is possible to construct appropriate functions using this local structure, the standard operation of “extension by 0” of local objects to the full space does not pro duce continuous functions. M. Crainic and I. Moerdijk \cite{cra_moe:nhaus} proposed a very natural way of dealing with this problem, and this preliminary section de scribes it. (That paper develops an extended sheaf  theory for non-Hausdorff  manifolds.
	Here $\sX$ will denote a separable topological space having the structure of a foliated space, but it is not required that the topology be Hausdorff . It is only required that $\sX$ can be covered by countably many open sets homeomorphic to a product $\D \times \mathcal Z$, where $D$ is an open ball in Euclidean $n$-space and $\mathcal Z$ is a separable locally compact Hausdorff  space. Let $\mathcal C\infty$ denote the structure sheaf  of the foliated space $\sX$, that is, the sheaf  of smooth functions on $\sX$. Let $\A$ be a sheaf of $\mathcal C\infty$-modules over $\sX$, for instance, the sheaf  of differential forms or other sheaves of smooth sections of foliated vector bundles. For such a sheaf  $\A$ over $\sX$, let $\A$ denote its Godement resolution: $A'\left(\sU\right)$ is the set of all sections (continuous or not) of $\A$ over $\sU\subset\sX$. For a Hausdorff  open subset $\mathcal W$ of $\sX$, let $\Ga_c\left( \mathcal W, \A\right) $ denote the set of continuous compactly supported sections of $\A$ over $\mathcal W$. If $\mathcal W\subset\sU$, then “extension by 0” induces a well defined homomorphism $\Ga_c\left( \mathcal W, \A\right)\to \A'\left(\sU \right)$. For an open subset $\sU$ of $\sX$, let $\Ga_c\left( \mathcal U, \A\right)$ denote the image of the homomorphism $\oplus \Ga_c\left( \mathcal W, \A\right)\to \A'\left(\sU\right)$ 
	\be\label{foli_incc_eqn}
	\Ga_c\left(\mathcal W, \A \right) \hookto \Ga_c\left(\mathcal U, \A \right).
	\ee
	Let $\G\bydef \G\left(M, \sF \right)$ be a foliation chart. 	The bundle $\Omega_M^{1/2}$ is trivial on $M$, and we
	could choose once and for all  a trivialization $\nu$ turning
	elements of $\Ga_c \left(\mathcal G , \Omega_{\mathcal G}^{1/2}\right)$ into functions.
	Let us
	however stress that the usage of half densities makes all the
	construction completely canonical.
	For $f,g \in \Ga_c \left(\mathcal G , \Omega_{\mathcal G}^{1/2}\right)$, the convolution
	product $f * g$ is defined by the equality
	\be\label{foli_prod_eqn}
	(f * g) (\gamma) = \int_{\gamma_1 \circ \gamma_2 = \gamma}
	f(\gamma_1) \, g(\gamma_2) \, .
	\ee
	This makes sense because, for fixed $\gamma : x \to y$ and fixing $v_x
	\in \wedge^k \,  \mathcal{F}_x$ and $v_y \in \wedge^k \,  \mathcal{F}_y$, the product
	$f(\gamma_1) \, g(\gamma_1^{-1} \gamma)$ defines a $1$-density on
	$G^y = \{ \gamma_1 \in G , \, r (\gamma_1) = y \}$, which is smooth
	with compact support (it vanishes if $\gamma_1 \notin\supp f$),
	and hence can be integrated over $G^y$ to give a scalar, namely $(f * g)
	(\gamma)$ evaluated on $v_x , v_y$.
	The $*$ operation is defined by $f^* (\gamma) =
	\overline{f(\gamma^{-1})}$,  i.e. if $\gamma : x \to y$ and
	$v_x \in \wedge^k \, \mathcal{F}_x$, $v_y \in \wedge^k \, \mathcal{F}_y$ then $f^*
	(\gamma)$ evaluated on $v_x , v_y$ is equal to
	$\overline{f(\gamma^{-1})}$ evaluated on $v_y , v_x$. We thus get a
	$*$-algebra $\Ga_c \left(\mathcal G , \Omega_{\mathcal G}^{1/2}\right)$. 
	where $\xi$ is a square integrable half density on $\mathcal G_x$. 
	For each leaf $L$ of
	$\left(M, \mathcal{F}\right)$ one has a natural representation of this $*$-algebra on the
	$L^2$ space of the my covering $\tilde L$ of $L$. Fixing a
	base point $x \in L$, one identifies $\tilde L$ with $\mathcal G_x = \{
	\gamma , s(\gamma) = x \}$ and defines
	\begin{equation}\label{foliepr_eqn}
		(\rho_x (f) \, \xi) \, (\gamma) = \int_{\gamma_1 \circ \gamma_2 =
			\gamma} f(\gamma_1) \, \xi (\gamma_2) \qquad \forall \, \xi \in L^2
		(\mathcal G_x),\
	\end{equation}

	Given
	$\gamma : x \to y$ one has a natural isometry of $L^2 (\mathcal G_x)$ on $L^2
	(G_y)$ which transforms the representation $\rho_x$ in $\rho_y$.
	\begin{lemma}
		If $f_1 \in \Ga_c\left(\sU_{   \a_1},\Om^{1/2} \right)$ and $f_2 \in \Ga_c\left(\sU_{   \a_2},\Om^{1/2} \right)$ then their convolution is a well-defined element $f_1*f_2 \in \Ga_c\left(\sU_{   \a_1}\cdot\sU_{   \a_2},\Om^{1/2} \right)$
	\end{lemma}
	
	\begin{proposition}\label{foliepr_prop}
		If $\sV \subset \G$ is a foliated chart for the graph of $\left(M, \sF\right)$ and $f \in \Ga_c\left(\sV, \Om^{1/2}\right)$ , then $\rho_x\left( f\right)$, given by \eqref{foliepr_eqn}, is a bounded integral operator on $L^2\left(\G_x \right)$.
	\end{proposition}
	\begin{empt}
		The space of compactly supported half-densities on $\G$ is taken as given by the exact sequence 
		\be\label{foli_ga_p_eqn}
		\bigoplus_{\a_0\a_1}\Ga_c\left(\sU_{   \a_0\a_1}, \Om^{1/2} \right) \to \bigoplus_{\a_0}\Ga_c\left(\sU_{   \a_0},\Om^{1/2} \right) \xrightarrow{\Ga_\oplus}  \Ga_c\left(\G,\Om^{1/2}\right) 
		\ee
		associated to a regular cover for $\left((M, \sF)\right)$ as above. The first step for defining a convolution is to do it at the level of $\bigoplus_{\a_0}\Ga_c\left(\sU_{   \a_0}\Om^{1/2} \right)$, as the following lemma indicates. 
	\end{empt}

	\begin{defn}\label{folied_defn}
		The \textit{reduced} $C^*$-algebra of the foliated space $\left(M,\sF\right)$ is the completion of $\Ga_c\left( \G,\Om^{1/2}\right)$ with respect to the pseudonorm \be\label{foli_pseudo_norm_eqn}
		\left\|f \right\| =\sup_{x \in M}\left\|  \rho_x\left(f\right)\right\|
		\ee where $\rho_x$ is given by  \eqref{foliepr_eqn}.
		This $C^*$-algebra is denoted by $C^*\left(M,\sF\right)$.
	\end{defn} 
	An obvious consequence of the construction of 
	$C^*\left(M,\sF\right)$ is the following. 
	
	\begin{cor}\label{foli_cov_alg_cor}
		Let $M$ be a foliated space and let $\mathfrak{A}$ be a regular cover by foliated charts. Then the algebra generated by the convolution algebras $\Ga_c\left( \G\left(\sU \right), \Om^{1/2}\right)$, $~\sU\in\mathfrak{A}$ is dense in  $C^*\left(M,\sF\right)$.
	\end{cor}
	
	\begin{empt}\label{folies_inc_empt}
		Let $\left(M,\sF\right)$ be an arbitrary foliated space and let $\sU\subset  M$ be an open subset.  Then $\left(\sU,\sF|_\sU\right)$ is a foliated space and the inclusion  $\sU\hookto  M$ induces a homomorphism of groupoids $\G\left(\sU \right)\hookto \G$ , hence a mapping
		\be\label{foli_inc_gc_eqn}
		j_\sU : \Ga_c\left(  \G\left(\sU \right) ,\Om^{1/2}\right)  \hookto \Ga_c\left( \G\left( M\right) ,\Om^{1/2}\right)
		\ee
		that is an injective homomorphism of involutive algebras. 
	\end{empt}
	\begin{prop}\label{folies_inc_prop} 
		Let $\sU$ be an open subset of the foliated space $M$. Then the inclusion $\sU \hookto M$ induces an isometry of $C^*\left(\sU,\sF|_\sU\right)$  into $C^*\left(M,\sF\right)$.
	\end{prop}
	\begin{remark}
		Similarly  to \cite{hilsum_scandalis:stab} the following notation will be used 
		\be
		C^*\left(\G^\sU_\sU \right) \bydef C^*\left(\sU,\sF|_\sU\right)
		\ee 
		where $\G^\sU_\sU$ is the reduction of  $\G$ on $\sU$ (cf. Definition \ref{groupoideduction_defn}.
	\end{remark}
	\begin{lemma}\label{foli_leaf_lem}
		Each element $\ga \in \G$ induces a unitary operator $\rho_\ga : L^2\left(s\left( \ga\right)\right)   \xrightarrow{\approx}  L^2\left(r\left( \ga\right)\right) $   that conjugates the operators $\rho_{s\left(\ga \right)} \left(f \right) $ and $\rho_{r\left(\ga \right)} \left(f \right)$. In particular, the norm of $\rho_x\left(f \right)$  is independent of the point in the leaf through $x$.
	\end{lemma}
	
	\begin{lem}\label{foli_point_lem}
		If $f \in \Ga_c\left(\G, \Om^{1/2}\right)$ does not evaluate to zero at each $\ga \in \G$, then there exists a point $x$ in $M$ such that $\rho_x\left( f\right)  \neq  0$.
	\end{lem}
	\begin{definition}\label{foli_fibration_defn}\cite{candel:foliI}
		A foliated space $\left(M, \sF\right)$ is a \textit{fibration} if for any $x$ there is an open transversal $N$ such that $x\in N$ and for every  leaf $L$ of $\left(M, \sF\right)$ the intersection $L\cap  N$ contains no more then one point.
	\end{definition}
	\begin{proposition}\label{foli_one_leaf_prop}
		The reduced $C^*$-algebra of a foliated space $M$ consisting of exactly one leaf is the algebra $\K\left( L^2\left(M \right) \right)$  of compact operators on $L^2\left(M \right)$.
	\end{proposition}

	\begin{prop}\label{foli_tens_comp_prop}\cite{candel:foliII}
		The reduced  $C^*$-algebra $C^*\left(\mathcal N\times Z \right)$ of the trivial foliated space $\mathcal N \times \mathcal Z$ is the tensor product $\K\otimes C_0\left(\mathcal Z\right)$, where $\K$ is the algebra of compact operators on $L^2\left(\mathcal N \right)$  and $ C_0\left(\mathcal Z\right)$ is the space of continuous functions on $\mathcal Z$ that vanish at infinity.
	\end{prop}
	\begin{thm}\label{foli_tens_comp_thm}\cite{candel:foliII}
		Assume that $\left(M, \sF\right)$ is given by the fibers of a fibration $p : M \to B$ with fiber $F$. Then $\left(M, \sF\right)$ is isomorphic to $C_0\left(B\right)\otimes \K\left(L^2\left(N\right)\right)$.
	\end{thm}
	\begin{remark}\label{foli_x_stateem}
		It is proven in \cite{candel:foliII} (See Claim 2, page 56) that for any $x \in M$ the given by \eqref{foliepr_eqn} representation $\rho_x: C^*\left(M, \sF \right)\to B\left( L^2\left(\G_x \right)\right)$   corresponds to a state $\tau_x : C^*\left(M, \sF \right)\to \C$.
	\end{remark}

	
	\begin{theorem}\label{foli_irred_hol_thm}\cite{candel:foliII}
		Let $(M,\sF)$ be a foliated space and let $x \in M$. Then the representation $\rho_x$ is irreducible if and only if the leaf through $x$ has no holonomy.
	\end{theorem}
		\begin{definition}\label{foli_groupoided_defn}\cite{renault:gropoid_ca}
		If $\Pi$ is the set of irreducible representations of $\C\left[\G \right]$ then the completion of $\C\left[\G \right]$ with respect to $C^*$-norm
		\be\label{groupoided_norm}
		\left\| a\right\| = \sup_{\pi\in \Pi}\left\|\pi\left( a\right)\right\| 
		\ee
		is said to be the \textit{reduced algebra} of $\G$. It will be denoted by $C^*\left(\G\right)$.
	\end{definition} 
	\begin{empt}\label{foli_pseudo_empt}
		Let $\G^N_N$ be a  given by \eqref{foli_gnn_eqn} pseudogroup
		\be\label{foli_pseudo_eqn}
		\begin{split}
			\left(a\cdot b\right)\left(\ga\right)= \sum_{\ga_1\circ\ga_2=\ga}a\left( \ga_1\right)a\left( \ga_1\right),\\
			a^*\left(\ga\right)= \overline{a\left(\ga^{-1}\right)} 
		\end{split}
		\ee
		Denote by $C^*\left(	\G^N_N \right)$ the reduced algebra of $\G^N_N$ (cf. Definition \ref{foli_groupoided_defn}).  
	\end{empt} 	
	
	\begin{theorem}\label{foli_mor_thm}\cite{candel:foliII}
		Let $(M,\F)$ be a foliated space and let $(Z, \mathscr H)$ be the holonomy pseudogroup associated to a complete transversal $Z$. Then the foliation $C^*$-algebra is Morita equivalent to the pseudogroup $C^*$-algebra.
	\end{theorem}
	\begin{remark}\label{foli_comp_transem}
		A transversal is \textit{complete} if it meets any leaf.
	\end{remark}
	\begin{remark}\label{foli_pseudo_algem}
		A pseudogroup $C^*$-algebra is $C^*\left(	\G^N_N \right)$.
	\end{remark}

	The following  lemma is a more strong version of the Theorem \ref{foli_mor_thm}.
	\begin{lemma}\label{foli_stab_lem}\cite{hilsum_scandalis:stab}
	Si $N$ est une transversale fidèle de $M$
	on a: $C^*\left(M, \sF \right)\cong C^*\left(\G^N_N \right)\otimes \K$.
	\end{lemma}
	\begin{remark}\label{foli_stabem}
	The English translation of the Theorem \ref{foli_stab_lem} is "if $N$ is a complete transversal then $C^*\left(M, \sF \right)\cong C^*\left(\G^N_N \right)\otimes \K$".
	\end{remark}
	\begin{empt}\label{foli_stab_empt}
	Besides the Lemma \ref{foli_stab_lem} we need some details of its proof. If $N$ is a complete transversal and
	$$
	E^M_N \bydef \Ga_c\left(r^{-1}\left(M \right), s^*\left(\Om^{1/2} \right)   \right) 
	$$
	then there is a $\Ga_c\left(\G, \Om \right)$-valued product  on $E^M_N$ given by
	$$
	\left\langle \xi, \eta \right\rangle\left( \ga\right) = \sum_{\substack{ s\left(\ga\right)=s\left(\ga'\right)\\r\left(\ga'\right)\in N}}\overline\xi\left(\ga'\circ\ga^{-1}\right)\eta\left(\ga'\right), \quad 
	$$
	If $\E^M_N$ is a completion of $E^M_N$ with respect to the norm $\left\| \xi\right|  \bydef \sqrt{\left\langle \xi, \xi \right\rangle}$ then $\E^M_N$ is $C^*\left(M,\sF\right)$-$C^*\left(\G^N_N\right)$-imprimitivity bimodule. (cf. Definition \ref{strong_morita_defn}). It follows that
	\be\label{foli_k1_eqn}
	C^*\left(M,\sF\right)\cong \K\left(\E^M_N\right).
	\ee
	On the other hand in \cite{hilsum_scandalis:stab} it is proven that there is an isomorphism 
	\be\label{foli_k2_eqn}
	\E^M_N\cong\ell^2\left(C^*\left(\G^N_N\right) \right)\cong C^*\left(\G^N_N\right)\otimes \ell^2\left( \N\right).  
	\ee
	The combination of equations \eqref{foli_k1_eqn}, \eqref{foli_k2_eqn} yields the Lemma  \ref{foli_stab_lem}.
	Below we briefly remind  the idea described in \cite{hilsum_scandalis:stab} proof of the equation  \ref{foli_k2_eqn}. There is a tubular neighborhood $V_N$ of $N$ which is homeomorphic to $N \times [0,1]^q$. This neighborhood yields the isomorphism
	\be\label{foli_e_sum_eqn}
	\E^M_N\cong C^*\left(\G^N_N\right)\otimes L^2\left( [0,1]^q\right) \oplus  \E^W_N
	\ee
	where $\E^W_N$ is a countably generated Hilbert $C^*\left(\G^N_N\right)$-module.
	From $ L^2\left( [0,1]^q\right)\cong\ell^2\left(\N\right)$ and the Kasparov stabilization theorem \ref{kasparov_stab_thm}, it follows that\\ $\E^M_N\cong\ell^2\left(C^*\left(\G^N_N\right) \right)$. If we  select any
	\be\label{foli_vec_eqn}
	\xi \in L^2\left( [0,1]^q\right)
	\ee
	such that $\left\|\xi\right\|= 1$ then  is an orthogonal basis
	\be\label{foli_basis_eqn}
	\left\{\eta_j\right\}_{j\in \N}\subset \ell^2\left(\N \right), \quad\eta_1\bydef\xi
	\ee
	In this basis the element  $a\otimes \xi\left\rangle \right\langle \xi\in C^*\left(\G^N_N\right)\otimes \K$ is represented by the following matrix
	\be\label{foli_mat_eqn}
	a\otimes \xi\left\rangle \right\langle \xi  =  	\begin{pmatrix}
		a& 0 &\ldots \\
		0& 0 &\ldots \\
		\vdots& \vdots &\ddots\\
	\end{pmatrix}.
	\ee
	
	\end{empt}

	\subsection{Restriction of foliation}

	\begin{lemma}\label{foles_lem}\cite{connes:ncg94}
	If $\sU \subset M$ is an open set and  $\left(\mathcal{U},\mathcal F|_{\mathcal{U}}\right) $ is the {restriction}   of $\left(M,\mathcal F \right)$ {to}  $\mathcal{U}$ then the graph 
	$\G\left(\mathcal{U},\mathcal F|_{\mathcal{U}}\right)$ is an open set in the graph $\G\left(M,\mathcal F \right)$, and the inclusion
	$$
	\Coo\left(\mathcal{U},\mathcal F|_{\mathcal{U}}\right)\hookto\Coo\left(M,\mathcal F \right)
	$$
	extends to an isometric *-homomorphism of $C^*$-algebras
	$$
	C^*\left(\mathcal{U},\mathcal F|_{\mathcal{U}}\right)\hookto C^*\left(M,\mathcal F \right).
	$$
	\end{lemma}
	
	\begin{remark}\label{folesem}\cite{connes:ncg94}
	This lemma, which is still valid in the non-Hausdorff  case \cite{connes:foli_survey}, allows one to reflect
	algebraically the local triviality of the foliation.
	Thus one can cover the manifold $M$ by
	open sets $\sU_\la$ such that $\sF$ restricted to $\sU_\la$ has a Hausdorff  space of leaves, $\sV_{\la} = \sU_\la/\sF$. and hence such that the $C^*$-algebras $C^*\left(\mathcal{U}_\la,\mathcal F|_{\mathcal{U}_\la}\right)$  are strongly Morita equivalent to the commutative $C^*$-algebras $C_0\left( B_\la\right)$. These subalgebras $C^*$-algebras $C^*\left(\mathcal{U}_\la,\mathcal F|_{\mathcal{U}_\la}\right)$ generate $ C^*\left(M,\mathcal F \right)$. 
	but of course they fit together in a very complicated way which is related to the global
	properties of the foliation.
	\end{remark}
	\subsection{Algebraic counterpart of geometric quotient}\label{quotions_a_sec}
	\paragraph{}
	Here I follow to \cite{{connes:ncg94}}
	We shall formulate the algebraic counterpart of the geometric operation of forming the
	quotient of a space by an equivalence relation, and show how to handle non-Hausdorff quotients by using noncommutative $C^*$-algebras. 
	
	\subsubsection{Quotient of two elements}
	Let $\sY \bydef \left\{a, b\right\}$ be a set consisting in two elements $a$ and $b$, so that the algebra $C\left(\sY \right)$ 
	of (complex-valued) functions on $\sY$ is the commutative algebra $\C \bigoplus 
	\C$ of $2 \times 2$ diagonal
	matrices. There are two ways of declaring that the two points a and b of Y are identical,
	i.e. of quotienting $\sY$ by the equivalence relation $a \sim b$.
	\begin{enumerate}
		\item The first is to consider the subalgebra $A \subset C\left( \sY\right)$  of functions $f$ on $\sY$ which take the
		same values at $a$ and $b: f(a) = f(b)$.
		\item The second is to consider the larger algebra $B \supset C\left(\sY \right)$ of all $2 \times 2$ matrices:
		\bean
		f = \begin{pmatrix}
			f_{aa} & f_{ab}\\
			f_{ba} & f_{bb}
		\end{pmatrix}.
		\eean
	\end{enumerate}
	The relation between the two algebras is the notion of strong Morita equivalence of
	$C^*$-algebras. This relation preserves many
	invariants of $C^*$-algebras, such as $K$-theory and the topology of the space of irreducible
	representations. Thus, in our example the pure states $\om_a$ and $\om_b$ of $B$ which come
	directly from the points $a$ and $b$ of $\sY$ by the formula
	\bean
	\om_a\left( f\right) = f_{aa}, \quad \om_a\left( f\right) = f_{bb}\quad \forall a \in B
	\eean
	yield equivalent irreducible representations of $B = \mathbb{M}_2\left(\C \right)$, which corresponds to the
	identification $a \sim b$ in the spectrum of $B$.
	\subsubsection{Quotients of foliated manifolds}
	\paragraph*{}
	The notions of foliation and its $C^*$-algebra are explained in \cite{candel:foliI,candel:foliII}. A foliated manifold $\left(M, \F \right)$ contains a differential manifold $M$ such that
	$$
	M = \bigsqcup_{\la \in \La} \L_\la
	$$
	where $\L_\la$ is a set of leaves. If any leaf simply connected then any irreducible representation correspond to a leaf. There is an equivalence relation on $M$ such that
	$$
	x' \sim x'' \quad \exists \la\in \La \quad x', x'' \in \L_\la
	$$ 
	It is proven that the spectrum of the explained in \cite{candel:foliII} reduced $C^*$-algebra of 	$C^*_r\left(M, \F \right)$ naturally homeomorphic to $M/ \sim$. Otherwise there is Hausdorff blowing-up $C_0\left( M\right) \to M\left(C^*_r\left(M, \F \right) \right)$. In result one has the following diagram
	\bean
	\begin{tikzcd}
		\mathfrak{Gelfand}\left(C^*_r\left(M, \F \right)\right)	 \arrow[rr, "\phi_{(C^*_r\left(M, \F \right)}"]\arrow[rd]  & ~ & M \arrow[ld, , "\mathfrak{Blowing}"]\\
		& M/\sim &
	\end{tikzcd}
	\eean

		\section{Noncommutative torus $\mathbb{T}^n_{\Theta}$}\label{nt_descr_subsec}
	\begin{definition}\label{nt_qirr_defn}\cite{wagner:pb}
		Denote by "$\cdot$" the scalar product on $\R^n$.
		The matrix $\Theta$ is called \emph{quite irrational} if, for all $\lambda \in \Z^n$, the condition $\exp(2\pi i \, \lambda\cdot \Theta  \mu) = 1$ for all $
		\la,\mu \in \Z^n$ implies $\lambda = 0$.
	\end{definition}
	\begin{definition}\label{nt_defn}\cite{wagner:pb}
		Let $\Theta$ be an invertible, real  skew-symmetric quite irrational $n \times n$ matrix. A \textit{noncommutative torus} $C\left(\mathbb{T}^n_{\Theta}\right)$ is the universal $C^*$-algebra generated by the set $\left\{U_k\right\}_{k \in \Z^n}$ of unitary elements which satisfy to the following relations.
		\begin{equation}\label{nt_unitary_product_eqn}
			U_k U_p = e^{-\pi ik ~\cdot~ \Theta p} U_{k + p}; ~~~   \end{equation}
	\end{definition}
	Following condition holds
	\begin{equation}\label{nt_unitary_product_comm_eqn}
		U_k U_p = e^{-2\pi ik ~\cdot~ \Theta p}U_p U_k.
	\end{equation}
	An alternative description of $\C\left(\mathbb{T}^n_{\Theta}\right)$ is such that if
	\begin{equation}\label{nt_th_eqn}
		\Th = \begin{pmatrix}
			0& \th_{12} &\ldots & \th_{1n}\\
			\th_{21}& 0 &\ldots & \th_{2n}\\
			\vdots& \vdots &\ddots & \vdots\\
			\th_{n1}& \th_{n2} &\ldots & 0
		\end{pmatrix}=\begin{pmatrix}
			0& \th_{12} &\ldots & \th_{1n}\\
			-\th_{12}& 0 &\ldots & \th_{2n}\\
			\vdots& \vdots &\ddots & \vdots\\
			-\th_{1n}& -\th_{2n} &\ldots & 0
		\end{pmatrix}
	\end{equation}
	then $C\left(\mathbb{T}^n_{\Theta}\right)$ is the universal $C^*$-algebra generated by unitary elements   $u_1,..., u_n \in U\left( C\left(\mathbb{T}^n_{\Theta}\right)\right) $ such that following condition holds
	\begin{equation}\label{nt_com_eqn}
		\begin{split}
			u_j u_k = e^{-2\pi i \theta_{jk} }u_k u_j.
		\end{split}
	\end{equation}
	Unitary  operators $u_1,..., u_n$ correspond to the standard basis of $\mathbb{Z}^n$ and they are given by
	\be\label{nt_gen_eqn}
	u_j = U_{k_j },\quad \text{ where } k_j=\left(0,...,\underbrace{ 1}_{j^{\text{th}}-\text{place}},...,0 \right)
	\ee
	\begin{defn}\label{nt_uni_defn}
		The unitary elements 
		$u_1,..., u_n \in U\left(C\left(\mathbb{T}^n_{\theta}\right)\right)$ which satisfy the relations \eqref{nt_com_eqn}, \eqref{nt_gen_eqn}
		are said to be \textit{generators} of $C\left(\mathbb{T}^n_{\Theta}\right)$. The set $\left\{U_l\right\}_{l \in \Z^n}$ is said to be the \textit{basis} of $C\left(\mathbb{T}^n_{\Theta}\right)$.
	\end{defn}
	If $a \in C\left(\mathbb{T}^n_{\Th}\right)$ is presented by a series
	\be\label{nt_series_eqn}
	a = \sum_{l \in \mathbb{Z}^{n}}c_l U_l;~~ c_l \in \mathbb{C}
	\ee
	and the series $\sum_{l \in \mathbb{Z}^{n}}\left| c_l\right| $ is convergent then from the triangle inequality it follows that the series is $C^*$-norm convergent and the following condition holds.
	\begin{equation}\label{nt_norm_estimation_eqn}
		\left\|a \right\| \le \sum_{l \in \mathbb{Z}^{n}}\left| c_l\right|.
	\end{equation}
	In particular if $\mathrm{sup}_{l \in \mathbb{Z}^n}\left(1 + \|l\|\right)^s \left|c_l\right| < \infty, ~ \forall s \in \mathbb{N}$ then $\sum_{l \in \mathbb{Z}^{n}}\left| c_l\right|< \infty $ and  there is the natural inclusion $\phi_\infty: \sS\left(\mathbb{Z}^n\right) \subset C\left(\mathbb{T}^n_{\Theta}\right)$ of vector spaces. If
	\be\label{nt_coo_eqn}
	\Coo\left(\mathbb{T}^n_{\Theta}\right)\bydef \phi_\infty \left( \sS\left(\mathbb{Z}^n\right)\right) \subset C\left(\mathbb{T}^n_{\Theta}\right)
	\ee
	then $\Coo\left(\mathbb{T}^n_{\Theta}\right)$ is a pre-$C^*$-algebra, and the Fourier transformation  yields the $\C$-isomorphism
	\be\label{nt_coo_iso_eqn}
	\Coo\left(\mathbb{T}^n\right)\cong \Coo\left(\mathbb{T}^n_{\Theta}\right).
	\ee
	There is a  state 
	\be\label{nt_state_eqn}
	\begin{split}
		\tau: C\left(\mathbb{T}^n_{\Theta}\right) \to \C;\\
		\sum_{k \in \Z^n} a_k U_k \mapsto a_{\left(0,...,0\right) }; \quad \text{ where } a_k \in \C,
	\end{split}
	\ee
	which induces the faithful GNS representation. 
	The $C^*$-norm completion  $C\left(\mathbb{T}^n_{\Theta}\right)$ of $\Coo\left(\mathbb{T}^n_{\Theta}\right)$ is a $C^*$-algebra and there is a faithful representation
	\begin{equation}\label{ntepr_eqn}
		C\left(\mathbb{T}^n_{\Theta}\right) \to B\left( L^2\left(C\left(\mathbb{T}^n_{\Theta}\right), \tau\right)\right) .
	\end{equation}
	(cf. Definition \ref{gns_defn}).  Similarly to the equation \eqref{from_a_to_l2_eqn} there is a $\C$-linear map 
	\begin{equation}\label{nt_to_hilbert_eqn}
		\Psi_\Th:C\left(\mathbb{T}^n_{\Theta}\right) \hookto L^2\left(C\left(\mathbb{T}^n_{\Theta}\right), \tau\right).
	\end{equation}
	If 
	\be\label{nt_xik_eqn}
	\xi_k \bydef \Psi_\Th\left(U_k \right)
	\ee 
	then from \eqref{nt_unitary_product_eqn}, \eqref{nt_state_eqn} it turns out
	\begin{equation}\label{nt_h_product_eqn}
		\tau\left(U^*_k  U_l \right) = \left(\xi_k, \xi_l \right)  = \delta_{kl},   
	\end{equation} 
	i.e. the subset $\left\{\xi_k\right\}_{k \in \mathbb{Z}^n}\subset L^2\left(C\left(\mathbb{T}^n_{\Theta}\right), \tau\right)$ is an orthogonal basis of  $L^2\left(C\left(\mathbb{T}^n_{\Theta}\right), \tau\right)$.
	Hence the Hilbert space  $L^2\left(C\left(\mathbb{T}^n_{\Theta}\right), \tau\right)$ is naturally isomorphic to the Hilbert space $\ell^2\left(\mathbb{Z}^n\right)$ given by
	\begin{equation*}
		\ell^2\left(\mathbb{Z}^n\right) = \left\{\xi = \left\{\xi_k \in \mathbb{C}\right\}_{k\in \mathbb{Z}^n} \in \mathbb{C}^{\mathbb{Z}^n}~|~ \sum_{k\in \mathbb{Z}^n} \left|\xi_k\right|^2 < \infty\right\}
	\end{equation*}
	and the $\C$-valued scalar product on $\ell^2\left(\mathbb{Z}^n\right)$ is given by
	\begin{equation}\label{nt_xi_eqn}
		\left(\xi,\eta\right)_{ \ell^2\left(\mathbb{Z}^n\right)}= \sum_{k\in \mathbb{Z}^n}    \overline{\xi}_k\eta_k.
	\end{equation}
	From \eqref{nt_unitary_product_eqn} and \eqref{nt_xik_eqn} it follows that the representation $C\left(\mathbb{T}^n_{\Theta}\right)\hookto B\left( L^2\left(C\left(\mathbb{T}^n_{\Theta}\right), \tau\right)\right)$ corresponds to the following action
	\be\label{nt_l2_eqn}
	\begin{split}
		C\left(\mathbb{T}^n_{\Theta}\right)\times  L^2\left(C\left(\mathbb{T}^n_{\Theta}\right), \tau\right)\to  L^2\left(C\left(\mathbb{T}^n_{\Theta}\right), \tau\right);\\
		U_k\xi_l = e^{-\pi ik ~\cdot~ \Theta l}\xi_{k + l} .
	\end{split}
	\ee

	
	
	\subsection{$K$-theory and Poincar\'e duality}\label{rieffel_sec}\cite{varilly:nc_intro}
	The noncommutative torus provides an example of a pre-$C^*$-algebra,
	which is neither commutative nor approximately finite but has an
	interesting and computable $K$-theory. The group
	$K_1(C\left(\T^2_\th \right) )$ is fairly easy to find. There are two generating
	unitaries, $u$ and~$v$, and all the $u^m v^n$ are mutually
	non-homotopic in $U(A_\th)$. Indeed, since $\tau_0(1) = 1$ and
	$\tau_0(u^m v^n) = 0$ for $(m,n) \neq (0,0)$, there cannot be a
	continuous path in $U(A_\th)$ from~$1$ to $u^m v^n$. Passing to
	matrices in $M_k(A_\th)$ cannot remedy this, since the same argument
	works, with $\tau_0$ replaced by~$\tau_0 \ox \tr$. Thus
	$K_1(\A_\th) = K_1(A_\th) = \Z[u] \oplus \Z[v]$.
	
	To determine $K_0(\A_\th)$, we seek projectors in $M_k(\A_\th)$ not
	equivalent to $e := 1 \oplus 0_{k-1}$. In fact, due to the irrationality
	of~$\th$, such projectors may be found in $\A_\th$ itself: the
	\textit{Powers - Rieffel projector}
	\be\label{rp_eqn}
	\begin{split}
		p_\th \in \Coo\left( \T^2_\th\right) \subset C\left( \T^2_\th\right)
	\end{split} 
	\ee
	has the
	characteristic property that $\tau_0(p) = \th$. Given this projector,
	the map $m[e] + n[p] \mapsto m\tau_0(e) + n\tau_0(p) = m + n\th$
	defines a map from $K_0(\A_\th)$ to $\Z + \Z\th$ which, by a theorem
	of Pimsner and Voiculescu~\cite{p-v:irr}, is an isomorphism of
	ordered groups.
	
	The projector $p$ is constructed as follows. Write elements of
	$\A_\th$ as $a = \sum_s f_s v^s$, where $f_s = \sum a_{rs} u^r$ is a
	Fourier series expansion of $f_s(t) = \sum a_{rs} e^{2\pi irt}$ in
	$\Coo(\T)$. Now we look for $p$ of the form
	$$
	p_\th = gv + f + hv^{-1}.
	$$
	Since $p^* = p$, the function $f$ is real and
	$\Bar{h(t)} = g(t + \th)$. Assuming $\thalf < \th < 1$, as we may, we
	choose $f$ to be a smooth increasing function on
	$[0, 1 - \th]$, define $f(t) := 1$ if $t \in [1 - \th, \th]$, and
	$f(t) := 1 - f(t - \th)$ if $t \in [\th,1]$; then let $g$ be the
	smooth bump function supported in~$[\th,1]$ given by
	$g(t) := \sqrt{f(t) - f(t)^2}$ for $\th \leq t \leq 1$. One checks
	that these conditions guarantee $p^2 = p$ (look at the coefficients
	of $v^2$, $v$ and~$1$ in the expansion of~$p^2$). Moreover,
	$$
	\tau_0(p_\th) = a_{00} = \int_0^1 f(t) \,dt
	= \int_0^{1-\th} f(t)\,dt + (2\th - 1) + \int_\th^1 f(t)\,dt = \th.
	$$
	The existence of these projectors (variation of $f$ on the interval
	$[0, 1 - \th]$ gives rise to many homotopic projectors) shows that
	the topology of the noncommutative torus is very disconnected, in
	contrast to the ordinary torus $\A_0$, whose only projectors are~$0$
	and $1$.
	\begin{theorem}\label{rp_nc_thm}\cite{Rieffel:irrat}
		For each $\bt \in \Z\th \cap \left[0,1\right]$ there is a projection
		$p_\th \in C\left(\T^2_\th \right)$  such that $\tau\left( p\right) =\bt$.
	\end{theorem}
	[The commutative torus $\Coo(\T^2)$ has the same $K$-groups:
	$K_j(\A_0) = \Z \oplus \Z$ for $j = 1,2$. However, this algebra has no
	Powers--Rieffel projector: the second generator of $K_0(\A_0)$ is
	obtained by pulling back the Bott projector from $K_0(\Coo(S^2))$.]
	
	Thus, $K_\8(\A_\th)$ has four generators: $[e]$, $[p]$, $[u]$
	and~$[v]$. The intersection form is antisymmetric (this is typical of
	dimension two) and the nonzero pairings of the generators
	are~\cite{varilly:nc_intro}:
	$$
	\bigl<[e], [p]\bigr> = - \bigl<[p], [e]\bigr> = 1,  \qquad
	\bigl<[u], [v]\bigr> = - \bigl<[v], [u]\bigr> = 1.
	$$
	The matrix of the form is a direct sum of two
	$\begin{pmatrix}
		0 & -1 \cr 1 & 0 \cr	\end{pmatrix}$ blocks, so it is nondegenerate, and
	Poincar\'e duality holds.
	
	\bigskip
	
	We have constructed a geometry $(\A_\th, \H, D_\tau, \Ga, J)$, which
	may be denoted by $\T^2_{\th,\tau}$. When $\th = 0$, $\tau$ plays the
	r\^ole of the modular parameter of an elliptic curve. Whether or not
	this provides an opening to a noncommutative theory of elliptic curves
	is a tantalizing speculation; it remains to be seen whether
	$\T^2_{\th,\tau}$ can yield useful arithmetic information.
	
	\begin{lemma}\label{nt_gen_lem}\cite{Rieffel:stable_irrat}
		Let $\th$ be an automorphisms of the unital $C^*$-algebra $A$
		which is in the connected component of the identity automorphisms of
		$A$, and let $\a$. also denote the corresponding action of $Z$ on $A$. Suppose
		that
		\begin{enumerate}
			\item 	Every element of $K_1\left(A \right)$  is represented y an invertible element
			in $A$ itself.
			\item The projections in  $A$ generate $K_0\left(A \right)$.  
			
		\end{enumerate}
		Then every element in $K_0\left(A\otimes_\a \Z \right)$  is represented by an invertible
		element of $A\otimes_\a \Z$. 
		
	\end{lemma}
	\begin{rem}
		It is proven in \cite{Rieffel:stable_irrat} that
		\be
		\begin{split}
			K_0\left(C\left(\T^n_\th \right)  \right) \cong \Z^{2^{n-1}}
		\end{split}
		\ee
		
	\end{rem}
	\begin{theorem}\cite{Rieffel:stable_irrat}
		If $\th$ is not rational, then the positive cone of 	$K_0\left(C\left(\T^n_\th \right)  \right)$
		consists of exactly the elements on which $\tau$ is positive.
	\end{theorem}
	
	\begin{theorem}\label{k_0_tor_thm}\cite{Rieffel:stable_irrat}
		If $\th$ is not rational, then the projections in $C\left(\T^n_\th \right)$
		generate $K_0\left(C\left(\T^n_\th \right)  \right)$.
	\end{theorem}
	
	\begin{theorem}\label{k_1_tor_thm} \cite{Rieffel:stable_irrat}
		If $\th$ is not rational, then every element of $K_1\left(C\left(\T^n_\th \right)  \right)$ is represented by an invertible element of $C\left(\T^n_\th \right)$
	\end{theorem}
	
	\begin{theorem}\cite{Rieffel:stable_irrat}
		If $\th$ is not rational, then the natural map from
		$U_1\left( C\left(\T^n_\th \right)\right)/ U^0_1\left( C\left(\T^n_\th \right)\right)$  to $U_1\left( C\left(\T^n_\th \right)\right)$ is an isomorphism.
	\end{theorem}

	\begin{theorem}\cite{Rieffel:stable_irrat}
		If $\th$ is not rational, then any two projections in
		$\mathbb{M}_m\left(C\left(\T^n_\th \right) \right)$ which represent the same element of $K_1\left(C\left(\T^n_\th \right)  \right)$ are in the  same
		connected component of the set of projections in $\mathbb{M}_m\left(C\left(\T^n_\th \right) \right)$.
	\end{theorem}
	
\section{Neo realism and Bohr locales}
\paragraph*{}
Long time The quantum physics was not easy to understand. This fact was caused by the absence of the strict mathematical formulation of it. The classical physics can be explained by the classical mathematical logic called {\it realism}. The quantum physics requires the application of the  topos logic called {\it neo realism}.

\subsection{Topos logic and quantum physics}

\paragraph{Realism in classical physics.} Here I follow to \cite{doering:neo}. In classical physics, one has a space
of states $\mathcal{S}$, and physical quantities $A$ are represented by
real-valued measured functions $f_{A}:\mathcal{S}\rightarrow\mathbb{R}$. A proposition about a physical
quantity $A$ is of the form \textquotedblleft$A\in\Delta$\textquotedblright,
which means \textquotedblleft the physical quantity $A$ has a value in the
(Borel) set $\Delta$\textquotedblright. This proposition is represented by the
inverse image $f_{A}^{-1}(\Delta)\subseteq\mathcal{S}$. In general,
propositions about the physical system correspond to Borel subsets of the
state space $\mathcal{S}$. If we have two propositions \textquotedblleft%
$A\in\Delta_{1}$\textquotedblright, \textquotedblleft$B\in\Delta_{2}%
$\textquotedblright\ and the corresponding subsets $f_{A}^{-1}(\Delta_{1})$,
$f_{B}^{-1}(\Delta_{2})$, then the intersection $f_{A}^{-1}(\Delta_{1})\cap
f_{B}^{-1}(\Delta_{2})$ corresponds to the proposition \textquotedblleft%
$A\in\Delta_{1}$ and $B\in\Delta_{2}$\textquotedblright, the union $f_{A}%
^{-1}(\Delta_{1})\cup f_{B}^{-1}(\Delta_{2})$ corresponds to \textquotedblleft%
$A\in\Delta_{1}$ or $B\in\Delta_{2}$\textquotedblright, and the complement
$\mathcal{S}\backslash f_{A}^{-1}(\Delta_{1})$ corresponds to the negation
\textquotedblleft$A\notin\Delta_{1}$\textquotedblright. Moreover, given a
state $s$, i.e., an element of the state space $\mathcal{S}$, each proposition
is either true or false: if $s$ lies in the subset of $\mathcal{S}$
representing the proposition, then the proposition is true, otherwise it is
false. Every physical quantity $A$ has a value in the state $s$, namely
$f_{A}(s)\in\mathbb{R}$. Thus classical physics is a \emph{realist} theory in
which propositions have truth-values independent of measurements, observers
etc. The logic is Boolean, since classical physics is based on constructions
with sets and functions, i.e., it takes place in the topos
$\operatorname*{\mathbf{Set}}$. We take this as a rule: if we want to describe a
physical system $S$ as a classical system, then the topos $\operatorname*{\mathbf{Set}}%
$ is used. This means no departure from what is ordinarily done, but it
emphasises certain structural and logical aspects of the theory.

\paragraph{Instrumentalism in quantum theory.} In quantum theory, the
mathematical description is very different. Physical quantities $A$ are
represented by self-adjoint operators $\widehat{A}$ on a Hilbert space
$\mathcal{H}$. While $\mathcal{H}$ can be called a space of states, the states
$\psi\in\mathcal{H}$ play a very different r\^{o}le from those in classical
theory. In particular, a state $\psi$ does not assign values to all physical
quantities, only to those for which $\psi$ happens to be an eigenstate. The
spectral theorem shows that propositions \textquotedblleft$A\in\Delta
$\textquotedblright\ are represented by projection operators $\widehat{E}%
[A\in\Delta]$ on Hilbert space. Unless $\psi$ is an eigenstate of $A$, such a
proposition is neither true nor false (except for the trivial cases
$\widehat{E}[A\in\Delta]=\widehat{0}$, which represents trivially false
propositions, and $\widehat{E}[A\in\Delta]=\widehat{1}$, which represents
trivially true propositions). The mathematical formalism of quantum theory is
interpreted in an \emph{instrumentalist} manner: given a state $\psi$, the
proposition \textquotedblleft$A\in\Delta$\textquotedblright\ is assigned a
probability of being true, given by the expectation value $p(A\in\Delta
;\psi):=\left\langle \psi\right\vert \widehat{E}[A\in\Delta]\left\vert
\psi\right\rangle $. This means that upon measurement of the physical quantity
$A$, one will find the measurement result to lie in $\Delta$ with probability 
$p(A\in\Delta;\psi)$. This interpretation depends on measurements and an external
observer. Moreover, the measurement devices (and the observer) are described
in terms of classical physics, not quantum physics.  The explanation of logical aspects of physics and their relations to the Topos theory is  explained in \cite{doering:neo}.
\subsection{The Bohr topos}

The aim of Bohr toposes is to relate algebraic quantum mechanics to
topos theory, so as to construct new foundations for quantum logic
and quantum spaces.	Bohr toposes provide the topos-theoretical notion of a space which
intrinsically carries the intuitionistic logical structure of a Heyting algebra. We need the following.
\begin{definition}\label{alex_top_defn}\cite{johnstone:stone_spaces}.
	Let $X$ be a partially ordered set. We define the \textit{Alexandrov topology} $\Upsilon(X, \le)$ (or simply 
	$\Upsilon(X)$, if the partial order is obvious) to be the collection of all \textit{upper sets}
	in $X$ (i.e. sets $U$ such that $x \in  U$ and $x \le y$ imply $y \in U$); this is clearly a 
	topology, since it is closed under arbitrary unions and intersections. 
	
\end{definition}

\begin{remark}\label{alex_not_sober}\cite{beyond_sober_duality}
	Alexandrov topology are not always sober.
\end{remark} 
For any $C^*$-algebra 
one can consider the partially ordered set $\mathcal C(A)$  of commutative
subalgebras of $A$, ordered by inclusion. The assignment of this partially ordered to
a $C^*$-algebra induces a functor $\mathcal C$ from the category of $C^*$-algebras to the category of partially ordered sets, since one has for any
homomorphism $A \xrightarrow{h}B$ a functor $
\mathcal C(A)\xrightarrow{\mathcal C}  \mathcal C(B)$
\newline
\begin{tikzcd}
	C \arrow[d] \arrow[r, mapsto]
	& h(C) \arrow[d]\\
	B\arrow[r, mapsto]	&h (B)
\end{tikzcd}
\\
\\
A partially ordered set yields an Alexandrov space $L(A)$.
\begin{definition}\label{bohr_topos_defn}\cite{joost:borification}
	The \textit{Bohr topos} of $C^*$-algebra $A$ is a  topos sheaves $\mathbf{Sh}\left( L(A)\right)$ over an Alexandrov space $L(A)$.
\end{definition}
\begin{remark}
	The open sets of the Alexandrov space is a partially ordered set, so from the Theorem \ref{locale_thm} it follows that the {Bohr topos} is localic. 
\end{remark}
\begin{definition}\label{bohr_locale_defn}
	The locale of the {Bohr topos} of $C^*$-algebra $A$ is said to be the \textit{Bohr} $A$-\textit{locale}, denoted by $\mathfrak{BohrLocale}(A)$.
\end{definition}
\begin{remark}\label{bohr_not_spat_rem}
	From the Remark \ref{alex_not_sober} and the Proposition \ref{sober_prop} it follows that the {Bohr topos} should not be always spatial.
\end{remark}

\subsection{Bohrification}
Here I follow to \cite{topos_quant}.
\paragraph{}Whereas the previous subsection considered the topos
$\asstopos(\alg{A})$ associated to a $C^*$-algebra $\alg{A}$, this
one is devoted to a particular object $\functor{\alg{A}}$ in this
topos.
In fact, the definition of $\functor{\alg{A}}$ is `tautological' in a literal
sense.
\begin{definition}
	\label{def:internalalgebra}
	Let $\alg{A}$ be a $C^*$-algebra in $\Sets$.
	The functor $\functor{\alg{A}} :
	\mathcal C\left(A \right) \to \Sets$ is given on objects by
	\[
	\functor{\alg{A}}(\alg{C}) = \alg{C},
	\]
	and on morphisms $\alg{D} \subseteq \alg{C}$ of
	$\sC \left( A\right) $ as the inclusion $\functor{\alg{A}}(\alg{D})
	\hookrightarrow \functor{\alg{A}}(\alg{C})$.
\end{definition}
Note that the functor
$\functor{\alg{A}} : \sC\left( A\right)  \to\Sets$ factors through $\bf{Cstar}$ or $\cat{cCstar}$
via the  forgetful embedding of $\bf{CStar}$ or $\cat{cCstar}$ in the ambient topos
$\Sets$.

\emph{The functor $\functor{\alg{A}}$ is our internal
	$C^*$-algebra} or {\it Bohrification} of $A$. 
topos \Sets\ to $\asstopos(\alg{A})$, the (generally) noncommutative
 $C^*$-algebra $\alg{A}$ has become a commutative $C^*$-algebra $\functor{\alg{A}}$.
Multiplication of two non-commuting operators is no longer
defined, since they live in different commutative subalgebras.{Kochen
	and Specker refer to such a structure as a partial
	algebra~\cite{kochenspecker67} and stress its relevance for
	the foundations of quantum theory; in a partial algebra both  addition
	and multiplication need only be defined for commuting operators.}
	\begin{definition}\label{gelfand_spectrum_defn}
		If we consider a functor 
	\bean
\Sigma : \cat{cCstar} \to \cat{Top},\\
C\left(\sX \right) \mapsto \sX
\eean 
then the composition $\Sigma \circ  \functor{\alg{A}}$ is functor. The image of  $\Sigma \circ  \functor{\alg{A}}$ is a topos denoted by $\ulS(\uA)$. We say that $\ulS(\uA)$ is the {\it Gelfand spectrum} of $A$.

	\end{definition}

\subsection{The Kochen--Specker Theorem}
\label{sec:spectrum}
The previous discussion yields a mapping $A\mapsto\ulS(\uA)$, which associates a certain internal
locale
to a (generally) noncommutative \ca. As argued in the Introduction,  $\ulS(\uA)$
describes the quantum logic of the physical system whose algebra of observables
is $A$.

An important property of the  internal spectrum $\underline{\Sigma}$  is that it may
typically be
highly non-spatial from an external point of view.
\begin{theorem}
	\label{prop:kochenspecker}
	Let $H$ be a Hilbert space with $\dim(\sH)>2$ and let $\alg{A}$ be the
	$C^*$-algebra of bounded operators on $\sH$. Then the locale $\ulS(\uA)$
	has no  points.
\end{theorem}
\subsection{Valuations}
 We say that a \emph{valuation} on a $C^*$-algebra $A$ is a nonzero map $\la:A_{\mathrm{sa}}\to \R$  that is linear on commuting operators and dispersion-free, i.e., $\la(a^2)=\la(a)^2$ for all $a\in A_{\mathrm{sa}}$. 
If $A$ is commutative, the Gelfand spectrum $\Sg(A)$ consists precisely of the valuations on $A$.{
\begin{theorem}\label{t3}
	There is a bijective correspondence between:
	\begin{itemize}
		\item 
		Valuations on $A$;
		\item Points  of $\underline{\Sg}(\uA)$  in $\Sh(\CA)$;
		\item Continuous cross-sections $\sg:\CA\to\Sg$  of the bundle $\pi:  \Sg\to \CA$.
	\end{itemize}
	In particular, this bundle admits no continuous cross-sections as soon as $A$ has no valuations,\footnote{The claim following this footnote sign  is the content of the original Kochen--Specker Theorem \cite{kochenspecker67}.} as in the case $A=B(H)$ with $\dim(H)>2$.
\end{theorem}
	Physically, a valuation correspond to a so-called \emph{noncontextual hidden variable}, which
assigns a sharp value to each observable $a$ \emph{per se}. A \emph{contextual hidden variable}
gives a sharp value to $a$ \emph{seen in a specific measurement context in which it, in particular, may be measured}. In our mathematization,  measurement contexts are identified with commutative $C^*$-subalgebras of some ambient noncommutative C*-algebra $A$, so that a contextual  hidden variable assigns a value to a pair $(a,C)$ where $a\in C$. Hence Theorem \ref{t3} identifies \emph{noncontextual} hidden variables with \emph{continuous} cross-section of $\pi:  \Sg\to\CA$, 
\emph{contextual} hidden variable corresponding to possibly \emph{discontinuous} cross-sections.

The mathematics neatly fits the physics here, but it should be realized that specific examples of C*-algebras $A$ may suggest coarser natural topologies on $\CA$ than the Alexandrov topology (like the Scott topology), which in turn may imply stronger continuity conditions. We thank the referee for this comment.} 
	
		\end{appendices}

\section*{Acknowledgment}

\paragraph*{}
Author would like to acknowledge members of the Moscow State University Seminars
"Noncommutative geometry and topology", "Algebras in analysis" leaded by professors A. S. Mishchenko and  A. Ya. Helemskii for a discussion
of this work.


\begin{thebibliography}{10}
	
 	\bibitem{apt_mult}	Charles A. Akemann, Gert K. Pedersen, Jun Tomiyama. \textit{Multipliers of $C^*$-algebras}. Journal of Functional Analysis Volume 13, Issue 3, July 1973, Pages 277-301, 1973. 	
	
\bibitem{akerman_gelfand} C.A. Akemann,{\it A Gelfand representation theory for $C*$-algebras}. PACIFIC JOURNAL OF MATHEMATICS
Vol. 39, No. 1, 1971/
	
	
	
		\bibitem{akerman_stone} C.A. Akemann, {\it The general Stone–Weierstrass problem}, J. Funct. Anal. 4 (1969) 277–294, 1969. 
	
		\bibitem{arveson:c_alg_invt} W. Arveson. {\it An Invitation to $C^*$-Algebras}, Springer-Verlag. ISBN 0-387-90176-0, 1981.
		
		
	
\bibitem{atiyah:kt}\textit{$K$-theory}.
By Michael Atiyah
W. A. BENJAMIN, INC. New York, Amsterdam 1967
i for these note is was partially supported by NSF Grant GP-1217. 1967.


\bibitem{bass} H. Bass. {\it Algebraic K-theory.} W.A. Benjamin, Inc. 1968. 

		
	

\bibitem{blackadar:ko} B. Blackadar. {\it K-theory for Operator Algebras}, Second edition. Cambridge University Press. 1998.

	
\bibitem{beyond_sober_duality} Marcello M. Bonsanguea, Bart Jacobs, Joost N. Kok	\textit{Duality beyond sober spaces: Topological spaces and 	observation frames}. Theoretical Computer Science 151 ( 1995) 79-124, 1995. 

\bibitem{cuntz_meyer_ros:bivariant} Joachim Cuntz, Ralf Meyer, Jonathan M. Rosenberg.
\textit{Topological and Bivariant K-Theory}. (Oberwolfach Seminars, 36, Band 36) 2010.

		
\bibitem{bourbaki_sp:gt} N. Bourbaki, {\it Elements of Mathematics. General Topology}, Part 1. \newline HERMANN, \'{E}DITEURS DES SCIENCES ET DAS ARTS \newline 115 Boulevard Saint-Germain. Paris \newline ADDISON-WESLEY PUBLISHING COMPANY. \newline Reading, Massachusets - Palo Ito - London - Don Mills, Ontario \newline A translation of \newline \'{E}L\'{E}MENTS DE MATH\'{E}MATIQUE, TOPOLOGIE G\'{E}N\'{E}RALE, \newline originally published in French by Hermann, Paris. 1966.


	
\bibitem{bredon:topology_geometry} Bredon, Glen E. \textit{Topology and geometry} (Graduate texts in mathematics) Corr. 2nd print Edition, 1993.

		\bibitem{bredon:sheaf} Bredon, Glen E. (1997), \textit{Sheaf theory}. Graduate Texts in Mathematics, 170 (2nd ed.), Berlin, New York: Springer-Verlag.  ISBN 978-0-387-94905-5, MR 1481706 (oriented towards conventional topological applications), 1997.
		
	\bibitem{brown:stable} Lawrence G. Brown, Philip Green, and Marc A. Rieffel. \textit{Stable isomorphism and strong Morita equivalence of $C^*$-algebras}. Pacific J. Math., Volume 71, Number 2 (1977), 349-363. 1977.


\bibitem{bryl:loop}Jean-Luc Brylinski. \textit{Loop Spaces, Characteristic Classes and Geometric Quantization}. Springer Science \& Business Media, Nov 15, 2007.

\bibitem{nt:ky}Chakraborty, S., Echterhoff, S., Kranz, J. et al. {\it $K$-theory of noncommutative Bernoulli shifts}. Math. Ann. 388, 2671–2703 (2024). https://doi.org/10.1007/s00208-023-02587-w, 2024.

\bibitem{connes:foli_survey} A. Connes. \textit{A survey of foliations and operator algebras}. Operator algebras and applications, Part 1, pp. 521-628, Proc. Sympos. Pure Math., 38, Amer. Math. Soc, Providence, R.I., 1982; MR 84m:58140. 1982. 

\bibitem{candel:foliI}Alberto Candel, Lawrence Conlon. \textit{Foliations I}. Graduate Studies in Mathematics, American Mathematical Society (1999), 1999.

\bibitem{cra_moe:nhaus} M. Crainic and I. Moerdijk. \textit{A remark on sheaf  theory for non-Hausdorff  manifolds}. Tech. Report 1119, Utrecht University, 1999.

\bibitem{candel:foliII}Alberto Candel, Lawrence Conlon. \textit{Foliations II}. American Mathematical Society; 1 edition (April 1 2003), 2003.

\bibitem{dauns:w_envelope} John Dauns. {\it Enveloping $W^*$-algebras}. 
Rocky Mountain J. Math. 8(4): 589-626 (Fall 1978). DOI: 10.1216/RMJ-1978-8-4-589, 1978.

\bibitem{doering:neo} Andreas Doering {\it Topos theory and ‘neo-realist’
quantum theory}. arXiv:0712.4003 [quant-ph]. 2007.

\bibitem{dorin:cuntzepr}Dorin Ervin Dutkay, John Haussermann, Palle E.T.Jorgensen. {\it Atomic representations of Cuntz algebras}. J. Math. Anal. Appl. 421(2015)215–243, 2015.


\bibitem{cuntz_meyeros:bivariant} Joachim Cuntz, Ralf Meyer, Jonathan M. Rosenberg.
\textit{Topological and Bivariant K-Theory}. (Oberwolfach Seminars, 36, Band 36) 2010.

	\bibitem{motivic} B.I. Dundas,	M. Levine,	P.A. Østvær,	O. R\"{o}ndigs,	V. Voevodsky.
{\it Motivic	Homotopy Theory}.	Lectures at a Summer School	in Nordfjordeid, Norway, Springer, August 2002.


\bibitem{eust} Samuel Eilenberg, Norman E. Steenrod. \textit{Foundations of Algebraic Topology}. Published by Princeton University Press 1952. 

\bibitem{engelking:general_topology} Ryszard Engelking. \textit{General topology}, PWN, Warsaw. 1977.


\bibitem{fell:operator_fields}J. M. G. Fell. \textit{The structure of algebras of operator fields}. Acta Math. Volume 106, Number 3-4 (1961), 233-280. 1961.


\bibitem{godement:sheaf} Roger Godement, \textit{Topologie Algébrique et Théorie des Faisceaux}. Actualités Sci. Ind. No. 1252. Publ. Math. Univ. Strasbourg. No. 13 Hermann, Paris. 1958.

\bibitem{goldblatt:topoi} Robert Goldblatt. \textit{Topoi: The Categorial Analysis of Logic}. Revised edition of XLVII 445. Studies in logic and the foundations of mathematics, vol. 98. North-Holland, Amsterdam, New York, and Oxford, 1984, xvi + 551 pp. 1984.

	
\bibitem{connes:ncg94} Alain Connes. {\it Noncommutative Geometry}, Academic Press, San Diego, CA,  661 p., ISBN 0-12-185860-X, 1994.

\bibitem{engelking:general_topology} Ryszard Engelking. \textit{General topology}, PWN, Warsaw. 1977.






	\bibitem{grothendieck:sheaf_noabelian} Alexander Grothendieck. \textit{A General Theory of Fibre Spaces With Structure Sheaf}, University of Kansas, Report No. 4 (1955, 1958), 1958.

\bibitem{fomenko_fucs_hot} Anatoly Fomenko, Dmitry Fuchs, {\it Homotopical Topology}, Graduate Texts in Mathematics 273, Springer (2016).

\bibitem{gelfand_manin} Sergei I. Gelfand, Yuri I. Manin. \textit{Methods of Homological Algebra},  Springer Monographs in Mathematics (SMM), 1996. 

	\bibitem{hartshorne:ag} Robin Hartshorne. {\it Algebraic Geometry.} Graduate Texts in Mathematics, Volume 52, 1977.
	
	\bibitem{hatcher:at}Allen Hatcher, \textit{Algebraic topology}. Cambridge University Presses, Cambridge, 2002. 
	 	
	 	\bibitem{hatcher:kt}Allen Hatcher. \textit{Vector Bundles \& $K$-Theory}. Version 2.2, November 2017. Copyright c 2003.
	 	Paper or electronic copies for noncommercial use may be made freely without explicit permission from the author.
	 	All other rights reserved.  2017.
	 	
	
	
	 \bibitem{topos_quant} Chris Heunen, Nicolaas P. Landsman, Bas Spitters, Sander Wolters, {\it A topos for algebraic quantum theory.}  	arXiv:0709.4364 [quant-ph], 2017.
	 
	 \bibitem{gelfand_spec} Chris Heunen, Nicolaas P. Landsman, Bas Spitters, Sander Wolters, {\it The Gelfand spectrum of a noncommutative $C^*$-algebra: a topos-theoretic approach.}  arXiv:1010.2050 [math-ph], 2024.
	 
		 
	 
	\bibitem{hilsum_scandalis:stab}Hilsum, Michel; Skandalis, Georges. \textit{Stabilit\'{e} des {$C\sp{\ast} $}-alg\`ebres de feuilletages}. Annales de l'Institut Fourier, Volume 33 (1983) no. 3, p. 201-208, 1983.  
	\bibitem{johnstone:stone_spaces} Peter Johnstone. \textit{Stone Spaces} Cambridge Studies in Advanced Mathematics 3, Cambridge University Press (1982, 1986).
	
	\bibitem{johnstone:topos}Peter Johnstone. \textit{Topos Theory}, L. M. S. Monographs no. 10, Academic Press 1977.

	

\bibitem{isely:akt}Olivier Isely {\it Algebraic K-Theory} Semester project Chaire of the Prof. Kathryn Hess Directed by Sverre Lunøe-Nielsen winter semester 2005-2006, 2006.

\bibitem{kelley:gt} John L. Kelley. {\it General Topology}. Springer, 1975.

\bibitem{kasch:mr} \textit{Modules and Rings}. A translation of \textit{Moduln und Ringe}. German text by F. Kasch, Ludwig-Maximilian University, Munich, Germany, Translation and editing by D. A. R. WALLACE University of Stirling, Stirling, Scotland 1982 ACADEMIC PRESS. A Subsidiary of Harcourt Brace Jovanovich, Publishers LONDON, NEW YORK, PARIS, SAN DIEGO, SAN FRANCISCO, S\~{A}O PAULO, SYDNEY, TOKYO, TORONTO.  1982.
	
\bibitem{karoubi:k} M. Karoubi. {\it K-theory, An Introduction.} Springer-Verlag. 1978. 
	
	
	\bibitem{cat:work} Saunders Mac Lane.	\textit{Categories for the Working Mathematician}.  Springer Graduate Texts in Mathematics series, with a second edition following in 1998. 1998.
	
	
			\bibitem{matro:hcm} Manuilov V.M., Troitsky E.V. \textit{Hilbert $C^*$-modules}. 
	Translations of Mathematical Monographs, vol. 226, 2005.
	
\bibitem{kochenspecker67}Simon Kochen and Ernst Specker. {\it The problem of hidden variables in quantum mechanics}.
Journal of Mathematics and Mechanics, 17:59–87, 1967.

\bibitem{topos:intro} Saunders Mac Lane, Ieke Moerdijk
\textit{Sheaves in Geometry and Logic. A First Introduction to Topos Theory}. Springer-Verlag - Berlin : New York. 1994.

\bibitem{milne:lec} J.S. Milne. \textit{Lectures on \'Etale Cohomology}. Version 2.21 March 22, 2013.

	
	
\bibitem{munkres:topology} James R. Munkres. {\it Topology.} Prentice Hall, Incorporated, 2000.
	
	\bibitem{murphy}G.J. Murphy. {\it $C^*$-Algebras and Operator Theory.} Academic Press 1990.
		
		\bibitem{neshv:non_haudorff} Sergey Neshveyev and Gaute Schwartz. \textit{Non-Hausdorff  \'etale groupoids and $C^*$-algebras of left  cancellative monoids}. M\"unster J. of Math. 16 (2023), 147 175, 2023.
	
		\bibitem{pedersen:ca_aut}Gert  Kjærgård Pedersen. {\it $C^*$-algebras and their automorphism groups}. London ; New York : Academic Press, 1979.
		
			\bibitem{joost:borification} Joost Nuiten. \textit{Bohrification of local nets of observables}. arXiv:1109.1397, 2011.
		
		
\bibitem{ped_semi} Gert K Pedersen.  {\it Applications of weak$*$ semicontinuity in $C^*$-algebra theory}, Duke Math.
J., 39 (1972), 431-450, 1972.

	
	
	\bibitem{rae:ctr_morita} Iain Raeburn, Dana P. Williams. \textit{Morita Equivalence and Continuous-trace $C^*$-algebras}. American Mathematical Soc., 1998.
	
\bibitem{p-v:irr} M. V. Pimsner and D. Voiculescu, {\it Imbedding the irrational rotation $C^*$-algebra into an AF-algebra}, J. Oper. Theory 4 (1980), 201–210.  1980.

\bibitem{renault:gropoid_ca} Jean Renault, \emph{A groupoid approach to {$C\sp{\ast} $}-algebras}, Lecture Notes in Mathematics, vol. 793, Springer, Berlin, 1980. 

\bibitem{rudin:fa}Walter Rudin. \textit{Functional Analysis}, Second Edition, McGraw-Hill, Inc. New York St. Louis San Francisco Auckland Bogota Caracas Hamburg Lisbon London Madrid Mexico Milan Montreal New Delhi Paris San Juan Sao Paulo Singapore Sydney Tokyo Toronto, 1991.

\bibitem{Rieffel:irrat}M. A. Rieffel, {\it $C^*$-algebras associated with irrational rotations”}, Pac. J. Math. 93 (1981), 415–429. 1981.

\bibitem{Rieffel:stable_irrat}
Marc~A. Rieffel, \emph{Non-stable $K$-theory and noncommutative tori}.
Contemporary Mathematics, Volume 62, 1987.

\bibitem{Rieffel:74a}
Marc~A. Rieffel, \emph{Induced representations of {C}{$^\ast$}-algebras},
Advances in Mathematics \textbf{13} (1974), 176--257. 1974.


\bibitem{rieffel:c_quant} Marc A. Rieffel.
{\it Quantization and $C^*$-algebras}, Contemporary   Mathematics,  167 (1994), 67-97. 1994.

\bibitem{rieffel_morita} Marc A. Reiffel, {\it Morita equivalence for $C^*$-algebras and $W^*$-algebras }, Journal of Pure and Applied Algebra 5 (1974), 51-96. 1974.

\bibitem{rosen:algebraicK} J. Rosenberg. {\it The algebraic K-theory of operator algebras}, $K$-Theory 12:1 (1997), (75-99) 1997.

\bibitem{rordam:kc} M. Rørdam, Flemming Larsen, N. Laustsen. {\it An Introduction to $K$-Theory for $C^*$-Algebras}. Cambridge University Press, Jul 20, 2000 - Mathematics, 2000. 

\bibitem{royden:fa} H.L. Royden. {\it Function algebras}.
Technical report. No 21. January 21, 1963. Prepared under grant DA AR0(D)31-124-G80. For U.S. Army Research Office. Applied Mathematics and statistics laboratories. Stanford University. Stanford, California. 1963.
	


\bibitem{rudin:pa} Rudin, Walter. \textit{ Principles of mathematical analysis}. (3rd. ed.), McGraw-Hill, ISBN 978-0-07-054235-8. 1976.

	\bibitem{schwieger:nt_cov}
Kay Schwieger, Stefan Wagner. \textit{Noncommutative Coverings of Quantum Tori}. 	Math. Scand. 26 (2020), 99–116. 2020.

	
		\bibitem{spanier:at} E.H. Spanier. {\it Algebraic Topology.} McGraw-Hill. New York. 1966.
		
\bibitem{counter_topology}	Lynn Arthur Steen,
J. Arthur Seebach. \textit{Counterexamples in topology}, Springer-Verlag,
1970.			
	
	
\bibitem{suslin:exc} A. Suslin and M. Wodzicki, {\it Excision in algebraic $K$-theory}, Ann. of Math. (2) 136 (1992), no. 1, 51–122. 1992.	

\bibitem{switzer:at} Switzer R M, {\it Algebraic Topology - Homotopy and Homology}, Springer. 2002.
\bibitem{takeda:inductive} Zir\^{o} Takeda. \textit{Inductive limit and infinite direct product of operator algebras.} Tohoku Math. J. (2) 	Volume 7, Number 1-2 (1955), 67-86. 1955.
\bibitem{a_thomas_kt_hom} Alan Thomas. {\it A Relation Between K-Theory and Cohomology}.Transactions of the American Mathematical Society Vol. 193 (Jun., 1974), pp. 133-142 (10 pages), 1972



\bibitem{varilly:nc_intro} J. Varilly. \textit{An Introduction to Noncommutative Geometry}. arXiv:physics/9709045v1. 1997.
	

\bibitem{wagner:pb} S. Wagner. \textit{On noncommutative principal bundles with finite Abelian structure group.} J. Noncommut. Geom., 8(4):987.  2014.
		
		\bibitem{torsten:sheaves} Torsten Wedhorn. \textit{Manifolds, sheaves, and cohomology}. (Springer Studium Mathematik - Master) (Englisch) Taschenbuch . August 2016.  	
		
\bibitem{wegge_olsen} N. E. Wegge-Olsen. \textit{$K$-Theory
	 and $C^*$-Algebras: A Friendly Approach.} Oxford University Press,
Oxford, England, 1993. 

\bibitem{zhi:cov_group} Zhi-Ming Luo. \textit{Covering groupoids} arXiv:math/0412230, 2004.

\end{thebibliography}
	\end{document}